\documentclass[11pt,leqno]{article}
\usepackage[all]{xy}
\usepackage{epsfig}

\usepackage{amsfonts,amssymb,amsthm,amsmath,amscd}
\usepackage[mathscr]{eucal}
\usepackage{mathrsfs}
\usepackage{hyperref}
\usepackage{url}
\usepackage{leftidx}
\usepackage{tikz}
\usetikzlibrary{calc, intersections}
\usepackage{marginnote}
\usepackage{graphics}

      \makeatletter

\usepackage{titlesec}

\titleformat{\subsection}[runin]
{\normalfont\large\bfseries}{\thesubsection}{1em}{}

\titleformat{\subsubsection}[runin]
{\normalfont\large\bfseries}{\thesubsubsection}{1em}{}

\makeatletter
\newsavebox{\@brx}
\newcommand{\llangle}[1][]{\savebox{\@brx}{\(\m@th{#1\langle}\)}%
  \mathopen{\copy\@brx\kern-0.5\wd\@brx\usebox{\@brx}}}
\newcommand{\rrangle}[1][]{\savebox{\@brx}{\(\m@th{#1\rangle}\)}%
  \mathclose{\copy\@brx\kern-0.5\wd\@brx\usebox{\@brx}}}
\makeatother

\newcommand{\PLH}{{\mkern-2mu\times\mkern-2mu}}

\theoremstyle{plain}
\newtheorem{thm}[subsubsection]{Theorem}
\newtheorem{thm*}{Theorem}
\newtheorem{cor}[subsubsection]{Corollary}
\newtheorem{lem}[subsubsection]{Lemma}
\newtheorem{prop}[subsubsection]{Proposition}
\newtheorem{conj}[subsubsection]{Conjecture}

\theoremstyle{definition}
\newtheorem{defn}[subsubsection]{Definition}
\newtheorem{notation}[subsubsection]{Notation}

\theoremstyle{remark}
\newtheorem{rem}[subsubsection]{Remark}

\numberwithin{equation}{subsubsection}

\hypersetup{pdfauthor={Name}}

\newcommand{\N}{\mathbb N}
\newcommand{\Z}{\mathbb Z}
\newcommand{\Q}{\mathbb Q}
\newcommand{\R}{\mathbb R}
\newcommand{\C}{\mathbb C}
\newcommand{\A}{\mathbb A}

\newcommand{\Qb}{\overline{\mathbb Q}}

\newcommand{\tpi}{2\pi i}
\newcommand{\alg}{\mathrm{alg}} 

\newcommand{\Hom}{\mathrm{Hom}}
\newcommand{\End}{\mathrm{End}}
\newcommand{\Aut}{\mathrm{Aut}}
\newcommand{\Ker}{\mathrm{Ker}}
 
\newcommand{\Tr}{\mathrm{Tr}} 
\newcommand{\tr}{\mathrm{tr}} 
\newcommand{\Gr}{\mathrm{Gr}} %Graded
\newcommand{\gr}{\mathrm{gr}} %graded
\newcommand{\Alt}{\mathrm{Alt}}

\newcommand{\ra}{\rightarrow}

\newcommand{\lra}{\longrightarrow}

\newcommand{\isom}{\stackrel{\sim}{\rightarrow}} 
\newcommand{\lisom}{\stackrel{\sim}{\longrightarrow}}
\newcommand{\hra}{\hookrightarrow}

\newcommand{\leqs}{\leqslant}

\newcommand{\Spec}{\mathrm{Spec}}
 
\newcommand{\dR}{\mathrm{dR}}
\newcommand{\Ext}{\mathrm{Ext}}

\newcommand{\GL}{\mathrm{GL}}
\newcommand{\SL}{\mathrm{SL}}
\newcommand{\PGL}{\mathrm{PGL}}
\newcommand{\PSL}{\mathrm{PSL}}

\newcommand{\SU}{\mathrm{SU}}
\newcommand{\su}{\mathfrak{su}}

\newcommand{\SO}{\mathrm{SO}}
\newcommand{\so}{\mathfrak{so}}
\newcommand{\Spin}{\mathrm{Spin}}
\newcommand{\Gm}{\mathbb{G}_{\mathrm{m}}}

\newcommand{\univ}{\mathrm{univ}}

\newcommand{\mcA}{\mathcal{A}}
\newcommand{\mcB}{\mathcal{B}}
\newcommand{\mcC}{\mathcal{C}}

\newcommand{\mcF}{\mathcal{F}} 
\newcommand{\mcG}{\mathcal{G}}
\newcommand{\mcH}{\mathcal{H}}

\newcommand{\mcL}{\mathcal{L}}
\newcommand{\mcM}{\mathcal{M}} 
\newcommand{\mcO}{\mathcal{O}}
\newcommand{\mcP}{\mathcal{P}}
\newcommand{\mcR}{\mathcal{R}}

\newcommand{\mcT}{\mathcal{T}} 
 
\newcommand{\mcV}{\mathcal{V}}

\newcommand{\mrM}{\mathrm{M}}

\newcommand{\msF}{\mathscr{F}}
\newcommand{\msZ}{\mathscr{Z}}

\newcommand{\mbE}{\mathbb{E}}
\newcommand{\mbF}{\mathbb{F}}
\newcommand{\mbH}{\mathbb{H}}
\newcommand{\mbS}{\mathbb{S}}
\newcommand{\mbV}{\mathbb{V}}
\newcommand{\mbW}{\mathbb{W}}

\newcommand{\mfg}{\mathfrak{g}}
\newcommand{\mfl}{\mathfrak{l}} %added in this paper
\newcommand{\mfm}{\mathfrak{m}} 
\newcommand{\mfu}{\mathfrak{u}}
\newcommand{\mfv}{\mathfrak{v}}
\newcommand{\mfX}{\mathfrak{X}}

\DeclareMathSymbol{\mhyphen}{\mathord}{AMSa}{"39}
\newcommand{\sslash}{\mathbin{/\mkern-6mu/}}

\newcommand{\pd}{\partial/\partial} %partial derivative
\newcommand{\Vol}{\mathrm{Vol}}
\newcommand{\CS}{\mathrm{CS}}
\newcommand{\Li}{\mathrm{Li}}
\newcommand{\MT}{\mathrm{MT}} %cat. of mixed Tate motives
\newcommand{\QMH}{\Q\text{-}\mathrm{MH}} %cat. of Q-mixed Hodge structures
\newcommand{\QHT}{\Q\text{-}\mathrm{HT}} %Q-linear cat. of Hodge-Tate structure
\newcommand{\DM}{\mathrm{DM}} %derived category of mixed motives
\newcommand{\DMT}{\mathrm{DMT}} %derived category of mixed Tate motives
\newcommand{\Bo}{\mathrm{Bo}} %Borel
\newcommand{\mot}{\mathrm{mot}} %motivic
\newcommand{\bdM}{\partial M} %boundary of M
\newcommand{\bmbH}{\overline{\mathbb{H}}} 
\newcommand{\bdmbH}{\partial \mathbb{H}} 
\newcommand{\CP}{\mathbb{CP}} %complex projective space

\newcommand{\MC}{\mathrm{M}.\mathrm{C}.} %Mauer Cartan
\newcommand{\red}{\mathrm{red}}

\newcommand{\mcHb}{\mathcal{H}_{\bullet}} %Hodge-Tate algebra (calligraphic)
\newcommand{\mcAb}{\mathcal{A}_{\bullet}} %Mixed Tate algebra  (calligraphic)
\newcommand{\mcMn}{\mathcal{M}^{(n)}} %period matrix of poly HTS of order n  (calligraphic)
\newcommand{\PHT}{\mathcal{P}} %period matrix of polyl HTS  (calligraphic)
\newcommand{\mcVb}{\overline{\mathcal{V}}} 
\newcommand{\fp}{2\pi i}
\newcommand{\Xb}{\overline{X}}

\newcommand{\tv}{\frac{\partial }{\partial z}}
\newcommand{\orav}{\overrightarrow{v}}
\newcommand{\nub}{\bar{\nu}}
\newcommand{\grR}{\mathrm{gr}\mhyphen R} %graded R-module
\newcommand{\grk}{\mathrm{gr}\mhyphen k} %graded k-spaces
\newcommand{\real}{\mathrm{real}} %Hodge realization
\newcommand{\Lsl}{\mathfrak{sl}} %Lie algebra SL

\newcommand{\WgrbH}{\gr^W_{\bullet}H} 
\newcommand{\Gon}{\mathrm{Gon}}
\newcommand{\mcMG}{\mathcal{M}_{\mathrm{Gon}}}
\newcommand{\PhiG}{\Phi_{\mathrm{Gon}}}
\newcommand{\mcPt}{\mathcal{P}^{(2)}}
\newcommand{\In}{\mathrm{I}_n}
\newcommand{\Ih}{\mathrm{I}_h}

\newcommand{\Onen}{1_{n}}
\newcommand{\Oneh}{1_{h}}
\newcommand{\KK}{\mathbf{K}\!\mathbf{K}}
\newcommand{\MoTe}{\mathbf{M}\!\mathbf{T}}

\begin{document}

\title{Chern-Simons invariants of hyperbolic three-manifolds, mixed Tate motives, and motivic path torsor of augmented character varieties}
\author{Dong Uk Lee}
\date{}
\maketitle

\abstract{For any complete hyperbolic three-manifold of finite volume, we construct a mixed Tate motive defined over the invariant trace field whose image under Beilinson regulator equals the $\PSL_2(\C)$-Chern-Simons invariant, thus equals $-\sqrt{-1}$ times the complex volume of the manifold. Further, we show that when $M$ has single torus boundary, under some assumption on asymptotic behaviour of the Chern-Simons invariant near an ideal point, its Hodge realization is a quotient of the mixed Hodge structure on the path torsor of the smooth locus of the canonical curve component of the augmented character variety of the three-manifold between a geometric point (giving the complete hyperbolic structure) and some tangential base point at an ideal point whose existence is asserted by the assumption. We explain its motivic implication.
In the appendix, we verify some cases of the assumption. The theory developed here is parallel to the motivic theory of polylogarithms.}

\tableofcontents

%%%%%%%%%%%%%%%%%%%%%%%%%%%%%%%%%%%%%%%%

%%%%%%%%%%%%%%%%%%%%%%%%%%%%%%%%%%%%%%%%
\section{Introduction}

%\subsection{}

One way of finding the volume of an oriented Riemannian manifold of finite volume is to use, if it exists, a geodesic triangulation, i.e. a triangulation whose constituent simplexes have geodesic faces. In hyperbolic geometry, the volume of any geodesic simplex has algebro-geometric expression as an integral of a rational differential form over an algebraically defined domain. Indeed, in the Klein model $K^n$ of the $m$-dimensional hyperbolic space $\mbH^m$: 
\begin{equation}  \label{eq:Klein_model}
K_m=\{(x_1,\cdots,x_{m},1)\in \R^{m+1}\ :\ Q_m(x_1,\cdots,x_m):=x_1^2+\cdots+x_m^2-1<0\},
\end{equation}
geodesic simplexes are the same as ``linear'' simplexes inside the ball $Q_m(x_1,\cdots,x_m)\leq0$, where linearity is taken with respect to the affine structure of $\A^{m+1}_{\R}$, so that every geodesic simplex is determined by a collection of hyperplanes $M=(M_1,\cdots,M_{m+1})$. It is known that the volume of such simplex is given by the integral $\int_{\Delta_M}\omega_Q$, where $\Delta_M$ is a (real) $m$-dimensional simplex in $\CP^m$ determined by $M$ and $\omega_Q$ is a certain rational differential $m$-form on $\CP^m$ defined over $\Q$ (up to a constant multiple in $\Qb$) with poles on the projecctivization of the quadratic $Q_m=0$ (cf. \cite{Goncharov99}). Therefore, when each $M_i$ is defined over the algebraic numbers $\Qb\subset\C$, this integral is a period, in the sense of Kontsevich-Zagier \cite{KontsevichZagier01} (i.e. integral of a rational differential form over a domain defined by algebraic equations or inequalities, all having algebraic numbers as coefficients), and so we see that the volume of any complete hyperbolic manifold of finite volume, as it is known to have a geodesic triangulation defined over $\Qb$, is a sum of periods, or, more suggestively, is an element of the ring of periods. 

Now, assume that the dimension $m=2n-1$ is odd. Then, Goncharov \cite{Goncharov99} showed that the volumes of complete hyperbolic manifolds of finite volume as well as hyperbolic geodesic simplexes are all certain \emph{canonically defined} periods of some special kind of mixed Hodge structures (which are in fact motivic). First, he showed that the volume of a hyperbolic geodesic simplex is the \textit{real period}, a novel notion which he introduced in \textit{ibid.}, of the mixed $\Q$-Hodge structure
\begin{equation}  \label{eq:rel.coh.}
H^m(\CP^m\backslash Q_m, \cup M_i\backslash Q_m,\Q)
\end{equation}
 (the relative cohomology group of a pair of algebraic variety and closed subvariety), provided with $\Delta_M$ (a relative cycle in $H_{m}(\CP^m,M_1\cup\cdots\cup M_m)$) and $\omega_Q$ (an element of $H^m(\CP^m\backslash Q_m,\Q)$) (cf. \cite[$\S$1.5]{Goncharov99}).
The key point is that this mixed $\Q$-Hodge structure $H$ is a \emph{Hodge-Tate} structure, i.e. its associated weight graded's $\gr^W_{l}H$ satisfy that for all $k\in\Z$, $\gr^W_{2k+1}H=0$ and $\gr^W_{2k}H\simeq \Q(-k)^{\oplus d_k}$ for some $d_k\in\Z_{\geq0}$, and that $\Delta_M$ and $\omega_Q$ form a \emph{framing} of $H$, i.e. are elements of $(\gr^W_0H)^{\vee}$ (dual space) and $\gr^W_{2n}H$ in a natural manner; such triple $(H,\Delta_M,\omega_Q)$ is called ($n$-)\emph{framed Hodge-Tate structure}.

Let $\QHT$ denote the abelian category of $\Q$-Hodge-Tate structures and for a number field $F\subset\Qb$, $\MT(F)$ the abelian category of mixed Tate motives over $F$ \cite{Levine93}, and $R^{\Bo}:K_{2n-1}(\Qb)\ra K_{2n-1}(\C) \ra \R$ the Borel regulator map from higher algebraic K-theory to $\R$ (cf. \cite{BurgosGil02}). Via the canonical isomorphism
\[ \Ext^1_{\MT(\Qb)}(\Q(0),\Q(n))=K_{2n-1}(\Qb)\otimes\Q \]
(cf. \cite{Levine98}), where for an abelian category $\mcC$, $\Ext^1_{\mcC}(-,-)$ is the Yoneda extension group,
$R^{\Bo}$ is identified with \emph{twice}  the ``real/imaginary part'' of the canonical period map:
\begin{equation} \label{eq:Hodge-realization}
\Ext^1_{\MT(\Qb)}(\Q(0),\Q(n)) \stackrel{\real^H}{\lra} \Ext^1_{\QHT}(\Q(0),\Q(n))\cong \C/\Q(n)\ \stackrel{p_n}{\lra}\ \R(n-1)\cong \R \ \stackrel{\times2}{\ra} \ \R,
\end{equation}
where $\real^H$ is the Hodge realization functor, $p_n$ is the projection onto the real/imaginary part (for a subring $R\subset\C$, $R(n):=(2\pi i)^nR\subset\C$; $\C=\R(n)\oplus \R(n-1)$), and $\R(n-1)\isom \R$ is the map $i^{n-1}r\mapsto r$. Using the above interpretation of the volume of geodesic simplex as a real period, Goncharov showed the following remarkable fact:

\begin{thm*} \cite[Thm.1.1]{Goncharov99} \label{thm:Goncharov99}
With every hyperbolic manifold of \emph{odd} dimension $m=2n-1$, there exists a mixed Tate motive defined over $\Qb\subset\C$ which is a simple extension of $\Q(0)$ by $\Q(n)$ in the abelian category $\MT(\Qb)$ of mixed Tate motives over $\Qb$ and whose image under the Borel regulator equals the volume.
\end{thm*}

This highly nontrivial result is a much more precise statement than that the volume is a sum of some periods, and suggests the possibility of understanding invariants of odd-dimensional hyperbolic manifolds in terms of mixed Tate motives over $\Qb$ (equiv. of elements of algebraic K-groups of $\Qb$) and their periods.

Now, we specialize to the dimension $m=3$. In this case, there is another important differential-geometric invariant, that is, the (metric) Chern-Simons invariant. For a \emph{closed} (i.e. compact without boundary) oriented Riemannian $3$-manifold $M$, the metric \emph{Chern-Simons invariant} $cs(M)$ was first introduced by Chern and Simons \cite{ChernSimons74} as the integral of a certain real $C^{\infty}$ $3$-form $Q_M$ (\emph{Chern-Simons $3$-form}) on the oriented orthonormal frame bundle $F(M)$ along a frame field $s$ (i.e. a section $s:M\ra F(M)$), more explicitly:
\begin{equation} \label{eq:metric_CS-inv_closed}
cs(M) := \int_{s(M)}Q_M := -\frac{1}{8\pi^2} \int_{s(M)}  \Tr(\omega\wedge d\omega +\frac{2}{3}\omega\wedge \omega\wedge \wedge \omega),
\end{equation}
where $\omega\in \mathcal{A}^1(F(M),\mathfrak{so}(3))$ is the Levi-Civita connection on $F(M)$.
For $3$-manifolds with non-empty boundary, Meyerhoff in his thesis \cite{Meyerhoff86} extended the definition of metric Chern-Simons invariant to complete hyperbolic three-manifolds $M$ having cusps: 
\begin{equation} \label{eq:metric_CS-inv_cusped}
cs(M):=\int_{s(M-L)} Q_M - \frac{1}{2\pi}\sum_{K\in L} \tau(K).
\end{equation}
Here, $L$ is some link in $M$ such that there exists a frame field $s$ defined on the complement $M- L$ which has certain special singularity along $L$ and also linearity near cusps. For each component $K$ of $L$, $\tau(K)$ is the \emph{complex torsion} of that curve. 
Both integrals, as being valued in $\R/2\Z$ in the closed case or in $\R/\Z$ in the cusped case, are known to be well-defined independent of the choices made (of $s$, $L$), that is $cs(M)$ is an invariant of $M$. 

Now, one can ask the same question whether this invariant is also motivic, especially is another canonically defined period of a mixed Tate motive, preferably the motive constructed by Goncharov. One may attempt to adapt Goncharov's original argument for volume computation which is to add local contributions over constituent geodesic simplexes. For this to work, it is necessary that the local contribution is some canonically defined period of some (mixed Tate) Hodge structure which is defined by the geodesic simplex only. But one faces difficulties. First, in the closed case, $cs(M)$ is the integral over $M$ of a $3$-form $s^{\ast}Q_M$ on $M$ which itself depends on the choice of the section $s$, and it is not even clear in the first place whether there exists $s$ such that the restriction of $s^{\ast}Q_M$ to every constituent geodesic simplex becomes a \emph{rational} differential form. In the cusped case, the local contribution of each geodesic simplex to (\ref{eq:metric_CS-inv_cusped}) is not even an integral of some $3$-form over the simplex, much less depends on $L$ as well as $s$. So a naive guess at answer is ``unlikely".

On the other hand, according to Thurston, the Riemannian volume $\Vol(M)$ and the metric Chern-Simons invariant $cs(M)$ should be considered together, as a single $\C/\sqrt{-1}\pi^2\Z$-valued invariant: \[ \Vol_{\C}(M)=\Vol(M)+\sqrt{-1} \pi^2 cs(M), \]
called \emph{complex volume}. It has better properties, such as analyticity on deformation space, and arises naturally in the theory of hyperbolic three-manifolds and especially in quantum topology.
Neumann and Yang \cite{Neumann92} constructed an invariant $\beta(M)$ living in the Bloch group $\mcB(\Qb)$ and Neumann \cite{NeumannYang99} proved that the complex volume of a hyperbolic three-manifold $M$ equals, up to a constant, the value at $\beta(M)$ of the Bloch regulator $\tilde{\rho}(z)$ introduced by Dupont and Sah. 
Also, Kirk and Klassen \cite{KirkKlassen93} defined the $\PSL_2(\C)$-Chern-Simons invariant for any (closed or not) hyperbolic three-manifold and showed that it equals the complex volume, up to a constant (\ref{eq:normalized_PSL_2(C)-CS.inv}).

We recall the Beilinson regulator $K_n(\Qb)\ra K_n(\C)\ra \C/\Q(n)$ (\ref{eq:Beilinson_regulator}) which is identified with the Hodge realization 
$ \Ext^1_{\MT(\Qb)}(\Q(0),\Q(n))\ra \Ext^1_{\QHT}(\Q(0),\Q(n))\cong \C/\Q(n) $ via the canonical isomorphism $\Ext^1_{\MT(\Qb)}(\Q(0),\Q(n))=K_{2n-1}(\Qb)$ . 

Our first main result is the following:

\begin{thm*} (Theorem \ref{thm:Main.Thm1:CS-MTM})
For any hyperbolic three-manifold $X$, there exists a mixed Tate motive $M(X)$ in $\Ext^1_{\MT(k(M))}(\Q(0),\Q(2))$ defined over the invariant trace field $k(M)\subset \Qb$ whose image under the Beilinson regulator equals the (normalized) $\PSL_2(\C)$-Chern-Simons invariant $\CS_{\PSL_2}(X)$.
\end{thm*}

Since the Borel regulator $K_3(\Qb)\ra \R$ is twice the imaginary part of the Beilinson regulator ($\C/\Q(2)=\R/\pi^2\Q\oplus \sqrt{-1}\R$), this result improves upon the Goncharov' result (Theorem \ref{thm:Goncharov99}). 

Our definition of the mixed Tate motive in Theorem A depends on the choice of a representative of the Bloch invariant $\beta(M)$, and at the moment we do not have a proof that this mixed Tate motive is independent of such choice, although it should be so according to some general motivic conjectures (cf. Remark \ref{rem:uniqueness_of_CS-motive}). Nevertheless, in this work, since there is no need for clarifying which one is being used, we call any mixed Tate motive as in Theorem 2, \emph{Chern-Simons mixed Tate motive}.

Now, there arises the natural question whether our Chern-Simons mixed Tate motive is the same, up to a constant multiple, as the Goncharov's mixed Tate motive (if we know that our motive is well-defined). The answer is ``presumably yes'' and ``probably no''. In fact, Goncharov constructed one element for each of the two groups $\Ext^1_{\MT(\Qb)}(\Q(0),\Q(2))=K_{3}(\Qb)_{\Q}$ and raised the question of equality of these two elements. His first homological method constructs an element of $H_3(\SL_2(\Qb),\Q)\cong K_{3}(\Qb)_{\Q}$. We believe that this element should be the same as our mixed Tate motive. On the other hand, his second method produces a certain sum of mixed Tate motives which turns out to be an extension of $\Q(0)$ by $\Q(2)$ (like our mixed Tate motive). But, we observe that because the input for his construction of mixed Tate motive is an element in the scissors congruence group $\mcP(\mbH^3)$, there is no guarantee that his mixed Tate motive detects the Chern-Simons invariant; in contrast, the Bloch group $\mcB(\C)$ which was used for our construction is finer than $\mcP(\mbH^3)$. So there is no guarantee that his (second) mixed Tate motive is the same as our mixed Tate motive or with his (first) K-theory element. We will give more detailed explanation in Subsection \ref{subsec:comp_wy_Goncharov_MTM}.

Our second main result relates this ``Chern-Simons mixed Tate motive" to the mixed Hodge structure defined on the path torsor of the agumented character variety of the hyperbolic manifold between two some specific base points. 

\begin{thm*} (Theorem \ref{thm:CS-VMHS_is_motivic})
Suppose $M$ has a single cusp. Let $X$ be the smooth open part of the geometric (curve) component of the augemented character variety $\tilde{X}(M)$; choose a point $z^0\in X$ giving the complete hyperbolic structure. Assume Conjecture \ref{conj:CS-VMHS_at_ideal_pt}; let $\orav$ be any tangent vector satisfying the condition of this conjecture. 

Then, the Chern-Simons Hodge-Tate structure constructed in Theorem 2 is a quotient of the completed path torsor  $\Q[P_{\orav,z^0}X]^{\wedge}$.
\end{thm*}

This theorem is an analogue of a well-known result about the polylogarithm variation of mixed Hodge structure on the projective line minus three points $\mathbb{P}^1\backslash\{0,1,\infty\}$ \cite[11.3]{Hain94}. In fact, our proof follows the Deligne's original argument faithfully, and our Conjecture \ref{conj:CS-VMHS_at_ideal_pt} is one of the ingredients needed to mimic his argument. The most essential ingredient in the original argument is computation of local monodromies of the local system underlying the polylogarithm VMHS around the two ideal points $0$, $1$ (which is reduced to (known) monodromy computations of polylogarithm functions \cite{Ramakrishnan82}); the statements corresponding to our Conjecture \ref{conj:CS-VMHS_at_ideal_pt} and our Theorem 1 (i.e. \cite{Hain94}, Thm.7.2, Thm.11.3) then follow easily from the local and global monodromy informations, respectively. But, in our case, to obtain necessary informations on these monodromies, we need some non-trivial results in three-dimensional topology (about Culler-Shalen theory and A-polynomials, cf. \cite{CullerShalen83}, \cite{CullerShalen84}, \cite{CCGLS94}).

This result strongly suggests (an observation due to Deligne in the polylogarithm case) that the Chern-Simons mixed Tate motive constructed in Theorem 2 is also a quotient of the mixed motive $A_{\orav,z^0}(X)$ whose dual is the affine algebra of the motivic path torsor with (tangential) base points $\orav$, $z^0$; the motivic path torsor $P^{\mot}_{\orav,z^0}(X)$ is a scheme $\Spec (A_{\orav,z^0}(X)^{\vee})$ in the (hypothetical) Tannakian category of mixed motives, in the sense of Deligne \cite{Deligne89}, having the Betti realization $\real^{Betti}(A_{\orav,z^0}(X))=(\Q[P_{\orav,z^0}X]^{\wedge})^{\vee}$. The main obstacle for turning this expectation into a rigorous statement is that at the moment, for non-rational curves, the motivic path torsor is a hypothetical object (due to lack of motivic $t$-structure and difficulty with tangential base points), although its existence is a standard belief. As a matter of fact, we obtained Theorem 3 before Theorem 2. The motivic implication of Theorem 3 led us to look for a mixed Tate motive with the property in Theorem 2.

We also bring the readers' attention to that in view of that the (augmented) character variety is defined only by the topological fundamental group of the three-manifold, this theorem is a vivid example of Mostow rigidity theorem, which says that the differential-geometric invariants of complete hyperbolic three-manifolds are also homotopy invariants. 

This article is organized as follows. 

Section 2 is devoted to a review of three invariants attached to arbitrary (i.e. closed or not) complete hyperbolic three-manifolds $M$ of finite volume and their relations:
(i) the complex volume $\Vol_{\C}(M)$, (ii) the Bloch invariant $\beta(M)$ and its Bloch regulator value, and (iii) the $\mathrm{PSL}_2(\C)$ Chern-Simons invariant $\CS_{\PSL_2}(M)$. 
The main result is that the Bloch regulator value at the Bloch invariant $\beta(M)$ of $M$ equals half the $\mathrm{PSL}_2(\C)$ Chern-Simons invariant of $M$.

An essential tool in the Goncharov's work \cite{Goncharov99} is the novel notion of \emph{big period} which is defined for framed $\Q$-Hodge-Tate structures and is valued in the ``big'' group $\C\otimes_{\Q}\C$. In fact, for his purpose of constructing a mixed Tate motive giving the Riemannian volume of hyperbolic manifold as Borel regulator value, Goncharov uses only the image $P_2$ of the big period under the map $x\otimes y\mapsto \tpi\cdot x\otimes\exp(\tpi \cdot y)$, which  thus loses some information. The aforementioned real period of a framed Hodge-Tate structure is derived from this simplified big period $P_2$. In Section 3, we give a review of the theory of big period and also introduce a new notion of ``skew-symmetric period''. Our key observation leading to Theorem \ref{thm:Main.Thm1:CS-MTM} is that if for $z\in\mathbb{P}^1(\C)\backslash\{0,1,\infty\}$, we let $\mcP^{(2)}(z)$ denote the polylogarithm framed Hodge-Tate structure of rank $3$ (\ref{eq:polylogarithm_Li_k}), the skew-symmetrization of the bid period (valued in $\C\wedge_{\Q}\C$) of $\mcP^{(2)}(z)$ equals the Bloch regulator value $\tilde{\rho}(z)$, cf. (\ref{eq:ss-period=rho}). 
In the last subsection, we recall the description, due to \cite{BMS87}, \cite{BGSV90}, of the Tannakian category of $\Q$-Hodge-Tate structures in terms of framed $\Q$-Hodge-Tate structures and explain the formula $\Ext_{\QHT}(\Q(0),\Q(n))=\mathrm{Ker}(\nub|_{\mcH_n})$, the right side being the kernel of the reduced comultiplication $\nub$ on the associated fundamental Hopf algebra $\mcH=\oplus \mcH_n$, called ``Hodge-Tate Hopf algebra" (Proposition \ref{prop:coh_red_cobar_cx}).

In Section 4, we start with a brief survey of the theory of mixed Tate motives over number fields, and
prove our first main theorem: Theorem \ref{thm:Main.Thm1:CS-MTM}, constructing (the) Chern-Simons mixed Tate motive as a certain sum of various polylogarithm mixed Tate motives. Its ingredient is a representative of the Bloch invariant $\beta(M)$. We give a comparison of this mixed Tate motive with Goncharov's mixed Tate motive (his second construction). The three sections 2, 3, 4 constitute the first part of this article. 

Our second main theorem, in its statement and proof as well, has a model in the theory of polylogarithm sheaves on the projective line minus three points $\mathbb{P}^1\backslash\{0,1,\infty\}$. For motivation,  we recall this theory in Section 5.

In Section 6, we review the theory of character varieties of three-manifolds and Thurston deformation curves in the single torus boundary case. For our purpose, the more useful objects is a certain double covering of the canonical (curve) component of the character variety of $M$, called \emph{augmented character variety}.
For each hyperbolic three-manifold $M$ with single torus boundary, following a previous construction of Morishita-Terashima \cite{MorishitaTerashima09}, we construct a variation of mixed Hodge structure (which we call Chern-Simons VMHS) over the smooth locus of the canonical component of the augmented character variety. The primary ingredient for this construction is the Chern-Simons invariants of flat connections on the trivial principal $\mathrm{PSL}_2(\C)$-bundle on $M$; here the Chern-Simons invariant is regarded as a section of certain line bundle over the character variety of the boundary torus $\partial M$, the idea going back to Kirk and Klassen \cite{KirkKlassen93}. We related their construction with that of Morishita-Terashima. 

In Section 7, we prove our second main theorem: Theorem \ref{thm:CS-VMHS_is_motivic}. As discussed briefly earlier, we follow the arguments in the polylogarithm theory. But, the main ingredients are provided by three-manifold topology theory.

In appendix, we verify Conjecture \ref{conj:CS-VMHS_at_ideal_pt} for the figure eight knot complement.

%%%%%%%%%%%%%%%%%%%%
%\textbf{Acknowledgement} This work was supported by the National Research Foundation of Korea(NRF) grant funded by the Korea government (NRF-2019R1I1A1A01062321, 2022R1I1A1A01073245, 2021R1A4A3033098). 

%%%%%%%%%%%%%%%%%%%%%%%%%%%%%%%%%%%%%%%%
%%%%%%%%%%%%%%%%%%%%%%%%%%%%%%%%%%%%%%%%
\section{Chern-Simons invariant of hyperbolic three-manifolds and Bloch regulator}

%%%%%%%%%%%%%%%%%%%%%%%%%%%%%%%%%%%%%%%%
\subsection{$\SO(3)$ Chern-Simons invariant of hyperbolic $3$-manifolds}

Here, we give a review of the theory of metric Chern-Simons invariant of hyperbolic three-manifolds, focusing on the cusped case. Our main references are \cite{ChernSimons74}, \cite{Meyerhoff86}, and \cite{Yoshida85}.

For our work, we also need constructions of general Chern-Simons invariants of connections on principal bundles for various Lie groups (the only Lie groups which we will deal with in this work are $\SO(3)$, $\SU(2)$, $\SL_2(\C)$, $\PSL_2(\C)$).

%%%%%%%%%%%%%%%%%%%%
\begin{defn} \label{defn:CSform}
Let $G$ be a (real or complex) Lie group with finitely many components. Suppose fixed an $\mathrm{Ad}$-invariant symmetric bilinear form on the Lie algebra $\mfg$
\[ \langle -, - \rangle\, :\, \mfg \times \mfg \lra \C. \]

(1) Let $P$ be a principal $G$-bundle on a three-manifold $M$.
The \emph{Chern-Simons $3$-form} of a connection $\omega\in \mcA^1(P,\mfg)$ on $P$ is the closed $3$-form
\[ Q(\omega):= \langle \omega\wedge d\omega +\frac{1}{3}\omega\wedge [\omega\wedge \omega] \rangle. \]
(recall that $\langle -, - \rangle:\mfg \times \mfg \ra \C$ (resp. the Lie bracket $[-,-]$) extends to $\mfg$-valued differential forms on any smooth manifold $X$: $\langle -, - \rangle:\mcA^k(X,\mfg)\times \mcA^l(X,\mfg) \ra \mcA^{k+l}(X,\mfg\otimes \mfg) \ra \mcA^{k+l}(X,\C)$ (resp. $[ -, - ]:\mcA^k(X,\mfg)\times \mcA^l(X,\mfg)\ra \mcA^{k+l}(X,\mfg\otimes \mfg) \ra \mcA^{k+l}(X,\mfg)$); $[-\wedge -]$ is another notation for $[-,-]$.)

(2) For an oriented Riemannian three-manifold $M$, the \emph{metric Chern-Simons $3$-form} $Q_M$ is the Chern-Simons $3$-form of the Levi-Civita connection $A_{\mathrm{L.C.}} \in \mcA^1(F(M),\mathfrak{so}_3)$ on the oriented orthonormal frame bundle $F(M)$
\[ Q_M:=Q(\omega_{\mathrm{L.C.}})= -\frac{1}{8\pi^2} \Tr( \omega_{\mathrm{L.C.}}\wedge d\omega_{\mathrm{L.C.}} +\frac{2}{3}\omega_{\mathrm{L.C.}}\wedge \omega_{\mathrm{L.C.}}\wedge \wedge \omega_{\mathrm{L.C.}})\]
for the choice $ \langle -, - \rangle\:=-\frac{1}{8\pi^2} \Tr$, where $\Tr:\so(3)\times \so(3)\ra \R$ is the restriction to $\so(3)$ of the trace function on $M_3(\R)$ (we used the identity: $\omega\wedge \omega=2[\omega\wedge \omega]$).
\end{defn}

The choice of $ \langle -, - \rangle\:=-\frac{1}{8\pi^2} \Tr$ in (2) will be explained in Remark \ref{rem:normalization}.

Now, we specialize to metric Chern-Simons invariant of cusped hyperbolic three-manifolds. 
Let $M$ be an oriented, hyperbolic three-manifold of finite volume with cusps and $F(M)\ra M$ its oriented orthonormal frame bundle (so, a principal $\SO(3)$-bundle). 

We have to choose the frame field on $M$ over which the Chern-Simons $3$-form $Q_M$ is integrated, cf.  (\ref{eq:metric_CS-inv_closed}). One wishes to use a frame field which is ``linear on the cusps''. 

%%%%%%%%%%%%%%%%%%%%
\begin{defn} \cite[$\S$4.1]{Meyerhoff86} \label{defn:frame_linear_on_cusps}
Suppose that a cusp of $M$ correspond to the point at infinity of the Poincare upper-half space model $\mbH^3$ with the coordinate $(x,y,t)\ (t>0)$. A \emph{linear frame field on a cusp} is given by 
\[ e_1=t\pd t,\ e_2=t[\cos(r) \pd x+\sin(r)\pd y],\ e_e=t[-\sin(r) \pd x+\cos(r)\pd y] \]
where $r$ is a constant.
\end{defn}
(See also \cite[p.486]{Yoshida85}.)

But, there does not always exist such a frame field: one might want to start with a frame field defined linearly in a neighborhood of cusps and attempt to extend it to the entire manifold. But this is not always possible: there exists an obstruction to such extension, a cohomology class in $H^1(M,\partial M;\pi_1(G))$, cf. \cite[$\S$34.2]{Steenrod99}. In trying to extend a given local frame field, one typically ends up meeting singularities along some link $L$.  Meyerhoff’s key idea for defining the metric Chern-Simons invariant in the cusped case is to require for the section $s:M-L\ra F(M)$ thus obtained to have only special kind of singularities at $L$, and to compensate the integral  $\int_{s(M-L)}Q_M$ by the torsion of $L$.

%%%%%%%%%%%%%%%%%%%%
\begin{defn}  \cite[$\S$3.1]{Meyerhoff86}, \cite[Def.1.3]{Yoshida85}
A \emph{special singular frame field} on $M$ is an orthonormal frame field on $M-L$ for some link $L$ which has the following behaviour at each component $K$ in $L$:

i) (in the limit) $e_1$ is tangent to $K$.

ii) $e_2$ and $e_3$ determine a singularity transverse to $K$ of index $1$ or $-1$.
\end{defn}

%%%%%%%%%%%%%%%%%%%%
\begin{defn} \label{Defn:torsion}
Let $M$ be an oriented Riemannian three-manifold. Let $(\theta_{ij})\in \mcA^1(F(M),\mathfrak{so}_3)$ be the Levi-Civita connection $1$-form defined by an orthonormal frame field (i.e. $(\theta_{ij})$ is a matrix of $1$-forms on $F(M)$ such that $\theta_{ji}=-\theta_{ij}$ and $d\theta^i=-\theta_{ij}\wedge \theta^j$ for the dual coframe field $(\theta^i)$). Then, the \emph{torsion} $\tau(L)$ of a link $L$ in $M$ is defined by
\[ \tau(L):=-\int_{s(L)}\theta_{23} \mod 2\pi \Z \]
where $s:L\ra F(M)$ is the section given by some orthonormal framing $\alpha:=(e_1,e_2,e_3)$ defined on a subset of $M$ containing $L$ such that $e_1(y)$ is tangent to $L$ at each $y\in L$ and the orientation of $L$ is given by $e_1$. 
\end{defn}
It is known \cite[Lem.1.1]{Yoshida85} that this integral is well-defined modulo $2\pi\Z$, independent of the choice of the orthonormal framing $\alpha$. 

Now, we present the Meyerhoff's definition of metric Chern-Simons invariant for complete hyperbolic three-manifolds of finite volume, either closed or not.%

%%%%%%%%%%%%%%%%%%%%
\begin{thm} \label{thm:CS_cusped}
(1) For every hyperbolic manifold $M$, closed or not, there exist a (possibly empty) link $L$ in $M$ and an oriented orthonormal framing $s$ on $M-L$ that has special singularity at $L$ and is linear on cusps. 
When $M$ is closed, one can take $L=\emptyset$, i.e. there exists a framing $s$ defined on the entire $M$.

(2) For any link $L$ and a framing $s$ on $M-L$ as in (1), the number (mod $\Z$)
\begin{equation} \label{eq:metric_CS-inv_cusped}
cs(M):=\int_{s(M-L)} Q_M - \frac{1}{2\pi}\sum_{K\in L} \tau(K) \quad (\mathrm{mod }\Z)
\end{equation}
is independent of the choice of the datum $(L,s)$, and if $M$ is closed and $L=\emptyset$, this is even well-defined mod $2\Z$;
by definition, this is the metric (or $\SO(3)$) Chern-Simons invariant of the hyperbolic three manifold $M$.
\end{thm}

%%%%%%%%%%%%%%%%%%%%
\begin{rem} \label{rem:different_conventions}
At this point, it seems worthwhile to clarify different conventions, normalizations, and notations appearing in literatures, centering around Chern and Simons \cite{ChernSimons74}, Meyerhoff \cite{Meyerhoff86}, Yoshida \cite{Yoshida85}, Neumann \cite{Neumann92}, and Kirk and Klassen \cite{KirkKlassen93}, and to compare them with our own choice.%%
\footnote{Often, such differences among different people, especially in the definition/normalization of various CS-invariants, cause confusions and incorrect details in statements.\label{ftn:normalization_CS-inv}} 

(i) Yoshida \cite[Def.1.1]{Yoshida85} and Meyerhoff \cite[$\S$3.2]{Meyerhoff86} use opposite signs for the definition of torsion of curves. We chose Yoshida's sign, thus have minus sign in Definition \ref{Defn:torsion}) and accordingly in front of the torsion term in (\ref{eq:metric_CS-inv_cusped}).

(ii) For the metric Chern-Simons $3$-form, which we denoted by $Q_M$, Yoshida uses the letter $Q$  \cite[l.-6 on p.480]{Yoshida85}, while Meyerhoff uses the same letter $Q$ to mean $4\pi^2 Q_M$ \cite[$\S$2.2]{Meyerhoff86}.

(iii) The metric Chern-Simons invariants defined by Yoshida \cite[l.13 on p.491]{Yoshida85} and Meyerhoff \cite[Theorem in $\S$1]{Meyerhoff86}, with the same notation $CS(M)$ (capital letters, oblique style), are the same, and is equal to \emph{half} our metric CS-invariant $cs(M)$ (\ref{eq:metric_CS-inv_cusped}); so, for cusped hyperbolic three-manifolds, their CS-invariants are well-defined modulo $\frac{1}{2}\Z$.
 In the case of closed three-manifolds, their definitions are also the same as the original definition of Chern and Simons \cite[(6.2)]{ChernSimons74}, and thus are well-defined even modulo $\Z$ (see Remark \ref{rem:normalization}, (2)). 

(iv) Neumann (\cite{Neumann92},  \cite{Neumann98}, \cite{NeumannYang99}) also adopts the original definition of metric Chern-Simon invariant. But, in these works he almost exclusively works with a normalized metric Chern-Simon invariant which he denotes by $\CS(M)$ (capital letters, upright style). We will also work with a normalized metric Chern-Simon invariant with the same notation $\CS(M)$: of course, they will have the same meanings (see Footnote \ref{ftn:CS(M)}). 

(v) Kirk and Klassen's definition of metric Chern-Simons invariant is the same as ours. But, for $\PSL_2(\C)$-Chern-Simons invariant, they use the first Pontryagin polynomial $P_1$, while our definition (to be given later) uses the second Chern polynomial $C_2$ (cf.  \cite[p.555, line17]{KirkKlassen93}, Remark \ref{rem:normalization}, (3)). \end{rem}

\begin{proof}
(1) First, the case when $M$ is closed follows from the well-known fact (\cite[VII.Thm.1]{Kirby89}) that every orientable three-manifold is parallellizable and the Gram-Schmidt orthonormalization. In the general case, this is proved in \cite[$\S$4.2]{Meyerhoff86}. A stronger statement is established in \cite[Prop.3.1]{Yoshida85}. 

(2) This is proved in Section 3.2 to 3.4 of \cite{Meyerhoff86}, and a different proof is given in \cite[Cor.1.1]{Yoshida85}. 
The mod-$\Z$ (rather than mod-$2\Z$) independence in the presence of the link $L$ comes from the fact that for two different links $L_0$, $L_1$ which are moved by a homotopy, the change in their Chern-Simons integrals $\int_{s(M-L_i)} Q_M$ is given by $\frac{1}{2\pi}(\tau(L_0)-\tau(L_1))$, while $\tau(L_i)$ is well-defined modulo $2\pi\Z$ (rather than $4\pi\Z$), see the discussion of $\S$3.2 of \cite{Meyerhoff86}. When $M$ is closed and $L=\emptyset$, the Chern-Simons integral is well-defined even modulo $2\Z$, independent of the choice of $s$ (see Remark \ref{rem:normalization}, (2)).
\end{proof}

%%%%%%%%%%%%%%%%%%%%
\begin{rem} \label{rem:normalization}
(1) The choice of a specific bilinear form $\langle -, - \rangle$ on $\mfg$ in Definition \ref{defn:CSform}, (1) is determined by the requirement that for \emph{closed} three-manifolds $M$, the Chern-Simons integral $\int_{s(M)}Q(\omega)$ is a \emph{mod $\Z$}-well-defined invariant $cs(P,A)$ of a principal $G$-bundle $P$ over $M$ endowed with a connection $\omega$ (i.e. the integral is independent of the choice of a section $s:M\ra P$). This condition is reduced to the one that 

($\dagger$) \emph{for the Maurer-Cartan form $\omega_{\MC}$ on $G$, the closed $3$-form $-\frac{1}{6}\langle \omega_{\MC}\wedge [\omega_{\MC}\wedge \omega_{\MC}]\rangle$ represents an integral class in $H^3(G,\R)$} (``Hypothesis 2.5.'' in \cite{Freed95}). 

This can be seen, for instance, from the transformation formula of $Q(\omega)$ under gauge transformation (\cite[Prop.2.3]{Freed95}). We follow the original discussion of Chern and Simons \cite{ChernSimons74} (See also \cite[$\S$4.2-4.3]{DijkgraafWitten90}), to find the bilinear forms $\langle -, - \rangle$ satisfying ($\dagger$).

For a real Lie group $G$ and $k\in\Z_{\geq0}$, let $I^k(G)=S^k(\mfg^{\vee})^G$, the $\R$-vector space of $\R$-valued degree-$k$ polynomial functions on $\mfg$ which are invariant under the adjoint action of $G$, and set $I^{\ast}(G):=\bigoplus_{k=0}^{\infty}I^k(G)$; for a complex Lie group $G$, we use the same notation $I^k(G)$ for the $\C$-space of similarly defined $\C$-valued functions on $\mfg$. We recall the (universal) Chern-Weil homomorphism $W:I^k(G) \ra H^{2k}(BG,\mbF)$ which sends $P\in I^k(G)$ to the cohomology class $[P(F)]$ of the closed $2k$-form $P(F)$, where $F$ is the curvature of an \emph{arbitrary} connection on the universal principal $G$-bundle $EG$ over the classifying space $BG$, and $\mbF=\R$ or $\C$ according as $G$ is real or complex; it is a fundamental fact that this map is well-defined, i.e. the cohomology class $[P(F)]$ is independent of the choice of the connection, and further is an algebra isomorphism either if $G$ is a connected compact real Lie group or if $G$ is a connected reductive complex Lie group \cite[5.23,5.30]{BurgosGil02}.  
A key ingredient in the work of Chern and Simons is the \emph{universal suspension map} \cite[p.55]{ChernSimons74} (cf. \cite[$\S$9]{Borel55}):
\begin{align*} I^k(G)\cong H^{2k}(BG,\mbF) \ra & \ H^{2k-1}(G,\mbF) \ ;\\ 
[P] \qquad \mapsto &\  [TP(\omega_{\MC}):=\frac{(-1)^{k-1}}{2^{k-1}{2k-1 \choose k} }P(\omega_{\MC}\wedge[\omega_{\MC},\omega_{\MC}]^{k-1})]
\end{align*}
($\mbF=\R$ or $\C$); the cohomology classes in the image are called \emph{universally transgressive}. This suspension map preserves the integral subspaces: 
\[ I^k_0(G)\cong H^{2k}(BG,\Z) \ra H^{2k-1}(G,\Z), \] 
where $I^k_0(G):=\{ P\in I^k(G)\ |\ W(P)\in H^{2k}(BG,\Z)\}$ (\cite[Prop.3.15]{ChernSimons74}).

Now, the $3$-form in the condition ($\dagger$) is $TP(\omega_{\MC})$ for $P=\langle -, - \rangle$ ($k=2$): for any $1$-form $\theta$,
\[ -\frac{1}{6}\langle \theta\wedge [\theta\wedge \theta]\rangle=TP(\theta); \] 
also, when the curvature $F=d\theta+\frac{1}{2}[\theta\wedge \theta]$ is zero, we have $TP(\theta)=\langle \theta\wedge d\theta +\frac{1}{3}\theta\wedge [\theta\wedge \theta] \rangle$, i.e. $TP(\theta)$ is the Chern-Simons form $Q(\theta)$.

Therefore, in order for ($\dagger$) to hold, it suffices that $\langle -, - \rangle\in I^2(G)$ lies in $I^2_0(G)$.
For example, for $G=\SO(3)$, it is known that $I^2_0(\SO(3))=H^3(B\SO(3),\Z)$ is generated by the first Pontryagin polynomial $P_1(X)=-\frac{1}{8\pi^2}\tr(X^2)$ (here, the product $X^2$ and $\tr(-)$ is defined via the embedding $X\in \so(3)\subset M_{3\times3}(\R)$) and the Stiefel-Whitney class (which is a $2$-torsion) \cite[1.5]{Brown82}.
For $G=\SU(2)$, $I^2_0(\SU(2))=H^3(B\SU(2),\Z)$ is generated by the second Chern polynomial $C_2(X)=\frac{1}{8\pi^2}\tr(X^2)$ (here, the product $X^2$ and $\tr(-)$ is defined via $\su(2)\hra M_{2\times2}(\C)$) \cite[Cor.5.7]{MimuraToda91}.

(2) The integral universal suspension map $H^{2k}(BG,\Z) \ra H^{2k-1}(G,\Z)$ is not surjective for general $G$, especially for $\SO(3)$. In this case, for $\langle -, - \rangle$, Chern and Simons take $\frac{1}{2}P_1$ which lies in $H^{3}(G,\Z)$, \cite[5.13, $\S$6]{ChernSimons74}. So, their original definition \cite[(6.2)]{ChernSimons74} of CS-invariant is half our metric CS-invariant (\ref{eq:metric_CS-inv_closed}), which is still well-defined mod $\Z$ for closed manifolds.

(3) There exists a double covering map $\sigma:\SU(2)\ra \SO(3)$ with kernel the center $Z(\SU(2))(\cong \Z/2\Z)$. Since $\SL_2(\C)=\SU(2)_{\C}$ (complexification of $\SU(2)$) and $Z(\SL_2(\C))=Z(\SU(2))$, we also have $\SO(3)_{\C}=\PSL_2(\C)$, and there are commutative diagrams
\begin{equation} \label{eq:I(G)}
\xymatrix{ \mathfrak{sl}_2(\C) \ar@{=}[r] & \mathfrak{sl}_2(\C) \\
\su(2) \ar[r]^{\simeq} \ar@{^(->}[u] & \so(3) \ , \ar@{^(->}[u] } \quad \xymatrix{ I^{\ast}(\SL_2(\C)) \ar@{=}[r] &I^{\ast}(\PSL_2(\C)) \\
I^{\ast}(\SU(2))\ar@{^(->}[u] & I^{\ast}(\SO(3))  \ar[l]_{\sigma^{\ast}}  \ar@{^(->}[u] }
\end{equation}

It is known that $\sigma^{\ast}: I^2_0(\SO(3))=H^4(B\SO(3),\Z)\ra I^2_0(\SU(2))=H^4(B\SU(2),\Z)$ sends the first Pontryagin polynomial $P_1=-\frac{1}{8\pi^2}\tr(X^2)_{\so(3)}$ to the minus four times the second Chern polynomial $-4C_2=-\frac{1}{2\pi^2}\tr(X^2)_{\su(2)}$. Indeed, one has $4\tr(XY)_{\Lsl_{2,\C}}|_{\so(3)}=\tr(XY)_{\so(3)}$ (the trace bilinear form $\tr(XY)_{\su(2)}$ on $\su(2)$ is the restriction of the trace form $\tr(XY)_{\Lsl_{2,\C}}$ on $\Lsl_{2,\C}$, which, when restricted to $\so(3)$, equals one fourth of the trace form $\tr(XY)_{\so(3)}$ on $\so(3)\subset \mathfrak{sl}_3$). This is also a statement equating the Killing forms on $\so(3)$ and $\su(2)$ via the isomorphism $\sigma^{\ast}$. 
\end{rem}

%%%%%%%%%%%%%%%%%%%%%%%%%%%%%%%%%%%%%%%%
\subsection{Complex volume of hyperbolic $3$-manifolds as $\PSL_2(\C)$ Chern-Simons invariant}

It has been proved to be quite useful to treat two most important invariants of complete hyperbolic three-manifold, Riemannian volume and metric Chern-Simons invariant, together as the real part and the imaginary part of a single complex-number (modulo $\sqrt{-1}\Z$). In this section, we explain how the complex volume $\Vol_{\C}(M)$ of \emph{any} (closed or not) complete hyperbolic three-manifold $M$ of finite volume can be viewed as the Chern-Simons invariant of some connection on a (triviallizable) $\PSL_2(\C)$-principal bundle over $M$. We discuss the case that $M$ is closed first, following Yoshida \cite{Yoshida85}, and then treat the cusped case, where the definition of $\PSL_2(\C)$ Chern-Simons invariant is not even clear and is due to Kirk and Klassen \cite{KirkKlassen93}.

%%%%%%%%%%%%%%%%%%%%
\begin{defn}
The \emph{complex volume} of a complete hyperbolic three-manifold is defined to be
\[ \Vol_{\C}(M)= \Vol(M)+i\CS(M) \in\ \C/2\pi^2i\Z \text{ or } \C/\pi^2i\Z,\]
where $\CS(M):=\pi^2cs(M)$.%%
\footnote{This normalized metric CS-invariant $\CS(M)$ is the one that appears with same notation in  \cite[Introduction]{Neumann92}. Indeed, Neumann's $\CS(M)$ is $2\pi^2$ times his metric CS-invariant which is half our metric CS-invariant $cs(M)$ (Remark \ref{rem:different_conventions}, (iv)).\label{ftn:CS(M)}}
This is valued in $\C/2\pi^2i\Z$ if $\partial M=\emptyset$, and in $\C/\pi^2i\Z$ if $\partial M\neq\emptyset$.
\end{defn}

We begin by recalling the general fact that for every smooth manifold $M$, with any representation $\rho:\pi_1(M)\ra G$ into a Lie group, there is associated a principal bundle $P(\rho):=\widetilde{M}\times_{\pi_1(M)} G$ (in this work, principal bundles are endowed with right action by structure group). This principal bundle is endowed with a flat connection $\omega_{\rho}$ which is induced from the trivial connection on $\widetilde{M}\times G$.

For a complete hyperbolic three-manifold $M$, let $\rho_0:\pi_1(M) \ra \PSL_2(\C)$ be the \emph{geometric representation}, the holonomy representation corresponding to the hyperbolic metric of $M$ and $P(\rho_0)$ be the associated principal bundle with flat connection $\omega_{\rho_0}$. Then, there exists a map $F(M) \ra P(\rho_0)$ equivariant with respect to $\SO(3)\hra \PSL_2(\C)$, constructed as follows. Fixing a frame over a point of the hyperbolic space $\mbH$, we identify the frame bundle $F(\mbH)\ra \mbH$ with $\PSL_2(\C)\ra \SO(3)\backslash\PSL_2(\C)$.
This gives a $\PSL_2(\C)$-equivariant map
\begin{equation} \label{eq:p_times_id}
p\times \mathrm{id}: F(\mbH) \ra \mbH \times \PSL_2(\C),
\end{equation}
where $p: F(\mbH) \ra \mbH$ is the bundle projection map. 
Using the existence of an isomorphism $\widetilde{M}\cong \mbH$ that is equivariant with respect to $\rho_0:\pi_1(M) \ra \PSL_2(\C)$ (by completeness of the hyperbolic metric), by taking quotients of (\ref{eq:p_times_id}) by $\pi_1(M)$ we obtain a map
\[ q: F(M)= F(\widetilde{M})/\pi_1(M) \ra P(\rho_0)=\widetilde{M}\times_{\pi_1(M)} \PSL_2(\C) \]
covering the identity map of $M$. So, any section $s:M\ra F(M)$ gives rise to a section $\hat{s}=q\circ s:M\ra P(\rho_0)$. In particular, $P(\rho_0)$ is trivial since $F(M)$ is so, and it follows that the (trivial) principal bundle $P(\rho_0)$ is the extension of the (trivial) bundle $F(M)$ via the closed embedding of structure groups $\SO(3)\hra \PSL_2(\C)$.

Then, the $\PSL_2(\C)$-Chern-Simons invariant of closed complete hyperbolic three-manifold $M$ is defined to be
\begin{equation} \label{eq:PSL_2-CS-inv:closed}
cs_{\PSL_2}(M):=cs(\rho_0,\hat{s}):=\int_{\hat{s}(M)} Q(\omega_{\rho_0}) \in \C/\Z.
\end{equation}
Here, we emphasize that for the Ad-invariant bilinear form $\langle-,-\rangle$ in the definition of $Q(\omega_{\rho_0})$, we use the second Chern polynomial $C_2(X)=\frac{1}{8\pi^2}\tr(X^2)\in I^2(\Lsl_{2,\C})$ (while for the metric Chern-Simons invariant we used the first Pontryagin polynomial $P_1(X)=-\frac{1}{8\pi^2}\tr(X^2)\in I^2(\so(3))$); we caution readers that some people, e.g. Kirk and Klassen \cite[p.555, line17]{KirkKlassen90}, use the first Pontryagin polynomial $P_1$ to define $\PSL_2(\C)$-Chern-Simons invariant.
%%

%%%%%%%%%%%%%%%%%%%%
\begin{lem} \label{lem:Yoshida_Lem.31}
We have
\[ -4\sqrt{-1}q^{\ast}Q(\omega_{\rho_0})=\frac{1}{\pi^2}C_{\MoTe}^{\ast}\mathrm{Vol}_{M}+\sqrt{-1}Q_{M}+d \gamma \]	
for some $2$-form $\gamma$, where $C_{\MoTe}:F(M)\ra M$ is the frame bundle map.
\end{lem}

\begin{proof} (cf. \cite[p.554]{KirkKlassen90}) 
As the first claim is a local statement, it suffices to check it on the covering space (\ref{eq:p_times_id}). 
Let $\mathrm{pr}_2:\mbH \times \PSL_2(\C) \ra \PSL_2(\C)$ be the projection and $\theta$ the Maurer-Cartan form on $\PSL_2(\C)$.
Since $\omega_{\rho_0}$ is locally given by $\mathrm{pr}_2^{\ast}\theta$,  $q^{\ast}Q(\omega_{\rho_0})$ descends from the Chern-Simons $3$-form $Q(\theta)=-\frac{1}{6}\langle \theta\wedge [\theta\wedge \theta]\rangle$ on $\PSL_2(\C)$ (the curvature of $\theta$ is zero), where $\langle-,-\rangle$ is the second Chern polynomial $C_2(X)=\frac{1}{8\pi^2}\tr(X^2)$ on $\Lsl_{2,\C}$. Hence,  the claim reduces to the equality of two $\C$-valued $3$-forms on $\PSL_2(\C)$:
\begin{equation} \label{eq:Yoshida_Lem.3.1} 
-4\sqrt{-1} Q(\theta)=\frac{1}{\pi^2} f_{\mbH}^{\ast}\mathrm{Vol}_{\mbH}+\sqrt{-1}Q_{\mbH}+d \widetilde{\gamma}
\end{equation}
for some $2$-form $\widetilde{\gamma}$ on $\PSL_2(\C)$. This is Lemma 3.1 of \cite{Yoshida85}. Indeed, choosing a basis  
\[ h=\left(\begin{array}{cc} 1 & 0 \\ 0 & -1 \end{array}\right),\ e=\left(\begin{array}{cc} 0 & 1 \\ 0 & 0 \end{array}\right),\ f=\left(\begin{array}{cc} 0 & 0 \\ 1 & 0 \end{array}\right) \]
of $\mfg=\mathfrak{sl}_{2,\C}$ with dual basis $\{h^{\vee}, e^{\vee}, f^{\vee}\}$, and using that 
\[ [h,e]=2e,\ [h,f]=-2f,\ [e,f]=h, \]
and $\theta=h\otimes h^{\vee}+e\otimes e^{\vee}+f\otimes f^{\vee}$ (regarding elements of $\Lsl_{2,\C}^{\vee}$ as left-invariant differential forms, as usual), we see that $-4Q(\theta)=\frac{2}{3}\langle \theta\wedge [\theta\wedge \theta]\rangle$ equals $\frac{1}{\pi^2}h^{\vee}\wedge e^{\vee}\wedge f^{\vee}$. Now, the Yoshida's lemma in question says that the $\sqrt{-1}$-multiple of the latter $3$-form (which he denoted by $C$) equals the right-hand side of (\ref{eq:Yoshida_Lem.3.1}) (for some $\widetilde{\gamma}$). 
\end{proof}

Consequently, the $\PSL_2(\C)$-Chern-Simons invariant of any closed complete hyperbolic three-manifold $M$ equals the complex volume of $M$, up to a constant:

%%%%%%%%%%%%%%%%%%%%
\begin{thm} \label{thm:PSL_2-CS-inv=C-vol:closed}
For any closed complete hyperbolic three-manifold $M$, we have
\[ -4\pi^2\sqrt{-1} cs_{\PSL_2}(M)=\Vol_{\C}(M). \]
%($\hat{s}=q\circ s$).
\end{thm}

Even when $M$ is not closed, Kirk and Klassen \cite{KirkKlassen93} define the $\PSL_2(\C)$-Chern-Simon invariant of \emph{cusped} hyperbolic three-manifolds. Recall that $P(\rho_0)$ is trivial.
Their idea is based on the viewpoint that via a choice of a section $s$ of $P(\rho_0)$, the integral (\ref{eq:PSL_2-CS-inv:closed}) is an integral over $M$ of the $3$-form $Q(A_0)$ for an $1$-form $A_0=s^{\ast}(\omega_{\rho_0})\in \mcA^1(M,\Lsl_{2,\C})$ on $M$, and that choices of $s$ are in bijection with the bundle automorphisms of the trivial bundle $M\times\PSL_2(\C)$ (\emph{gauge group}), also with gauge-transformations of $A_0$. Then the $1$-forms corresponding to the sections $s$ that are linear on cusps will be of certain special shape near cusps, called \emph{normal form} (cf. \cite[p.228]{Meyerhoff86}). A result of Kirk and Klassen says that the Chern-Simons integral $\int_MQ(A)$ of a connection $A\in \mcA^1(M,\Lsl_{2,\C})$ is independent of the choice of $A$, as long as $A$ lies in a single gauge-equivalence class and is in normal form near cusps. Although there might not exist a section $s:M\ra P(\rho_0)$ that is linear on cusps, we can still show the existence of a flat connection $A\in \mcA^1(M,\Lsl_{2,\C})$ in normal form near cusps and having the geometric representation as its holonomy. Applying the result of Kirk and Klassen, we define the $\PSL_2(\C)$-Chern-Simons invariant of $M$ as the Chern-Simons integral $\int_MQ(A)$ of such connection $A$. This turns out to have the expected relation to the complex volume.

In the next discussion, for simplicity, let us assume that $\bdM$ has only one connected component.
Let $\partial M\times [0,1]\subset \overline{M}$ be a collar with $\partial M$ identified with $\partial M\times\{1\}$. The choice of an ordered basis $\mu,\nu$ of $H_1(\bdM,\Z)$ allows to identify $\R^2$ with the universal covering: $\R^2\ra \partial M:(x,y)\ra (e^{i x},e^{i y})$. We assume that the orientation of $\partial M$ as the boundary of $M$ agrees with that inherited from the cover $\R^2\ra \partial M$, so that $\{dx,dy,dr\}$ is an oriented basis of $1$-forms on $\partial M\times I$, where $r$ is the coordinate of $I=[0,1]$.

Suppose given a trivialized principal bundle $P= \SL_2\times G$ for $G=\SL_2(\C)$ or $\PSL_2(\C)$. With respect to the trivialization, any connection $A$ is written as $A=\alpha dx+\beta dy+\gamma dr$ in a neighborhood of $\partial M$, where $\alpha,\beta,\gamma$ are functions $\partial M\times I\ra \mathfrak{sl}_2$. 
Note that since $\pi_1(\partial M)$ is abelian, any representation $\rho:\pi_1(\partial M)\ra \SL_2(\C)$ is conjugate to one whose image lies in (the image in $\PSL_2(\C)$ of) the group of upper triangular representations, and further to a diagonal representation unless $\rho$ is (non-central) parabolic, i.e. $\Tr(\rho(\pi_1(\partial M)))=\pm2$.

%%%%%%%%%%%%%%%%%%%%
\begin{defn} \label{defn:normal_form}
A  flat connection $A$ on $M$ is in \textit{normal form near boundary} if one of the following two conditions holds:

\textsc{Case 1}. 
in a neighborhood of the boundary, 
\begin{equation} \label{eq:normal_form1}
A=\left(\begin{array}{cc} i\alpha & 0 \\ 0 & -i\alpha \end{array} \right) dx+ \left(\begin{array}{cc} i\beta & 0 \\ 0 & -i\beta \end{array} \right) dy
\end{equation}
for some $\alpha,\beta\in \C$: we write $A=A^{(1)}(\alpha,\beta)$.

\textsc{Case 2}. 
in a neighborhood of the boundary, 
\begin{equation} \label{eq:normal_form2}
A=\left(\begin{array}{cc} -\frac{iu}{2} & \frac{a}{2\pi}\exp(i(ux+vy)) \\ 0 & \frac{iu}{2} \end{array} \right) dx+ \left(\begin{array}{cc} -\frac{iv}{2} & \frac{b}{2\pi}\exp(i(ux+vy)) \\ 0 & \frac{iv}{2} \end{array} \right) dy.
\end{equation}
\end{defn}
for some $u,v\in\Z$, $a,b\in\C$: we write $A=A^{(2)}(u,v;a,b)$.

The point here is that the connection $A$ is constant near boundary (i.e. $\alpha$, $\beta$, $u$, $v$, $a$, $b$ are all constants). 

The local holonomy $\pi_1(\partial M)\ra (\mathrm{P})\SL_2(\C)$ of $A$ at the boundary is described explicitly by the shape of $A$ near the boundary as follows.

%%%%%%%%%%%%%%%%%%%%
\begin{lem} \label{lem:holonomy_of_normal_form}
Every connection in normal form near boundary is flat in a neighborhood of the boundary, and for a suitable base point near $\partial M$, the holonomy representation $\rho_A$ is given by: 
\begin{align} 
 \text{ Case }1:\qquad & \mu\mapsto \left(\begin{array}{cc} e^{\tpi \alpha} & 0 \\ 0 & e^{-\tpi \alpha} \end{array} \right),\ \nu\mapsto \left(\begin{array}{cc} e^{\tpi \beta} & 0 \\ 0 & e^{-\tpi \beta} \end{array} \right), \label{eq:bdry_holonomy1} \\
 \text{ Case }2:\qquad & \mu\mapsto (-1)^u\left(\begin{array}{cc} 1 & a \\ 0 & 1 \end{array} \right),\quad   \nu\mapsto (-1)^v\left(\begin{array}{cc} 1 & b \\ 0 & 1 \end{array} \right) \label{eq:bdry_holonomy2}
\end{align}
(in the case $G=\PSL_2(\C)$, we consider their images in $\PSL_2(\C)$).
\end{lem}

From this lemma, we see that given a local holonomy representation $\rho:\pi_1(\partial M)\ra \SL_2(\C)$, there are as many local connections in normal form having holonomy $\rho$ as choices of branches of $\log \rho(\mu)_{11}$, $\log \rho(\nu)_{11}$, where $\rho$ is conjugated to be of the form (\ref{eq:bdry_holonomy1}) or  (\ref{eq:bdry_holonomy2}) (in Case 2, these branches are really parities of $u$ and $v$). Note that any homomorphism $\rho:\pi_1(\partial M)\ra \SL_2(\C)$ has a conjugate in one of these forms (since $\pi_1(\partial M)$ is abelian). It turns out that for every global homomorphism $\rho :\pi_1(M)\ra \SL_2(\C)$ and for any choice of constants $(\alpha,\beta)$, $(u,v,a,b)$ in (\ref{eq:bdry_holonomy1}), (\ref{eq:bdry_holonomy2}) (for some conjugate of $\rho|_{\pi_1(\partial M)}$), there exists a global flat connection giving $\rho$ as holonomy which is in normal form for these constants.

%%%%%%%%%%%%%%%%%%%%
\begin{prop} \cite[2.3,3.2]{KirkKlassen93} \label{prop:KK93_2.3,3.2} 
Let $N_1,\cdots,N_h$ be the connected components (tori) of $\partial M$.

(1) Any connection on $M$ which is flat in a neighborhood of $\partial M$ can be gauge-transformed to a connection in normal form near boundary.

(2) Suppose given a representation $\rho:\pi_1(M)\ra \SL_2(\C)$. For any choice of constants $(\alpha_k,\beta_k)$ in Case 1 and $(u_k,v_k,a_k,b_k)$ in Case 2 ($k=1,\cdots,h$) such that for each $k=1,\cdots,h$, $\rho|_{N_k}$ is conjugate to either (\ref{eq:bdry_holonomy1}) or (\ref{eq:bdry_holonomy2}), there exists a flat connection on $P=M\times\SL_2$ whose holonomy is conjugate to $\rho$, and which at each cusp is in normal form (\ref{eq:normal_form1}), (\ref{eq:normal_form2}) for these constants. 

The same holds for a representation $\rho:\pi_1(M)\ra \PSL_2(\C)$.
\end{prop}

In (2), there is an overlap among the two cases, i.e. Case 1 with $(\alpha,\beta)\in \frac{1}{2}\Z^2$  equals Case 2 with $(a,b)=(0,0)$, $(\alpha,\beta)=(-\frac{u}{2},-\frac{v}{2})$, in which case we could choose any one among these two.

\begin{proof}
(1) This is proved in \cite[2.3]{KirkKlassen93} in the case that the structure group of the principal bundle is $\SU(2)$, in which case the holonomy at $\pi_1(\partial M)$ is always diagonalizable. This gives a proof in Case 1 of Definition \ref{defn:normal_form}. The same argument carries over to $\SL_2(\C)$ in Case 2, too. 
In more detail, the obstructions to extending a given local $1$-form to entire $M$ lie in $H^i(M,\partial M,\pi_i(\SL_2(\C)))$, which are all zero.
So, the question is reduced to the problem of finding a local connection of the form (\ref{eq:normal_form2}) having the (local) holonomy (\ref{eq:bdry_holonomy2}). This problem is solved in the explanation (p.538) of Case 2 in the proof of ibid. Theorem 3.2.

(2) For $\SL_2(\C)$, in case 1, this is ibid. Theorem 2.3 (3), while in case 2 it is shown in the proof (p. 538) of ibid. Theorem 3.2. One easily sees that the argument carries over to $\PSL_2(\C)$.
\end{proof}

We define the Chern-Simons invariant of a flat connection on the trivial bundle $P=M\times G$ ($G=\SL_2(\C)$ or $\PSL_2(\C)$), for flat connections that are in normal form near boundary,%%
\footnote{Later, we will also need to consider the Chern-Simons integral of an arbitrary flat connection.}
as
\begin{equation} \label{eq:cs(A)_normal_form_A}
cs_M(A) :=\frac{1}{8\pi^2} \int_{M}\mathrm{Tr}(dA\wedge A+\frac{2}{3}A\wedge A\wedge A)
\end{equation}
(note that here we use the second Chern polynomial $C_2=\frac{1}{8\pi^2}\tr$).

%%%%%%%%%%%%%%%%%%%%
\begin{thm} \cite[2.4]{KirkKlassen93} \label{thm:KK93_2.4}
If two connections $A$, $B$ on $P$ are gauge-equivalent and also are in the same normal form near $\bdM$, then $cs(A)=cs(B)$$\mod \Z$.
\end{thm}

\begin{proof}
First, suppose that $G=\SL_2(\C)$. When the boundary holonomy is diagonalizable, 
we can repeat the proof of the same statement in the case $G=\SU(2)$ (Theorem 2.4 of \cite{KirkKlassen93}). When the boundary holonomy is parabolic but non-diagonalizable, this is proved in \textit{ibid.} Lemma 3.3. When $G=\PSL_2(\C)$, it is not difficult to check that the same arguments continue to work.
\end{proof}

%%%%%%%%%%%%%%%%%%%%
\begin{cor} \label{cor:PSL_2(C)-CS-inv}
There exists a flat connection $A$ on $M\times \PSL_2(\C)$ such that at each cusp $N_k$, $A$ is in normal form of Case 2 for $(u_k,v_k)=(0,0)$ and some $(a_k,b_k)$ and that the holonomy representation of $A$ is the geometric representation $\rho_0$, up to conjugacy. 

The Chern-Simons invariant $cs_M(A)$ of such a flat connection $A$ is well-defined mod $\Z$, independent of the choice of $A$; by definition, this is the $\PSL_2(\C)$-Chern-Simons invariant $cs_{\PSL_2}(M)$ of the cusped complete hyperbolic three-manifold $M$.
\end{cor}

\begin{proof}
It is known (\cite[$\S$4]{NeumannZagier85}) that the geometric representation $\rho_0$ is \emph{boundary parabolic}, meaning that at each cusp, after a conjugation of $\rho_0$, the images of the meridian and the longitude lie in the unipotent subgroup (i.e. the image in $\PSL_2(\C)$ of the subgroup of $\SL_2(\C)$ of unipotent matrices). Hence, for each $k=1,\cdots,h$, there exists a unique $(a_k,b_k)\in \C^2$ such that $\rho|_{N_k}$ is conjugate to (\ref{eq:bdry_holonomy2}) for the constants $(0,0,a_k,b_k)$.
Then, for this family of constants $\{(0,0,a_k,b_k)\}_{k=1,\cdots,h}$ and  $\rho=\rho_0$, 
Proposition \ref{prop:KK93_2.3,3.2}, (2) gives a flat connection $A$ with the stated properties.

The second statement is an immediate consequence of Theorem \ref{thm:KK93_2.4}, in view of the well-known fact that the holonomy gives an injection
\[ \{\text{flat connections on }M\times G\}/\mcG \hra \Hom(\pi_1(M),G)/\text{conjugation}, \]
where $\mcG=\Hom(M,G)$ (gauge group), cf. \cite[Lem.2.2]{KirkKlassen90}. 
\end{proof}

We normalize the $\PSL_2(\C)$-Chern-Simons invariant $cs_{\PSL_2}(M)$ by 
\begin{equation} \label{eq:normalized_PSL_2(C)-CS.inv}
\CS_{\PSL_2}(M):=(2\pi i)^2 cs_{\PSL_2}(M)\ \in \C/\Q(2).
\end{equation}
%%%%%%%%%%%%%%%%%%%%
\begin{thm} \label{thm:PSL_2-CS-inv=C-vol:cusped}
The (normalized) $\PSL_2(\C)$-Chern-Simons invariant $\CS_{\PSL_2}(M)$ of any complete hyperbolic three-manifold $M$ of finite volume is related to the complex volume $\Vol_{\C}(M)$ by
\[  \sqrt{-1}\CS_{\PSL_2}(M)=\Vol_{\C}(M). \]
\end{thm}

\begin{proof}
The case that $M$ is closed is Theorem \ref{thm:PSL_2-CS-inv=C-vol:closed}. In the case that $M$ has cusps, this is proved in p. 555-556 (especially, line 16-20 on p.556) of \cite{KirkKlassen93}: we remind the readers that to define the $\PSL_2(\C)$ Chern-Simons invariant $cs_{\PSL_2}(M)$, Kirk and Klassen use the first Pontryagin polynomial for the bilinear form $\langle-,-\rangle$ in Definition \ref{defn:CSform}, so that (in terms of $cs_{\PSL_2}(M)$) their $\PSL_2(\C)$ Chern-Simon invariant equals $-4$ times our $\PSL_2(\C)$ Chern-Simon invariant (\cite[p.555,line17]{KirkKlassen93}).
\end{proof}

We can also define $\SL_2(\C)$-Chern-Simons invariant of hyperbolic three-manifolds $M$ which however depends on a lift of the geometric representation $\rho_0:\pi_1(M)\ra \PSL_2(\C)$ to $\SL_2(\C)$; such a lift is called a \emph{spin structure} on $M$ (recall that the bundle $P(\rho_0)$ is trivial and $\SL_2(\C)=\Spin(3)_{\C}$, so that if $\tilde{\rho}_0:\pi_1(M)\ra \SL_2(\C)$ is a lift of $\rho_0$, there exists a double covering $P(\tilde{\rho}_0)\ra P(\rho_0)$ of the associated principal bundles which is compatible with the actions of $\SL_2(\C)\ra \PSL_2(\C)$; in fact, $P(\rho_0)=P(\tilde{\rho}_0)\times_{\SL_2(\C)}\PSL_2(\C)$, extension via change of structure group). 
Here, we only consider the case $M$ is closed; the cusped case will be considered in Subsection \ref{subsec:CS-inv_flat-conn:cusped}.

Like $P(\rho_0)$, the $\SL_2(\C)$-bundle $P(\tilde{\rho}_0)$ affords a section $\tilde{s}$, since $\pi_1(\SL_2(\C))=\{1\}$, as one can see by the obstruction theory \cite[$\S$34.2]{Steenrod99}.
When $M$ is closed, using $\tilde{s}$ one defines the $\SL_2(\C)$-Chern-Simons invariant of a lift $\tilde{\rho}_0:\pi_1(M)\ra \SL_2(\C)$ of $\rho_0$ by
\begin{equation} \label{eq:SL_2-CS}
cs(\tilde{\rho}_0,\tilde{s}):=\int_{\tilde{s}(M)} Q(\omega_{\tilde{\rho}_0}) \in \C/\Z,
\end{equation}
where as before $Q(\omega_{\tilde{\rho}_0})\in \mcA^3(P(\tilde{\rho}_0),\Lsl_{2,\C})$ is the Chern-Simons $3$-form of the flat connection $\omega(\tilde{\rho}_0)$ on $P(\tilde{\rho}_0)$ (induced from the trivial connection on $\widetilde{M}\times \SL_2(\C)$), and the $3$-form $Q(-)$ is defined using the second Chern polynomial $C_2\in I^2(\Lsl_{2,\C})$.
It is then clear that
\begin{equation} \label{eq:SL_2-CS-inv=PSL_2-CS-inv}
cs(\tilde{\rho}_0,\tilde{s})=cs(\rho_0,\hat{s}),
\end{equation}
where $\hat{s}$ is the section on $P(\rho_0)$ induced from $\tilde{s}$ (recall that we use $C_2\in I^2(\Lsl_{2,\C})$ in the definition of $\PSL_2(\C)$-Chern-Simons invariant $cs(\rho_0,\hat{s})$).

%%%%%%%%%%%%%%%%%%%%%%%%%%%%%%%%%%%%%%%%
\subsection{Ideal triangulation of cusped hyperbolic three-manifolds}
In this section, $M$ is a complete hyperbolic three manifold (of finite volume). Suppose $M$ is non-compact and has an ideal triangulation: $M= \Delta_1\cup \cdots \cup \Delta_N$. 
This means that each simplex $\Delta_i$ is a (hyperbolic) ideal tetrahedron, i.e. a tetrahedron in $\overline{\mbH^3}$ which has its vertices in $\CP^1=\partial \overline{\mbH^3}$ and has geodesic faces, and that $M$ is obtained by glueing together these ideal tetrahedra by means of face-pairing maps and then removing vertices or horospheres (closed balls centered at vertices); in this case $M$ is the interior of a compact manifold having a union of tori as boundary, and these \emph{boundary tori} have the induced flat metric. For the question of existence of an ideal triangulation, see the introduction of \cite{Francaviglia04}. 
Any ideal tetrahedron with ordered vertices $\{z_1,z_2,z_3,z_4\}$ is moved by $\mathrm{Iso}^+(\mbH^3)=\PSL_2(\C)$ to the ideal tetrahedron with ordered vertices $\{\infty,0,1,z\}$ for some uniquely determined $z\in \CP^1-\{0,1,\infty\}$, which is the cross ratio%
\footnote{This definition of cross ratio differs from the ones used in \cite{NeumannYang99}, \cite{Neumann92} by $z\leftrightarrow z^{-1}$.}
\[ z=[z_1:z_2:z_3:z_4]=\frac{(z_3-z_1)(z_4-z_2)}{(z_3-z_2)(z_4-z_1)} \in  \C\backslash \{0,1\}.\]
Even permutation of the vertices changes $z$ to one of $z$, $z'=\frac{1}{1-z}$, $z''=1-\frac{1}{z}$; one of these numbers is called the \emph{shape parameter} (or \emph{cross-ratios} or \emph{modulus}) of the oriented ideal tetrahedron. We fix an ordering of the vertices of each ideal tetrahedron compatible with the orientation of $M$, and choose one of the three associated shape parameters $z_j^0$ of each $\Delta_j$; put 
\[ \mathscr{Z}^0:=[\log z^0_1,\cdots, \log z^0_N,\log (1-z^0_1),\cdots, \log (1-z^0_N) ] \in \C^{2N}. \]
(row vectors).%% 
\footnote{in this subsection, however, we will work with column vectors, in accordance with the convention of Neumann and Zagier \cite{NeumannZagier85} (but in opposite to that in \cite{Neumann92}).}
In fact, as we fix an ordering of the vertices, to each edge of the tetrahedron, 
one can assign one of the three shape parameters $z$, $z'$, $z''$, in such a way that the parameters associated with two opposite edges are equal and all $z$, $z'$, $z''$ are attached to some (pairs of opposite) edges (see \cite[$\S$2]{Neumann92} for details).

It is known (\cite[$\S$2]{NeumannZagier85}, \cite[$\S$2]{Neumann92}) that when $M$ has $h$ cusps, 
the complete hyperbolic structure of $M$ is determined by $\msZ^0$ subject to $N+2h$ linear equations with integer coefficients, written in the form
\begin{equation} \label{eq:CCequation}
U(\msZ^0)^t =\pi \sqrt{-1} D,
\end{equation}
where $U$ is a $(N+2h)\times 2N$ integral matrix and $D$ is a column vector in $\Z^{N+2h}$. More precisely, this system of equations consists of two parts: we write
\begin{equation} \label{eq:U=(R,C)}
U= \left(\begin{array}{c} R \\ C \end{array}\right),\quad R\in M_{N\times 2N}(\Z),\ C\in M_{2h\times 2N}(\Z)
\end{equation}
and $D= \left(\begin{array}{c} D_1 \\ D_2 \end{array}\right)$ for $D_1\in M_{N\times 1}(\Z)$, $D_2\in M_{2h\times 1}(\Z)$, so that (\ref{eq:CCequation}) splits to 
\begin{equation} \label{eq:CCequation2}
R(\msZ^0)^t=\pi \sqrt{-1} D_1,\quad C(\msZ^0)^t=\pi \sqrt{-1} D_2.
\end{equation}
The geometric meanings of these equations are as follows.
The first equation $R(\msZ^0)^t=\pi \sqrt{-1} D_1$, called \emph{consistency relations}, expresses the condition that the tetrahedra fit together around each edge of the triangulation.
The second one $C(\msZ^0)^t=\pi \sqrt{-1} D_2$, called \emph{cusp relations}, means that at each end (or cusp), if $T_i$ is the corresponding connected component (torus) of $\partial M$, the elements of the peripheral subgroup $\pi_1(T_i,t_i)$ have parabolic holonomy (i.e. are unipotent matrices).

For the following (and later uses), we introduce ``tensor-multiplication over $\Q$'' $M \otimes N$ of two matrices $M\in M_{l\times m}(\C)$, $N\in M_{m\times n}(\C)$ whose output is a matrix in $M_{l\times n}(\C\otimes_{\Q} \C)$:
\begin{equation}  \label{eq:tensor_mult_matrices}
(M\otimes N)_{ij}=\sum_{k=1}^m M_{ik}\otimes N_{kj} \in \C\otimes_{\Q} \C \ (1\leq i\leq l,1\leq j\leq n)
\end{equation}
(namely, we multiply each row $X$ of $M$ and each column $Y$ of $N$ as in the usual matrix multiplication, but using the rule that two entries $x$, $y$ in $X$, $Y$ are multiplied to $x\otimes y\in \C\otimes\C$.)
In the similar manner, we can also define wedge product $M\wedge N$ of two matrices $M$, $N$. The usual (scalar) matrix product as well as this wedge product are all induced from the tensor product $M\otimes N$ via multiplication $\C\otimes\C\ra\C$ and the skew-symmetrization $\C\otimes\C \ra \C\wedge\C$, respectively. Note that for any rational $Q\in M_{m\times m}(\Q)$, one has 
\begin{equation} \label{eq:Q-move_tensor_matrix_product} 
MQ\otimes N=M\otimes QN.
\end{equation}

%%%%%%%%%%%%%%%%%%%%
\begin{lem} \label{lem:Neumann92_(3.10)}
Let $R$ and $C$ be as in (\ref{eq:U=(R,C)}). 
Suppose $x,y\in \C^{2N}$ (row vectors) satisfy $Rx^t=Ry^t=0$. Then, 
$xJ_{2N}\otimes y^t=\frac{1}{2} (Cx^t)^tJ_{2h}\otimes Cy^t$.
\end{lem}

In particular, by applying skew-symmetrization, we recover the Neumann's formula \cite[(3.10)]{Neumann92}:
\[ \text{for }x,y\in \C^{2N},\ Rx=Ry=0\ \Rightarrow\  x\wedge y=\frac{1}{2} Cx\wedge Cy \]
(He works with column vectors, and his ``wedge product'' $x\wedge y$ is our $xJ_{2N}\wedge y^t$). Also, by taking multiplication of our formula, we recover Corollary 2.4. of \cite{NeumannZagier85}; Neumann says that his formula follows from this corollary  by ``formal observation''. Our formula is a further expansion of this ``formal observation''.%%
\footnote{If we write our bilinear form $\langle x,y\rangle:=xJ_{2N}\otimes y^t$ as a sum of symmetric form $S(x,y)$ and alternating form $A(x,y)$ in the standard way, $S(x,y)$ is the Neumann's ``wedge product'' $x\wedge y$, and $A((x_1,x_2),(y_1,y_2))=(x_1\otimes y_2^t+y_2\otimes x_1^t) -(x_2\otimes y_1^t+y_1\otimes x_2^t)$, which is non-trivial in general; thus our formula is stronger than Neumann's one. But, for the purpose of showing vanishing of complex Dehn invariant, we just need $x=y$, so the Neumann's formula suffices.}

\begin{proof}
This is a re-enactment of the proof of Corollary 2.4. of \cite{NeumannZagier85}, replacing the usual matrix multiplications by matrix tensor multiplications appropriately.
For any matrix $A$ with $2N$ columns we shall denote by $[A]\subset \C^{2N}$ the subspace of $ \C^{2N}$ generated by the rows of $A$. Also, we let $(-)^{\perp}$ denote the orthogonal complement of $(-)\subset \C^{2N}$ with respect to the symplectic form $\frac{1}{2}xJ_{2N}y^t$, where $J_{2N}=\left(\begin{array}{cc} 0 & I_N \\ -I_N & 0 \end{array}\right)$. 
 
We have $0=Rx^t=R J_{2N}(x J_{2N})^t$, i.e.  $xJ_{2N}\in [R]^{\perp}$.
Since $[R]^{\perp}=[U]$ (\cite[Prop.2.3]{NeumannZagier85}), there exists $z\in \C^{N+2h}$ such that $xJ_{2N}=zU$. Also, $UJ_{2N}C^t=2  \left(\begin{array}{c} J_{2h} \\ 0 \end{array}\right)$ by \cite[Thm.2.2]{NeumannZagier85} (``Neumann-Zagier symplectic relation''), so $xC^t=-xJ_{2N}^2C^t=-2z \left(\begin{array}{c} J_{2h} \\ 0 \end{array}\right)$. Now, 
\begin{align*}
\frac{1}{2} xC^t J_{2h}\otimes Cy^t=& -z \left(\begin{array}{c} J_{2h} \\ 0 \end{array}\right) J_{2h}\otimes Cy^t= z\otimes \left(\begin{array}{c} C \\ 0 \end{array}\right) y^t \\
=& z\otimes \left(\begin{array}{c} C \\ R \end{array}\right) y^t =z\otimes U y^t=zU\otimes y^t=z J_{2N}\otimes y^t.\ \qedhere
\end{align*}
\end{proof}

%%%%%%%%%%%%%%%%%%%%%%%%%%%%%%%%%%%%%%%%
\subsection{Complex volume of hyperbolic $3$-manifolds as Bloch regulator}
The complex volume of a complete hyperbolic three-manifold is expressed in terms of Bloch regulator \cite[Lecture6]{Bloch00}. For the brach of $\log(-)$, we take the standard one with branch cut $(-\infty,0]$. 

The Bloch regulator is the map $\C-\{0,1\} \ra \C\wedge_{\Z}\C$ defined by
\begin{align} \label{eq:Bloch_regulator}
\tilde{\rho}(z) =& \frac{\log(z)}{2\pi i}\wedge\frac{\log(1-z)}{2\pi i}+ 1\wedge \frac{1}{(2\pi i)^2} \int_0^z\left(\frac{\log(1-t)}{t}+\frac{\log(t)}{1-t}\right) dt.
\end{align}
This expression is interpreted as follows: given the choice of the standard branch of $\log(z)$, a value for any of the functions $\log(z)$, $\log(1-z)$ and $\int_0^z\left(\frac{\log(1-t)}{t}+\frac{\log(t)}{1-t}\right) dt$ is determined by the choice of a path from $0$ to $z$ in $\C-\{0,1\}$. In more detail, we take a path made up of the straight path from $0$ to $\frac{1}{2}$ followed by a path from $\frac{1}{2}$ to $z$ and continue analytically the (standard) branches of $\log(t)$, $\log(1-t)$ along that path (picking up their values at $z$), while integrating the function in the integrand thus determined along the same path. One can prove (cf. proof of Lemma 6.1.1 of \cite{Bloch00}) that this gives a well-defined (continuous) map $\C-\{0,1\} \ra \C\wedge_{\Z}\C$ (i.e. the value is independent of the choice of branch and path): although Bloch's original regulator (\ref{eq:original_Bloch_regulator}) is different from the above (his regulator is valued in $\C\otimes_{\Z}\C^{\ast}$), the argument of \textit{loc. cit.} continues to apply in our case, since we are working with wedge products. 
Also, the Bloch regulator can be regarded as an ``analytic mapping'' from $\C-\{0,1\}$ to $\C\wedge_{\Z}\C$, in the sense that locally on an open neighborhood $U$, it is an element of $\mcO(U)\wedge\mcO(U)$ for the ring $\mcO(U)$ of analytic functions on $U$.

The function appearing in the second term of (\ref{eq:Bloch_regulator}) (defined by the same procedure)
\begin{align} \label{eq:Rogers_dilogarithm}
\mcR(z) =& -\frac{1}{2} \int_0^z\left(\frac{\log(1-t)}{t}+\frac{\log(t)}{1-t}\right) \\ 
=&\ \Li_2(z)+\frac{1}{2}\log(z)\log(1-z). \nonumber
\end{align}
is called \emph{Rogers dilogarithm function}, where $\Li_2(z)$ is the dilogarithm function which is defined on $|z|<1$ by $\sum_{n=1}^{\infty} \frac{z^n}{n^2}$ and analytically continued to $\C-[1,\infty)$ using the expression \begin{equation} \label{eq:Li_2} \Li_2(z)=-\int_0^z\log(1-t)\frac{dt}{t} \end{equation}
and the standard branch of $\log(z)$; so, $\mcR(z)$ can be regarded as an analytic function on $\C- (-\infty,0]\cup [1,\infty)$. 
However, the Rogers dilogarithm itself (so, the two terms of the Bloch regulator individually) does not extend to $\C-\{0,1\}$ as a continuous single-valued function (in the case of the Bloch regulator, the ambiguities in the two terms in (\ref{eq:Bloch_regulator}) created by monodromy cancel each out); nevertheless, its restriction to the interval $(0,1)$ can be extended to $\R$ as a continuous (even analytic except at the two points $\{0,1\}$) real-valued function (\cite[II.1.A]{Zagier07}).

For a field $k$, the pre-Bloch group $\mcP(k)$ is the abelian group generated by symbols $[z]$, $z\in k-\{0,1\}$, subject to the relations:
\begin{align}
 \left[ x \right]-\left[ y \right] +\left[ \frac{y}{x} \right] - & \left[ \frac{1-x^{-1}}{1-y^{-1}} \right] +\left[ \frac{1-x}{1-y} \right]=0  \label{eq:five-term} \\
 [x]+[1-x]=&[x]+[\frac{1}{x}]=0.  \label{eq:6-torsion}
\end{align}
The first of these conditions is usually called the ``five term relation'' and is equivalent to
 \[ \sum_{i=0}^4 (-1)^i \left[ [z_0:\cdots:\hat{z}_i:\cdots:z_4] \right]=0 \]
for distinct $z_0,\cdots,z_4\in k-\{0,1\}$ (the inner brackets $[-]$ are cross-ratios).
Sometimes, one considers the group $\mcP'(k)$ similarly defined using the first condition only.
Then, it is known  (\cite[5.4,5.6,5.11]{DupontSah82})  that when the characteristic of $k$ is not $2$, the elements in $\mcP'(k)$ of the form (\ref{eq:6-torsion}) are torsions of exponent $6$ and that when $k$ is algebraically closed, $\mcP'(k)=\mcP(k)$.
The map $z\in \C-\{0,1\}\mapsto z\wedge(1-z)$ induces a homomorphism
\begin{equation} \label{eq:cx_Dehn_inv}
\delta_{\C}\, :\, \mcP(\C)\ra \C^{\times}\wedge_{\Z} \C^{\times}\ ;\  [z]\mapsto z\wedge (1-z)
\end{equation} 
(\cite[Lem.1.1]{Suslin90}). 
 The Bloch regulator (\ref{eq:Bloch_regulator}) satisfies the relation (\ref{eq:five-term}) (\cite[p.173]{DupontSah82}), so it induces a homomorphism (again denoted by $\rho$) 
 \[ \tilde{\rho}:\mcP(\C)\ra \Lambda^2_{\Z}(\C).\] 
 
In the following, we normalize $\tilde{\rho}$ by $\rho:=\Xi \circ \tilde{\rho}$, where $\Xi(z\wedge w)=2\pi i z\wedge 2\pi i w:  \Lambda^2_{\Z}(\C)\ra  \Lambda^2_{\Z}(\C)$: 
\begin{equation} \label{eq:normalized_Bloch-regulator}
\rho(z)= \log(z) \wedge \log(1-z) + (2\pi i)\wedge \frac{1}{2\pi i} \int_0^z\left(\frac{\log(1-t)}{t}+\frac{\log(t)}{1-t}\right) dt.
\end{equation}

We recall the commutative diagram with exact rows, where the upper row is (a part of) the Bloch-Wigner exact sequence \cite[(1.7)]{Dupont87}, \cite[Prop.4.17]{DupontSah82}:
\begin{equation} \label{eq:Bloch-Wigner_exact_seq}
\xymatrix{ 0 \ar[r] & H_3(\SL_2(\C)^{\delta},\Z)/(\Q/\Z) \ar[d]^c \ar[r]^(.7){\sigma} & \mcP(\C) \ar[d]^{\rho} \ar[r]^{\delta_{\C}} &  \Lambda^2_{\Z}(\C^{\times})  \ar[d]^{=} \\
 0 \ar[r] & \C/\Q(2) \ar[r]^{2\pi i \wedge \frac{1}{2\pi i}\mathrm{id}} & \Lambda^2_{\Z}(\C) \ar[r]^{e} & \Lambda^2_{\Z}(\C^{\times}) }
\end{equation}
Here, $H_3(\SL_2(\C)^{\delta},\Z)$ is the Eilenberg-MacLane integral group homology of the discrete group $\SL_2(\C)^{\delta}$ (for a Lie group $G$, $G^{\delta}$ denotes the same group with discrete topology); $\Q/\Z$ sits inside this homology group and see \cite[(1.4)]{Dupont87} for the definition of $\sigma$. Also, $e(z\wedge w)=\exp(z)\wedge \exp(w)$, and $c$ is induced by $\rho\circ \sigma$ by commutativity. To see the injection $2\pi i \wedge \frac{1}{2\pi i}\mathrm{id}:\C/\Q(2) \hra  \Lambda^2_{\Z}(\C)$, we observe (using that $\C$ is uniquely divisible) the canonical identifications
\begin{equation} \label{eq:CwedgeC}
\C\wedge_{\Z}\C=\C\wedge_{\Q}\C,\quad \C/\Z(k) \wedge_{\Z} \C/\Z(k) =\C/\Q(k) \wedge_{\Q} \C/\Q(k)\ (k\in\Z),
\end{equation}
(which hold by divisibility of $\C$, $\C/\Z(k)$) and $2\pi i \wedge \frac{1}{2\pi i} (\Q(2))=2\pi i\wedge \Q 2\pi i=\Q\cdot 2\pi i\wedge 2\pi i=0$.

Therefore, $\rho:\mcP(\C)\ra \Lambda^2_{\Z}(\C)$ induces a map from the \emph{Bloch group} $ \mcB(\C):=\mathrm{Ker}(\delta_{\C})$ to $\C/\Q(2)$:
 \begin{equation} \label{eq:Bloch_regulator_on_Bloch_gp}
 \rho: \mcB(\C) \ra \C/\Q(2),
\end{equation}
which we once again denote by $\rho$ and call (normalized) Bloch regulator.

For the unnormalized Bloch regulator $\tilde{\rho}$ (\ref{eq:Bloch_regulator}), there is an analogue of
the Bloch-Wigner exact sequence (\ref{eq:Bloch-Wigner_exact_seq}), where the bottom exact sequence is replaced by $0\ra \C/\Q\stackrel{1\wedge \mathrm{id}}{\ra} \Lambda_{\Z}^2(\C) \stackrel{\tilde{e}}{\ra} \Lambda_{\Z}^2(\C^{\times})$ with $\tilde{e}(z\wedge w)=\exp(2\pi iz)\wedge \exp(2\pi iw)$ (\cite[(1.7)]{Dupont87}); so, we obtain the unnormalized Bloch regulator $\tilde{\rho}: \mcB(\C) \ra \C/\Q$.
It is easy to see that there exists the relation
\begin{equation} \label{eq:rho=(2pii)^2tilde{rho}}
\rho=(2\pi i)^2 \tilde{\rho}.
\end{equation}
under the identification $\C/\Q\isom \C/(2\pi i)^2\Q:x\mapsto (2\pi i)^{2}x$.

Also, via $\sigma$, we identify $\rho$ (\ref{eq:Bloch_regulator_on_Bloch_gp}) as an element of 
\[ c\in H^3(\SL_2(\C)^{\delta},\C/\Q(2))=\Hom(H_3(\SL_2(\C)^{\delta},\Z),\C/\Q(2)).\]

The map $\delta_{\C}$ is often called \emph{complex Dehn invariant} map, because the usual Dehn invariant of the ideal tetrahedron $\Delta(z):=\Delta(\infty,0,1,z)$ with shape parameter $z=[\infty:0:1:z]$ equals twice the ``imaginary part'' of $\delta_{\C}([z])$; the ``imaginary part'' of an element of $\C/\Z(1)\wedge\C/\Z(1)$ refers to the component corresponding to the summand $(\R\otimes \R/\Z(1))$ in the decomposition
\[ \C/\Lambda\wedge\C/\Lambda \cong (\R\oplus \R/\Lambda)\wedge  (\R\oplus \R/\Lambda)  \cong \left[(\R\wedge \R)\oplus (\R/\Lambda\wedge \R/\Lambda)\right] \oplus  (\R\otimes \R/\Lambda) \]  
with $\Lambda=\Z(1)$ (\cite[2.5]{Neumann98}).

For a complete hyperbolic three-manifold $M$ of finite volume, let $\rho_0:\pi_1(M)\ra \PSL_2(\C)$ be the holonomy representation, unique up to conjugacy, of the complete hyperbolic structure of $M$: $M\simeq \mathbb{H}^3/\rho_0(\pi_1(M))$. The \emph{invariant trace field} of $M$ is the subfield of $\C$ generated by squares of the traces of elements of $\rho_0(\pi_1(M))\subset \PSL_2(\C)$; this field is known to be a subfield of $\Qb\subset\C$.

%%%%%%%%%%%%%%%%%%%%
\begin{thm} \label{thm:2Complex_volume=Bloch_regulator}
Let $M$ be a complete hyperbolic three-manifold of finite volume. 

(1) The element $\sum_i[z_i]\in \mcP(\C)$ attached to any ideal triangulation $M= \Delta(z_1)\cup \cdots \cup \Delta(z_N)$ of $M$ has zero complex Dehn invariant: $\delta_{\C}(\sum_i [z_i])=0$.

(2) The element $\sum_i [z_i]$ of $\mcB(\C)$ is independent of the choice of ideal triangulation of $M$; we denote it by $\beta(M)$. Moreover, the image of $\beta(M)$ in $\mcB(\C)\otimes\Q$ lies in $\mcB(k(M))\otimes\Q$.

(3) We have
\[ 2\Vol_{\C}(M)=\sqrt{-1} \rho(\beta(M)) \mod \sqrt{-1}\pi^2\Q,\]
where $\rho$ is the (normzalized) Bloch regulator (\ref{eq:Bloch_regulator_on_Bloch_gp}). 
So, by Theorem \ref{thm:PSL_2-CS-inv=C-vol:cusped},
\begin{equation} \label{eq:2PSL_2-CS-inv=rho(beta(M))}
2\CS_{\PSL_2}(M)= \rho(\beta(M))\ \text{ in }\C/\Q(2).
\end{equation} 
\end{thm}

\begin{proof}
(1) This has been known for long, first proved by Thurston (see Introduction of  \cite{NeumannYang99}), and different proofs have been given, e.g. \cite[$\S$5]{NeumannYang99}. 
We give a proof due to Neumann \cite[p.255-256]{Neumann92} (using Lemma \ref{lem:Neumann92_(3.10)}). Let $Z^0_1:=(\log z^0_1,\cdots, \log z^0_N)\in \C^N$, $Z^0_2:=(\log (1-z^0_1),\cdots, \log (1-z^0_N))\in \C^N$. Choose a rational solution $Q\in \Q^{2N}$ to (\ref{eq:CCequation}), and put $W:=Z^0-Q\pi\sqrt{-1}\in \C^{2N}$; so $W=Z^0$ in $\C/\Q(1)$.
Then, adopting our matrix tensor/wedge product notation, 
\[ Z^0_1 \wedge (Z^0_2)^t =Z^0 J_{2N}\otimes (Z^0)^t= W J_{2N}\otimes W^t \stackrel{(\ast)}{=}\ \frac{1}{2} (CW^t)^tJ_{2h}\otimes CW^t =0\]
Here, $(\ast)$ is Lemma \ref{lem:Neumann92_(3.10)} and the last equality follows from the fact that the entries of $CW^t\in \C^{2h}$ are holonomies of the meridians and the longitudes (\cite[(2.6), (3.1), (3.2)]{Neumann92}), so are zeros, as we assume $M$ to be complete. Now, in view of (\ref{eq:CwedgeC}), this proves the claim since $\delta_{\C}(\sum_i [z_i])$ is the image of $Z^0_1 \wedge (Z^0_2)^t$ under the homomorphism $\exp\wedge\exp: \C/\Z(1)\wedge\C/\Z(1)\ra \C^{\times}\wedge\C^{\times}$ .

(2) The first statement is Proposition 4.3 of \cite{NeumannYang99}; see also \cite{Cisneros-MolinaJones03} for another proof based on a different construction of $\beta(M)$. The second statement is Theorem 1.2. of \textit{ibid.}

(3) This is Theorem 1.3 of \cite{NeumannYang99}, which was proved (with a different formulation) in \cite[Thm.1 $\&$ (3.8)]{Neumann92}. It seems worth pointing out three main ingredients of that proof: (i) the fact that the Cheeger-Chern-Simon class $\hat{C}_2\in H^3(\PSL_2(\C)^{\delta},\C/\Q(2))$ (associated with the $2$nd Chern polynomial $C_2$) equals half the Bloch regulator $\rho$  \cite[Thm.1.8]{Dupont87}, (ii) Yoshida's analytic formula for complex volumes \cite[Thm.2]{Yoshida85}, and finally (iii)  Neumann's explicit description of the (Riemannian) volume of (not-necessarily complete) hyperbolic three-manifolds in terms of shape parameters and the values of ``Bloch-Wigner dilogarithm (the imaginary part of the Bloch regulator) at these points, together with a rational solution to (\ref{eq:CCequation}) \cite{Neumann92}. 
\end{proof}

For later use, we recall one functional equation (among many) of $\tilde{\rho}(z)$ (equiv. of $\rho(z)$):
\begin{align}
\Li_2(1-z)=&-\Li_2(z)+\frac{\pi^2}{6} - \log(z)\otimes \log(1-z), \nonumber \\
\tilde{\rho}(1-z)=& -\tilde{\rho}(z) \quad (z\in \C-(-\infty,0]\cup [1,\infty)) \label{eq:functional_eq_Bloch-regulator}
\end{align}
The first relation is well-known \cite[$\S$2]{Zagier07} and the second one is deduced from the first one.

%%%%%%%%%%%%%%%%%%%%%%%%%%%%%%%%%%%%%%%%
%%%%%%%%%%%%%%%%%%%%%%%%%%%%%%%%%%%%%%%%
\section{Big and skew-symmetric periods of framed mixed Tate motives}

%%%%%%%%%%%%%%%%%%%%
\subsection{Period matrix of splitted Hodge-Tate structure}

A period matrix of a pure $\Q$-Hodge structure $H$ is a transition matrix for two bases of $H_{\C}$, one a $\Q$-basis $\{v_i\}$ of $H$ and the other a $\C$-basis $\{e_i\}$ of $H_{\C}$, often the latter basis being adapted to the Hodge filtration in the sense that there exists a sequence $\{r_1,r_2,\cdots,r_m\}$ of positive integers such that the flag of subspaces in the Hodge filtration matches that of the subspaces spanned by the increasing subsets $\{e_{1},\cdots, e_{r_1+\cdots +r_m}\}_{m}$ of basis elements. When $H$ is the singular cohomology space $H=H^i_{sing}(X(\C),\Q)$ (of some degree) of a smooth projective variety (or a pure motive) over a field $k\subset \C$, one usually chooses a basis of the algebraic deRham cohomology group $H^i_{\dR}(X/k)$, regarded as a basis of $H_{\C}$ via the canonical isomorphism $\Phi:H^i_{sing}(X(\C),\Q)\otimes_{\Q}\C=H^i_{\dR}(X/k)\otimes_k\C$. Namely, the period matrix equals the matrix expression with respect to some bases for the isomorphism $\Phi$, an isomorphism defined over $\C$ between two spaces $H$, $H^i_{\dR}(X/k)$ which are defined over different subfields $\Q,k$ of $\C$. 

For a pure Hodge structure, any period matrix thus defined determines the Hodge structure. While there does not seem to exist an efficient matrix expression for general mixed Hodge structures, the situation for Hodge-Tate structure is just as good as in pure cases, as we explain now. The basic idea is to exploit splittings, defined over $\Q$ and $\C$, of the weight filtration of Hodge-Tate structure.

Let $E\subset E'$ be a field extension. An $E'$-\emph{splitting} of an increasing filtration $W^{\bullet}H$ on a $E$-vector space $H$ is a choice of an isomorphism $\varphi:\bigoplus_k \gr^W_k H\otimes_EE' \isom H_{E'}$ compatible with the weight filtrations on both sides (i.e. for every $k\in\Z$, $\varphi$ maps $\bigoplus_{l\leq k} \gr^W_l H\otimes_EE' $ to $W_kH\otimes_EE'$ and the induced (by quotient) automorphism of $\gr^W_k H\otimes_{E}E'$  is the identity map). Any $E$-splitting $\varphi$ gives an $E'$-splitting $\varphi_{E'}$ by scalar extension.
The same discussion also applies to decreasing filtration.

%%%%%%%%%%%%%%%%%%%%
\begin{lem} \label{lem:splitting_of_HTS}
The following conditions on a mixed $\R$-Hodge structure $H$ are equivalent:

(i) $H$ is an ($\R$-)Hodge-Tate structure, i.e. for all $k\in\Z$, $\gr^W_{2k-1}=0$ and $\gr^W_{2k} H$ is a direct sum of copies of $\R(-k)$.

(ii) For every $p\in\Z$, the canonical map 
\begin{equation} \label{eq:HT(ii)}
F^pH_{\C}\cap W_{2p}H_{\C}\ra \gr^W_{2p}H\otimes_{\R}\C 
\end{equation}
 is an isomorphism.

(iii) For every $p\in\Z$, the canonical map 
\begin{equation} \label{eq:HT(iii)} 
F^pH_{\C}\cap W_{2p}H_{\C}\ra \gr^{p}_FH_{\C}
\end{equation} 
is an isomorphism.

(iv) The canonical map
\begin{equation} \label{eq:HT(iv)}
\bigoplus_p F^pH_{\C}\cap W_{2p}H_{\C} \ra H_{\C}
\end{equation}
is an isomorhpism.

(iv') There exists a decomposition:
\begin{equation*} \label{eq:HT(iv')}
H_{\C}=\bigoplus_p (H_{\C})_{p}
\end{equation*}
such that the complex weight and Hodge filtrations are given by
\begin{align} 
W_{2k} &=\bigoplus_{p\leq k} (H_{\C})_{p},  \label{eq:can_SP_WF} \\
F^{k} &=\bigoplus_{p\geq k} (H_{\C})_{p}. \label{eq:can_SP_HF}
\end{align}
\end{lem}

\begin{proof}
For short, we write $F^p$, $W_{2p}$ and $\gr^W_{2p}$ instead of $F^pH_{\C}$, $W_{2p}H_{\C}$ and $\gr^W_{2p}H_{\C}$, respectively.

(i) $\Rightarrow$ (ii): 
We have $F^p\gr^W_{2k}:=\mathrm{Im}(F^p\cap W_{2k} \ra \gr^W_{2k})$. Since $\gr^W_{2p}=F^p\gr^W_{2p}$, the canonical map $F^p\cap W_{2p} \ra \gr^W_{2p}$ is surjective. Since $F^{p+1}\gr^W_{2p}=0$, we have $F^{p+1}\cap W_{2p}\subset W_{2p-2}$ for all $p\in\Z$ (so that $F^{p+1}\cap W_{2p}=F^{p+1}\cap W_{2p-2}$). By the same reason, $F^{p+1}\cap W_{2p-2}\subset F^{p}\cap W_{2p-2} \subset W_{2p-4}$, and we obtain
\[ F^{p+1}\cap W_{2p}=F^{p+1}\cap W_{2p-2}=F^{p+1}\cap W_{2p-4} \ (\subset F^{p-1}\cap W_{2p-4}\subset W_{2p-6}). \]
Continuing with this process, we see that $F^{p+1}\cap W_{2p}=F^{p+1}\cap W_{2(p-N)}$ for every $N\geq1$, and taking $N\gg1$, that the canonical map $F^p\cap W_{2p} \ra \gr^W_{2p}$ is injective. 

The implication (ii) $\Rightarrow$ (i) and the equivalences (ii) $\Leftrightarrow$ (iv) $\Leftrightarrow$ (iv') are obvious.

(ii) $\Rightarrow$ (iii): Let $k\in\Z_{\geq 0}$. Since the natural map $F^{p-k}\cap W_{2p} \ra \gr^W_{2p}$ has the same image as its subspace $F^p\cap W_{2p}$ by assumption, $F^{p-k}\cap W_{2p}$ is spanned by the two subspaces $F^p\cap W_{2p}$, $F^{p-k}\cap W_{2p-2}$ whose intersection $F^p\cap W_{2p-2}$ is zero, again by assumption. Namely, we have a direct sum decomposition: 
\[ F^{p-k}\cap W_{2p} =F^p\cap W_{2p} \oplus F^{p-k}\cap W_{2p-2}.\] 
From this, we deduce that for all $p\in\Z, k\in\Z_{\geq0}$, 
\begin{align*} 
F^{p}\cap W_{2(p+k)} & = F^{p+k}\cap W_{2(p+k)} \oplus F^{p}\cap W_{2(p+k-1)} \\
& = F^{p+k}\cap W_{2(p+k)} \oplus F^{p+k-1}\cap W_{2(p+k-1)} \oplus F^{p}\cap W_{2(p+k-2)} =\cdots \\
& = F^{p+k}\cap W_{2(p+k)} \oplus F^{p+k-1}\cap W_{2(p+k-1)} \oplus \cdots \oplus F^{p}\cap W_{2p}.
\end{align*}
Taking $k$ sufficiently large, we obtain
\[ F^p= \bigoplus_{k\geq0} F^{p+k}\cap W_{2(p+k)}.\]
which proves the required implications. 

(iii) $\Rightarrow$ (ii): Both are statements about relative position of two filtrations $F^{\bullet}$, $(W_{\bullet})_{\C}$ defined over $\C$ of a $\C$-vector space $H_{\C}$ (for $(W_{\bullet})_{\C}$, we can ignore odd-degree subspaces and reindex the even-degree subspaces so that the new $W_{k}$ is the original $W_{2k}$). Accordingly one can resort to the symmetry of the statements, by using the standard trick of changing a decreasing (resp. increasing) filtration to an increasing (resp. decreasing) filtration: $F'_{k}:=W^{-2k}$, $W'^{2p}:=F_{-p}$.
\end{proof}

In particular, the weight filtration of every $\R$-Hodge-Tate structure 
has a canonical $\C$-splitting: 
\begin{equation} \label{eq:ST}
S_{HT}:\bigoplus_p \gr^W_{2p}H\otimes_{\R}\C \lisom \bigoplus_p (H_{\C})_{p}=H_{\C}.
\end{equation}
Note that while one can always find a (non-canonical) $\R$-splitting of a filtered $\R$-vector space, this $\C$-splitting of a $\R$-Hodge-Tate structure, being constructed by means of Hodge filtration, is canonical.

\begin{defn} \label{defn:good_basis}
Let $V=\bigoplus_{k\in\Z} V_k$ be a graded $\Q$-vector space. A \emph{good} basis of $V$ is a $\Q$-basis $\{v_k\}$ of $V$ adapted to the splitting (i.e. each $v_k$ is an element of some $V_{n_k}$), and a \emph{fine} change of good basis is a transformation $\varphi$ from one good basis $\{v_k\}$ to another good basis $\{\varphi(v_k)\}$ such that for all $k\in\Z$, $\varphi(v_k)\in V_{n_k}$ when $v_k \in V_{n_k}$.
\end{defn}

%%%%%%%%%%%%%%%%%%%%
\begin{defn} \label{defn:Period_matrix}
Let $H$ be a $\Q$-Hodge-Tate structure and $\varphi:\bigoplus_k \gr^W_{2k} H \isom H$ a $\Q$-splitting of the weight filtration of $H$. For each $k\in\Z$, choose a $\Q$-basis $\{v^{(k)}_i ; i=1,\cdots, r_k\}\ (r_k\in\Z_{\geq 0})$ of $\gr^W_{-2k}H$, so that $\{v^{(k)}_i\}_{i,k}$ is a good basis of $\WgrbH:=\bigoplus_k \gr^W_{2k} H$. 

(1) The \emph{period matrix} $\mcM$ of the splitted $\Q$-Hodge-Tate structure $(H,\varphi)$ with respect to a good basis $\{v^{(k)}_i\}_{i,k}$ of $\gr^W_{\bullet}H$ is the transition matrix expressing the basis $\mcB_1=\{\lambda^{(k)}_i\}$ of $H_{\C}$ in terms of the basis $\mcB_2=\{e^{(k)}_i\}$, where
\[  \lambda^{(k)}_i:=\varphi(v^{(k)}_i)\in H,\quad e^{(k)}_i:=(2\pi i)^{-k}S_{HT}(v^{(k)}_i)\in H_{\C} \]
That is, $\mcM$ is the matrix%%
\footnote{Our convention in this section for matrix expression $M$ of an endomorphism $\phi\in\End(V)$ with respect to a basis $\{v_1,\cdots,v_n\}$ of $V$ is that $\phi(v_i)=\sum_j M_{ij}v_j$ (i.e. transpose of the usual matrix expression). See Example below.}
of the \emph{period operator}
\begin{equation} \label{defn:period_operator}
\Phi(\varphi):=J_H\circ S_{HT}^{-1}\circ\varphi_{\C} \ \in \Aut(\WgrbH_{\C})
\end{equation}
with respect to the (same) basis $\{v^{(k)}_i\}_{i,k}$, where $J_H$ is the endomorphism of $\WgrbH_{\C}$ that acts as scalar multiplication with $(2\pi i)^{k}$ on $\gr^W_{-2k}H_{\C}$ for all $k\in\Z$.

(2) The \emph{Goncharov period matrix} $\mcMG$ of is the matrix of the \emph{Goncharov period operator}
\begin{equation} \label{defn:Goncharov_period_operator}
\Phi(\varphi)_{\Gon}:=S_{HT}^{-1}\circ\varphi_{\C} \ \in \Aut(\WgrbH_{\C})
\end{equation}
with respect to the basis $\{v^{(k)}_i\}_{i,k}$ (or equiv. the transition matrix expressing the basis $\mcB_1=\{\lambda^{(k)}_i\}$ of $H_{\C}$ in terms of the basis $\mcB_2'=\{S_{HT}(v^{(k)}_i)\}$.)
\end{defn}
 
%%%%%%%%%%%%%%%%%%%%
\begin{rem} \label{rem:correct_big_period}
Our terminology is different from that used by Goncharov \cite{Goncharov99}.
What we call period matrix and denote by $\mcM$ is called canonical period matrix and denoted by $\tilde{\mcM}$ by him, whereas what we call Goncharov period matrix/operator is called period matrix/operator by him: 

\begin{table} [h]
\centering
\small
%\resizebox{\columnwidth}{!}{%
\begin{tabular}{clcl}
We & & Goncharov \\
Period operator/matrix & $\mcM$ & Canonical period operator/matrix & $\tilde{\mcM}$ \\
Goncharov period operator/matrix & $\mcM_{\Gon}$ & Period operator/matrix & $\mcM$
\end{tabular}%}
\end{table}

Our choice of names/notations are due to two things: first,  the ``period matrix' which is more commonly used (especially, to describe Hodge-Tate structures, as explained below), thus seems to be called by this conventional name is the period matrix in our terminology, and as such there is a danger of confusion if one follows Goncharov's terminology. Secondly, Goncharov period operator(matrix) has its own significance and uses, thus we believe it to merit a distinguished name.
\end{rem}

 For our discussion of periods, it is convenient to use the following notational rules for matrices
%%%%%%%%%%%%%%%%%%%%
\begin{notation}
(i) For $A,B\in \mrM_{N\times N}(\C)$, $A\cdot B$ is the matrix multiplication. 

(ii) For $A=(A_{ij})\in \mrM_{N\times N}(\C)$ and $x=(x_1,\cdots,x_N)^t \in \mrM_{N\times 1}(W)$  for a $\C$-vector space $W$ (such as $\WgrbH_{\C}$), $A\cdot x\in \mrM_{N\times 1}(W)$ is the matrix multiplication, i.e. $(A\cdot x)_{i}=\sum_jA_{ij}x_j$. 

(iii) For $\phi\in \End(W)$ and $x=(x_1,\cdots,x_N)^t \in \mrM_{N\times 1}(W)$, $\phi(x):=(\phi(x_1),\cdots,\phi(x_N))^t \in \mrM_{N\times 1}(W)$. 
\end{notation}

\textbf{Example} Suppose $W$ is a $\C$-vector space with a basis $\mcB=\{w_1,\cdots,w_N\}$ and  $\phi_1,\phi_2\in\End(W)$. The matrix $M_i$ of $\phi_i$ with respect to $\mcB$ in our convention is such that for $w:=(w_1,\cdots,w_N)^t\in \mrM_{N\times 1}(W)$, $\phi_i(w)=M_i\cdot w$, and thus 
\begin{equation} \label{eq:row-matrices_involution}
\phi_1\circ\phi_2(w)=\phi_1(M_2\cdot w)=M_2\cdot \phi_1(w)=M_2\cdot M_1\cdot w.
\end{equation}
Here we are using the notations introduced above.

Now, writing simply $\{v^{(k)}_i\}_{i,k}=\{v_1,\cdots,v_N\}$ and setting $V:=(v_1,\cdots,v_N)^t\in M_{N\times 1}(\WgrbH_{\C})$, by definition of $\mcM$, we have $\varphi_{\C}(V)=\mcM\cdot S_{HT}\circ J_H^{-1}(V)$. Since $\mcM\cdot S_{HT}\circ J_H^{-1}(V)=S_{HT}\circ J_H^{-1}(\mcM\cdot V)$, we obtain $J_H\circ S_{HT}^{-1}\circ\varphi_{\C}( V)=\mcM\cdot V$, showing that $\mcM$ is the matrix of $\Phi(\varphi)$ (\ref{defn:period_operator}).
By (\ref{eq:row-matrices_involution}), we have 
\begin{equation} \label{eq:PM=GPMJM} 
\mcM=\mcMG\cdot J_{\mcM}, \end{equation} 
where $J_{\mcM}$ is the matrix of $J_H$ with respect to the same basis $\{v^{(k)}_i\}_{i,k}$, i.e. the block-diagonal matrix of the same block-type as $\mcM$ such that the diagonal block corresponding to $\gr^W_{-2k}$ is $(2\pi i)^k \mathrm{Id}_{r_{i}}$.

Since $\varphi_{\C}$ and $S_{HT}$ are both splittings of the same filtration, the period matrix $\mcM$ of a splitted Hodge-Tate structure defined in this way is a block upper-triangular matrix such that each block-diagonal equals the identity matrix up to a power of $2\pi i$: more precisely, if we put $\pmb{\lambda}^{(k)}:={}^t(\lambda^{(k)}_1,\cdots,\lambda^{(k)}_{r_k})$ and $\mathbf{e}^{(k)}:={}^t(e^{(k)}_1,\cdots,e^{(k)}_{r_k})$ (column vectors), we have
\begin{equation}  \label{eq:CPM_HT}
\left(\begin{array}{c} \vdots \\ \pmb{\lambda}^{(0)} \\ \pmb{\lambda}^{(1)} \\ \pmb{\lambda}^{(2)} \\ \vdots \end{array}\right) =
 \left(\begin{array}{ccccc}  \ddots & \ast & \ast & \ast & \ast \\ 0 & I_{r_0} & \ast & \ast & \ast \\ 0 & 0 & (2\pi i) I_{r_{1}} & \ast & \ast \\  0 & 0 & 0 & (2\pi i)^2 I_{r_{2}} & \ast \\ 0 & 0 & 0 & 0 & \ddots \end{array}\right)
 \left(\begin{array}{c} \vdots \\ \mathbf{e}^{(0)} \\ \mathbf{e}^{(1)} \\ \mathbf{e}^{(2)} \\ \vdots \end{array}\right). 
\end{equation}
Here, the square matrix is the period matrix $\mcM$.

From the period matrix, we can read off the Hodge-Tate structure information: $H_{\C}$ is the $\C$-vector space $\C^N\ (N=\sum_i r_i)$ with $\{ e^{(i)}_j\}_{i,j}$ as the standard basis, and $H$ is the $\Q$-vector subspace of $H_{\C}$ with a basis consisting of $\lambda^{(i)}_{j}$, the $j$-th row of the weight $-2i$ block of $\mcM$. The weight and Hodge filtration are given by
\begin{align}
W_{-2k+1}H=W_{-2k}H \, & =\ \Q\left\langle \lambda^{(i)}_{1},\cdots, \lambda^{(i)}_{r_i}\, ;\, i\geq k \right\rangle \quad \subset H \label{eq:Wt_Hod_Fil_PeriodMatrix}
\\
F^{-p}H_{\C} \, & =\ \C\left\langle e^{(i)}_{1},\cdots, e^{(i)}_{r_i}\, ;\, i\leq p \right\rangle \quad \subset H_{\C}; \nonumber
\end{align}
so, $W_{-2k}H_{\C}=\C \left\langle e^{(i)}_{1},\cdots, e^{(i)}_{r_i}\, ;\, i\geq k \right\rangle $ and
\begin{equation} \label{eq:can_splitting_MTHS}
F^{-k}H_{\C}\cap W_{-2k}H_{\C}=\C\langle e^{(k)}_1,\cdots, e^{(k)}_{r_k}\rangle.
\end{equation}

Conversely, by this recipe (\ref{eq:Wt_Hod_Fil_PeriodMatrix}), any such matrix $\mcM$ (\ref{eq:CPM_HT}) determines a $\Q$-Hodge-Tate structure $H$; every $H$ constructed in this way is already equipped with a $\Q$-splitting (provided by the map $v^{(i)}_j\mapsto\lambda^{(i)}_j$).
Also, note that with a fixed choice of a basis $\{v^{(k)}_i\}_{i,k}$ of $\gr^W_{\bullet}H$, a different choice of a $\Q$-splitting of the weight filtration of $H$ corresponds to $N\mcM$ for a block-unipotent matrix $N$ of the same block-type as $\mcM$ which has \emph{rational} coefficients.

This discussion establishes the following well-known statement (cf. \cite[Prop.9.1]{Hain94}):

\begin{prop} \label{prop:moduli_of_MTH}
The set of all Hodge-Tate structures whose weight graded quotients are isomorphic to $H_0=\bigoplus_k \Q(k)^{r_k}$ via a fixed isomorphism is the quotient $U(\Q)\backslash M(\C)$ of the set $M(\C)$ of complex upper triangular matrices of the form (\ref{eq:CPM_HT}) by the subgroup $U(\Q)$ of $\GL_r(\Q)$ of block-unipotent matrices of same block-type as (\ref{eq:CPM_HT}), where $r=\sum r_k\in\N$.

In particular, attaching to a period matrix $\left(\begin{array}{cc} 1 & \omega \\ 0 & (2\pi i)^n \end{array}\right)$ the element $\omega\ \text{mod}\,(2\pi i)^n\Q$ of $\C/(2\pi i)^n\Q$ gives a canonical isomorphism
\begin{equation} \label{eq:Ext^1_MH(Q(0),Q(n)} 
\Ext^1_{\QMH}(\Q(0),\Q(n)) = \C/(2\pi i)^n\Q \quad (n>0). 
\end{equation}
\end{prop}

\begin{proof}
We just need to prove the formula (\ref{eq:Ext^1_MH(Q(0),Q(n)}). We have
\begin{align*}
\Ext^1_{\QMH}(\Q(0),\Q(n)) & = \ U(\Q)\backslash M(\C) = \left(\begin{array}{cc} 1 & \Q \\ 0 & 1 \end{array}\right) \backslash  \left(\begin{array}{cc} 1 & \C \\ 0 & (2\pi i)^n \end{array}\right) \\
& = \ \C/ (2\pi i)^n \Q. \qquad \qedhere \nonumber
\end{align*}
\end{proof}

See Example 23 of \cite{Hain03} for an explicit geometric description of the simple extension Hodge-Tate structure corresponding to an element of $\C/(2\pi i)^n \Q$.

%%%%%%%%%%%%%%%%%%%%
\subsection{Period of framed Hodge-Tate structure: Big period and skew-symmetric period}

A period of a Hodge structure is an entry of a period matrix, so is highly choice-sensitive. But, for framed Hodge-Tate structures, Goncharov \cite{Goncharov99} found some canonically defined period (``big period''), at the cost of enlarging the domain where the period lives. 
For our purpose of interpreting Chern-Simons invariant of hyperbolic three-manifold as the (canonical) period of some Hodge-Tate structure in $\Ext_{\QHT}(\Q(0),\Q(2))$, we introduce a variant of the big period: ``skew-symmetric period''.

We give the definition of maximal period of splitted $n$-framed Hodge-Tate structures first and then that of big period of $n$-framed Hodge-Tate structures.

%%%%%%%%%%%%%%%%%%%%
\begin{defn} Let $(m,n)\in \Z^2$. \label{defn:framed_HTS}
An \emph{$(m,n)$-framing} on a Hodge-Tate structure $H$ over $\Q$ is 
a choice of a non-zero vector in $\gr^W_{-2m} H$ and a non-zero covector $\gr^W_{-2n} H\ra \Q$.
\end{defn}

We call a $\Q$-Hodge-Tate structure $H$ \emph{$(m,n)$-framed} if it is endowed with an $(m,n)$-framing.
We remark on two things. First, while this definition itself makes sense for every integer pairs $(m,n)$, the case $m>n$ will be of little use for our purpose. Secondly, a $(0,n)$-framing is often called simply $n$-framing. In literature, there exist some different conventions for the definition of $n$-framing (\cite[p.589]{Goncharov99}, \cite[Def.5.1]{Brown13}, \cite[$\S$1.2]{GoncharovZhu18}), all of which are related to each other by Tate twist with $\Q(n)$ or $\Q(-n)$: our definition of $n$-framing is the same as that used in  \cite[Def.5.1]{Brown13} and \cite[$\S$1.2]{GoncharovZhu18}, while an $n$-framed Hodge-Tate structure $H'$ in the sense of \cite{Goncharov99} equals the Tate-twist $H(-n)$ of an $n$-framed Hodge-Tate structure $H$ in our sense.

%%%%%%%%%%%%%%%%%%%%
Given a vector space $V$ over a field $k$, for $v\in v$, $f\in V^{\vee}:=\Hom(V,k)$ and $\phi\in \End(V)$, we use the notation
\[ \langle v|\phi|f \rangle:=f(\phi(v)). \]
If we choose a basis $\mcB=\{v_i\}$ of $V$ such that $v\in\mcB$ and $f\in\mcB^{\vee}$ (dual basis), and $M$ is the matrix of $\phi$ with respect to $\mcB$, then $\langle v|\phi|f \rangle$ is just the entry of $M$ in the position $(v,w)$, where $w\in \mcB$ is dual to $f$ (recall that the rows and columns of $M$ are indexed by elements of $\mcB$).

%%%%%%%%%%%%%%%%%%%%
\begin{defn} \label{defn:maximal_periods} \cite[$\S$4.1]{Goncharov99}
Let  $(H,v_m,f^n)$ be an $(m,n)$-framed $\Q$-Hodge-Tate structure, $\varphi:\bigoplus_k \gr^W_{2k} H \isom H$ a $\Q$-splitting of the weight filtration of $H$; let $\PhiG(\varphi)=S_{HT}^{-1}\circ \varphi_{\C} \in \Aut(\WgrbH_{\C})$ be the Goncharov period operator (\ref{defn:Goncharov_period_operator}) defined by $\varphi$. 

The \emph{maximal period} of the splitted $(m,n)$-framed $\Q$-Hodge-Tate structure $((H,v_m,f^n);\varphi)$ is 
\begin{equation*} 
p((H,v_m,f^n);\varphi) :=\langle v_m | \PhiG(\varphi) | f^n \rangle\ \in \C, 
\end{equation*}
where $f^n$ is extended to a map $\WgrbH\ra\C$ by setting $f^n|_{\gr^W_{-2i}}=0$ for $i\neq n$.
\end{defn}

We have more explicit descriptions for this maximal period.
\begin{lem}
(1) Let $\PhiG(m,n)$ be the $(m,n)$-component of $\PhiG(\varphi)$ in the decomposition $\End(\WgrbH)=\bigoplus_{(m,n)\in \Z^2} \Hom(\gr^W_{-2m} H_{\C}, \gr^W_{-2n} H_{\C})$; so, the maximal period $p((H,v_m,f^n);\varphi)$ equals the image of $1$ under the composite map $\C \stackrel{v_m}{\lra} \gr^W_{-2m} H_{\C} \stackrel{\PhiG(m,n)}{\lra} \gr^W_{-2n}H_{\C} \stackrel{f^n}{\lra} \C$.

The map $\PhiG(m,n)$ is given by the composite map
\[ \gr^W_{-2m} H_{\C} \stackrel{\varphi_{\C}}{\lra}  W_{-2m}H_{\C} \stackrel{(\ref{eq:can_SP_WF})}{=} \bigoplus_{p\leq -m} (H_{\C})_{p} \stackrel{pr}{\lra}  (H_{\C})_{-n}=\gr^W_{-2n}H_{\C}. \]

(3) Suppose chosen a good basis of $\mcB=\{v^{(k)}_i \}_{k,i}$ of $\WgrbH$ such that $v_m\in\mcB$ and $f^n\in\mcB^{\vee}$ (dual basis). The maximal period $p((H,v_m,f^n);\varphi)$ equals the entry of $\mcMG=\mcM J_{\mcM}^{-1}$ in the position $(v_m,v_n)$, where $v_n\in \mcB$ is dual to $f^n$. 
\end{lem}

Let $V$ be a $\Q$-vector space. For $v\in V$, $f\in V^{\vee}$ and $\phi_1, \phi_2\in \End(V_{\C})$, we define
\begin{equation} \label{eq:<v|AotimesB|f>}
\langle v|\phi_1\otimes_{\Q}\phi_2|f\rangle := \sum_k \langle v|\phi_1|f^k\rangle \otimes_{\Q} \langle v_k|\phi_2|f\rangle \ \in\ \C\otimes_{\Q}\C,
\end{equation}
where $\{v_k\}$ is an arbitrary $\Q$-basis of $V$ with dual basis $\{f^k\}$: it follows from (\ref{eq:Q-move_tensor_matrix_product}) that this expression is independent of the choice of basis $\{v_k\}$ over $\Q$ (but not necessarily so for a basis over $\C$).

%%%%%%%%%%%%%%%%%%%%
\begin{defn} \cite[$\S$4.3]{Goncharov99} \label{defn:bigPeriod}
Let $(H,v_m,f^n)$ be an $(m,n)$-framed $\Q$-Hodge-Tate structure, and let
$\PhiG(\varphi)=S_{HT}^{-1}\circ \varphi_{\C}$ be the Goncharov period operator (\ref{defn:Goncharov_period_operator}) defined by a $\Q$-splitting $\varphi:\bigoplus_k \gr^W_{2k} H \isom H$.

The \emph{big period} of $(H,v_m,f^n)$ is 
\begin{align}
\mathscr{P}_n((H,v_m,f^n)):=&\ \langle v_m| \Phi_{\Gon}(\varphi)^{-1}\otimes_{\Q} \Phi_{\Gon}(\varphi) |f^n \rangle \quad \in \C\otimes_{\Q}\C.  \label{eq:bigPeriod} \\
= &\ \sum_{v_k\in\mcB} \langle v_m|\Phi_{\Gon}(\varphi)^{-1}|f^k\rangle \otimes_{\Q} \langle v_k| \Phi_{\Gon}(\varphi)|f^n\rangle, \nonumber
\end{align}
where $\mcB=\{v_k\}$ is an arbitrary basis of $\WgrbH$ with dual basis $\mcB^{\vee}=\{f^k\}$ such that $v_m\in \mcB$ and $f^n\in\mcB^{\vee}$.
\end{defn}

\begin{lem} \label{lem:big_period=(tensor_product)_{m,n}}
Let $(H,v_m,f^n)$ be an $(m,n)$-framed $\Q$-Hodge-Tate structure. 
Choose a splitting $\varphi:\WgrbH\isom H$, a good basis $\mcB=\{v_1,\cdots,v_N\}$ of $\WgrbH$ such that $v_m\in \mcB$ and $f^n\in\mcB^{\vee}$; let $\mcM$ be the associated period matrix (\ref{defn:period_operator}).
Then, the big period $\mathscr{P}_n((H,v_m,f^n))$ equals
\[ \sum_k (2\pi i)^{m} \langle v_m|\mcM^{-1}|f^k\rangle \otimes_{\Q} (2\pi i)^{-n}\langle v_k|\mcM |f^n\rangle, \]
i.e. is the matrix tensor-product (\ref{eq:tensor_mult_matrices}) of the two matrices in $\mrM_{1\times N}(\C)$, $\mrM_{N\times 1}(\C)$: 
\begin{equation*} \label{eq:big_period=(tensor_product)_{m,n}}
(\tpi)^{m}\left( (\mcM^{-1})_{m,1},\cdots,(\mcM^{-1})_{m,N}\right),\ (\tpi)^{-n}\left(\mcM_{1,n},\cdots, \mcM_{N,n}\right)^t.
\end{equation*}
\end{lem}

This is obvious from definition and the relation $\mcMG=\mcM J_{\mcM}^{-1}$ (\ref{eq:PM=GPMJM}).

%%%%%%%%%%%%%%%%%%%%
\begin{prop} \label{prop:Goncharov99;4.2}
The big period $\mathscr{P}_n$ does not depend on the choice of the splitting $\varphi$.
\end{prop}

\begin{proof}
This is Prop. 4.2 of \cite{Goncharov99}. Indeed, this follows from Lemma \ref{lem:big_period=(tensor_product)_{m,n}}, the fact that the period matrix corresponding to a different splitting is $N\mcM$ for an upper-triangular block-unipotent matrix $N$ with rational entries, and the identity (\ref{eq:Q-move_tensor_matrix_product}). 
\end{proof}

%%%%%%%%%%%%%%%%%%%% \subsection{Periods of framed Hodge-Tate structures 2: skew-symmetric period}

We introduce another period which is derived from the big period. 

%%%%%%%%%%%%%%%%%%%%
\begin{defn} \label{defn:skew-symmetric_period}
The \emph{skew-symmetric period} of an $n$-framed $\Q$-Hodge-Tate structure $(H,v_0,f^n)$ is the image of the big period $\mathscr{P}_n((H,v_0,f^n))$ (\ref{eq:bigPeriod}) under the skew-symmetrization $\Alt:\C\otimes_{\Q}\C \ra \C\wedge_{\Q}\C\ :\ x\otimes y\mapsto x\wedge y$:
\begin{equation} \label{eq:skew-symmetricPeriod}
\mathscr{A}_n((H,v_0,f^n)):=  \Alt \left( \langle v_0| \PhiG(\varphi)^{-1}\otimes_{\Q} \PhiG(\varphi) |f_n\rangle \right)  \in \C\wedge_{\Q}\C
\end{equation}
(for any choice of a $\Q$-splitting $\varphi:\WgrbH\isom H$).
 \end{defn}

%%%%%%%%%%%%%%%%%%%%
\subsection{Example: Polylogarithm Hodge-Tate structure \cite{BeilinsonDeligne94}}
For $k\in\N$, the $k$-th polylogarithm function $\Li_k(z)$ is an analytic function on $|z|<1$ defined by
\begin{equation} \label{eq:polylogarithm_Li_k}
\Li_k(z)=\sum_{n=1}^{\infty}\frac{z^n}{n^k}.
\end{equation}
These are iterated integrals of differential forms with logarithmic poles on $\mathbb{P}^1(\C)\backslash\{0,1,\infty\}$:
\begin{align*}
\Li_1(z) &=-\log(1-z)=\int_0^z\frac{dt}{1-t}, \\
\Li_{k+1}(z) &= \int_0^z\Li_k(t)\frac{dt}{t}\ \quad (k\geq1).
\end{align*}
This expression as iterated integrals shows that $\Li_k(z)$, defined by (\ref{eq:Li_k}) for $|z|<1$, can be analytically continued as a \emph{multivalued} function on $\mathbb{P}^1(\C)\backslash\{0,1,\infty\}$. 

The $n$-th polylogarithm Hodge-Tate structure ($n\in\N$) is the Hodge-Tate structure of dimension $n+1$ whose period matrix (Definition \ref{defn:Period_matrix}, (\ref{eq:CPM_HT})) is
\begin{equation} \label{eq:PolyHT}
\mcP^{(n)}(z)= \left( \begin{array} {cccccc}
1 & -\Li_1(z) & -\Li_2(z) & \cdots & \cdots & -\Li_n(z) \\
0 & \tpi & \tpi \log z & \tpi \frac{(\log z)^2}{2} & \cdots & \tpi \frac{(\log z)^{n-1}}{(n-1)!} \\
\vdots & 0 & (\tpi)^2 & (\tpi)^2 \log z & & (\tpi)^2 \frac{(\log z)^{n-2}}{(n-2)!} \\ 
& &  & \ddots  & & \vdots \\
0 & & & && (\tpi)^{n-1} \log z \\
0 & & & & & (\tpi)^n
\end{array} \right); 
\end{equation}
If $\{e_0,\cdots,e_n\}$ is the standard basis of $\Q^{n+1}$ and $\lambda_i(z)\in \C^{n+1}\ (0\leq i\leq n)$ is the $i$-th%%
\footnote{We index the rows and columns from $0$ to $n$.}
row of $\mcP^{(n)}$, this Hodge-Tate structure $H$ is given by $H=\Q\langle \lambda_0(z),\cdots,\lambda_{n+1}(z)\rangle \subset H_{\C}=\C^{n+1}=\oplus_{i=0}^{n} \C e_i$ (\ref{eq:CPM_HT}) with the weight and Hodge filtrations  (\ref{eq:Wt_Hod_Fil_PeriodMatrix}).
Note that each weight graded quotient $\gr^W_{2k}$ is one-dimensional and there exists a natural $n$-framing:
\[ v_0:=\text{image of }\lambda_0(z)\text{ in }\gr^W_0H,\quad f^n \in (\gr^W_{-2n}H)^{\vee}=\Q\langle (\tpi)^{n}e_n\rangle^{\vee} \ (f^n((\tpi)^{n}e_n)=1). \]

We compute the big and skew-symmetric period of this framed Hodge-Tate structure in the case $n=2$. Consider a general period matrix (with its inverse showing first on the left):
\begin{equation} 
\mcM(z)^{-1}=  \left( \begin{array} {ccc} 1 & - \frac{a}{\tpi} & \frac{ab-c}{(\tpi)^2} \\ 0 & \frac{1}{\tpi} & \frac{-b}{(\tpi)^2} \\ 0 & 0 & \frac{1}{(\tpi)^{2}} \end{array} \right);  \quad \label{eq:generalHT3} \mcM(z)=\left( \begin{array} {ccc} 1 & a & c \\ 0 & \tpi & \tpi\, b \\ 0 & 0 & (\tpi)^2 \end{array} \right) 
\end{equation}
By Lemma \ref{lem:big_period=(tensor_product)_{m,n}}, the big period $\mathscr{P}_2$ of this framed Hodge-Tate structure is the matrix tensor-multiplication of the first row $\mcM^{-1}$ and $(\tpi)^{-2}$ times the third column of $\mcM$:
\begin{align*} 
\mathscr{P}_2(\mcPt(z)) = & 1\otimes \frac{c}{(\tpi)^2} -  \frac{a}{\tpi} \otimes \frac{b}{\tpi}  +  \frac{ab-c}{(\tpi)^2} \otimes 1, \\
\mathscr{A}_2(\mcPt(z)) = & 1\wedge \frac{2c-ab}{(2\pi i)^2} -\frac{a}{\tpi} \wedge \frac{b}{\tpi} .
\end{align*}

So, for the framed polylogarithm Hodge-Tate structure of order $2$
\begin{equation} \label{eq:PolyHT2}
\mcP^{(2)}(z)= \left( \begin{array} {ccc}
1 & -\Li_1(z) & -\Li_2(z)  \\
0 & \tpi & \tpi\,\log z  \\
0 & 0 & (\tpi)^2
\end{array} \right),
\end{equation}
we have ($a=\log(1-z)$, $b=\log(z)$, $c=-\Li_2(z)$)
\begin{align} 
\mathscr{P}_2(\mcP^{(2)}(z)) = & 1\otimes \frac{-\Li_2(z)}{(\tpi)^2} - \frac{\log(1-z)}{\tpi} \otimes \frac{\log(z)}{\tpi}    +  \frac{\log(1-z)\log z+\Li_2(z)}{(\tpi)^2} \otimes  \frac{1}{(\tpi)^2}, \nonumber \\
\mathscr{A}_2(\mcP^{(2)}(z)) = & 1\wedge \frac{-2}{(2\pi i)^2} \left(\Li_2(z) +\frac{1}{2}\log(1-z)\log(z) \right) + \frac{\log(z)}{2\pi i}\wedge\frac{\log(1-z)}{2\pi i} \label{eq:ss-period=rho} \\
= & \tilde{\rho}([z]) \nonumber
\end{align}
for the (unnormalized) Bloch regulator $\tilde{\rho}([z])$ (\ref{eq:Bloch_regulator}), (\ref{eq:Rogers_dilogarithm}); this last equality is the reason for us having introduced the notion of skew-symmetric period!

As pointed by Goncharov \cite[$\S$4.5]{Goncharov99}, the original Bloch regulator \cite[Lem.6.1.1]{Bloch00}:
\begin{equation} \label{eq:original_Bloch_regulator}
\tpi \otimes \exp(\frac{-\Li_2(z)}{\tpi})-\log(1-z)\otimes z 
\end{equation}
is the image of the big period $\mathscr{P}_2(\mcP^{(2)}(z))$ under the homomorphism
\[ \C\otimes_{\Q}\C \ra \C\otimes_{\Q}\C^{\ast};\ x\otimes y\mapsto  \tpi\cdot x \otimes \exp(\tpi\cdot y).  \]

\begin{prop}
The skew-symmetric period of the polylogarithm Hodge-Tate structure $\mcM^{(2)}(z)$ of order $2$ equals the image $[z]\in \mcB(\C)$ under the Bloch regulator $\tilde{\rho}(z)$ (\ref{eq:Bloch_regulator}).
\end{prop}

%%%%%%%%%%%%%%%%%%%%
\subsection{Hodge-Tate Hopf algebra}

Let $T$ be a $\Q$-linear abelian category and $\omega:T\ra \mathrm{Vec}_{\Q}$ a $\Q$-linear exact faithful functor. For an object $X$ of $T$, let $\langle X\rangle$ be the strictly full subcategory of $T$ whose objects are those isomorphic to a subquotient of $X^n$ for some $n\in\N$, and let $ \End(\omega|_{\langle X\rangle})$ denote the endomorphism functor of $\omega|_{\langle X\rangle}$ (it is the largest $\Q$-subalgebra of $\End(\omega(X))$ stabilizng $\omega(Y)$ for all $Y\subset X^n$). Then, $\omega$ identifies $T$ with the category of finite-dimensional left $\End(\omega)$-modules, where $\End(\omega):=\varprojlim \End(\omega|_{\langle X\rangle})$ (the endomorphism functor of $\omega$), \cite[II. 2.13]{DeligneMilne82}. The latter is also equivalent to the category of finite-dimensional right $\End(\omega)^{\vee}$-comodules for the coalgebra $\End(\omega)^{\vee}:=\Hom_{\Q-\mathrm{vec}}(\End(\omega),\Q)$ (loc. cit. II. Prop. 2.14). 
In particular, this is true for any neutral Tannakian category, in which case $\End(\omega)$ is further a cocommutative Hopf algebra. 

Now, for the $\Q$-linear Tannakian category $\QHT$ of Hodge-Tate $\Q$-structures, every object $H$ has a canonical weight filtration $W_{\bullet}H$ such that morphisms in the category are strictly compatible with the weight filtration, and there exists a fibre functor $\omega$, now to the category of $\Z$-graded $\Q$-vector spaces:
\begin{equation} \label{eq:can_FF_HT}
\omega(H)=\bigoplus_{n\in\Z} \omega(H)_n:=\bigoplus_{n\in\Z} \Hom(\Q(n),\gr^W_{-2n}H),
\end{equation}
In this case, it turns out that there exists an explicit description of the graded Hopf algebra $\End(\omega)^{\vee}$ in terms of framed mixed Hodge-Tate structures \emph{up to some equivalence relation}. Here, we review this theory. For details, see \cite[$\S$2.1]{BMS87} (the original source), \cite[$\S$1]{BGSV90}, \cite[$\S$3]{Goncharov98}, \cite[$\S$3.2]{Goncharov99}, \cite[Appendix.A]{Goncharov05}.

This equivalence relation is defined to be the coarsest equivalence relation $\sim$ on the set of $n$-framed Hodge-Tate structures for which two $n$-framed Hodge-Tate structures $H$, $H'$ are \emph{equivalent} if there exists a morphism $H\ra H'$ of mixed $\Q$-Hodge structures preserving $n$-framings. One can show (\cite[1.3.4]{BGSV90}) that every $n$-framed Hodge-Tate structure $H$ is equivalent to $\gr^W_0H/\gr^W_{-2n}H$, i.e. to an $n$-framed Hodge-Tate structure $H'$ such that $H'=W_{\geq 0}H'$ and $W_{<-2n}H'=0$, and further to an $n$-framed Hodge-Tate structure $H''$ such that $\gr^W_0H''\simeq \Q(0)$ and $\gr^W_{-2n}H''\simeq \Q(n)$.  

Let $\mcH_n$ be the set of equivalence classes of $n$-framed Hodge-Tate structures; we denote the class of $n$-framed Hodge structure $(H,v_0,f^n)$ by $[H,v_0,f^n]$. This set $\mcH_n$  has a $\Q$-vector space structure defined as follows (\cite[p.166]{Goncharov98}):
\[ (H,v_0,f^n)+(H',v_0',{f^n}') =(H\oplus H',v_0+v_0',f^n+{f^n}');\quad
c (H,v_0,f^n) =(H,v_0,cf^n). \]
The zero object is $H=\Q(0)\oplus \Q(n)$ with the obvious $n$-framing.

The tensor product of mixed Hodge structures induce a commutative multiplication on the graded $\Q$-vector space $\mcH_{\bullet}:=\bigoplus_n \mcH_n$: 
\[ \mu: \mcH_p\otimes \mcH_q \ra \mcH_{p+q}. \]
More interestingly, one can define a comultiplication on $\mcH_{\bullet}$:
\begin{equation} \label{eq:coproduct_on_mcH}
\nu= \bigoplus_{p+q=n} \nu_{pq}\ :\ \mcH_n \ra \bigoplus_{p+q=n} \mcH_p\otimes \mcH_q .
\end{equation}
We define its component $\nu_{pq}:\mcH_n \ra \mcH_p\otimes \mcH_q$ as follows:  Choose a basis $\{b_i\}$ of $\Hom(\Q(p),\gr^W_{-2p}H)$ and let $\{b_i^{\vee}\}\subset \Hom(\gr^W_{-2p}H,\Q(p))$ be the dual basis. Then, for $[H,v_0,f^n]\in \mcH_n$ and every $i$, the datum $(v_0,b_i^{\vee})$ is a $p$-framing on $H$. Further, $(b_i(-p),f^n(-p))$ is a $q$-framing on $H(-p)$, since
\begin{align}  \label{eq:Tate_twist_framing}
b_i(-p)\in\ \Hom(\Q(p),\gr^W_{-2p}H)(-p) &=\Hom(\Q(0),\gr^W_{0}(H(-p))),\\
f^n(-p)\in\ \Hom(\gr^W_{-2n}H,\Q(n))(-p) &=\Hom(\gr^W_{-2q}(H(-p)),\Q(q)). \nonumber
 \end{align}
In view of this, it makes sense to define
\begin{equation} \label{eq:comultiplication_on_framed-HTS}
\nu_{pq}([H,v_0,f^n]) :=\sum_i [H,v_0,b_i^{\vee}] \otimes [H,b_i,f^n](-p).
\end{equation}
(Here and later, we write $[H,b_i,f^n](-p)$ for the equivalence class $[(H,b_i,f^n)(-p)]$ of the Tate-twisted $(n-p)$-framed Hodge-Tate structure $(H,b_i,f^n)(-p):=(H(-p),b_i(-p),f^n(-p))$, by abuse of notation.)

%%%%%%%%%%%%%%%%%%%%
\begin{lem}  \label{lem:Goncharov99;3.5}
This definition of $\nu_{pq}$ is independent of the choice of the basis $\{e_i\}$ and also depends only on the equivalence class of framed Hodge-Tate structures, so that it gives a well-defined group homomorphism (\ref{eq:coproduct_on_mcH}). It is furthermore an algebra homomorphism (i.e. $\nu(a\cdot b)=\nu(a)\cdot \nu(b)$ for $a\cdot b:=\mu(a\otimes b)$).

The abelian group $\mcH_{\bullet}$ has a structure of graded Hopf algebra with the commutative multiplication $\mu$ and the comultiplication $\nu$. 
\end{lem}
See \cite[3.5]{Goncharov99} for a proof. 

As $\omega$ is a fibre functor to the category of graded vector spaces, $\End(\omega)$ has a natural grading $\End(\omega)=\bigoplus_{n\leq 0}\End(\omega)_n$ with $\End(\omega)_n:=\{ r\in \End(\omega) \ |\ r_H(\omega(H)_k)\subset\omega(H)_{k-n},\ \forall H\in \QHT, k\in\Z\}$, where $r_H\in\End(\omega(H))$ denotes the value on $\omega(H)$ of the endomorphism functor $r$. The linear dual $\End(\omega)^{\vee}=\bigoplus_n\End(\omega)_n^{\vee}$ is a graded Hopf algebra. Another way to look at this structure is to observe that $\End(\omega)$ is canonically isomorphic to the universal enveloping algebra of the Lie algebra $L(\QHT)$ of derivations of $\omega$:
\[ L(\QHT):=\{ E\in \End(\omega) \ |\ E_{H_1\otimes H_2}=E_{H_1}\otimes\mathrm{id}_{H_2} + \mathrm{id}_{H_1}\otimes E_{H_2}\},\]
As $L(\QHT)$ itself has a similar non-positive grading, so does the universal enveloping algebra, which thus has a structure of a graded Hopf algebra.

%%%%%%%%%%%%%%%%%%%%
\begin{thm}  \label{thm:Goncharov98;3.3}
The $\Q$-algebra $\mcH_{\bullet}$ is canonically isomorphic, as graded Hopf algebra, to the linear dual $R^{\vee}$ of $R:=\End(\omega)$ by the map which to $[H,v_{0}\in \omega(H)_0,f^{n}\in\omega(H)_n^{\vee}]\in \mcH_n$ attaches the dual element in $(\End(\omega)_{-n})^{\vee}:=\Hom(\End(\omega)_{-n},\Q)$ given by
\begin{equation} \label{eq:mcH^{vee}=End(omega)}
 \End(\omega)_{-n} \ra \Q\, :\, r \mapsto f^n (r_H\cdot v_0).
\end{equation}
\end{thm}

See \cite[3.3]{Goncharov98}, \cite[A.2]{Goncharov05} for a proof. The unit is $[\Q(0),\mathrm{id}_{\Q(0)},\mathrm{id}_{\Q(0)}]\in \mcH_0$, from which via the theorem, we obtain an augmentation $\epsilon:R\ra \Q$; clearly this also equals the ring homomorphism $r\mapsto r|_{\Q(0)}\in\End(\Q(0))=\Q$.

We generalize slightly more the notion of $n$-framing on a Hodge-Tate structure. For two integers $m,n\in\Z$, we consider Hodge-Tate structures $H$ framed by two non-zero morphisms $v_m\in \omega(H)_{m}=\Hom(\Q(m),\gr^W_{-2m}H)$, $f^n\in\omega(H)_{n}^{\vee}= \Hom(\gr^W_{-2n}H,\Q(n))$; for short, we call it $(m,n)$-framed Hodge-Tate structure. Obviously, the previous discussion for $n$-framed Hodge-Tate structures applies to this notion as well, cf. \cite[$\S$3]{Goncharov98}.

%%%%%%%%%%%%%%%%%%%%
\begin{defn} \label{defn:mixed_HT-algebra=dual_End-alg}
Let $\mcH(m,n)$ be the $\Q$-vector space of equivalence classes of $(m,n)$-framed Hodge-Tate structures.
For $n\geq m$, we let $\varphi:\mcH(m,n) \ra (\End(\omega)_{m-n})^{\vee}$ be the canonical morphism defined by:
\[ \varphi([H,v_m,f^n]): r\in \End(\omega)_{m-n} \mapsto \langle f^n, r_H(v_m) \rangle\in\Q. \]
Via Theorem \ref{thm:Goncharov98;3.3}, $\varphi$ is regarded as a map $\mcH(m,n) \ra \mcH_{n-m}$.  \end{defn}

Like $\mcHb$, the bi-graded space $\underline{\mcH}:=\bigoplus_{n\geq m}\mcH(m,n)$ has a structure of $\Q$-coagebra; this is a ``mixed colagebra" in the sense of \cite[$\S$2.1]{BMS87}.

In the following, we will view a comodule map $\varphi:V\ra V\otimes \mcH$ for a Hopf algebra $\mcH$ equivalently as a linear map $\varphi':V\otimes V^{\vee}\ra \mcH$ (satisfying some properties) defined by that $\varphi'(v,f)=\langle f,\varphi(v)\rangle$ for $v\in V$ and $f\in V^{\vee}$.

%%%%%%%%%%%%%%%%%%%%
\begin{cor} \cite[1.6]{BGSV90}, \cite[3.6]{Goncharov99} \label{cor:Goncharov99;3.6}
The category $\QHT$ of $\Q$-Hodge-Tate structures is canonically equivalent to the category of finite-dimensional graded right $\mcH_{\bullet}$-comodules by the map which assigns to a Hodge-Tate structure $H$ the graded comodule $\omega(H)=\bigoplus_n \omega(H)_n$, with $\mcH_{\bullet}$-coaction $\omega(H)\otimes \omega(H)^{\vee}\ra \mcH_{\bullet}$ being defined by that $v_m\otimes f^n \in \omega(H)_m\otimes \omega(H)_n^{\vee}$ maps to $\varphi([H,v_m,f^n])$ if $n\geq m$, and to zero otherwise, where $\varphi:\mcH(m,n) \ra \mcH_{n-m}$ is the map from Definition \ref{defn:mixed_HT-algebra=dual_End-alg}.

In particular, the equivalence class $[H,v_0,f^n]\in  \mcH_{\bullet}$ of an $n$-framed Hodge-Tate structure $(H,v_0,f^n)$ is determined by the associated $\mcH_{\bullet}$-comodule structure $\nu_{H}:\omega(H)\otimes \omega(H)^{\vee}\ra \mcH_{\bullet}$:
\[  [H,v_0,f^n]=\nu_H(v_0\otimes f^n)\ \in \mcH_{\bullet}. \]
\end{cor}

\begin{proof} (of Corollary)
The $\Q$-linear abelian category $(\QHT,\widetilde{\omega})$ ($\widetilde{\omega}$ being the same fiber functor $\omega$, except for forgetting the grading) is identified with the abelian category of finite-dimensional left $R=\End(\omega)$-modules. 
For a general $\Q$-algebra $R$, the latter category is equivalent to the category of finite-dimensional right $R^{\vee}$-comodules, via the bijections $\Hom_{\Q}(R\otimes_{\Q}V,V) \cong \Hom_{\Q}(V,\Hom(R,V)) \cong \Hom_{\Q}(V,V\otimes_{\Q} R^{\vee})$ for $R$-modules $V$. Let $\Delta_{V}:V\ra V\otimes R^{\vee}$ be the comodule map corresponding to a given $R$-module $\rho_{V}:R\otimes V\ra V$; then, there exists the relation:
\begin{equation} \label{eq:action-coaction}
(f,\Delta_{V}(v))(r)=\langle f,\rho_{V}(r\otimes v) \rangle
\end{equation}
for every $r\in R$, $v\in V$, $f\in V^{\vee}$. The same correspondence induced by $\omega$ (\ref{eq:can_FF_HT}) is further an equivalence of tensor categories between $\QHT$ and the tensor category of finite-dimensional graded right $R^{\vee}$-comodules for the graded Hopf algebra $R^{\vee}$, (cf. the proof of Theorem II. 2.11 in \cite{DeligneMilne82}).

So, to prove the corollary, we have to show that for every Hodge-Tate structure $H$ (so, being endowed with $\End(\omega)$-action $\rho_{\omega(H)}:\End(\omega)\otimes \omega(H)\ra \omega(H)$), the associated comodule structure $\Delta_{\omega(H)}:\omega(H)\otimes \omega(H)^{\vee} \ra R^{\vee}$ equals the comodule structure described in the statement, namely the map $\omega(H)\otimes \omega(H)^{\vee} \ra \mcHb\, :\, (v_m,f^n)\mapsto \varphi([H,v_m,f^n])$ if $n\geq m$, or zero if $n<m$, via the map 
$\mcH_{n-m} \stackrel{\varphi}{\ra} (R_{m-n})^{\vee}$ (Definition \ref{defn:mixed_HT-algebra=dual_End-alg}).
But, by (\ref{eq:action-coaction}) applied to $V=\omega(H)$, we see that 
\[ \Delta_{\omega(H)}(v_m,f^n)(r)=\langle f^n, \rho_{\omega(H)}(r \otimes v_m)\rangle =\langle f^n, r_{H}\cdot v_m \rangle,\] 
for every $v_m \in \omega(H)_m$, $f^n\in \omega(H)_n^{\vee}$, and $r\in R=\End(\omega)$. Note that $r_{H}\cdot v_m \in \bigoplus_{l\geq m}\omega(H)_l$ since $r$ preserves the weight filtrations of Hodge-Tate structures, so that $f^n(r_{H}\cdot v_m)$ is zero if $n<m$, and also that $ \Delta_{\omega(H)}(v_m,f^n)\in (R_{m-n})^{\vee}$.
So, the claim follows from Definition \ref{defn:mixed_HT-algebra=dual_End-alg}.
\end{proof}

\begin{lem} \label{lem:Q_{(n)}}
For $n\in\N$, the pure Hodge-Tate structure $\Q(n)$
 corresponds, via Corollary \ref{cor:Goncharov99;3.6}, to the one-dimensional $\Q$-vector space $\Q$ located in degree $-n$ with $R$-module structure being given by the augmentation $\epsilon\in R^{\vee}$; we denote this graded $R$-module by $\Q_{(n)}$.
 \end{lem}

Clearly, a mixed Hodge structure $H$ which is a simple extension of $\Q(0)$ by $\Q(n)\ (n\geq 0)$ in the abelian category of mixed Hodge structures (i.e. a short exact sequence $0\ra \Q(n)\ra H\ra \Q(0)\ra 0$) has a natural structure of $n$-framed Hodge-Tate structure, providing a monomorphism 
\begin{equation} \label{eq:one-extension_mcH}
\iota_n:\Ext^1_{\QHT}(\Q(0),\Q(n)) \hra \mcH_n.
\end{equation}
(The left group is also equals $\Ext^1_{\QMH}(\Q(0),\Q(n))$ \cite[(6)]{Goncharov99}, since $\QHT$ is closed under extension.)

%%%%%%%%%%%%%%%%%%%%
\begin{rem}
The dual Hopf algebra $\mcH_{\bullet}^{\vee}\simeq \bigoplus_n\mcH_n^{\vee}$ is a universal enveloping algebra of a Lie algebra, a free pro-nilpotent graded Lie algebra $L$ with generators in degree $-n$ equal to $(\C/(2\pi i)^n\Q)^{\vee}$, cf. \cite[Lem.1.4.3]{BGSV90}, \cite[$\S$2]{DeligneGoncharov05}.
\end{rem}

Let $\nub:=\nu|_{ \mcH_{>0}}-1\otimes\mathrm{id}_{ \mcH_{>0}}-\mathrm{id}_{ \mcH_{>0}}\otimes 1$ be the reduced comultiplication.
Then, clearly, the image of the inclusion $\iota_n:$ (\ref{eq:one-extension_mcH}) is contained in $\mathrm{Ker}(\nub|_{\mcH_n}:\mcH_n \ra \bigoplus_{k+n=n,k,l>0} \mcH_k\otimes\mcH_l)$, since any $H\in \Ext^1_{\QMH}(\Q(0),\Q(n))$ has $\gr^W_{-2k}H=0$ for all $0<k<n$.

%%%%%%%%%%%%%%%%%%%%
\begin{prop} \label{prop:coh_red_cobar_cx}
The image of $\iota_n$ is $\Ker(\nub|_{\mcH_n})$.
\end{prop}

In particular, there exists a natural bijection $\Ker(\nub|_{\mcH_n})\isom \Ext^1_{\QHT}(\Q(0),\Q(n))$.
The mere existence of such a natural bijection can be deduced rather quickly from the following general facts: for an arbitrary augmented, graded Hopf algebra $\mcHb$ over a field $k$ and any graded $\mcHb$-comodule $N$, there exists a canonical bijection 
\begin{equation*} \label{eq:Yoneda_Ext=Hopf-alg_coh}
\Ext^1_{\mcHb\mhyphen\text{comod}}(k,N) \cong H^1((co\bar{B}^{\ast}(\mcHb)\otimes N)_0) 
\end{equation*}
between the Yoneda extension group of graded $\mcHb$-comodules and the degree-$1$ cohomology group of the degree $0$-part of the tensor product (complex) of $N$ and the \emph{reduced cobar complex} $co\bar{B}^{\ast}(\mcHb)$ of the Hopf algebra $\mcHb$:
\[ \Q \stackrel{0}{\lra} \mcH_{>0} \stackrel{\nub}{\lra}  \mcH_{>0}^{\otimes2}  \lra  \mcH_{>0}^{\otimes3} \lra \cdots. \]
When $N$ is $k_{(n)}$, i.e. the comodule $k$ (via the augmentation $\mcHb\ra k$) located in degree $-n$, it is clear that \[ H^1((co\bar{B}^{\ast}(\mcHb)\otimes N)_0)=\Ker(\nub|_{\mcH_n}). \]
Now, when $\mcHb$ is the Hodge-Tate Hopf algebra, the comodule $N$ corresponding to the Hodge-Tate structure $\Q(n)$ is the comodule $\Q_{(n)}$ (Lemma \ref{lem:Q_{(n)}}), and also by Corollary \ref{cor:Goncharov99;3.6}, we have $\Ext^1_{\mcHb\mhyphen\text{comod}}(k,N)=\Ext^1_{\QHT}(\Q(0),\Q(n))$. That is, we have a natural bijection:
\begin{equation} \label{eq:Ker(coproduct)=Yoneda_Ext}
\Ker(\nub|_{\mcH_n}) =H^1((co\bar{B}^{\ast}(\mcHb)\otimes\Q_{(n)})_0)\cong \Ext^1_{\mcHb\mhyphen\text{comod}}(k,\Q(n)) \cong \Ext^1_{\QHT}(\Q(0),\Q(n)).
\end{equation}

But, in our case, the construction of $\mcH_{\bullet}$ in terms of equivalence classes of framed Hodge-Tate structures provides another, a priori different, map $\iota_n:\Ext^1_{\QHT}(\Q(0),\Q(n))\hra \Ker(\nub|_{\mcH_n})$ (\ref{eq:one-extension_mcH}), and the proposition asserts the stronger statement%%
\footnote{We will need this stronger fact in the proof of Theorem \ref{thm:Main.Thm1:CS-MTM}: see the last paragraph of that proof.} 
that $\iota_n$ \emph{is} the bijection (\ref{eq:Ker(coproduct)=Yoneda_Ext}), namely that for every element $\gamma$ of $\Ker(\nub|_{\mcH_n})$, the simple extension in $\Ext^1_{\QHT}(\Q(0),\Q(n))$ constructed by the canonical bijection (\ref{eq:Ker(coproduct)=Yoneda_Ext}) has the equivalence class (i.e. image under $\iota_n$) $\gamma\in \mcH_n$. We verify this non-obvious statement, by unravelling the constructions involved in (\ref{eq:Ker(coproduct)=Yoneda_Ext}). 

\begin{proof}
For the first preliminary step, let us work with a general graded algebra $R=\oplus R_n$ over a field $k$. We use the bar notation for elements of $R^{\otimes n+1}$: $r_0|r_1|\cdots|r_n:=r_0\otimes r_1\otimes\cdots \otimes r_n$.
The \emph{bar resolution} of a left graded $R$-module $B$ is the complex $R^{\otimes n+1}\otimes_kB\ (n\geq 0)$
\begin{equation} \label{eq:bar_resolution}
\cdots R^{\otimes3}\otimes_kB \stackrel{d_2}{\lra}  R^{\otimes2}\otimes_kB  \stackrel{d_1}{\lra} R\otimes_kB  (\lra B)
\end{equation}
with differentials given by
\begin{align*} 
d_n(r_0|\cdots |r_n|b) = & \sum_{i=0}^{n} (-1)^{i} r_0|\cdots|r_{i}r_{i+1}|\cdots|b
\end{align*}
(here the last summand ($i=n$) is $(-1)^nr_0|\cdots|r_{n-1}|r_n b$ ($r_{n+1}:=b$); note that the differentials preserve degrees.
This is known to be a resolution of $B$ (by free $R$-modules) in the abelian category of graded $R$-modules, cf. \cite[8.6.12]{Weibel94}. Thus, for a graded (left) $R$-module $C$, $\Ext^i_{\grR}(B,C)$ can be computed as $H^i$ of the dual complex (all tensorizations will be over $k$):
\begin{align}
&\, \Hom_{\grR}(R\otimes B,C)\, \stackrel{d^{0}}{\ra}\, \Hom_{\grR}(R^{\otimes 2}\otimes B,C)\, \stackrel{d^{1}}{\ra}\,  \Hom_{\grR}(R^{\otimes 3}\otimes B,C)\, \stackrel{d^{2}}{\ra}\,\cdots  \nonumber \\
\cong\, &\, \Hom_{\grk}(B,C)\, \stackrel{d^{0}}{\lra}\, \Hom_{\grk}(R\otimes B,C)\, \stackrel{d^{1}}{\lra}\,  \Hom_{\grk}(R^{\otimes 2}\otimes B,C)\, \stackrel{d^{2}}{\lra}\,\cdots \nonumber \\
\cong\, &\, (C\otimes B^{\vee})_0\, \stackrel{d^{0}}{\lra}\, (C\otimes (R\otimes B)^{\vee})_0 \, \stackrel{d^{1}}{\lra}\,  (C\otimes (R^{\otimes 2}\otimes B)^{\vee})_0 \, \stackrel{d^{2}}{\lra}\,\cdots \nonumber \\
\cong\, &\, (C\otimes B^{\vee})_0 \, \stackrel{d^{0}}{\lra}\, (C\otimes R^{\vee}\otimes B^{\vee})_0 \, \stackrel{d^{1}}{\lra}\,  (C\otimes (R^{\vee})^{\otimes 2}\otimes B^{\vee})_0 \stackrel{d^{2}}{\lra}\,\cdots \label{eq:cobar_cx_B}
\end{align}
where $(-)_0$ means the degree $0$-subspace (the tensor product of two graded spaces is equipped with the usual grading).
The differentials $d^n:C\otimes (R^{\vee})^{\otimes n}\otimes B^{\vee}\ra C\otimes (R^{\vee})^{\otimes n+1}\otimes B^{\vee}$ are 
\begin{align*}  d^{n}(c|f_1|\cdots|f_n|w)=\Delta_{C}(c)|f_1|\cdots|f_n |w \, & \,+\,  \sum_{i=1}^n (-1)^i c|f_1|\cdots|\nu(f_i)|\cdots|f_n |w \\ & \ +\,  (-1)^{n+1}  c|f_1|\cdots|f_n|\nu_{B^{\vee}}(w),
\end{align*}
where $\Delta_{C}:C\ra C\otimes R^{\vee}$ is the \emph{right} $R^{\vee}$-comodule map for the coalgebra $R^{\vee}$ induced from the module map $R\otimes C\ra C$ (\ref{eq:action-coaction}) and $\nu_{B^{\vee}}:B^{\vee}\ra R^{\vee}\otimes B^{\vee}$ is the dual of the module map $R\otimes B\ra B$ (which endows $B^{\vee}$ with the structure of \emph{left} $R^{\vee}$-comodule).
Indeed, for $c|f|w\in C\otimes R^{\vee}\otimes B^{\vee}$ (regarded as an element of $\Hom_R(R\otimes R\otimes B,C)$) and $1|r_1|r_2|b$ of $R\otimes R^{\otimes2}\otimes B$, we have 
\begin{align*} d^1(c|f|w) \left(1|r_1|r_2|b \right)=& c|f|w \left(r_1|r_2|b-1|r_1r_2|b+1|r_1|r_2b \right) \\ =& f(r_2)w(b)(r_1 c) -f(r_1r_2)w(b)c+f(r_1)w(r_2b)c \\ =& \Delta_{C}(c)(r_1)f(r_2)w(b) -\nu(f)(r_1|r_2)w(b)c + f(r_1)|\nu_{B^{\vee}}(w)(r_2|b)c. \end{align*}
The dual complex (\ref{eq:cobar_cx_B}) is called the \emph{cobar complex} of $C\otimes_k B^{\vee}$.

The bar resolution of $B$ has a quasi-isomorphic quotient complex $R\otimes \bar{R}^{\otimes n}\otimes_kB\ (n\geq 0)$, $\bar{R}:=\mathrm{Coker}(k\ra R)$, called \emph{normalized bar resolution}
\[ \cdots R\otimes \bar{R}^{\otimes2}\otimes_kB \stackrel{d_2}{\lra}  R\otimes \bar{R}\otimes_kB  \stackrel{d_1}{\lra} R\otimes_kB \stackrel{}{\lra} B \]
with differential $d_n$ being essentially the same as that for the bar complex \cite[Ex.8.6.4]{Weibel94}:
\begin{align*} d_n(r_0|\bar{r}_1|\cdots |\bar{r}_n|b) = r_0r_1|\bar{r}_2|\cdots |\bar{r}_n|b\, & \,+\,  \sum_{i=1}^{n-1} (-1)^ir_0|\cdots|\overline{r_{i}r_{i+1}}|\cdots|b 
\\ & \, +\, (-1)^nr_0|\bar{r}_1|\cdots |\bar{r}_{n-1}| r_n b.
\end{align*}
One can check that this is well-defined, i.e. the right expression is independent of the choice of representatives $r_i\in R$ of $\bar{r}_i\in\bar{R}$. Furthermore, when $B=k$, the last term vanishes since for $r_n':=r_n-\epsilon(r_n)\cdot1$, one has $\bar{r}_n=\bar{r}_n'$ and $r_n'm=\epsilon(r_n')m=0$.

From now on, we assume that $R$ has an augmentation $\epsilon:R\ra k$, so that there exist a canonical identification $I_R:=\mathrm{Ker}(\epsilon:R\ra k)\cong \mathrm{Coker}(k\ra R)=\bar{R}$ and a splitting $R=I_R\oplus k$ given by $\mathrm{id}_k:k\stackrel{\eta}{\ra} R\stackrel{\epsilon}{\ra}k$. In view of this, it follows that the cobar complex of $B^{\vee}$ has a natural quasi-isomorphic subcomplex, called the \emph{reduced cobar complex} of $C\otimes_k B^{\vee}$: 
\begin{equation} \label{eq:rcbr}
(C\otimes_k B^{\vee})_0 \  \stackrel{d^0_0}{\lra} \ (C\otimes_k I_R^{\vee}\otimes B^{\vee})_0 \ \stackrel{d^1_0}{\lra} \ (C\otimes_k (I_R^{\vee})^{\otimes2}\otimes B^{\vee})_0 \  \stackrel{d^2_0}{\lra} \cdots.
\end{equation}
Namely, each component $(C\otimes (I_R^{\vee})^{\otimes n}\otimes B^{\vee})_0$ is a component in the decomposition of $(C\otimes (R^{\vee})^{\otimes n}\otimes B^{\vee})_0$ (\ref{eq:cobar_cx_B}) resulting from the splitting $R^{\vee}=I_R^{\vee}\oplus k\epsilon$ and the differential $d^n$ is the degree $0$-part of 
the map \[ C\otimes (I_R^{\vee})^{\otimes n}\otimes B^{\vee} \stackrel{d^n}{\lra} C\otimes (R^{\vee})^{\otimes n+1}\otimes B^{\vee} \twoheadrightarrow C\otimes (I_R^{\vee})^{\otimes n+1}\otimes B^{\vee}.\] 
More explicitly, if $\bar{\Delta}_C:C\ra C\otimes I_R^{\vee}$ denotes the projection of $\Delta_C:C\ra C\otimes R^{\vee}$ via $R^{\vee}=I_R^{\vee}\oplus k\epsilon$ and $\bar{\nu}_{B^{\vee}}:B^{\vee}\ra I_R^{\vee}\otimes B^{\vee}$ is a similar projection of $\nu_{B^{\vee}}$, 
the differentials $d^0$, $d^1$ are given by the degree $0$-part of 
\begin{align} 
d^0 =&\, \bar{\Delta}_{C}\otimes\mathrm{id}_{B^{\vee}} -\mathrm{id}_{C}\otimes \bar{\nu}_{B^{\vee}}, \nonumber \\
d^1=&\, \bar{\Delta}_C\otimes\mathrm{id}_{I_R^{\vee}\otimes B^{\vee}} -\mathrm{id}_{C}\otimes\bar{\nu}\otimes\mathrm{id}_{B^{\vee}} + \mathrm{id}_{C\otimes I_R^{\vee}}\otimes\bar{\nu}_{B^{\vee}}, \label{eq:rcbr_d1} 
\end{align}
where $\bar{\nu}$ is the reduced comultiplication
\[ \bar{\nu}:=\nu|_{I_R^{\vee}}-\epsilon\otimes\mathrm{id}_{I_R^{\vee}}-\mathrm{id}_{I_R^{\vee}}\otimes \epsilon\, :\, I_R^{\vee}\lra I_R^{\vee}\otimes I_R^{\vee} \]
which also equals the dual of the map $I_R\otimes I_R\stackrel{\nu}{\ra} R\twoheadrightarrow \bar{R}\cong I_R$ ($R\twoheadrightarrow \bar{R}\cong I_R$ is the same as the projection $R\ra I_R$).
If $B=k$, then we have $\bar{\nu}_{B^{\vee}}=0$, $d^0=\bar{\Delta}_{C}$ and $d^1=\bar{\Delta}_C\otimes\mathrm{id}_{I_R^{\vee}} -\bar{\nu}$.

Therefore, for any graded $R$-module $C$, we obtain the canonical isomorphism
\begin{align}  \label{eq:Ext^1=Ker(d^1)}
 \Ext^1_{\grR}(B,C)  \cong\,  &   \mathrm{Ker}(d^1_0:(C\otimes_k I_R^{\vee}\otimes B^{\vee})_0 \lra (C\otimes_k (I_R^{\vee})^{\otimes2}\otimes B^{\vee})_0 )/\mathrm{Im}(d^0_0).
 \end{align}
On the other hand, this Ext group $\Ext^i$ is also the cohomology group $H^i$ of the complex (\ref{eq:cobar_cx_B}). When $i=1$, for any $\gamma\in \mathrm{Ker}(d^1):=\mathrm{Ker}( \Hom_{\grR}(R^{\otimes2}\otimes B,C)  \stackrel{d^1}{\lra}\Hom_{\grR}(R^{\otimes3}\otimes B,C) )$, the corresponding extension is constructed explicitly as follows:  
from the bar resolution (\ref{eq:bar_resolution}) there arise two exact sequences of graded $R$-modules
\[ R^{\otimes3}\otimes B \stackrel{d_2}{\ra}  R^{\otimes}\otimes B  \stackrel{d_1}{\ra} \mathrm{Im}(d_1) \ra 0,\quad 0 \ra \mathrm{Im}(d_1) \ra R\otimes B \stackrel{}{\ra} B \ra 0 \]
Applying the right exact functor $\Hom_{\grR}(-,C)$ to these sequences, we obtain exact sequences
\begin{align*}  \mathrm{Ker}( \Hom_{\grR}(R^{\otimes2}\otimes B,C)  \stackrel{d^1}{\lra}\Hom_{\grR}(R^{\otimes3}\otimes B,C) )  \cong\, & \Hom_{\grR}(\mathrm{Im}(d_1),C), \end{align*}
 \[ \Hom_{\grR}(R\otimes B,C) \ra \Hom_{\grR}(\mathrm{Im}(d_1),C) \stackrel{\partial}{\ra} \Ext^1_{\grR}(B,C) \ra 0. \]
The coboundary map $\partial$ is defined by: given a map $\gamma\in \Hom_{\grR}(\mathrm{Im}(d_1),C)$, we form the push-out
\begin{equation} \label{eq:push-out}
\xymatrix{ 0 \ar[r] & \mathrm{Im}(d_1) \ar[r] \ar[d]^{\gamma} & R\otimes B \ar[r] \ar[d]^{\tilde{\gamma}} & B \ar[r] \ar@{=}[d] & 0 \\
0 \ar[r] & C \ar[r] & M \ar[r] & B \ar[r] &0},
\end{equation}
obtaining an extension $M$ of $B$ by $C$. When $B=k$ (as $R$-module via $\epsilon$), the top short exact sequence is identified with $0\ra I_R\ra R\stackrel{\epsilon}{\ra} k\ra 0$, i.e, $\mathrm{Im}(d_1)=I_R$, and the sequence splits (as $k$-vector spaces) as the canonical decomposition $R=I_R\oplus k$. 

Now, we specialize to our case of the Hodge-Tate endomorphism algebra $R=\End(\omega)$ with $B=\omega(\Q(0))=\Q_{(0)}$, $C=\omega(\Q(n))=\Q_{(n)}$ (Lemma \ref{lem:Q_{(n)}}). 
We have $\Delta_{\Q_{(n)}}=\mathrm{id}_{\Q_{(n)}}\otimes \epsilon$, so $\bar{\Delta}_{\Q_{(n)}}=0$, $d^0=0$ and $d^1=-\bar{\nu}$.
Since $(\Q_{(n)}\otimes I_R^{\vee})_0\cong (R_{-n})^{\vee}$, the formula (\ref{eq:Ext^1=Ker(d^1)}) becomes
\begin{align} \label{eq:Gon99_(6)}
\Ext^1_{\QHT}(\Q(0),\Q(n))  \cong\, &  \mathrm{Ker}(d^1_0:(R_{-n})^{\vee} \lra \bigoplus_{k+n=-n,k,l>0} R_k^{\vee}\otimes R_l^{\vee}) 
\end{align}
Then, the simple extension in $\Ext^1_{\QHT}(\Q(0),\Q(n))$ attached to $\gamma \in \mathrm{Ker}(d^1_0)$ by this isomorphism (\ref{eq:Gon99_(6)}) is given as follows:
regarding $\gamma\in \mathrm{Ker}(d^1_0)\subset (R_{-n})^{\vee}\cong (\Q_{(n)}\otimes I_R^{\vee})_0$ as an element of $\Hom_{\grR}(I_R,\Q_{(n)})$,  
the bottom row of the associated diagram (\ref{eq:push-out}), being a sequence of graded $R$-modules, is the image under $\omega$ of a (uniuqe) extension of Hodge-Tate structures 
\begin{equation} \label{eq:push-out_HT}
0 \lra \Q(n) \lra H \lra \Q(0) \lra 0:
\end{equation}
this is the simple extension attached to $\gamma \in \mathrm{Ker}(d^1_0)$. 

Therefore, to prove the proposition, we have to prove that the equivalence class of the Hodge-Tate structure $H$ (\ref{eq:push-out_HT}) endowed with the obvious $n$-framing $v_{0}=\mathrm{id}_{\Q(0)}$ and $f^n=\mathrm{id}_{\Q(n)}$ is $\gamma$, this time regarded as an element of $\mcH_{>0}$ via the identification $(I_R)^{\vee}\cong\mcH_{>0}$ in Theorem \ref{thm:Goncharov98;3.3}.
By Corollary \ref{cor:Goncharov99;3.6}, the equivalence class of this framed Hodge-Tate structure $(H,v_{0},f^{n})$ is $\nu_{H}(v_{0},f^{n}) \in \mcH $. By Theorem \ref{thm:Goncharov98;3.3}, as an element of $R^{\vee}$, this is given by that
\[ \End(\omega)\ra \Q\, :\, r\mapsto \nu_{H}(v_{0},f^n)(r)=\langle f^n, r_{H}\cdot v_{0} \rangle,\]
where $r_H\cdot v_0$ is the action of $r_H\in\End(\omega(H))$ on $v_0\in \omega(H)=\omega(\Q(0))\oplus \omega(\Q(n))$.
Since (\ref{eq:push-out}) is a diagram of $R$-modules and its two rows split canonically as vector spaces, for $r\in I_R$, we have
\[ r_H\cdot v_0=r_H\cdot \tilde{\gamma}(1)=\tilde{\gamma}(r\cdot 1)=\gamma(r),\]
where $r\cdot 1$ is the multiplication in the ring $R$ of two elements $r\in I_R$ and $1\in \Q$ (so, lies in $I_R$). This means that the map $I_R:\ra \Q:r \mapsto \nu_{H}(v_{0},f^n)(r)$ is simply $\gamma\in I_R^{\vee}$, as asserted.
\end{proof}

A remarkable property of the big period is that it factors through the equivalence relation on framed Hodge-Tate structures.
%%%%%%%%%%%%%%%%%%%%
\begin{thm}  \label{thm:big_period_factors_thru_mcH_n}
The big period induces a homomorphism $\mathscr{P}_n:\mcH_n\ra \C/\Q\otimes_{\Q}\C/\Q$.
\end{thm}

This is Theorem 4.5. a) of \cite{Goncharov99}. We give another more direct proof.

\begin{proof}
Since the equivalence relation defining $\mcH_n$ is generated by maps $H'\ra H$ of $n$-framed Hodge-Tate structures preserving $n$-framings, it suffices to show that for any such map $g:H'\ra H$, $\mathscr{P}_n(H')=\mathscr{P}_n(H)$. But, since the category $\QMH$ of mixed $\Q$-Hodge structures is abelian, we may further assume that $g$ is either injective or surjective. First, let us assume that $g$ is injective. Since $H\mapsto\gr^W_{-2k}H$ is an exact functor on $\QMH$, the induced map $\gr^W_{\bullet}H'\ra \gr^W_{\bullet}H$ is injective. 
To compute the big periods of $H'$ and $H$, we may choose bases of $\gr^W_{\bullet}H'$, $\gr^W_{\bullet}H$ and splittings of $H'$, $H$ arbitrarily, since the big period is independent of such choice (Proposition \ref{prop:Goncharov99;4.2}).
Let us first choose a $\Q$-basis of $\gr^W_{\bullet}H'$ such that the covector $f^n\in \gr^W_{-2n}H'$ (part of the datum of $n$-framing of $H'$) is a member of the associated dual basis. We expand it to a basis $\{v_k\}_{k\in J}\ (J\subset\Z)$ of $\gr^W_{\bullet}H$; by abuse of notation, $v_0$ is also the vector in $\gr^W_{0}H'$ with the same notation which is a part of the datum of $n$-framing of $H'$ (and $H$). We have $J=J'\sqcup J''$ with $\{v_i\}_{i\in J'}$ (resp. $\{v_i\}_{i\in J''}$) a basis of $\gr^W_{\bullet}H'$ (resp. of $\gr^W_{\bullet}H''$), where $H'':=H/H'\in \QMH$; note that $0,n\in J'$, $0,n\notin J''$.
After choosing a splitting $\varphi:\bigoplus_k \gr^W_k H \isom H$ which also gives a splitting of $H'$, the Goncharov period matrix $\mcM$ of $H$ splits into two:
\[ \mcM=\mcM'+\mcM'', \]
where $\mcM'$ (resp. $\mcM''$) is the Goncharov period matrix of $H'$ (resp. $H''$); in this proof only, for simplicity let us write $\mcM$ for the Goncharov period matrix $\mcMG$.
This has the property (recalling that the rows and columns of $H$ are indexed by $i\in J$) that 

($\star$) for any pair $(i,j)\in J\times J$, the $(i,j)$-entry $\mcM'_{i,j}$ of $\mcM'$ is zero unless $(i,j)\in J'\times J'$ and similarly for $\mcM''_{i,j}$. 

We can also split the identity matrix $I$ of the same size as $\mcM$ accordingly into $I'+I''$, where $I'$ has only nonzero entry for $(i,i)$ with $i\in J'$ where it is $1$, and $I''$ is similarly defined with respect to $J''$.

Now, since $\mcM-I$ is nilpotent, we can rewrite the big period $\mathscr{P}_n((H,v_0,f^n))$ (\ref{eq:bigPeriod}) as
\begin{align*} 
\langle v_0| \mcM^{-1}\otimes_{\Q} \mcM |f^n\rangle 
=& \sum_k \langle v_0| (-1)^k (\mcM-I)^k \otimes_{\Q} \mcM |f^n\rangle \\
=& \sum_k (-1)^k \sum_{i\in J} \left((\mcM-I)^k\right)_{0,i}\otimes \mcM_{i,n}.
\end{align*}
We note that

($\heartsuit$) $\mcM'-I'$, $\mcM''-I''$ are both nilpotent (upper-triangular) matrices and have the same property ($\star$) above.

Since $\mcM-I=(\mcM'-I')+(\mcM''-I'')$, we have
\[ (\mcM-I)^k=\sum_{(\epsilon_1,\cdots,\epsilon_k)} (\mcM'-I')^{\epsilon_1}(\mcM''-I'')^{1-\epsilon_1}\cdots (\mcM'-I')^{\epsilon_k}(\mcM''-I'')^{1-\epsilon_k}, \]
where the sum runs through the set of all sequences $(\epsilon_1,\cdots,\epsilon_k)$ with $\epsilon_i\in \{0,1\}$. So, for a fixed $i\in J$, we have
\[ \left((\mcM-I)^k\right)_{0,i}=\sum_{(\epsilon_1,\cdots,\epsilon_k)} \sum_{0<i_1<i_2<\cdots<i_k=i} (\mcM'-I')_{0,i_1}^{\epsilon_1}(\mcM''-I'')_{0,i_1}^{1-\epsilon_1} \cdots (\mcM'-I')_{i_{k-1},i_k}^{\epsilon_k}(\mcM''-I'')_{i_{k-1},i_k}^{1-\epsilon_k}, \]
where the inner sum runs through all sequences $(i_1,\cdots,i_k)$ of strictly increasing positive integers in $J$ with $i_k=i$.
It is easy to see from the property ($\heartsuit$) that 
\[ \left((\mcM-I)^k\right)_{0,i}=\left((\mcM'-I')^k\right)_{0,i},\text{ and both are zero unless }i\in J': \] 
indeed, first, if $\epsilon_1=0$, then $(\mcM''-I'')_{0,i_1}^{1-\epsilon_1}= 0$ by  ($\heartsuit$), and if $i_1\in J''$, then $(\mcM'-I')_{0,i_1}^{\epsilon_1}=0$, so in the summation we may assume that $\epsilon_1=1$ and $i_1\in J'$. Next, by the same reason, we may assume that $\epsilon_2=1$ and $i_2\in J'$, and can continue.
Since clearly $\mcM_{i,n}=\mcM'_{i,n}$ for $i\in J'$, hence we have
\begin{align*} 
\langle v_0| \mcM^{-1}\otimes_{\Q} \mcM |f^n\rangle = & \sum_k (-1)^k \sum_{i\in J} \left((\mcM-I)^k\right)_{0,i}\otimes \mcM_{i,n} \\
= & \sum_k \langle v_0| (-1)^k (\mcM'-I')^k \otimes_{\Q} \mcM |f^n\rangle \\
= & \langle v_0| (\mcM')^{-1}\otimes_{\Q} \mcM' |f^n\rangle,
\end{align*}
as claimed. The case that $g$ is surjective is proved similarly.
\end{proof}

Composing $\mathscr{P}_n:\mcH_n\ra \C/\Q\otimes_{\Q}\C/\Q$ with the skew-symmetrization $\C/\Q\otimes_{\Q}\C/\Q \ra \C/\Q\wedge_{\Q}\C/\Q$, we see that the skew-period map (\ref{eq:skew-symmetricPeriod}) also induces a homomorphism
\begin{equation} \label{eq::skew-symmetric_period}
\mathscr{A}_n\ : \ \mcH_n \ra \C/\Q\wedge_{\Q}\C/\Q.
\end{equation}

%%%%%%%%%%%%%%%%%%%%
\begin{lem} \label{lem:Ext^1_HT=ss-period}
For $n>0$, the composite $\mathscr{A}_n\circ \iota_n$ (\ref{eq:one-extension_mcH}) induces an isomorphism
\[ \mathscr{A}_n\circ \iota_n: \Ext^1_{\QHT}(\Q(0),\Q(n))\ \isom \ 1\wedge \C/\Q=\C/\Q, \] 
which equals twice the canonical isomorphism (\ref{eq:Ext^1_MH(Q(0),Q(n)}), under the identification $\C/\Q \isom \C/(2\pi i)^n\Q:x\mapsto (2\pi i)^nx$.
 \end{lem}

\begin{proof}
Let $H\in \Ext^1_{\QHT}(\Q(0),\Q(n))$.
After choosing a splitting $\gr^W_0H\oplus \gr^W_{-2n}H \isom H$ (equiv. a lift of $1\in \gr^W_0H=\Q(0)$ to $H$), we obtain the period matrix $\mcM= \left( \begin{array} {cc} 1 & \omega \\ 0 & (2\pi i)^n \end{array} \right)$ $(\omega\in\C)$. One computes that $\mathscr{P}_n(H)=1\otimes \frac{\omega}{(2\pi i)^n} -\frac{\omega}{(2\pi i)^n}\otimes 1$, so $\mathscr{A}_n(H)=1\wedge\frac{2\omega}{(2\pi i)^n}$. On the other hand, to $H$, the isomorphism (\ref{eq:Ext^1_MH(Q(0),Q(n)}) attaches $\omega\text{ mod }(2\pi i)^n\Q$.
\end{proof}

%%%%%%%%%%%%%%%%%%%%%%%%%%%%%%%%%%%%%%%%
%%%%%%%%%%%%%%%%%%%%%%%%%%%%%%%%%%%%%%%%
\section{Chern-Simons mixed Tate motive}

%%%%%%%%%%%%%%%%%%%%
\subsection{Mixed Tate motives}

We give a condensed review of the theory of mixed Tate motives, cf. \cite[$\S$1, $\S$2]{DeligneGoncharov05}, \cite{Levine98},  \cite{Levine93}, \cite{Levine13}.

Let $k$ be a field of characteristic zero (so that the resolution of singularities is available). A triangulated category of mixed motives over $k$ was constructed by Hanamura \cite{Hanamura95}, Levine \cite{Levine98} and Voevodsky \cite{Voevodsky00} independently, and Levine  \cite{Levine98}, VI 2.5.5, constructs an equivalence between his triangulated category and that of Voevodsky. 
We will work with the Voevodsky's category (of geometric motives), which we denote by $\DM(k)$; let $\DM(k)_{\Q}$ be the triangulated category deduced from $\DM(k)$ by tensoring with $\Q$.
We have in $\DM(k)$ Tate objects $\Z(n)$ ($n\in\Z$), whose images in $\DM(k)_{\Q}$ will be also denoted $\Q(n)$: $\Q(n)=\Q(1)^{\otimes n}$ and different $\Q(n)$'s are mutually nonisomorphic. We will be interested only in the triangulated subcategory $\DMT(k)_{\Q}$ of $\DM(k)_{\Q}$ generated by the $\Q(n)$'s, i.e. the smallest strictly full subcategory of $\DM(k)_{\Q}$ containing all the objects $\Q(n)[m]\ (n,m\in\Z)$ and closed under extensions (recall that $E$ is an extension of $B$ by $A$ if there exists a distinguished triangle $A\ra E\ra B\ra A[1]$); its objects are the ``iterated extensions'' of various $\Q(n)[m]$'s.

Among many properties enjoyed by this triangulated category; we just mention two properties (specialized to the case of the motive of $\Spec(k)$): first, a relation of the Hom groups between shifts of Tate objects (``motivic cohomology'' of the smooth $k$-variety $\Spec(k)$) with other known invariants (\cite[$\S$4.2]{Voevodsky00}, \cite[$\S$VI,2.1.9, II.3.6.6]{Levine98}): for $q\geq0$ and $p\in\Z$, there exist canonical isomorphisms
\begin{equation} \label{eq:Motivic-coh=alg.K-theory}
\Hom_{\DM(T)(k)_{\Q}}(\Q(0),\Q(q)[p])=\mathrm{CH}^q(\Spec(k),2q-p)=K_{2q-p}(k)^{(q)}.
\end{equation}
Here, $\mathrm{CH}^q(\Spec(k),2q-p)$ is the Bloch's higher Chow group of $\Spec(k)$, and 
$K_n(k)^{(q)}$ is the $q$-th eigen-subspace of $K_{n}(k)_{\Q}$ for the Adams operator, which is isomorphic to $\gr^q_{\gamma}K_{2q-p}(k)_{\Q}$, the $q$-th associated graded part of the $\gamma$-filtration on $K_{2q-p}(k)_{\Q}$.
It follows that 
\[ \Hom_{\DM(T)(k)_{\Q}}(\Q(0),\Q(q)[p])=0 \text{ for }
\begin{cases} 
\ p>2q\geq0 & (\text{by the first equality}) \\
q=0\text{ and }p\neq0 & (\text{by the second equality})
\end{cases}; \]
one also has $K_{0}(k)^{(0)}=\Q$. 
Secondly, we have (\cite[10.6]{Levine13})
\[ \Hom_{\DM(T)(k)_{\Q}}(\Q(0),\Q(q)[p])=0\quad\text{ for }q<0. \]
Also, on $\DM(k)$, there exists an associative and commutative tensor product $\otimes$ with unity which is compatible with the triangulated structure and rigid (existence of a dual of every object with natural properties).

Recall the Beilinson-Soule vanishing conjecture:
\begin{equation} \label{eq:Beilinson-Soule vanishing} 
K_{2q-p}(k)^{(q)}=0\quad \text{ for }p\leqs0\text{ and }q>0 
\end{equation}
When this condition holds for a field $k$, 
Levine \cite{Levine93} constructs a $t$-structure on $\DMT(k)_{\Q}$ whose heart $\MT(k)$ (an admissible abelian category) has objects the iterated extensions of $\Q(n)$'s $(n\in\Z)$ \cite[1.3.14]{BBD82}; we call $\MT(k)$ the category of \emph{mixed Tate motives} over $k$. 

In general, when a triangulated category $D$ has a $t$-structure with heart $\mcA$, there exists a canonical map from $\Ext^p_{\mcA}(A,B)$, the Yoneda extension groups of objects in $\mcA$, to $\Hom_D(A,B[p])$, the Hom groups in $D$ of their shifts, which is a bijection for $p=0,1$ and an injection for $p=2$ (\cite[(1.1.4),(1.1.5)]{DeligneGoncharov05}).
So in the case when the Beilinson-Soule vanishing conjecture holds, we have a map
\begin{equation} \label{eq:Ext_{MT}->K-group}
\Ext^p_{\MT(k)}(\Q(0),\Q(q))\ra \Hom_{\DMT(k)_{\Q}}(\Q(0),\Q(q)[p])=K_{2q-p}(k)^{(q)}.
\end{equation}

From this point, until the end of this subsection, we assume that the Beilinson-Soule conjecture (\ref{eq:Beilinson-Soule vanishing}) holds for $k$, so that $\MT(k)$ exists. 

Then, first the tensor structure $\otimes$ on $\DM(k)$ provides $\MT(k)$ with a tensor product which is exact in each variable, associative, commutative, with unity and rigid.
On the other hand, from $\Ext^1(\Q(m),\Q(n))=0$ for $n\leq m$, it follows that every object $M$ of $\MT(k)$ admits a unique ``weight'' filtration $W$, finite, increasing and indexed by the even integers, such that
\[ \Gr^W_{-2n}(M):=W_{-2n}(M)/W_{-2(n+1)}(M) \]
is a sum of copies of $\Q(n)$. 
The filtration $W$ is functorial, strictly compatible with morphisms, exact, and compatible with the tensor product. 
Set
\[ \omega_n(M):=\Hom(\Q(n),\Gr^W_{-2n}(M)). \]
So we have
\[ \Gr^W_{-2n}(M)=\Q(n)\otimes_{\Q}\omega_n(M). \]
The exact functor $M\mapsto \omega(M):=\oplus \omega_n(M)$ is a $\otimes$-functor, thus a fiber functor and the category $\MT(k)$ is $\Q$-linear neutral Tannakian. 
Therefore, we can repeat the constructions and arguments from the previous section on the Tannakian category $\QHT$ of Hodge-Tate $\Q$-structures. One can define $n$-framed mixed Tate motives over $k$, the abelian group $\mcA_n(k)$ of their equivalence classes, and $\mcAb(k):=\bigoplus_{n\geq 0}\mcA_n(k)$ has a structure of an augmented commutative graded Hopf algebra over $\Q$; we denote the comultiplication $\mcAb(k)\ra \mcAb(k)\otimes\mcAb(k)$ by $\Delta$ and call it the \emph{motivic comultiplication}. The category $\MT(k)$ is canonically equivalent to the category of finite-dimensional graded (right) $\mcAb(k)$-comodules, and the analogue of Proposition \ref{prop:coh_red_cobar_cx} holds with $\QHT$ and $\mcH_n$ replaced by $\MT(k)$ and $\mcA_n$ respectively:
\begin{equation} \label{eq:motivic_coh=red_cobar_cx_coh} 
\Ext^1_{\MT(k)}(\Q(0),\Q(n))=H^i_{(n)}(\mcAb(k))=\mathrm{Ker}(\bar{\Delta}|_{\mcA_n(k)}), 
\end{equation}
where the middle term is the degree-$n$ part of the cohomology of the reduced cobar complex of $\mcAb(k)$ (graded vector space) and $\bar{\Delta}:=\Delta -1\otimes\mathrm{id} - \mathrm{id}\otimes1:\mcAb(k)_{>0}\ra \mcAb(k)_{>0}\otimes \mcAb(k)_{>0}$ is the reduced motivic comultiplication. 

Furthermore, for any embedding $\sigma:k\hra \C$, there exist the realization functor 
\[ \real_{\sigma}: \MT(k)\ra \QHT \]
\cite[$\S$1.5]{DeligneGoncharov05}. It is clear from functoriality of the construction that the canonical isomorphisms (\ref{eq:motivic_coh=red_cobar_cx_coh}), Proposition \ref{prop:coh_red_cobar_cx} are compatible with the induced realization functors, i.e. the diagram 
\begin{equation} \label{eq:Hod_real}
\xymatrix{ 
\Ext^1_{\MT(k)}(\Q(0),\Q(n)) \ar@{=}[r]^(0.6){\ref{eq:motivic_coh=red_cobar_cx_coh}} \ar[d]_{\real_{\sigma}} & \mathrm{Ker}(\bar{\Delta}|_{\mcAb(k)}) \ar[d]^{\real_{\sigma}}   \\
\Ext^1_{\QHT}(\Q(0),\Q(n)) \ar@{=}[r]^(0.6){\ref{prop:coh_red_cobar_cx}} &  \mathrm{Ker}(\bar{\nu}|_{\mcH_n}) .  }
\end{equation}
commutes.

Now, suppose that $k$ is a number field. According to Borel \cite{Borel77}, the Beilinson-Soule vanishing conjecture holds. Moreover, he showed that for $q>0$, 
\begin{equation} \label{eq:K_{2q-1}(k)^{(q)}=K_{2q-1}(k)_{Q}}
K_{2q}(k)_{\Q}=0\text{ and }K_{2q-1}(k)^{(q)}=K_{2q-1}(k)_{\Q},
\end{equation}
so that one has
\[ K_{2q-p}(k)^{(q)}=0 \text{ unless }(p,q)=(0,0)\text{ or }(1,\geq 1).\]
From this last property, Levine \cite[4.3]{Levine93} deduces that for any number field $k$, the abelian category $\MT(k)$ has the desired property that the map (\ref{eq:Ext_{MT}->K-group}) is an isomorphism for all $p$ and $q$:
\begin{equation} \label{eq:Ext_{MT}=alg.K-theory}
\Ext^p_{\MT(k)}(\Q(0),\Q(q))=K_{2q-p}(k)^{(q)};
\end{equation}
in fact, these groups are both zero for $p\geq2$, and this implies (cf. \cite[Prop.5.8]{Goncharov99}) that $\DMT(k)_{\Q}$ is also the derived category of $\MT(k)$.
Also, the realization functor $\real_{\sigma}$ (\ref{eq:Hod_real}) is compatible with the Beilinson regulator 
\begin{equation} \label{eq:Beilinson_regulator}
R^{Be}:K_{2n-1}(\C)_{\Q} \ra H^1_{\mathscr{D}}(\C,\Q(n))=\C/\Q(n)
\end{equation}
(where $H^1_{\mathscr{D}}(\C,\Q(n))$ is the Deligne-Beilinson cohomology of $\Spec(\C)$)
via the map (\ref{eq:Ext_{MT}->K-group}), i.e. the diagram
\begin{equation} \label{eq:Beil_reg=Hod_real}
\xymatrix{ \Ext^1_{\MT(k)}(\Q(0),\Q(n)) \ar[r]^{\real_{\sigma}} \ar[dd]_{(\ref{eq:Ext_{MT}->K-group})} & \Ext^1_{\QMH}(\Q(0),\Q(n)) \ar[rd]_(.5){(\ref{eq:Ext^1_MH(Q(0),Q(n)})}^{\simeq} & \\
& & \C/(2\pi i)^n\Q \\
K_{2n-1}(k)_{\Q} \ar[r]^{\sigma} & K_{2n-1}(\C)_{\Q} \ar[ru]_{R^{Be}} & }
\end{equation}
commutes (cf. the last paragraph on p.9 of \cite{DeligneGoncharov05}).

%%%%%%%%%%%%%%%%%%%%
\subsection{Polylogarithm mixed Tate motive}

\begin{lem} \label{lem:Delta(PolylogarithmHT2)}
The reduced comultiplication $\bar{\nu}$ (\ref{eq:coproduct_on_mcH}) of the framed polylogarithm Hodge-Tate structure $\mcP^{(2)}(z)$ (\ref{eq:PolyHT2}):
\[ \mcP^{(2)}(z)= \left( \begin{array} {ccc}
1 & -\Li_1(z) & -\Li_2(z)  \\
0 & \tpi & \tpi \log z  \\
0 & 0 & (\tpi)^2
\end{array} \right) \]
is given by
\[ \bar{\nu}(\mcP^{(2)}(z))= \log(1-z) \otimes \log(z)\ \in \C/\Z(1)\otimes \C/\Z(1)\ \]
under the canonical isomorphism $\mcH_1\cong \C/\Z(1)$ (Lemma \ref{lem:Ext^1_HT=ss-period}).
\end{lem}

\begin{proof}
This is obvious from the definition (\ref{eq:comultiplication_on_framed-HTS}) and the isomorphism $\mcH_1\cong \C/2\pi i\Z$.
\end{proof}

\begin{thm} \label{thm:PolyMTM}
For every $z\in\ \C-\{0,1\}$ and $n\in\N$, there exists a framed mixed Tate motive $(M^{(n)}_z,v_0(z),f^n(z))$ defined over $\Q(z)\subset\C$ whose Hodge realization is the framed polylogarithm Hodge-Tate structure $\mcP^{(n)}$ (\ref{eq:PolyHT}).
\end{thm}

The first proof is attributed to Deligne, and later more constructions (especially, in the case $n=2$) were given, cf. \cite{Wang06}, \cite{DupontFresan23}. The property about the field of definition $\Q(z)$ is a consequence of these constructions of the motive. 

\begin{rem} \label{rem:uniqueness_of_poly-motive}
Strictly speaking, it is not clear (at least to this author) that a priori all these constructions are the same objects in the category $\MT(\Q(z))$, although one expects this to be the case. But, we observe that there will be a unique framed mixed Tate motive defined over $k\subset\Qb$ whose Hodge realization is the framed polylogarithm Hodge-Tate structure $\mcP^{(n)}(z)$ if it is the fiber at $z$ of a mixed Tate motive $\mcM$ over $X:=\mathbb{P}^1\backslash \{0,1,\infty\}$ (or ``mixed Tate motivic sheaf", as one says) in the sense of \cite{CisinskiDeglise19} or \cite{HuberWildeshaus98} whose Hodge realization is the polylogarithm variation of mixed Hodge structure  \cite[1.1-1.3]{BeilinsonDeligne94}, \cite[$\S$7]{Hain94} (see also Subsec. \ref{subsec:PolyVMHS}); such $\mcM$ is the same datum as a $\MT(k)$-representation (in the sense of \cite[$\S$7]{Deligne89}) of the motivic fundamental group $\pi_1^{\mot}(X,\orav)$ with tangential base point, where $\orav$ is the tangent vector $\frac{\partial}{\partial z}$ at $0$ for the standard coordinate $z$ on $\A^1$, cf. \cite[Thm.1]{Levine10}. Indeed, 
then there exists a ``parallel transport'' morphism $\mcM_{\orav}\otimes A_{\orav,z}(X) \ra \mcM_{z}$ in the Tannakian category $\MT(k)$, where $A_{\orav,z}(X)$ is the pro-object in $\MT(k)$ whose dual is the affine algebra of the motivic path torsor $P_{\orav,z}^{\mot}(X)$ as constructed in \cite[Thm.4.4]{DeligneGoncharov05}. It is known \cite[$\S$11]{Hain94} that $\mcM_{\orav}$ contains $\Q(0)$ and that the induced morphism $\real^{Hod}(A_{\orav,z}(X)) \ra \real^{Hod}(\mcM_{z})$ of mixed Hodge structures is surjective and the kernel of its fiber at $\orav$ has an explicit set of generators given by local monodromoies. Since local monodromies are motivic \cite[4.4]{DeligneGoncharov05}, hence $\mcM_{z}$ must be the quotient mixed Tate motive of $A_{\orav,z}(X) $ by some ideal of $A_{\orav,\orav}(X) $ which is determined only by the Hodge realization.

Therefore, we call any framed mixed Tate motive which satisfies the above property \emph{polylogarithm mixed Tate motive} of order $n$ (framed by $v_0(z)$, $f^n(z)$).

\end{rem}

\begin{thm} \label{thm:Main.Thm1:CS-MTM}
Let $X$ be a hyperbolic three-manifold with invariant trace field $k(M)\subset \Qb$.
Then there exists a mixed Tate motive $M(X)$ in $\Ext^1_{\MT(k(M))}(\Q(0),\Q(2))$, whose image under the Beilinson regulator (\ref{eq:Beilinson_regulator}) equals the normalized $\PSL_2(\C)$-Chern-Simons invariant $\CS_{\PSL_2}(X)$ (\ref{eq:normalized_PSL_2(C)-CS.inv}) 
\end{thm}

\begin{proof}
The Bloch invariant $\beta(X)$ constructed by Neumann and Yang lies in $\mcB(k(X))_{\Q}$, where $k(X)\subset\Qb$ is the invariant trace field (\cite[Thm.1.2]{NeumannYang99}). Suppose that $\beta(X)=\sum_{i=1}^N [z_i]$ for $z_i\in k(X)$.
For each $z\in \C-\{0,1\}$, let $M(z)$ be the framed mixed Tate motive defined as a difference of two framed polylogarithm mixed Tate motives:
\begin{equation} \label{eq:M(z)}
M(z):=\frac{1}{2} \left([M^{(2)}_z,v_0(z),f^n(z)]-[M^{(2)}_{1-z},v_0(1-z),f^n(1-z)]\right)
\end{equation}
(Theorem \ref{thm:PolyMTM}), and put $M(X):=\sum_i M(z_i) \in \MT(F)$, where $F:=k(M)$.

First, we show that $M(X)\in \Ext^1_{\MT(F)}(\Q(0),\Q(2))$, more precisely, we will check that $\bar{\Delta}_2(M(X))=0$ ($\bar{\Delta}_2=\bar{\Delta}|_{\mcA_2}:\mcA_2\ra\mcA_1\otimes\mcA_1$ being the reduced motivic comultiplication) and obtain an element of $\Ext^1_{\MT(F)}(\Q(0),\Q(2))$ via the canonical isomorphism (\ref{eq:motivic_coh=red_cobar_cx_coh}).
For any embedding $\sigma:F\hra \C$, we have the Hodge realization functor $\real_{\sigma}:\Ext^1_{\MT(F)}(\Q(0),\Q(n)) \ra \Ext^1_{\QMH}(\Q(0),\Q(n))$.
Since the Hodge realization functor 
\begin{equation} \label{eq:Hod_real_F}
\real^{Hod}_F:=\bigoplus_{\sigma} \real_{\sigma}\ :\ \Ext^1_{\MT(F)}(\Q(0),\Q(n)) \lra \bigoplus_{\sigma} \Ext^1_{\QMH}(\Q(0),\Q(n))
\end{equation}
is fully faithful \cite[Prop.2.14]{DeligneGoncharov05}, by compatibility of motivic/Hodge-theoretic reduced comultiplications, it suffices to verify that for the reduced comultiplication $\bar{\nu}$ on $\QHT$, $\sum_i \bar{\nu}(\real_{\sigma}(M(z_i)))=0$ for every embedding $\sigma:F\hra \C$.
But, for $z\in \C-\{0,1\}$, by Lemma \ref{lem:Delta(PolylogarithmHT2)}, we have
\begin{align*}
& \sum_i \bar{\nu}([\PHT^{(2)}(z_i)])- \bar{\nu}([\PHT^{(2)}(1-z_i)]) \\
=& \sum_i \log(1-z_i)\otimes \log(z_i) - \log(z_i)\otimes \log(1-z_i) \\
=& 2\sum_i \log(1-z_i)\wedge \log(z_i) 
\end{align*}
(Here, for a $\Q$-vector space $V$, we identify $V\wedge_{\Q} V$ with a subspace of $V\otimes_{\Q}V$ via $x\wedge y\mapsto \frac{1}{2}(x\otimes y-y\otimes x)$). 
This is the complex Dehn invariant $\delta_{\C}(\sum_i [z_i])$, hence is zero by Theorem \ref{thm:2Complex_volume=Bloch_regulator}, (1). Since $\delta_{\C}([z])=(1-z)\wedge z$ obviously commutes with the action of $\Aut(\C)$, the claim is proved.

Next, for the inclusion $\mathrm{id}:F\hra\C$, the Beilinson regulator $K_{2n-1}(F)_{\Q}\ra K_{2n-1}(\C)_{\Q} \ra \C/\Q(2)$ is the same as the Hodge realization $ \real_{\mathrm{id}}$ (\ref{eq:Hod_real}) under the canonical isomorphism $\Ext^1_{\QHT}(\Q(0),\Q(n)) = \C/\Q(2)$ (\ref{eq:Ext^1_MH(Q(0),Q(n)}): commutativity of the diagram (\ref{eq:Beil_reg=Hod_real}). Since the normalized $\PSL_2(\C)$-Chern-Simons invariant $\CS_{\PSL_2}(X)$ is one-half of the normalized Bloch regulator $\rho(\beta(M))$ (\ref{eq:2PSL_2-CS-inv=rho(beta(M))}), to prove the second statement, we need to show the equality
 \[ \real_{\mathrm{id}}(M(X))=\frac{1}{2}\rho(\beta(X)). \] 
By commutativity of (\ref{eq:Hod_real}), the Hodge realization $\real_{\mathrm{id}}(M(X))$ of the mixed Tate motive $M(X)=\sum_i M(z_i)$ is the Hodge-Tate structure $H(X)$ corresponding to $\frac{1}{2}\sum_i \mcP^{(2)}(z_i)-\mcP^{(2)}(1-z_i)\in \mathrm{Ker}(\bar{\nu}|_{\mcH_2})$ under the isomorphism $\iota_2^{-1}:\mathrm{Ker}(\bar{\nu}|_{\mcH_2})\isom\Ext^1_{\QHT}(\Q(0),\Q(2))$ (Proposition \ref{prop:coh_red_cobar_cx}). 
Now, the composite of the isomorphisms
\[ \mathrm{Ker}(\bar{\nu}|_{\mcH_2}) \stackrel{\iota_2^{-1}}{\lra} \Ext^1_{\QHT}(\Q(0),\Q(2))  \stackrel{(\ref{eq:Ext^1_MH(Q(0),Q(n)})}{=}\C/\Q(2)\isom  \C/\Q \]
is the map $\frac{1}{2}\mathscr{A}_2$ (Lemma \ref{lem:Ext^1_HT=ss-period}). Hence, we have
\begin{align*} 
\real_{\mathrm{id}}(M(X))
\stackrel{(\star)}{=}& \frac{1}{4}\mathscr{A}_2(\sum_i \mcP^{(2)}(z_i)-\mcP^{(2)}(1-z_i)) \\
\stackrel{}{=}& \frac{1}{4}\sum_i \left[\mathscr{A}_2( \mcP^{(2)}(z_i))-\mathscr{A}_2(\mcP^{(2)}(1-z_i)) \right] \\
\stackrel{(i)}{=}& \frac{1}{4} \sum_i \left[\tilde{\rho}([z_i])-\tilde{\rho}([1-z_i]) \right] \\
\stackrel{(ii)}{=}& \frac{1}{2}\tilde{\rho}(\sum_i [z_i]).
\end{align*}
Here, $\tilde{\rho}$ is the unnormalized Bloch regulator (\ref{eq:Bloch_regulator}) and we have $\rho=(2\pi i)^2 \tilde{\rho}$ (\ref{eq:rho=(2pii)^2tilde{rho}}). The equalities (i) and (ii) are respectively (\ref{eq:ss-period=rho}) and the functional equation for $\tilde{\rho}(z)$ (\ref{eq:functional_eq_Bloch-regulator}). This completes the proof. 
\end{proof}

\begin{rem} \label{rem:uniqueness_of_CS-motive}
(1)  We note that the first equality $(\star)$ in the above uses Lemma \ref{lem:Ext^1_HT=ss-period} ant thus requires that the Hodge realization of our mixed Tate motive $M(X)$ is constructed from $\frac{1}{2}\sum_i \mcP^{(2)}(z_i)-\mcP^{(2)}(1-z_i) \in \mathrm{Ker}(\bar{\nu}|_{\mcH_2})$ via the specific isomorphism $\iota_2^{-1}:\mathrm{Ker}(\bar{\nu}|_{\mcH_2})\isom\Ext^1_{\QHT}(\Q(0),\Q(2))$. So, we had to know that $\iota_2$ was an isomorphism (Proposition \ref{prop:coh_red_cobar_cx}).

(2) The mixed Tate motive $M(X)$ constructed in the theorem depends on the choice of a representative $\sum_{i=1}^N[z_i]\ (z_i\in k(M))$ of the Bloch invariant $\beta(M)$. At the moment we do not know whether this mixed Tate motive is independent of such choice. The independent of our construction amounts to showing the ``five term relation'' for the framed mixed Tate motive 
$M(z)=\frac{1}{2} \left(M^{(2)}_z-M^{(2)}_{1-z} \right)$ (\ref{eq:M(z)}). 

Meanwhile, uniqueness of a mixed Tate motive with specified Hodge realization would be a consequence of the full faithfulness of the Hodge realization $\real_{\mathrm{id}}:\Ext^1_{\MT(k(M))}(\Q(0),\Q(2))\ra \Ext^1_{\QHT}(\Q(0),\Q(2))$ ($\mathrm{id}:k(M)\subset \C$). Note that this is stronger than the full faithfulness of the usual Hodge realization $\oplus_{\sigma:k(M)\hra \C}\real_{\sigma}$ which is known \cite[2.14]{DeligneGoncharov05}, and the author does not know whether such faithfulness is even reasonable to expect, like a consequence of some motivic conjectures. 
But, thanks to our second main theorem (Theorem \ref{thm:CS-VMHS_is_motivic}), the argument of Remark \ref{rem:uniqueness_of_poly-motive} can be invoked to show that any (Chern-Simons) mixed Tate motive as in Theorem \ref{thm:Main.Thm1:CS-MTM} will be unique, if one assumes existence of (the maximal Tate quotient of) the motivic path torsor of the augmented character variety of the manifold with arbitrary tangential base point and existence of ``Chern-Simons motivic sheaf" over the augmented character variety whose Hodge realization is the Chern-Simons variation of Hodge-Tate structure which will be defined in Section 7; we think that existence of these objects is a reasonable assumption and perhaps can be deduced from existence of the motivic $t$-structure \cite{Beilinson12}. 
\end{rem}

%%%%%%%%%%%%%%%%%%%%
\subsection{Comparison with Goncharov's mixed Tate motives} \label{subsec:comp_wy_Goncharov_MTM}

For any hyperbolic manifold $M$ of any odd dimension $2m-1$, Goncharov gave two constructions of elements in $K_{2m-1}(\Qb)_{\Q}=\Ext^1_{\MT(\Qb)}(\Q(0),\Q(2))$ whose images under the Borel regulator $R^{Bo}:K_{2m-1}(\C)\ra \R$ both equal the volume of $M$, and asked whether they are the same, expecting this to be the case.

When $m=2$, his second method of construction is at the very basis of our construction method of a mixed Tate motive in $\Ext^1_{\MT(\Qb)}(\Q(0),\Q(2))$. 
But, as we will explain, Goncharov's mixed Tate motive by his second construction is not canonical and not necessarily equal to our mixed Tate motive since it cannot capture the Chern-Simons invariant in general. Since his K-theory element by the first construction is expected to be the same as our mixed Tate motive which should be canonical, this would answer negatively Goncharov's question above.
We now give detailed explanation.

Goncharov's first method is of homological nature, constructing an element $c(M)$ of the group homology of $\SL_2(\C)^{\delta}$, the discrete group underlying the Lie group $\SL_2(\C)$:
\[ K_3^{\mathrm{ind}}(\C) \cong H_3(\SL_2(\C)^{\delta},\Z)\]  
(an isomorphism induced by the Hurewicz map $K_3(\Qb)\ra H_3(\GL(\C)^{\delta},\Z)$). 
For this he first constructs an element $\tilde{c}(M)$ of the relative group homology $H_3(\SO(3,1)^{\delta},T(s)^{\delta},\Q)$ for some subgroup $T(s)$ of the stabilizer subgroup $E(s)$ of an ideal point $s\in\partial\mbH^3$.
Then, $c(M)$ is a lift to $H_3(\SL_2(\C),\Q)$ of $\varphi(\tilde{c}(M))$ along $H_3(\SL_2(\C)^{\delta},\Q)\ra H_3(\SL_2(\C)^{\delta},P^{\delta},\Q)$, where $\varphi:(Spin(3,1),T(s))\ra (\SL_2(\C),P)$  is the half-spin representation and $P$ is some subgroup of $\SL_2(\C)$.
We believe that this element $c(M)$ equals our mixed Tate motive via the above isomorphism.

Next, we explain why the mixed Tate motive produced by Goncharov's second construction method is non-canonical and unsatisfactory from the viewpoint of Chern-Simons invariant, and as such is not equal to our mixed Tate motive.

Let $X$ be one of the three geometries: $\mbE$ (Euclidean geometry), $\mbS$ (Spherical geometry), $\mbH$ (Hyperbolic geometry), and $G$ be the isometry group.
The \emph{scissors congruence group} $\mcP(X)$ of $X$ is the quotient group of the free abelian group generated by ``polytopes'' $P$ in $X$ by the relations $\sim$:
\begin{align*} 
(i) &\quad P\sim P_1+P_2 \Leftrightarrow P=P_1\dot{\sqcup} P_2\quad  (\text{``interior disjoint union''}) \\
(ii) &\quad P\sim gP\quad \text{for all }g\in G.
\end{align*}
For the definition of a polytope, we refer to \cite[$\S$2]{Dupont82}, \cite{Dupont01} (it needs a little caution, when $X=\mbS$). We denote the class of a polytope $P$ by $[P]$. For two polytopes $P$, $P'$, we have $[P]=[P']$ in $\mcP(X)$ if and only if $P$ and $P'$ are \emph{stably scissors congruent} (not scissors congruent!) in the sense that  there exist polytopes $Q$ and $Q'$ such that $P\sqcup Q \simeq P'\sqcup Q'$ with $Q'=gQ$ for some $g\in G$. 
Also, when we fix an orientation of $X$, $\pm[P]$ can be regarded as an oriented polytope with the negative sign corresponding to the orientation opposite to that of $X$.

Goncharov \cite[$\S$1.4, $\S$3.1]{Goncharov99} also defines scissors congruence group as the abelian group generated by pairs $[M,\alpha]$, where $M$ is an oriented geodesic simplex and $\alpha$ is an orientation of $X$, with the relations:
\begin{align*} 
(i') &\quad [M,\alpha]=[M_1,\alpha]+[M_2,\alpha]\ \text{ if }M=M_1\dot{\sqcup}M_2 \\
(i'') &\quad [M,\alpha] \text{ changes sign if we change orientation of }M\text{ or }\alpha \\
(ii') &\quad [M,\alpha]=[gM,g\alpha]\ \text{ for any }g\in G.
\end{align*}

%%%%%%%%%%%%%%%%%%%%
\begin{lem} \label{lem:SCgroup}
When $X\neq\mbS$, 
Goncharov's scissors congruence group is the same as $\mcP(X)$. 
\end{lem}
\begin{proof} 
When we fix an orientation $\alpha$ of $X$, Goncharov's scissors congruence group is also generated by $[M,\alpha]$, and we claim that the map 
\begin{equation} \label{eq:isom_bet'b_two_SCgroups}
[M,\alpha]\mapsto \epsilon(M)[|M|]
\end{equation}
(defined on generators) induces an isomorphism between the two scissors congruence groups, where $\epsilon(M)=\pm1$ according as the orientation of $M$ agrees with $\alpha$ or not and $[|M|]$ is the class in $\mcP(X)$ of the polytope $|M|$ underlying $M$ (i.e ignoring orientation). This map is invariant under the relation (i$''$), hence it suffices to show that \emph{in the presence of the relations (i), (i$'$)} which are obviously equivalent under this correspondence, the two relations (ii), (ii$'$) match. If we set $\epsilon(g)=\pm1$ according as $g\in G$ is orientation preserving or not, we have $\epsilon(gM)=\epsilon(g)\epsilon(M)$, so that 
\[ [gM,g\alpha]\stackrel{(i'')}{=}\epsilon(g)[gM,\alpha] \mapsto  \epsilon(g)\epsilon(gM)[|gM|]=\epsilon(M)[|gM|], \]
thus, (ii$'$) holds if and only if $[P]=[gP]$ for all polytopes $P=|M|$ and $g\in G$ ($g|M|=|gM|$), namely if and only if $P$ and $gP$ are stably scissors congruent. Since $G$ acts transitively on $X$, it is known \cite[p.5]{Dupont01} that stable scissors congruence implies scissors congruence, i.e. (ii).
\end{proof}

Goncharov's second construction of a mixed Tate motive for cusped hyperbolic manifold hinges crucially on the equality $\mcP(\mbH^n)= \mcP(\bmbH^n)$, where the latter is the scissors congruence group of the extended hyperbolic space $\bmbH^n=\mbH^n\sqcup \bdmbH^n$ and the equality is induced by the inclusion $\mbH^n \hra \bmbH^n$.
It is the group generated by geodesic $n$-simplices whose vertices are allowed to lie on $\bmbH^n$ while all other points must be finite, modulo the same relations as in $\mcP(\mbH^n$) (either for $G$ or $G^+$) (recall that the isometry group $G=\SO(n,1)$ of $\mbH^n$ still acts on $\bmbH^n$ and $\bdmbH^n$). Dupont \cite[$\S$2.Remark2]{Dupont82} and Sah (cf. \cite[$\S$3]{DupontSah82}) also introduce the scissors congruence group $\mcP(\bdmbH^n)$ of ideal (or ``totally asymptotic'' in their terminology) polytopes (i.e. polytopes with all vertices lying on the boundary $\bdmbH^n$). When $n=3$, this group is related to the pre-Bloch group.

%%%%%%%%%%%%%%%%%%%%
\begin{defn}
(1) The scissors congruence group $\mcP(\bdmbH^n)$ of $\bdmbH^n$ is the abelian group generated by all (ordered) $(n+1)$-tuples $(a_0,\cdots,a_n)$ of ideal points $a_i\in \bdmbH^n$ subject to the relations
\begin{align*} 
(i) &\quad (a_0,\cdots,a_n)=0\text{ if all $a_i$'s lie in a geodesic subspace of dimension less than $n$},\\
(ii) &\quad \sum_{0\leq i\leq n+1} (-1)^i(a_0,\cdots,\hat{a}_i,\cdots,a_{n+1})=0\quad  (a_0,\cdots,a_{n+1}\in \bdmbH^n),\\
(iii) &\quad \epsilon(g) (ga_0,\cdots,ga_n)=(a_0,\cdots,a_n),\quad (\forall a_i\in \bdmbH^n,\ \forall g\in G).
\end{align*}

(2) Thurston scissors congruence group $\mcP'(\bdmbH^n)$ is the abelian group defined similarly to 
$\mcP(\bdmbH^n)$, except that (i) and (iii) in the definition of $\mcP(\bdmbH^n)$ are replaced respectively by
\begin{align*} 
(i') &\quad (a_0,\cdots,a_n)=0\text{ if }a_i=a_j \text{ for some }i\neq j,\\
(iii') &\quad (ga_0,\cdots,ga_n)=(a_0,\cdots,a_n),\quad (\forall a_i\in \bdmbH^n,\ \forall g\in G^+).
\end{align*}
\end{defn}

\begin{rem}
(1) We remark that in the definitions of the scissors congruence groups $\mcP(\mbH^n)$ (both definitions) and $\mcP(\bmbH^n)$, one could have used the subgroup $G^+$ of orientation-preserving isometries, instead of the full isometry group $G$ (\cite[Thm.2.2]{Dupont01}). As $G^+$ is a normal subgroup of $G$ of index $2$ which is generated by $G^+$ and any reflection in a hyperplane, this means that any two polytopes in $\mbH^n$ in mirror image are (stably) scissors congruent under orientation-preserving motions only (the trick is to use barycentric subdivision, \textit{loc. cit.}). 
But, the two definitions of  $\mcP(\bdmbH^n)$ and $\mcP'(\bdmbH^n)$ are distinguished mainly by use of different groups $G$ or $G^+$. The reason why in this situation use of different groups results in different definitions is that for \emph{ideal} polytopes in mirror image, the barycentric subdivision trick which was efficient for $\mcP(\mbH^n)$, $\mcP(\bmbH^n)$ (\text{lot. cit.}) does not work, since the barycentric subdivision of an ideal simplex is not any longer made of ideal simplices. As we will see later in this subsection, this difference makes the group $\mcP'(\mbH^n)$ a better object than the scissors congruence groups $\mcP(\mbH^n)$, $\mcP(\bmbH^n)$, $\mcP(\bdmbH^n)$.

(2) In the definition of $\mcP'(\bdmbH^n)$, we obtain the same group if we impose the restriction that all $a_i\neq a_j$ for $i\neq j$ in all the conditions, except for (i$'$) which is replaced back by (i) (\cite{DupontSah82}, Remark after Corollary 4.7). Therefore, when $n=3$, $\mcP'(\bdmbH^3)$ is identified with the pre-Bloch group $\mcP(\C)$, via the map $(g_0,\cdots,g_3)\mapsto [[g_0:\cdots:g_3]]$ (so, $(\infty,0,1,z)\mapsto [z]$). It is obvious that the composite map $\mcP(\C)=\mcP'(\bdmbH^3)\ra \mcP(\bdmbH^3) \ra \mcP(\bmbH^3)$ is given by (cf. the bijection (\ref{eq:isom_bet'b_two_SCgroups})):
\begin{equation} \label{eq:P(C)raP(bdH^3)}
\mcP(\C)\ra \mcP(\bmbH^3): [z]\mapsto \epsilon(z)[|\Delta(\infty,0,1,z)|],
\end{equation} 
where $\epsilon(z)=\pm$ depending on whether the oriented (by vertex order) simplex $\Delta(\infty,0,1,z)$ has the same orientation as the ambient $\mbH^3$ or not, and $|-|$ is the underlying polytope with $[-]$ being its class.

Any reflection $\tau$ with respect to a geodesic hyperplane acts on $\mcP'(\bdmbH^n)$, and likewise the complex conjugation $\iota$ on $\mcP(\C)$. Let $\mcP'(\bdmbH^n)_{-}:=\mcP'(\bdmbH^n)/\langle \tau x+x\rangle$ and $\mcP(\C)_{-}:=\mcP(\C)/\langle [\bar{z}]+[z] \rangle$ denote the corresponding $(-1)$-coeigenspaces. Note that the ideal simplices $\Delta(\infty,0,1,z)$, $\Delta(\infty,0,1,\bar{z})$ are the mirror image of each other (with respect to the geodesic hyperplane spanned by $\{\infty,0,1\}$). 
Therefore, as $\epsilon(\bar{z})=-\epsilon(z)$, the natural map (\ref{eq:P(C)raP(bdH^3)}) factors through $\mcP(\C)_{-}=\mcP'(\bdmbH^3)_{-}$.
\end{rem}

%%%%%%%%%%%%%%%%%%%%
\begin{thm} \label{thm:orientation-sensitive_SCG}
Suppose $n=3$. There are natural isomorphisms (induced by natural maps)
\[ \mcP(\C)_{-}\isom \mcP'(\bdmbH^3)_{-}  \isom  \mcP(\bdmbH^3) \isom \mcP(\bmbH^3) \stackrel{\sim}{\leftarrow} \mcP(\mbH^3). \]
\end{thm}

In particular, we have $\mcP(\C)_{-}\cong \mcP(\mbH^3)$ (\cite[Thm. 2.4]{Neumann98}), which can be interpreted as saying that the pre-Bloch group $\mcP(\C)$ is an ``orientation-sensitive scissors congruence group''. One should not take the equality $\mcP(\mbH^3)\isom \mcP(\bmbH^3)$ as meaning that every polytope in $\mcP(\bmbH^3)$ is scissors congruent to a polytope in $\mcP(\mbH^3)$: ``an infinite geodesic can never be cut up into a finite number of pieces and placed inside $\mbH^n\ (n>0)$''. It just says that it is \emph{stably scissors congruent} in $\mcP(\bmbH^3)$ to a finite polytope (\cite[Remark after Cor.4.7]{Sah81}).

\begin{proof}
This is Corollary 8.18 of \cite{Dupont01}. More precisely, we have seen the isomorphism $\mcP(\C)_{-}\isom \mcP'(\bdmbH^3)_{-}$.
The bijectivity $\mcP'(\bdmbH^3)_{-} \isom \mcP(\bdmbH^3)$ is shown in \cite[(5.24)]{DupontSah82}.
The bijectivity $\mcP(\bdmbH^3) \isom \mcP(\bmbH^3)$ follows from that for odd $n>2$, the natural map $\mcP(\bdmbH^n) \ra \mcP(\bmbH^n)$ is surjective with at most $2$-torsion kernel \cite[Prop.3.7,(ii)]{DupontSah82} and torsion-freeness of $\mcP(\bdmbH^3)$ (which holds since $\mcP(\C)$ is uniquely divisible \cite[Thm.8.16]{Dupont01}).
The last isomorphism $\mcP(\bmbH^3) \stackrel{\sim}{\leftarrow} \mcP(\mbH^3)$ (which holds for general $n>1$) is shown in \cite[Thm.2.1]{DupontSah82} (see also \cite[Prop.3.3]{Sah81}).
\end{proof}

In the Goncharov's second method of constructing a mixed Tate motive in $ \Ext^1_{\MT(\Qb)}(\Q(0),\Q(2))$ associated with a complete hyperbolic three-manifold of finite volume, the input is the scissor congruence class in $\mcP(\mbH^3)$ of the manifold (rather than the Neumann-Yang invariant $\beta(M)$ in the Bloch group); especially, even in the cusped case he works only with finite geodesic simplices, instead of ideal simplices, using the equality $\mcP(\mbH^3)\cong \mcP(\bmbH^3)$. But, as W. Neumann \cite[p.388-399]{Neumann98} points out, elements of the scissors congruence group cannot detect orientation-sensitive invariants, such as Chern-Simons invariants $CS$. Indeed, if an orientation-preserving diffeomorphism $\phi:M\ra \bar{M}$ of complete oriented hyperbolic three-manifolds is a mirror reflection, in the sense that for some geodesic triangulation of $M$, the restriction of $\phi$ to each constituent geodesic simplex is a mirror reflection (like the reflection $\Delta(\infty,0,1,z)\mapsto \Delta(\infty,0,1,\bar{z})$, both being oriented by vertex orders), then $M$ and $-\bar{M}$ ($\bar{M}$ with reversed orientation) have the same associated scissors congruence classes, but $CS(-\bar{M})=-CS(\bar{M})=-CS(M)$ is not equal to $CS(M)$ in general (unless it is trivial).%%
\footnote{This can be also deduced from the equality $2\Vol_{\C}(M)=\sqrt{-1} \rho(\beta(M))$ (Theorem \ref{thm:2Complex_volume=Bloch_regulator}) and $\overline{\rho(z)}=\rho(\bar{z})$ 
%(at least on $\{z\in \C|\ |z|<1,  |1-z|<1\}$), 
which follows from the same property for $\mcR(z)=\Li_2(z)+\frac{1}{2}\log(z)\log(1-z)$ (\ref{eq:Rogers_dilogarithm}).}
So there is no hope of expressing the complex volume of a hyperbolic three-manifold in terms of the Goncharov's mixed Tate motive constructed from scissors congruence classes. 

%%%%%%%%%%%%%%%%%%%%%%%%%%%%%%%%%%%%%%%%
%%%%%%%%%%%%%%%%%%%%%%%%%%%%%%%%%%%%%%%%
\section{Interlude: Polylogarithm variation of mixed Hodge structure}

\subsection{Polylogarithm variation of mixed Hodge structure over $\mathbb{P}^1\backslash\{0,1,\infty\}$} \label{subsec:PolyVMHS}

Here, we follow \cite{Hain94}, \cite{BeilinsonDeligne94} to give an account of the relevant part of the theory of polylogarithms. We also recommend \cite{BurgosGilFresan} for a detailed expansion of \cite{BeilinsonDeligne94}.

For $k\in\N$, the $k$-th polylogarithm function $\Li_k(z)$ is an analytic function on $|z|<1$ defined by
\begin{equation} \label{eq:Li_k}
\Li_k(z)=\sum_{n=1}^{\infty}\frac{z^n}{n^k}.
\end{equation}
These are iterated integrals of differential forms with logarithmic poles on $\mathbb{P}^1(\C)\backslash\{0,1,\infty\}$:
\begin{align*}
\Li_1(z) &=-\log(1-z)=\int_0^z\frac{dt}{1-t}, \\
\Li_{k+1}(z) &= \int_0^z\Li_k(t)\frac{dt}{t}\ \quad (k\geq1).
\end{align*}
This expression as iterated integrals shows that $\Li_k(z)$, defined by (\ref{eq:Li_k}) for $|z|<1$, can be analytically continued as a \emph{multivalued} function on $\mathbb{P}^1(\C)\backslash\{0,1,\infty\}$. Many of the properties of these polylogarithms, especially their monodromy properties, are understood best when considered together. For $n\in\N$, we recall the framed polylogarithm Hodge-Tate structure $\mcP^{(n)}$ (\ref{eq:PolyHT}); we denote it by $\Lambda^{(n)}(z)$ when we regard it as a matrix with entries in the ring of multivalued functions on $\mathbb{P}^1(\C)\backslash\{0,1,\infty\}$:
\begin{equation} \label{eq:Lambda;P^1-3pts}
\Lambda^{(n)}(z)= \left( \begin{array} {cccccc}
1 & -\Li_1(z) & -\Li_2(z) & \cdots & \cdots & -\Li_n(z) \\
0 & \tpi & \tpi \log z & \tpi \frac{(\log z)^2}{2} & \cdots & \tpi \frac{(\log z)^{n-1}}{(n-1)!} \\
\vdots & 0 & (\tpi)^2 & (\tpi)^2 \log z & & (\tpi)^2 \frac{(\log z)^{n-2}}{(n-2)!} \\ 
\vdots & &  & \ddots & & \vdots \\
0 & & & & (\tpi)^{n-1} & (\tpi)^{n-1} \log z \\
0 & & & &  & (\tpi)^n
\end{array} \right) 
\end{equation}
(when we index the rows and columns by the set $[0,n]:=\{0,1.\cdots,n\}$, for $j\geq1$ the $j$-th row $\lambda_j$ of $\Lambda(z)$ equals $(\tpi)^j$ times 
\[ [0,\cdots,0,1,\log z,\frac{\log z}{2},\cdots, \frac{(\log z)^{n-j}}{(n-j)!}] \] 
with $1$ appearing in the $j$-th position and for $k\geq j$ the $k$-th entry is $\frac{(\log z)^{k-j}}{(k-j)!}$).
It is known that the row vectors $\lambda_j(z)$ of $\Lambda^{(n)}(z)$ are fundamental solutions of the first order differential equation
\[ d\lambda =\lambda \mathbb{\omega}, \]
where 
\begin{equation} \label{eq:omega;P^1-3pts} 
\mathbb{\omega}=\left( \begin{array} {ccccc}
0 & \omega_1 & 0 & \cdots & 0 \\
& \ddots & \omega_0 & \ddots & \vdots \\
\vdots & & \ddots & \ddots & 0 \\
& & & \ddots & \omega_0 \\
0 & & \cdots & & 0
\end{array} \right) 
 \quad (\omega_0:=\frac{dz}{z},\ \omega_1:=\frac{dz}{1-z}).
 \end{equation}
 and that the monodromy representation 
 \[ M:\pi_1(\mathbb{P}^1(\C)\backslash\{0,1,\infty\}) \ra \GL_{n+1}(\C)\] 
 (i.e. the analytic continuation of $\Lambda(z)$ along $\gamma$ equals $M(\gamma)\Lambda^{(n)}(z)$) is valued in $\GL_{n+1}(\Q)$.

In other words, $\lambda_i(z)\ (0\leq i\leq n)$ is a flat section (in fact, a multivalued global flat section) of the trivial vector bundle $\mathcal{O}_X^{\oplus n+1}$ of rank $n+1$ over $X:=\mathbb{P}^1\backslash\{0,1,\infty\}$ with integral connection
 \begin{equation} \label{eq:nabla;P^1-3pts}
 \nabla f=df -f\mathbb{\omega} 
 \end{equation}
and the sections $\lambda_0,\cdots,\lambda_{n}$ generate a $\Q$-local system $\mathbb{V}$ over $X$. This is called the \emph{$n$-th polylogarithm local system}.

 %%%%%%%%%%%%%%%%%%%%
\begin{thm} \cite[$\S$2,$\S$7]{Hain94}, \cite[$\S$1]{BeilinsonDeligne94}
The $n$-th polylogarithm local system underlies a good variation of mixed Hodge structure whose weight-graded quotients are canonically isomorphic to $\Gr^W_{-2i}=\Q(i)$, $\Gr^W_{-2i+1}=0$ ($0\leq i\leq n$). It has unipotent monodromy at each point of $\{0,1,\infty\}$. They form a projective system of mixed Hodge structures as $n$ varies.
\end{thm}

For the notion of a good variation of mixed Hodge structure, see \cite{HainZucker87a}, \cite{HainZucker87b}, where a good variation of mixed Hodge structure whose weight-graded quotients are sums of copies of some $\Q(i)$ as in the theorem is called a \emph{Tate variation of mixed Hodge structure}.

The Hodge filtration  $\mcF^{\bullet}$ and the weight filtration $\mbW_{\bullet}$ of the $n$-th polylogarithm variation of mixed Hodge structure $\mcP^{(n)}=(\mcV=\mathcal{O}_X^{\oplus n+1},\nabla,\mbW_{\bullet},\mcF^{\bullet})$ in the theorem are defined as follows:  
Let $\{e_0,\cdots,e_n\}$ be the standard basis of $\C^{n+1}$. Define \textit{complex} weight filtration and Hodge filtration on $V=\C^3$ by:
\begin{align} \label{eq:PLVMHS_HF}
W_{-2l+1}\C^{n+1}=&W_{-2l}\C^{n+1}=\C\langle e_l, \cdots, e_n \rangle, \\
F^{-p}\C^{n+1}=& \C\langle e_0, \cdots, e_p \rangle, \nonumber
\end{align}
We put $\msF^p:=F^pV\otimes\mcO_X \subset \mcV=\C^{n+1}\otimes\mcO_X$. 
We define $\mbV_z$ to be the $\Q$-subspace of $V=\C^{n+1}$ spanned by $\{\lambda_0(z),\cdots,\lambda_n(z)\}$ endowed with the filtration by $\Q$-subspaces:
\begin{equation} \label{eq:PLVMHS_WF}
W_{-2l+1}\mbV_z=W_{-2l}\mbV_z=\Q\langle \lambda_l(z), \cdots, \lambda_n(z) \rangle 
\end{equation}
This $\Q$-filtration $W_{\bullet}\mbV_z$ is well-defined, i.e. independent of the choice of the branch of $\log z$ (as the monodromy is valued in $\GL_{n+1}(\Q)$), and is a $\Q$-structure of the $\C$-filtration $W_{\bullet}V$, so it follows that there exists an increasing filtration $\mbW_l$ of $\mbV$ by $\Q$-local sub-systems; the graded quotients of the mixed Hodge structure $V_z$ are $\Z(0)$, $\Z(1)$, $\cdots$, $\Z(n)$.

When $V(z)$ denotes the variation of mixed Hodge structure which is the extension of $\Q$ by $\Q(1)$ corresponding to $z\in\mathcal{O}^{\times}(\mathbb{P}^1\backslash\{0,1,\infty\})$, the $n$-th polylogarithm variation of mixed Hodge structure is an extension of $\Q$ by the shift $(\mathrm{Sym}^{n-1}V(z))\otimes\Q(1)$ of the $(n-1)$-rd symmetric power of $V(z)$ (\cite[Prop.9.5]{Hain94}), and the polylogarithm functions $\{\Li_k\}_{0\leq k\leq n}$ can be regarded as an extension data. 
In the case $n=2$, the entries of $\Lambda(z)$ can be understood more directly as encoding the various extensions of the Hodge-Tate structures (see $\S$9 of \cite{Hain94} or subsection \ref{subsec:TateVMHS} of the main body for more details); in particular, the extension of $\Q$ by $V(z)\otimes\Q(1)$ corresponds to the nontrivial $(0,2)$-entry of $\Lambda(z)$, namely the dilogarithm function $\Li_2(z)$ which appears in many different areas of mathematics.

On the other hand, some well-defined limits of polylogarithm functions at the points at infinity $\{0,1,\infty\}$  (obtained as sutiable regularized integrals) turn out to be quite interesting numbers. These limits are expressed best in terms of limit mixed Hodge structures (\`a la W. Schmid \cite{Schmid73}) at suitable tangent vectors at the points at infinity. 

\begin{thm} \cite[Thm.7.2]{Hain94} \label{thm:limitMHS_PLVMHS}
Let $z$ be the natural coordinate function on $\C-\{0,1\}$ and $\frac{\partial}{\partial z}$ be the associated tangent vector at $0$. The limit mixed Hodge structure at $\frac{\partial}{\partial z}$ of the $n$-th polylogarithm variation of mixed Hodge structure  splits, i.e. equals the split mixed Hodge-Tate structure $\Q(0)\oplus \Q(1)\oplus \cdots \oplus \Q(n)$, and that at 
$-\frac{\partial}{\partial z}$ (regarded as a tangent vector at $1$ ``pointing towards $0$'') has $\zeta(2),\cdots,\zeta(n)$ as periods (of mixed Hodge-Tate structure).
\end{thm}

Beilinson and Deligne \cite{BeilinsonDeligne94} conjectured that the polylgarithmic variation of mixed Hodge structure is motivic, i.e. is the Hodge realization of a ``mixed motivic sheaf'' over $\mathbb{P}^1\backslash\{0,1,\infty\}$; lacking a good formalism of (mixed) motivic sheaves at the time, their conjecture instead stated that it is the Hodge realization of a ``realization system'' in the sense of Deligne \cite{Deligne89}, which admits a geometric description in terms of the \emph{fundamental groupoid} of the projective line minus three points. We now explain this in more detail.

For a topological space $X$, let $P_{x,y}X$ denote the set of homotoy classes of paths in $X$ from $x$ to $y$, and $\Q[P_{x,y}X]$ the free abelian group on $P_{x,y}X$. As $x$, $y$ vary, these form a local system $\{\Q[P_{x,y}X]\}_{(x,y)}\ra X\times X$. There exists a filtration (in local systesms) given by the powers of the augmentation ideal $J$; denote by $\Q[P_{x,y}X]^{\wedge}$ the completion of $\Q[P_{x,y}X]$ with respect to the powers of $J$. 

Now, when $X$ is an algebraic variety over $\C$, Hain showed that the local system $\{\Q[P_{x,y}X]^{\wedge}\}_{(x,y)}\ra X\times X$ underlies a good variation of mixed Hodge structure whose fiber over $(x,x)$ is the canonical mixed Hodge structure constructed by Chen on the unipotent completion $\Q\pi_1(X,x)^{\wedge}$ of the group algebra $\Q\pi_1(X,x)=\Q[P_{x,y}X]$. 
Moreover, Hain and Zucker \cite{HainZucker87a} (cf. \cite[11.2]{Hain94}) showed that 

\begin{thm}
If $\mcV\ra X$ is a \emph{unipotent} variation of mixed Hodge structure over a smooth curve, the natural map
\[ \mcV_x\otimes \Q[P_{x,y}X]^{\wedge} \ra \mcV_y \]
induced by parallel transport is a morphism of mixed Hodge structures for every $x,y\in X$.
\end{thm}

Here, we stress that for base point $x$ and/or $y$, one can allow a ``tangential base point'' $\overrightarrow{v}$ in the sense of \cite[$\S$15]{Deligne89}, in which case the fiber at $\overrightarrow{v}$ of $\mcV$ is nothing other than the limit mixed Hodge structure of Schmid and $P_{\frac{\partial}{\partial z},x}X$ is interpreted as the set of homotopy classes of paths leaving $0$ in the direction of $\frac{\partial}{\partial z}$ towards $x$.%%
\footnote{or from a point in the punctured tangent space $T_0^{\ast}X$ with coordinate $\frac{\partial}{\partial z}$ to $x$}

Now, when $\mcV=\mcMn$ and $\overrightarrow{v}$ is $\frac{\partial}{\partial z}$ for the standard coordinate $z$ of $\C$, by Theorem \ref{thm:limitMHS_PLVMHS}, there exists a canonical inclusion of mixed Hodge structure $\Q(0)\ra \mcMn_{\frac{\partial}{\partial z}}$, hence the parallel transport map induces a canonical morphism of MHS
\[ \Q[P_{\frac{\partial}{\partial z},x}X]^{\wedge} \ra \mcMn_x \]

\begin{thm} \cite[Thm.11.3]{Hain94} \label{thm:Hain94;Thm.113}
The polylogarithm variation $\mcMn$ is the quotient of the variation of mixed Hodge structure $\Q[P_{\frac{\partial}{\partial z},\ast}X]^{\wedge}$ whose fiber 
at $\frac{\partial}{\partial z}$ is the quotient of $\Q\pi_1(X,\frac{\partial}{\partial z})^{\wedge}$ by the right ideal generated by
\[ M(\sigma_0)-1,\quad J(M(\sigma_1)-1),\quad J^n \]
where $M(\sigma_0)$ is the local monodromy at the tangential base point $\frac{\partial}{\partial z}$ wrapping $0$ counterclockwise, $M(\sigma_1)$ is the local monodromy at $\frac{\partial}{\partial z}$ which moves along the path $(0+\epsilon,1-\epsilon)\ (0<\epsilon\ll1)$ in the real line and wrapping $1$ counterclockwise once and follows the same path back, and $J$ is the augmentation ideal.
\end{thm}

This has the following remarkable consequence, observed by Deligne. For $X=\mathbb{P}^1\backslash\{0,1,\infty\}$, Deligne and Goncharov \cite[3.12]{DeligneGoncharov05} (cf. \cite[Thm.4.144]{BurgosGilFresan}) constructed the space of \emph{motivic paths} from $a$ to $b$: $P_{b,a}(X)=\Spec(A_{b,a}(X))$ (when $b=a$, by definition this is the \emph{motivic fundamental group} $\pi_1^{\mot}(X,a)$ of $X$ with base point $a$). This is a scheme in the Tannakian category $\MT(\Q)$ of mixed Tate motives over $\Q$ (see \cite[$\S$5]{Deligne89} for the notion of ``schemes in a Tannakian category'').
Under each realization functor, this motivic path torsor becomes the path torsor scheme in that realization category. For example, the Betti realization of $\pi_1^{\mot}(X,x)$ is the pro-unipotent algebraic envelope of $\pi_1(X(\C),s)$ \cite[$\S$9]{Deligne89}, that is the spectrum of the (commutative) Hopf algebra $(\Q\pi_1(X,x)^{\wedge})^{\vee}$ dual to the (cocommutative) Hopf algebra $\Q\pi_1(X,x)^{\wedge}$. 
When $a,b$ are (tangential) base points of $X$ defined over $\mcO_S$ for a ring $\mcO_S$ of $S$-integers ($k$ and $S$ being a number field and a finite set of places of $k$), this also defines an object in the similarly defined Tannakian category $\MT(\mcO_S)_{\Q}$ of mixed Tate motives over $\mcO_S$.
They also showed (\cite[Prop.2.4]{DeligneGoncharov05}) that the Hodge realization functor from the abelian category $\MT(\Q)$ of mixed Tate motives over $\Q$ to the category $\QHT$ of Hodge-Tate structures is fully faithful and its essential image is stable under taking sub-objects. 
Then, since the local monodromy is also a mixed Tate motive (\textit{ibid.} 5.4), from the above theorem (Theorem \ref{thm:Hain94;Thm.113}) we conclude that:
\begin{cor}
At any tangential base point $\overrightarrow{v}$, the limit mixed Hodge structure $\mcMn_{\overrightarrow{v}}$ is ``motivic''. More precisely, it is the Hodge realization of a mixed Tate motive over $\Z$. 
\end{cor}

We remark that Beilinson and Deligne \cite{BeilinsonDeligne92} sketched an explicit construction of \emph{motivic polylgoarithm}, ``motivic sheaf'' over $\mathbb{P}^1\backslash\{0,1,\infty\}$ whose Hodge realization is the polylogarithm VMHS) as an element of a certain K-group (see \cite{HuberWildeshaus98} for  details of their construction).

 %%%%%%%%%%%%%%%%%%%%
\subsection{What we do in the second part}
First we construct a unipotent variation of mixed $\Q$-Hodge structure over the smooth locus $\tilde{X}(M)_0^{\mathrm{sm}}$ of the canonical component $\tilde{X}(M)_0$ of the augemented character variety $\tilde{X}(M)$; by abuse of terminology, we will call $\tilde{X}(M)_0$ \textit{canonical curve}. Our construction follows the idea of Morishita-Terashima \cite{MorishitaTerashima09} who constructed a similar variation of mixed Hodge structure using Chern-Simons invariant, which is regarded as a section of a line bundle. Their variation of mixed Hodge structure however lives on the (Thurston) deformation curve, instead of our curve $\tilde{X}(M)_0^{\mathrm{sm}}$: the deformation curve of a hyperbolic three-manifold $M$ depends on the choice of an ideal triangulation of $M$, while the character variety (and thus $\tilde{X}(M)_0^{\mathrm{sm}}$) is canonically attached to $M$. We give a precise relation between Morishita-Terashima's construction and ours (based on Kirk-Klassen's work) in Proposition \ref{prop:KK=MT-CS_sections}.
As will be seen, our construction of Chern-Simons variation of mixed Hodge structure (CS VMHS, for short) is quite similar to the well-known construction of $2^{\text{nd}}$-polylogarithm variation of mixed Hodge structure (compare (\ref{eq:nabla;P^1-3pts}), (\ref{eq:PLVMHS_HF}), (\ref{eq:PLVMHS_WF}) with the corresponding constructions in \ref{subsec:CSVMHS}): the role of the dialogarithm function in the polylogarithm variation of mixed Hodge structure will be played by a sum of the Chern-Simons invariant and the product of the log-holonomies of the meridian and the longitude.
Then, from this CS VMHS, it is natural to aim to establish statements corresponding to Theorem \ref{thm:limitMHS_PLVMHS}, Theorem \ref{thm:Hain94;Thm.113} (with a view towards proving the motivicity of our CS VMHS) \emph{following the same strategy of proof}. Hence we need to consider the points at infinity of the canonical curve $\tilde{X}(M)_0$ (this curve is affine, so there always exist ``points at infinity'', which in three-dimensional topology are more often called \emph{idea points}).
Here, we make the statement corresponding to \ref{thm:limitMHS_PLVMHS} in our setting as an assumption (Conjecture \ref{conj:CS-VMHS_at_ideal_pt}); this concerns an asymptotic behavior of the Chern-Simons invariant near an ideal point. We verify this assumption in the appendix in the case of the figure-eight knot complement and another knot. Assuming this conjecture, we proceed to prove the analogue of Theorem \ref{thm:Hain94;Thm.113} by the same method of using local monodromies at tangential base points at ideal points. Here, we establish such an analogue whose proof requires some nontrivial facts in three-dimensional topology of $M$ (this is another satisfying aspect of our work). Namely, we show (Theorem \ref{thm:CS-VMHS_is_motivic}) that this so-called Chern-Simons variation of mixed Hodge structure is a quotient of the mixed Hodge structure on the unipotent completion of the path torsor.

At this point, if we accept some popular expectations in the theory of motives (most notably, existence of the motivic $t$-structure in the triangulated category of mixed motives \cite{Beilinson12}), 
our main theorem implies that the Chern-Simons variation of mixed Hodge structure is motivic, i.e. is the Hodge realization of a mixed motive over $\Qb$.
Here, we bring the reader's attention to the fact that unlike in the case of the projective line minus three points, there is no guarantee that the object in the triangulated category of mixed motives constructed by Deligne and Goncharov using cosimplicial model of the path torsor belongs to the triangulated subcategory of mixed Tate motives.

%%%%%%%%%%%%%%%%%%%%%%%%%%%%%%%%%%%%%%%%
%%%%%%%%%%%%%%%%%%%%%%%%%%%%%%%%%%%%%%%%
\section{Chern-Simons line bundle and Chern-Simons section}

%%%%%%%%%%%%%%%%%%%%%%%%%%%%%%%%%%%%%%%% \subsection{Chern-Simons invariant of flat connections on trivial principal $\SL_2$-bundle on cusped hyperbolic $3$-manifolds} \label{subsec:CS-inv_flat-conn:cusped}
\subsection{Chern-Simons invariant as a section of a line bundle}   \label{subsec:CS-inv_flat-conn:cusped}

For a closed Riemannian three-manifold $M$, the Chern-Simons integral  $\int_MQ(A)$ of a flat connection $A$ on a trivial(ized) principal $(\mathrm{P})\SL_2(\C)$-bundle (regarded as $\Lsl_{2,\C}$-valued $1$-form on $M$) is gauge-invariant, thus gives a well-defined $\C/\Z$-valued invariant $cs_M(A)$ of $A$.
For non-closed manifolds, this integral is not gauge-invariant. For cusped hyperbolic three-manifolds,
Kirk and Klassen \cite{KirkKlassen90}, \cite{KirkKlassen93} defined the Chern-Simons invariant of an aribtrary \emph{flat} connection $A$ by 
\begin{equation} \label{eq:cs_M(gA)}
cs_M(A):=cs_M(g\cdot A)=\int_MQ(g\cdot A),
\end{equation}
where $g\in\mathcal{G}$ is chosen such that the gauge transformed connection $g\cdot A$ is in normal form (Proposition \ref{prop:KK93_2.3,3.2}, (1)).
If one can choose a canonical such gauge transform $g\cdot A$, then this will give a well-defined invariant of $A$ likewise. But in general, there is no unique $A'$ that is gauge equivalent to $A$ and is in normal form, and the Chern-Simons integrals mod $\Z$ (\ref{eq:cs_M(gA)}) of such $A'$ and $A$ still might differ; especially, it is not possible to define $cs([\rho])$ unambiguously for a conjugacy class $[\rho]$ of holonomy representation $\rho: \pi_1(M)\ra \SL_2(\C)$. 

However, Kirk and Klassen determined the precise difference in the Chern-Simons integrals of two gauge equivalent $1$-forms $A'$, $A$ in normal form which turn out to depend only on their normal forms near boundary (Theorem \ref{thm:KK93_2.5}). This allowed them to view the Chern-Simons invariant (\ref{eq:cs_M(gA)}) as a section over the character variety $X(M)$ of $M$ of certain principal $\C^{\times}$-bundle(=line bundle) on the character variety $X(\partial M)$ of the boundary $\partial M$ (we will call this section \emph{Kirk-Klassen Chern-Simons section}). We remark that this viewpoint is originally due to Ramadas-Singer-Weitsman \cite{RSW89}.

On the other hand, following the ideas of Kirk and Klassen, Morishita and Terashima \cite{MorishitaTerashima07} constructed another line bundle over the algebraic torus $(\C^{\times})^{2h}$ ($h$ being the number of boundary tori) and a section of it over the Thurston deformation curve, again by means of the Chern-Simons integral (\ref{eq:cs_M(gA)}) (we will call this section \emph{Morishita-Terashima Chern-Simons section}). But, their construction of line bundle is based on Heisenberg group and its relation to the one of Kirk and Klassen is not clear. Here, we provide an exact relation between these two line bundles.

 A remarkable aspect of the Morishita-Terashima Chern-Simons section, which will be a key to our work, is that it can be interpreted (in fact, was constructed from the beginning) as a (good unipotent) variation of Hodge-Tate structures. But, whereas the Kirk-Klassen Chern-Simons section exists on the character variety and uniquely determined by the hyperbolic manifold, the Morishita-Terashima Chern-Simons section exists over the deformation curve, thus depends on the choice of an ideal triangulation and as such is non-canonical. 
For our purpose of construction of canonical motivic invariants of hyperbolic manifolds, we combine these two works. Especially, we mimic the Morishita-Terashima constructions over the \emph{augmented character variety}, a certain double covering of the character variety which can be thought of as a triangulation-choice-free analogue of the Thurston deformation curve.

%%%%%%%%%%%%%%%%%%%%
\subsection{Character varieties} \label{subsec:char_scheme}

In the works on three-manifolds, there are used a few related, but slightly different definitions of $\SL_2$-character varieties. Here, we explain three definitions: Culler-Shalen construction, GIT quotient of representation variety, affine scheme of trace ring. The latter two definitions are equivalent over fields and most commonly used nowadays, while old works (such as \cite{KirkKlassen90}, \cite{KirkKlassen93}) are built on the pioneering definition introduced by Culler and Shalen \cite{CullerShalen83}. We also discuss $\PSL_2$ character varieties. In this work, by a variety we mean a separated scheme of finite type over a field, which is not necessarily reduced or irreducible.

Over a field $k$ of characteristic \emph{zero}, the character variety of an abstract group is most often defined as the GIT-quotient of the representation variety. 

First, it is easily seen (\cite[$\S$5]{Sikora12}) that there exists an affine $\Z$-scheme $R(\Gamma)$ that represents the functor:
\begin{equation} \label{eq:representation_functor}
(\text{algebras})\ra (\text{sets})\ :\ A \mapsto \Hom_{\text{gp}}(\Gamma,\SL_2(A));
\end{equation}
we call it ($\SL_2$-)representation $\Z$-scheme of $\Gamma$. 
More precisely, there exist an algebra $A(\Gamma)$ over $\Z$ (called the \textit{universal representation ring}) and a ``universal representation'' 
\begin{equation} \label{eq:univ_rep'n}
\rho^{\univ}:\Gamma\ra \SL_2(A(\Gamma))
\end{equation}
such that every representation $\rho:\Gamma\ra \SL_2(A)$ is obtained from $\rho^{\univ}$ via a unique homomorphism $\phi:A(\Gamma)\ra A$: 
\begin{equation} \label{eq:Rep_scheme}
\Hom_{\text{gp}}(\Gamma,\SL_2(A))=\Hom_{\alg}(A(\Gamma), A).
\end{equation}

Since for any algebra $A$, the group $\mathrm{PGL}_2(A)$ acts on the left-hand side of \ref{eq:Rep_scheme} via the conjugation action on $\SL_2(A)$, it follows that the group scheme $\mathrm{PGL}_2$ acts on $A(\Gamma)$ (this is also clear from an explicit construction of $A(\Gamma)$, cf. loc. cit.).

\begin{defn} 
The $\SL_2$-character variety of a group $\Gamma$ is the affine $\Q$-variety:
 \[ X(\Gamma):= \Spec(A(\Gamma)_{\Q}^{\mathrm{PGL}_2}) \]
 (Spec of the subalgeba of functions in $A(\Gamma)_{\Q}$ invariant under $\mathrm{PGL}_2$).
 \end{defn}
This is also known as the (affine) GIT quotient of the representation variety $R(\Gamma)_{\Q}$ and is written as $R(\Gamma)_{\Q}\sslash \mathrm{SL}_2$.

The following fact is deduced from a standard result in geometric invariant theory, cf. \cite[Thm.2.2]{Marche16}: 

%%%%%%%%%%%%%%%%%%%%
\begin{thm} \label{thm:fund_property_GIT}
If $k$ is algebraically closed field, there is a bijection among the following sets:
\begin{itemize}
\item[(i)] the $k$-points of $X(\Gamma)$ (=$\Hom_{k\mhyphen\text{alg}}([A(\Gamma)_k]^{\PGL_2},k)$);
\item[(ii)] the \emph{closed} orbits of $\SL_2(k)$ acting on $R(\Gamma)(k)=\Hom(A(\Gamma)_k,k)$;
\item[(iii)] the conjugacy classes of \emph{semi-simple} representations $\Gamma\ra \SL_2(k)$;
\item[(iv)] the set of characters of representations $\Gamma\ra \SL_2(k)$.
\end{itemize}
\end{thm}

For $\PSL_2$-character varieties, there does not seem to be a standard definition for an arbitrary finitely generated group $\Gamma$: one difficulty is that the group-valued functor $R\mapsto \PSL_2(R):=\SL_2(R)/\{\pm 1\}$ on the category of $\Q$-algebras is not representable, so neither is the functor (\ref{eq:representation_functor}) for $\PSL_2$.%%
\footnote{Some people , e.g. \cite[$\S$3]{BoyerZhang98}, use the embedding $\PSL_2 \hra \SL_3$ (adjoint representation) to avoid this issue.}
However, when $H^2(\Gamma,\Z/2)=0$, there is a reasonable definition of $\PSL_2$-character variety, at least over $\Qb$, as the quotient of $X(\Gamma)_{\Qb}$ by $H^1(\Gamma,\Z/2)=\Hom(\Gamma,\{\pm 1\})$, where $\sigma\in \Hom(\Gamma, \{\pm 1\})$ acts on $\chi\in  X(\Gamma) (\Qb)$ by $\sigma\chi(\gamma):=\sigma(\gamma)\chi(\gamma)$. This is because every $\bar{\rho}\in R_{\PSL_2}(\Gamma)(\Qb)$ determines a cohomology class $w_2(\bar{\rho})\in H^2(\Gamma,\Z/2)$, which is zero if and only if $\bar{\rho}$ lifts to $\SL_2$. The condition $H^2(\Gamma,\Z/2)=0$ holds, for example, when $\Gamma=\pi_1(S^3-K)$ for a knot $K$ in $S^3$, cf. \cite[$\S$3]{BoyerZhang98}. In this case that $\Gamma=\pi_1(S^3-K)$, we further have $H^1(\Gamma,\Z/2)=\langle \iota\rangle \cong\Z/2$, and in fact, $A(\Gamma)\subset A(\Gamma)_{\Qb}$ is stable under $\iota$, hence we define the $\PSL_2$-character variety over $\Q$ as 
\[ Y(\Gamma):=\Spec(B(\Gamma)_{\Q}^{\langle \iota\rangle}) \] 
for $B(\Gamma)_{\Q}= A(\Gamma)_{\Q}^{\mathrm{PGL}_2}$, cf. \cite[2.1.2]{MPvL11}. For example, for the meridian $m\in \pi_1(S^3-K)$, $\iota(\tau_{m})=-\tau_{m}$, so that $\tau_{m^2}=\tau_m^2-2\in B(\Gamma)_{\Q}^{\langle \iota\rangle}$.

There is another candidate for the $\SL_2$-character scheme of $\Gamma$ which is defined as the $\Z$-affine scheme $\mathfrak{X}(\Gamma):=\Spec(T(\Gamma))$, where the affine algebra is the \textit{trace ring}:
\begin{equation}
T(\Gamma):=\Z\langle\, \tau_{\gamma}:=\mathrm{tr}(\rho^{\univ}(\gamma))\ |\ \gamma\in \Gamma\, \rangle \ \subset A(\Gamma)
\end{equation}
(subring of $A(\Gamma)$ generated by $\tau_{\gamma}$ for all $\gamma\in \Gamma$). 

Let $S(\Gamma)$ be the quotient algebra (\emph{Skein algebra}):
\[ \Q[Y_{\gamma},\gamma\in\Gamma]/\langle Y_e-2,\ Y_{\alpha\beta} +Y_{\alpha^{-1}\beta}-Y_{\alpha}Y_{\beta}\ \forall \alpha,\beta\in\Gamma \rangle .\]
By the skein relation for traces in $\SL_2$: for any representation $\rho:\Gamma\ra \SL_2$, if $\tau:=\tr(\rho)$,
\[ \tau(\alpha\beta) +\tau(\alpha^{-1}\beta) -\tau(\alpha)\tau(\beta)=0 ,\]
there exists a natural ring homomorphism
\begin{align} \label{eq:S(Gamma)->A(Gamma)}
S(\Gamma) \ & \stackrel{\theta}{\twoheadrightarrow}\ T(\Gamma)_{\Q} \stackrel{i}{\hookrightarrow} A(\Gamma)_{\Q}^{\PGL_2} \\
Y_{\gamma}\ & \longmapsto\ \tau_{\gamma}:=\mathrm{tr}(\rho^{\univ}(\gamma)) . \nonumber
\end{align}

\begin{thm} \label{thm:PS00_Thm.7.1} \cite[Thm.7.1]{PS00}
The maps $\theta$, $i$ are isomorphisms.
\end{thm}
See also \cite[Thm.2.5]{Marche16}.

\begin{rem}
There are at least two advantages of $\mathfrak{X}(\Gamma)$ over $X(\Gamma)$: first, $\mathfrak{X}(\Gamma)$ is a scheme defined over $\Spec(\Z)$ having a nice moduli interpretation, thus provides a natural integral model of $X(\Gamma)=\mathfrak{X}(\Gamma)_{\Q}$, whereas GIT quotient seems to work best only for field bases and flat base changes, so that it is not clear at all whether $\Spec(A(\Gamma)^{\PGL_2})$ is a ``good'' integral model of $X(\Gamma)$. 
The second advantage is the property (\cite{Procesi98}, \cite{Saito96}) that over the open subset of $\mathfrak{X}(\Gamma)$ consisting of \textit{absolutely irreducible} representations/characters, the morphism $R(\Gamma) \ra \mathfrak{X}(\Gamma)$ (of $\Z$-schemes) becomes a torsor under $\mathrm{PGL}_2$ for the \'etale topology, where absolutely irreducible characters in $\mathfrak{X}(\Gamma)$ are defined in terms of certain discriminant, cf. \cite[$\S$4]{Nakamoto00}. 
\end{rem}

Suppose that $\Gamma$ is finitely generated, say by $\gamma_1,\cdots,\gamma_r$. Let $\{\gamma_1,\cdots,\gamma_N\}$ be the set of all elements of $\Gamma$ of the form $\gamma_{i_1}\cdots\gamma_{i_r}$, where $i_1,\cdots,i_r$ are positive integers in $\{1,\cdots,m\}$ with $i_1<\cdots<i_r$ (e.g. if $m=2$, this set is $\{\gamma_1, \gamma_2,\gamma_3=\gamma_1\gamma_2\}$).
Then, Culler and Shalen (\cite[1.4.1,1.4.5]{CullerShalen83}) show that as a subalgebra of $A(\Gamma)$, $T(\Gamma)$ is generated by $\tau_{\gamma_1},\cdots, \tau_{\gamma_N}$, and that the image of the map
\begin{align*} 
R(\Gamma)(\C)\ \ra & \quad \C^N \\
\rho\quad \mapsto & \ (\tau_{\gamma_1}(\rho),\cdots,\tau_{\gamma_N}(\rho))=(\tr\rho(\gamma_1),\cdots,\tr\rho(\gamma_N)).
\end{align*}
is a Zariski-closed subset of $\C^N$. 
We give this image the structure of a reduced affine variety and denote it by $X(\Gamma;\C)$: this is called \emph{Culler-Shalen character variety}. 
If $z_1,\cdots,z_N$ are the coordinate functions of $\A^N$, we have an algebra homomorphism
$t^{\ast}:\C[z_1,\cdots,z_N]=\mcO(\A^N) \ra T(\Gamma)=\mcO(\mfX(\Gamma))$ defined by $t^{\ast}(z_i)=\tau_{\gamma_i}$, thus obtain a morphism of affine varieties $t:\mfX(\Gamma) \ra \A^N$. 
Clearly, the set $X(\Gamma;\C)(\C)$ is contained in the image of the map $\mfX(\Gamma)(\C) \ra \C^N$:  the map $\phi:T(\Gamma)\ra \C:\tau_{\gamma}\mapsto \tr(\rho(\gamma))$ attached to a representation $\rho:\Gamma \ra\SL_2(\C)$ is a ring homomorphism.
In fact, it follows from Theorem \ref{thm:fund_property_GIT} and Theorem \ref{thm:PS00_Thm.7.1} that
$X(\Gamma;\C)(\C)$ is the entire image, namely that any ring homomorphism $\phi:T(\Gamma)\ra \C$ is the character of a representation $\rho:\Gamma \ra\SL_2(\C)$: for (absolutely) irreducible characters, this is a part of the statement cited above about $R(\Gamma) \ra \mathfrak{X}(\Gamma)$ being a torsor over the irreducible locus (a torso is a surjective map, among others). Therefore, we have a finite surjective morphism of reduced affine varieties: 
 \[ t:  X(\Gamma)_{\C}^{\red} \isom \mfX(\Gamma)^{\red}_{\C} \ra X(\Gamma;\C), \]
 which are bijective on $\C$ (or any geometric) points.
Such a morphism does not need to be an isomorphism in general (but, it is so, for example, if the target is a normal variety).

%%%%%%%%%%%%%%%%%%%%
\subsection{Thurston deformation space of cusped hyperbolic three-manifolds}
For the discussion in this subsection, our main references are \cite{NeumannZagier85} and Ch. E, Sec. 5 (E.5) of \cite{BenedettiPetronio92}.
Let $M$ be the interior of a compact three-manifold $\overline{M}$ whose boundary $\partial \overline{M}$ is a union of tori (we will just write $\partial M$ for $\partial \overline{M}$).
Suppose that there exists a \textit{topological ideal triangulation} $\mcT$ of $M$, namely a collection of standard tetrahedra with vertices removed which glue together by a set of face-pairing maps with the resulting space $M$; removing open stars at vertices from the glued union of tetrahedra gives $\overline{M}$, cf. \cite[E.5-i]{BenedettiPetronio92}. 

Thurston's idea for finding a hyperbolic structure on $M$ endowed with such a (topological) ideal triangulation is, starting with some hyperbolic structures on the constituent ideal tetrahedra, to find conditions ensuring that these local hyperbolic structures glue consistently giving a global hyperbolic metric on $M$. This gluing condition is that at each edge $e$ in $\mcT$, the product of the shape parameters of the tetrahedra joined to $e$ is $1$ and the sum of the dihedral angles at $e$ of these tetrahedra is equal to $2\pi$. When this condition is satisfied at all edges, we obtain a hyperbolic structure on $M$ which is in general not complete. 
We recall that when an ordering of the vertices is fixed, to each edge of the tetrahedron, 
one can assign one of the three shape parameters $z$, $z'=\frac{1}{1-z}$, $z''=1-\frac{1}{z}$, in such a way that the parameters associated with two opposite edges are equal and all $z$, $z'$, $z''$ are attached to some different (pairs of opposite) edges (\cite[$\S$2]{Neumann92}).
Since the dihedral angle at an edge whose associated parameter is $z$ (resp. $z'$, $z''$) is $\log z$ (resp. $\log z'$, $\log z''$), from the condition on dihedral angle sum, by exponentiating we obtain an algebraic relation in the shape parameters $\{z_i\}_{i=1,\cdots,N}$ of the form: 
\begin{equation} \label{eq:gluing_eq}
\prod_{i=1}^Nz_i^{r_{ij}'}(1-z_i)^{r_{ij}''}=\pm 1.
\end{equation}
The set of these equations for all edges define a closed complex subvariety 
\[ Y(M;\mcT)\subset (\mathbb{P}^1(\C)\backslash\{0,1,\infty\})^N,\] 
where $N$ is the number of edges in $\mcT$; $N$ is also equal to the number of the tetrahera in $\mcT$ since the Euler characteristic of $\overline{M}$ is zero \cite[Lem.E.5.6]{BenedettiPetronio92}. This is called the \textit{deformation variety} for the chosen ideal triangulation, and the algebraic relations (\ref{eq:gluing_eq}) are called \textit{gluing equations}. 

\begin{rem} \label{rem:degree-one-ideal-triangulation}
(1) The gluing equation does not capture faithfully the dihedral-angle-sum condition (it only tells that the dihedral angle sum is a multiple of $2\pi i$). But, it is easy to see (\cite[Lem.E.6.1]{BenedettiPetronio92}) that a \emph{positive} solution (i.e. $\mathrm{Im}(z_i)>0$) does fulfill the dihedral angle sum condition, so indeed defines a hyperbolic structure. Unfortunately, given a topological ideal triangulation, there might not exist such positive solution giving the hyperbolic structure on $M$. 

(2) Nevertheless, Epstein and Penner \cite{EpsteinPenner88} showed that every complete hyperbolic three-manifold with cusps of finite volume admits an ideal topological triangulation that supports a solution to the gluing equations and completeness equations which ``give the hyperbolic structure''. But, the shape parameters have only non-negative imaginary parts and can be real numbers (i.e.
 $(M;\mcT)$ can have \emph{degenerate} ideal tetrahedra). 

(3) Epstein-Penner decomposition above can be used to give an \emph{degree-one ideal triangulation} in the sense of \cite[Def.2.1]{NeumannYang99}, which suffices for our purpose in this work.%% 
\footnote{Indeed, the only input for our construction of the mixed Tate motive attached to any complete hyperbolic three-manifold of finite volume in Theorem \ref{thm:Main.Thm1:CS-MTM} was the Bloch invariant $\beta(X)\in \mcB(k(X))_{\Q}$ which was constructed by Neuamnn and Yang (\textit{loc. cit.}) using any degree-one ideal triangulation.}
\end{rem}
 
It is known \cite{Thurston97} that the dimension of the deformation space of a complete hyperbolic structure equals the number of cusps, and there exists a unique component in $Y(M;\mcT)$ containing every solution giving the \emph{complete} hyperbolic structure: it is called the \textit{canonical component}.

For example, in the case of the Figure-eight knot complement $M=S^3\backslash K$, there exists an ideal triangulation with two tetrahedra, and such that
\[ Y(M;\mcT)=V(x(1-x)y(1-y)-1)\subset (\mathbb{P}^1(\C)\backslash\{0,1,\infty\})^2. \]
The smooth projective completion $\overline{Y(M;\mcT)}$ of this curve is an elliptic curve over $\Q$ with four ideal points (=points in $\overline{Y(M;\mcT)}\backslash Y(M;\mcT)$ whose images under the map $\overline{Y(M;\mcT)}\ra(\mathbb{P}^1(\C)\backslash\{0,1,\infty\})^2$ have all of its coordinates in $\{0,1,\infty\}$). Its (minimal) Weierstrass model is: $y^2+xy+y=x^3+x^2$, \cite{Fujii05}. On the other hand, the $\mathrm{SL}_2$-character variety $X(M)$ of $M$ has Weierstrass model $y^2=x^3-2x+1$, \cite{Harada11}.%%
\footnote{These two curves are both $\Q$-elliptic curves and have very similar arithmetic properties: $\overline{Y(M;\mcT)}$ has conductor $15$, rank $0$, torsion subgroup $\Z/4\Z$, while the character variety has conductor $40$, rank $0$, torsion subgroup $\Z/4\Z$.}
%%

%%%%%%%%%%%%%%%%%%%%
%\subsection{}
Suppose the chosen triangulation $\mcT$ supports a hyperbolic structure $\{z_i\}_{i=1}^N$ which is not necessarily complete. For simplicity, we assume that $M$ has a single cusp. 
We fix an ordered basis $\{\mu,\nu\}$ of $\pi_1(\partial M)$ such that the induced orientation on $\partial M$ (i.e. on the universal cover of $\partial M$ which is identified with $\R\mu\oplus\R\nu$) agrees with the orientation on $\partial M$ as the boundary of $M$
(in the case of a knot complement, we will choose $\mu$ and $\nu$ respectively to be the classes of the meridian and the longitude%%
\footnote{In this work, to denote the longitude, we use $\nu$ instead of the conventional $\lambda$, which we reserve for other use.} 
endowed with orientations, well-defined up to simultaneous reversing).
Any small horosphere at each cusp cuts the boundary in a torus and this torus has an affine structure provided by its triangulation into the triangles $\Delta(0,1,z)\subset \C$.  This defines a holonomy representation $\pi_1(\partial M)\ra \mathrm{Aff}(\C)$.  Assume that the hyperbolic structure is complete, namely that the affine structure is Euclidean.
Then, the \emph{derivatives} (i.e. $a$ of $x\mapsto ax+b$) of the holonomy along $\mu$ and $\nu$ are expressed as
\begin{align}
\mfm(z)=\prod_{i=0}^N \left(\frac{z_i}{z_i^0}\right)^{m_i'} \left(\frac{1-z_i}{1-z_i^0} \right)^{m_i''},\quad 	\mfl(z)=\prod_{i=0}^N \left(\frac{z_i}{z_i^0} \right)^{l_i'} \left(\frac{1-z_i}{1-z_i^0} \right)^{l_i''},
\end{align}
where $z^0$ is the point on $Y(M;\mcT)$ corresponding to the complete hyperbolic structure (\cite[the formula right after (27)]{NeumannZagier85}).
(The constants $(m_i',m_i'')$, $(l_i',l_i'')\in \Z^2$ are determined by the coefficients of the cusp relations  $C(\msZ^0)^t=\pi \sqrt{-1} D_2$ (\ref{eq:CCequation2}).)
These can be regarded as regular functions on the affine variety $Y(M;\mcT)$: 
\begin{equation} \label{eq:m,l}
\mfm,\mfl\ :\  Y(M;\mcT)\ra \C.
\end{equation}

Take an open neighborhood $U$ of $z^0$ in $Y(M;\mcT)$ and branches of 
\begin{equation} \label{eq:mfu,mfv}
\mfu=\log \mfm,\quad \mfv=\log \mfl
\end{equation}
on $U$ with $\mfu(z^0)=\mfv(z^0)=0$.
If $\rho:\pi_1(M)\ra \PSL_2(\C)$ is a honomomy representation corresponding to a point $z=\{z_i\}$ of $Y$, then there exists a conjugate of $\rho_z$ such that
\begin{equation} \label{eq:u,v_on_deformation_curve}
\rho_z(\mu)= \pm \left(\begin{array}{cc} e^{\frac{\mfu(z)}{2}} & \ast \\
0 & e^{-\frac{\mfu(z)}{2}} \end{array} \right),\quad \rho_z(\nu)= \pm \left(\begin{array}{cc} e^{\frac{\mfv(z)}{2}} & \ast \\ 0 & e^{-\frac{\mfv(z)}{2}} \end{array} \right)
\end{equation}
(the associated affine transformations on any boundary torus cut by a horosphere are then $w\mapsto e^{\mfu(z)}w+ e^{-\frac{\mfu(z)}{2}}\ast$, $w\mapsto e^{\mfv(z)}w+ e^{-\frac{\mfv(z)}{2}}\ast$).

By a basic theorem of W. Thurston (cf. \cite[$\S$4]{NeumannZagier85}), we may assume (after shrinking $U$ if necessary) that $\mfu:U\ra \C$ is a local holomorphic coordinate around $z^0$ (in this situation $U$ is often called a \textit{Dehn surgery space}).

%%%%%%%%%%%%%%%%%%%%%%%%%%%%%%%%%%%%%%%%
%%%%%%%%%%%%%%%%%%%%%%%%%%%%%%%%%%%%%%%%
\subsection{Kirk-Klassen construction of Chern-Simons line bundle and CS-section}

Let $D(\partial M)$ be the split algebraic torus over $\Q$ with character group $\pi_1(\partial M)$: it is the affine $\Q$-scheme whose affine algebra is the group algebra $\Q[\pi_1(\partial M)]\simeq \Q[\Z^{\oplus 2h}]$, so that it is isomorphic to $\Gm^{2h}$ (an isomorphism being determined by a choice of a basis of $\pi_1(\partial M)$ such as ours $\{\mu_i,\nu_i\}_i$).
In particular, for any $\Q$-algebra $k$, the group $D(\partial M)(k)$ of $k$-rational points consists of homomorphisms $\lambda:\pi_1(\partial M)\ra k^{\times}$ (group of units in $k$). We regard $D(\partial M)$ as a subvariety of $R(\partial M)$ via 
$\lambda\mapsto \left(\begin{array}{cc} \lambda &  \\  & \lambda^{-1} \end{array}\right)$. 
If $\iota$ denotes the involution on $D(\partial M)$ defined by $\lambda\mapsto \lambda^{-1}$, the canonical map $D(\partial M) \ra X(\partial M)$ induces a morphism
\begin{equation} \label{eq:q:D->X}
D(\partial M)\sslash \langle\iota\rangle \ra X(\partial M)=R(\partial M)\sslash \SL_2: [\lambda]\mapsto \lambda+\lambda^{-1}.
\end{equation}
This is an isomorphism of algebraic varieties \cite[1.14]{Benard20}, \cite[Thm.2.1]{Sikora14}. 

For a separated scheme $X$ of finite type over a field $k\subset\C$, let $X_{\C}^{an}$ (or even $X^{an}$) denote the complex analytic variety associated with the reduced closed subscheme $X^{red}$ of $X$.

For the next discussion, for simplicity, let us assume $h=1$: the argument easily generalizes to arbitrary $h$.
Let $\Delta:=\Z^2\rtimes\Z/2\Z$ with $\iota\neq 1\in\Z/2\Z$ acting on $\Z^2$ by $(a,b)\mapsto (-a,-b)$; let $\{\mu,\nu\}$ be the standard basis of $\Z^2\subset \Delta$ (this choice of the notation $\{\iota,\mu,\nu\}$ will be compatible with the previous choice). We make $\Delta$ act on $(\C\oplus\C)\times \C^{\times}$ by
\begin{align} \label{eq:CS_linebundle-action}
\nu :& (x,y;z) \mapsto (x+\tpi,y;ze^{y}) \\
\mu :& (x,y;z) \mapsto (x,y+\tpi;ze^{-x}) \nonumber \\
\iota :& (x,y;z) \mapsto (-x,-y;z) \nonumber
\end{align}
(be wary of the roles of $\mu$ and $\nu$).
Taking quotient of the trivial fibration $\C^2\times \C^{\times}\ra \C^2$ by this analytic group action, we obtain a principal $\C^{\times}$-bundle over the analytic variety $X(\partial M)^{an}$:
\[ E(\partial M) \ra X(\partial M)^{an}. \] 
In more detail, we first take the quotient of the trivial fibration $\C^2\times \C^{\times}\ra \C^2$ by the subgroup $\langle \mu,\nu\rangle$ of $\Delta$ (which acts freely on both spaces), obtaining a principal $\C^{\times}$-bundle $Q(\partial M)$ over the analytic variety $D(\partial M)^{an}$: 
\[ Q(\partial M) \ra D(\partial M)^{an}. \] 
Then from the fact that $\iota$ induces an action of $\C^{\times}$-bundle $Q(\partial M)$ covering an action on $D(\partial M)_{\C}$ which acts trivially on the fibers over its fixed points in $D(\partial M)_{\C}$ \cite[p.525]{KirkKlassen93}, it follows that this principal $\C^{\times}$-bundle $Q(\partial M)$ descends to $X(\partial M)^{an}$, namely there exists a principal $\C^{\times}$-bundle $E(\partial M)$ over $X(\partial M)^{an}$ whose pull-back to $D(\partial M)^{an}$ is isomorphic to $Q(\partial M)$. Indeed, the condition implies that for a suitable local trivialization of $Q(\partial M)$ over $D(\partial M)^{an}$, the transition functions are invariant under $\iota$, thus since the natural map $\mcO(X(\partial M)^{an})\ra \mcO(D(\partial M)^{an})^{\PGL_2}$ is an isomorphism (\cite[Thm.8]{Neeman88}), they belong to analytic functions on $X(\partial M)^{an}$, defining a $\C^{\times}$-bundle on the latter (for more details, see \cite[Theoreme2.3]{DrezetNarasimhan89} whose arguments in the algebraic setup carry over to our analytic setup, given the fact just cited).
This analytic principal $\C^{\times}$-bundle $E(\partial M)$ over $X(\partial M)^{an}$ is unique up to isomorphism; we call it \textit{Chern-Simons $\C^{\times}$-(or line) bundle}. Although the description makes use of the choice of an oriented basis of $\pi_1(\partial M)$ compatible with the orientation of $\partial M$, the bundle $E(\partial M)$ itself depends only on the orientation of $\partial M$.

%%%%%%%%%%%%%%%%%%%%
\begin{thm} \cite[2.5, 3.4]{KirkKlassen93} \label{thm:KK93_2.5}
Suppose that $A$, $B$ are gauge-equivalent connection $1$-forms on $M\times \SL_2(\C)$ which are in normal form near boundary and such that
\[ A=A^{(1)}(\alpha,\beta),\quad B=\epsilon A^{(1)}(\alpha+p,\beta+q), \]
(\ref{eq:normal_form1}) for some $\alpha,\beta\in\C$, $p,q\in\Z$, $\epsilon\in\{\pm1\}$ in Case 1, and
\[ A=A^{(2)}(-2\alpha,-2\beta;a,b),\quad B=\epsilon A^{(2)}(-2(\alpha+p),-2(\beta+q);a',b') \]
(\ref{eq:normal_form2}) for some $\alpha,\beta\in\frac{\Z}{2}$, $p,q\in\Z$, and $a,b,a',b'\in \C$, $\epsilon\in\{\pm1\}$ in Case 2. 

Then, we have $cs(B)-cs(A) =p\beta - q\alpha  \mod\Z$.
\end{thm}

\begin{proof}
When the boundary holonomy is diagonalizable, the proof of Theorem \ref{thm:KK93_2.5} in the $\SU(2)$ case works without change, too. When the boundary holonomy is parabolic, this is proved in \textit{ibid.} Lemma 3.4, 3.5.
\end{proof}

%%%%%%%%%%%%%%%%%%%%
\begin{thm} \label{thm:KK93_3.2}

Define a map $\bar{C}_{\KK}:R(M)(\C) \ra E(\partial M)$ by
\begin{equation} \label{eq:CS_section}
C_{\KK}(\rho)=[\, \tpi\, \alpha_1,\, \tpi\, \beta_1,\, \cdots,\, \tpi\, \alpha_h,\, \tpi\, \beta_h\, ;\, e^{\tpi\, cs(A)}\, ]. 
\end{equation}
Here, given $\rho$, for each $i=1,\cdots,h$, we choose any $(\alpha,\beta)$ such that a conjugate of $\rho|_{\pi_1(\partial M)}$ is of the form (\ref{eq:bdry_holonomy1}) for $(\alpha,\beta)$ if $\rho|_{\pi_1(\partial M)}$ is in Case 1, or any $(u,v,a,b)$ such that a conjugate of $\rho|_{\pi_1(\partial M)}$ is of the form (\ref{eq:bdry_holonomy2}) for $(u,v,a,b)$ if $\rho|_{\pi_1(\partial M)}$ is in Case 2, and set $(\alpha_i,\beta_i)\in\C^2$ to be $(\alpha,\beta)$ in Case 1 or $(-\frac{u}{2},-\frac{v}{2})$ in Case 2. And, $A$ is any flat connection with holonomy being (a conjugate of) $\rho$ and such that for each $i$, $A$ is in normal form for the just chosen constants $(\alpha,\beta)$, $(u,v,a,b)$ (Definition \ref{defn:normal_form}).

Then, $\bar{C}_{\KK}$ induces a well-defined map $C_{\KK}:X(M)(\C) \ra E(\partial M)$, giving a holomorphic section over $X(M)_{\C}$ of $E(\partial M)\ra X(\bdM)_{\C}$:
\[ \xymatrix@R=15pt{ & E(\partial M) \ar[d] \\ 
X(M)_{\C} \ar[ru]^{C_{\KK}} \ar[r] &  X(\bdM)_{\C}.} \]

(2) Let $\rho_t:\pi_1(M)\ra \SL_2(\C),\ t\in[0,1]$ be a path of representations avoiding non-central boundary-parabolic ones. 
Suppose that $\rho_t|_{\pi_1(\partial M)}$ is of the shape (\ref{eq:bdry_holonomy1}) for some smooth $(\underline{\alpha}(t),\underline{\beta}(t)):[0,1]\ra\C^{2h}$, where $\underline{\alpha}(t)=(\alpha_1(t),\cdots,\alpha_h(t))$, $\underline{\beta}(t)=(\beta_1(t),\cdots,\beta_h(t))$. If $C_{\KK}([\rho_t])=[\underline{\alpha}(t),\underline{\beta}(t);z(t)]$, we have
\[ z(1)z(0)^{-1}=\exp(\frac{1}{\tpi} (\int_0^1\underline{\alpha}(t)\frac{\underline{\beta}(t)}{dt}-\underline{\beta}(t)\frac{\underline{\alpha}(t)}{dt})). \]
\end{thm}

\begin{proof} 
This is \cite[Thm.3.2]{KirkKlassen93}. The holomorphicity in (1) is a consequence of (2).
\end{proof}

%%%%%%%%%%%%%%%%%%%%%%%%%%%%%%%%%%%%%%%%
%%%%%%%%%%%%%%%%%%%%%%%%%%%%%%%%%%%%%%%%
\subsection{Augmented character variety}

We introduce the augmented character variety of three manifolds with toral boundary. 

\begin{defn}
The \textit{augmented character variety} $\tilde{X}(M)$ is the fiber product of $X(M)\ra X(\partial M)$ and $D(\partial M)\ra R(\partial M)\ra X(\partial M)$:
\[ \tilde{X}(M):= X(M)\times_{X(\partial M)}D(\partial M). \]
\end{defn}

\begin{rem}
(1) In view of the isomorphism (\ref{eq:q:D->X}), $\tilde{X}(M)$ is a double covering of $X(M)$. 
For an algebraically closed field $k$ of characteristic zero, a $k$-rational point of $\tilde{X}(M)$ is a $k$-point $(\chi,\lambda)$ in $X(M)\times D(\partial M)$ such that the restriction of $\chi$ to $\pi_1(\partial M)$ is the character of the diagonal representation $\delta_{\lambda}:=\left(\begin{array}{cc} \lambda &  \\  & \lambda^{-1} \end{array}\right)$. Be warned, however, that this condition implies, but not equivalent to, that $\lambda(\gamma)+\lambda^{-1}(\gamma)=\chi(\gamma)$ for all $\gamma\in \pi_1(\partial M)$: for any given $\chi$, there are in general $2^{2h}$ many $\lambda\in \Hom(\pi_1(\partial M),k^{\times})$'s satisfying this latter condition since $\pi_1(\partial M)\simeq \Z^{2h}$, while in the definition of $\tilde{X}(M)$ we are choosing one among two $\lambda$'s satisfying the condition that there exists a representation $\rho\in R(M)(k)$ with $\chi(\rho)|_{\pi_1(\partial M)}=\lambda+\lambda^{-1}$ (equiv. that some conjugate of the diagonal representation $\delta_{\lambda}$ lifts to a representation of $\pi_1(M)$). 

(2) A major weakness of the character variety $X(M)$, compared with the Thurston deformation curves, is the absence of a single-valued eigenvalue (global regular) function $\lambda$ with the property as above: on the deformation curve, we have $\mfm$, $\mfl$, the derivatives of the holonomy along $\mu$ and $\nu$, which are of great utility in the study of cusped hyperbolic $3$-manifolds, as in the example of hyperbolic Dehn surgery.
But Thurston deformation curves also has the weakness that their definition depends on the choice of an ideal triangulation supporting the complete hyperbolic structure, while $X(M)$ is canonically defined by $M$ (equivalently, $\pi_1(M)\subset \PSL_2(\C)$) only. The augmented character variety has both features. It is a triangulation-choice-free analogue of deformation curve.

(3) The augmented character variety is also a GIT quotient $\tilde{X}(M)=\tilde{R}(M)\sslash \PGL_2$, where
$\tilde{R}(M)$ is the \emph{augmented representation variety}
\[ \tilde{R}(M):=R(M)\times_{R(\partial M)}D(\partial M) \] 
and $\PGL_2$ acts on $\tilde{R}(M)$ via $R(M)$.

(4) The augmented character variety was introduced earlier. Recently, Benard \cite{Benard20} showed that the adjoint Reidemeister torsion can be interpretated as a rational volume form on the augmented character variety, and related its vanishing order at a point is related to the singularity of the image of the point in the character variety.
\end{rem}

Recall that $N_1,\cdots,N_h$ are boundary tori of $\partial M$ and that the affine algebra $\Q[D(\partial M)]$ is isomorphic to the group algebra $\Q[\pi_1(\partial M)]$, so every element of $\pi_1(\partial M)$ can be regarded as a regular function on $D(\partial M)$.

\begin{defn} \label{defn:m(z),l(z)}
For each $i=1,\cdots,h$, let $f_{\mu_i}(z)$ and $f_{\nu_i}(z)$ be the regular functions on $D(\partial M)$ corresponding to $\mu_i,\nu_i\in \pi_1(N_i)$, and $m_i(z)$ and $l_i(z)$ the regular functions on $\tilde{X}(M)$ obtained by pull-back from $f_{\mu_i}(z)$ and $f_{\nu_i}(z)$, respectively. Put $\underline{l}:=(l_1,\cdots,l_h)\in \mcO( \tilde{X}(M))^{\oplus h}$, $\underline{m}:=(m_1,\cdots,m_h)\in \mcO( \tilde{X}(M))^{\oplus h}$.
\end{defn}

When we fix a point $\tilde{z}_0$ of $\tilde{X}(M)$ whose image in $X(M)$ is the conjugacy class of a lift to $\SL_2(\C)$ of the geometric representation $\rho_0$ (corresponding to the complete hyperbolic structure), one can choose a small neighborhood $\tilde{V}$ of $\tilde{z}^0$ and branches of 
\[ u_i=\log m_i\text{ and }v_i=\log l_i,\] defined on $\tilde{V}$, with $u_i(\tilde{z}^0)=0$, $v_i(\tilde{z}^0)=0$, and analytically continue them, obtaining multi-valued functions to the irreducible component of $\tilde{X}(M)$ containing $\tilde{z}^0$. On a small enough neighborhood $U$ of $z^0$ in the deformation curve $Y(M,\mcT)$ one can find an analytic function $f:U\ra \tilde{X}(M)$ with $f(z^0)=\tilde{z}^0$. It follows from (\ref{eq:u,v_on_deformation_curve}) that there exists the relation
\begin{equation} \label{eq:u,v_mfu,mfv}
\mfu=2u\circ f,\quad \mfv=2v\circ f.
\end{equation}

%%%%%%%%%%%%%%%%%%%%%%%%%%%%%%%%%%%%%%%%
%%%%%%%%%%%%%%%%%%%%%%%%%%%%%%%%%%%%%%%%
\subsection{Heisenberg line bundle and Morishita-Terashima Chern-Simons section} \label{subesc:Heisenberg-L.B.&MT-CS-section}

For $n\in\N$, let $H_{\Z}^{(n)}\subset H_{\C}^{(n)}$ be the integral and complex Heisenberg groups ``of order $n$'': they consist of $(2n+1)\times (2n+1)$ matrices such that
\[ H_{\Z}^{(n)}=\left\{ \left(\begin{array}{ccc}  \Onen & a & c\\  & \In & b^t \\  & & \Onen^t \\ \end{array}\right) \,\middle\vert\, \begin{array}{c} a,b\in\Z(1)^{n}, \\ c\in\Z(2)\end{array} \right\} \,\subset \, 
H_{\C}^{(n)}=\left\{ \left(\begin{array}{ccc}  \Onen & a & c\\  & \In & b^t \\  & & \Onen^t \\ \end{array}\right) \,\middle\vert\, \begin{array}{c} a,b\in\C^{n}, \\ c\in\C\end{array} \right\},  \]
where $\Onen=(1,\cdots,1)\in \mathrm{M}_{1\times n}$. 
We write $(a,b;c)'$ for the matrix $\left(\begin{array}{ccc}  \Onen & a & c\\  & \In & b^t \\  & & \Onen^t \\ \end{array}\right)$.
The quotient manifold $H_{\Z}^{(n)}\backslash H_{\C}^{(n)}$ is a principal $\C^{\times}$-bundle over $(\C^{\times})^n\times (\C^{\times})^n$ via the map $H_{\Z}^{(n)}(a,b;c)' \mapsto (\exp(a),\exp(b))$ and carries a flat connection $1$-form $\theta=dc-a\cdot db^t$. To see the principal $\C^{\times}$-bundle structure more explicitly, let us define put a group structure on the complex manifold $\C^{2n}\times \C^{\times}$ by
\[ (v;z)\ast (v';z')=(v+v';zz'e[v_1\cdot (v_2')^t]), \]
where $v=(v_1,v_2),v'=(v_1',v_2')\, (v_i,v_i'\in\C^{n})$, $z,z'\in \C^{\times}$, and $e[u]:=\exp(\frac{1}{2\pi i}u)$. This complex Lie group $H^{(n)}$ is an extension of Lie groups:  
\[ 1 \lra \C^{\times}  \stackrel{j}{\lra} H^{(n)} \stackrel{p_H} \lra \C^{2n} \lra 0 \]
with $j(z)=(0;z)$ and $p_H(v;z)=v$.
Then, the map $H_{\C}^{(n)} \ra H^{(n)}$ defined by $\left(\begin{array}{ccc} \Onen & a & c\\  & \In & b^t \\  & & \Onen^t \\ \end{array}\right) \mapsto (a,b;e[c])$
induces an isomorphism of discrete subgroups $H_{\Z}^{(n)}\isom \Z(1)^{\oplus 2n}\times\{1\}$, thus
an analytic isomorphism of principal $\C^{\times}$-bundle  
\[ H_{\Z}^{(n)}\backslash H_{\C}^{(n)} \ \isom\ \Z(1)^{\oplus 2n}\backslash H^{(n)} \]
over $\Z(1)^{\oplus 2n}\backslash \C^{2n} \cong (\C^{\times})^{2n}$, cf. \cite[$\S$4]{Ramakrishnan89}.

Let $X$ be a smooth irreducible subvariety of a Thurston deformation variety $Y(M;\mcT)$ containing the geometric point, and let $\langle \underline{l},\underline{m}^2\rangle$ denote the line bundle over $X$ obtained as the pull-back of $H_{\Z}^{(n)}\backslash H_{\C}^{(n)}$ by the holomorphic map $(\underline{l},\underline{m}^2):X \ra (\C^{\times})^n\times (\C^{\times})^n$.
When $h=1$ and $X$ is the canonical (curve) component of , Morishita and Terashima \cite[Thm.4.1]{MorishitaTerashima09} constructed a section of this bundle $\langle \underline{l},\underline{m}^2\rangle$ over $X$, which thus depends on the choice of an ideal triangulation defining the deformation curve. Here, we emulate their construction to (the canonical component of) the augmented character variety. Put $\zeta:(\C^{\times})^{2h} \ra (\C^{\times})^{2h}:(l,m)\mapsto (l,m^2)$.

\begin{thm} \label{thm:MT09_4.1}
Define a map $\bar{C}_{\MoTe}:\tilde{R}(M) \ra H_{\Z}^{(h)}\backslash H_{\C}^{(h)}$ by
\begin{equation}  \label{eq:MT09_4.1}
z\mapsto H_{\Z}^{(h)}\left(\begin{array}{ccc} \Onen & \underline{v}(z) & (\tpi)^2cs([A_z])+ \underline{u}(z)\cdot \underline{v}(z)^t \\  & \In &  2\underline{u}(z)^t \\  & & \Onen^t \\ \end{array}\right),
\end{equation}
where given $z$, for each $i=1,\cdots,h$, we take $(\alpha_i,\beta_i)\in\C^2$, plus $(a_i,b_i)\in \C^2$ in Case 2, as in Theorem \ref{thm:KK93_3.2} and put $\underline{v}(z):=(\alpha_1(z),\cdots,\alpha_h(z))$, $\underline{u}(z):=(\beta_1(z),\cdots,\beta_h(z))$, 
and $A_z$ is a flat connection with holonomy being (a conjugate of) $\rho_z$ and such that for each $i$, $A_z$ is in normal form for the constants $(\alpha_i(z),\beta_i(z))$, plus $(a_i(z),b_i(z))$ in Case 2, just chosen.

Then, $\bar{C}_{\MoTe}$ induces a well-defined map $C_{\MoTe}:\tilde{X}(M)_{\C} \ra H_{\Z}^{(h)} \backslash H_{\C}^{(h)}$ covering the composite of the canonical map $\tilde{X}(M)_{\C} \ra D(\partial M)_{\C}=(\C^{\times})^{2h}$ and $\zeta$. This is a holomorphic flat section over $\tilde{X}(M)_{\C}$ of $H_{\Z}^{(h)}\backslash H_{\C}^{(h)}\ra (\C^{\times})^h\times (\C^{\times})^h$:
\[ \xymatrix@R=15pt{ & H_{\Z}^{(h)} \backslash H_{\C}^{(h)} \ar[d] \\ 
\tilde{X}(M)_{\C} \ar[ru]^{C_{\MoTe}} \ar[r] & (\C^{\times})^h\times (\C^{\times})^h.} \]
\end{thm}

The same proof of Morishita-Terashima \cite[Thm.4.1]{MorishitaTerashima07}, which is based on \cite[2.5, 3.4]{KirkKlassen93} (Theorem \ref{thm:KK93_2.5} here), works without change for our definition. Next, we compare the Kirk-Klassen section $C_{\KK}$ (Theorem \ref{thm:KK93_3.2}) and this Morishita-Terashima section $C_{\MoTe}$.

We denote by $\langle \nu, \mu\rangle$ the free abelian group with generators $\{\nu_1,\mu_1,\cdots, \nu_h,\mu_h\}$ endowed with the action on the Heisenberg group $\C^{2h}\times\C^{\times}$ as defined in (\ref{eq:CS_linebundle-action}) and by $\Delta_h$ the group $\langle \nu, \mu\rangle \rtimes \langle \iota \rangle$ with $\iota$ being also as defined there. 

%%%%%%%%%%%%%%%%%%%%
\begin{prop} \label{prop:P=H_ZbsH_C}
(1) The natural map $\langle \nu, \mu\rangle \backslash \C^{2h}\times\C^{\times} \ra E(\partial M)=\Delta_h\backslash \C^{2h}\times\C^{\times}$ covering the quotient map $q:D(\partial M)=(\C^{\times})^{2h}\ra D(\partial M)\sslash \langle\iota\rangle = X(\partial M)$ induces an isomorphism
\[ \langle \nu, \mu\rangle \backslash \C^{2h}\times\C^{\times} \isom q^{\ast}E(\partial M) \] 
 of principal $\C^{\times}$-bundles.

(2) The (multi-valued) map $f: \C^{2h}\times\C^{\times} \ra H_{\C}^{(h)}$ defined by
\[ (x,y;z) \mapsto \left(\begin{array}{ccc} \Oneh & x & \pi i \log z+\frac{1}{2}x\cdot y^t \\  & \Ih &  y^t \\  & & \Oneh^t  \end{array}\right) \]
induces an isomorphism of analytic spaces
\begin{equation} \label{eq:isom_of_CS-line_bdls}
\langle \nu, \mu\rangle \backslash \C^{2h}\times \C^{\times} \isom \tilde{H}_{\Z}^{(h)}\backslash H_{\C}^{(h)}
\end{equation}
covering the identity map of $D(\partial M)=(\C^{\times})^{2h}$, where 
\[ \tilde{H}_{\Z}^{(n)}=\left\{ \left(\begin{array}{ccc}  \Onen & a & c\\  & \In & b^t \\  & & \Onen^t \\ \end{array}\right) \,\middle\vert\, \begin{array}{c} a,b\in\Z(1)^{n}, \\ c\in\frac{1}{2}\Z(2)\end{array} \right\} .\]

(3) Let $g\ :\ H_{\C}^{(h)} \ra H_{\C}^{(h)}$ be the group homomorphism $\left(\begin{array}{ccc} \Oneh & x & z \\  & \Ih &  y^t \\  & & \Oneh^t \\ \end{array}\right) \mapsto \left(\begin{array}{ccc} \Oneh & x & 2z \\  & \Ih &  2y^t \\  & & \Oneh^t \\ \end{array}\right)$. 
The induced map 
\begin{equation} \label{eq:pull-back_of_CS-line_bdls2} 
\tilde{H}_{\Z}^{(h)}\backslash H_{\C}^{(h)}  \twoheadrightarrow H_{\Z}^{(h)}\backslash H_{\C}^{(h)}
\end{equation}
which covers the map $\zeta:(\C^{\times})^{2h} \ra (\C^{\times})^{2h}:(l,m)\mapsto (l,m^2)$ is a pull-back diagram. 
\end{prop}

Note that the two maps in (2) and (3) are not compatible with the right $\C^{\times}$-actions, so are not maps of principal $\C^{\times}$-bundles, while their composite $g\circ f$ is so.

\begin{proof} 
For simplicity, we assume that $h=1$ and write $H_{\C}=H_{\C}^{(1)}$, $H_{\Z}=H_{\Z}^{(1)}$, $\tilde{H}_{\Z}=\tilde{H}_{\Z}^{(1)}$. 

(1) This is obvious since $E(\partial M)=\Delta_h \backslash \C^{2}\times\C^{\times}$.

(2) First, a different choice of $\log z$ changes the original image by a left-translation:
\[ \left(\begin{array}{ccc} 1 & x & \pi i (\log z+\tpi n)+\frac{1}{2}xy\\  & 1&  y \\  & & 1 \\ \end{array}\right)
=\left(\begin{array}{ccc} 1 & 0 & (\tpi)^2\frac{n}{2} \\  & 1&  0 \\  & & 1 \\ \end{array}\right) \left(\begin{array}{ccc} 1 & x & \pi i \log z+\frac{1}{2}xy\\  & 1&  y \\  & & 1 \\ \end{array}\right). \]
The action $\nu:(x,y;z)\mapsto (x+\tpi,y;ze^{y})$ corresponds to a left-translation:
\[ \left(\begin{array}{ccc} 1 & x+\tpi & \pi i(\log z+y)+\frac{1}{2}(x+\tpi)y\\  & 1&  y \\  & & 1 \\ \end{array}\right) =\left(\begin{array}{ccc} 1 & \tpi & 0 \\  & 1&  0 \\  & & 1 \\ \end{array}\right) \left(\begin{array}{ccc} 1 & x & \tpi \log z+\frac{1}{2}xy\\  & 1&  y \\  & & 1 \\ \end{array}\right).\]
Similarly, the action $\mu:(x,y;z)\mapsto (x,y+\tpi;ze^{-x})$ corresponds to another left-translation:
\[ \left(\begin{array}{ccc} 1 & x & \pi i (\log z-x)+\frac{1}{2}x(y+\tpi)\\  & 1&  y+\tpi \\  & & 1 \\ \end{array}\right) =\left(\begin{array}{ccc} 1 & 0 & 0 \\  & 1&  2\pi i \\ &  & 1 \\ \end{array}\right) \left(\begin{array}{ccc} 1 & x & \pi i\log z+\frac{1}{2}xy\\  & 1&  y \\  & & 1 \\ \end{array}\right). \]
Hence, the given map induces an isomorphism $\langle\nu,\mu\rangle \backslash \C^2\times \C^{\times} \isom \tilde{H}_{\Z}\backslash H_{\C}$ of (analytic) spaces covering the identity map. 

(3) For $x\in\C$, write $\bar{x}$ for $x\text{ mod }\tpi \Z$. We have to show that for each $(\bar{x},\bar{y})\in (\C^{\times})^{2}$, the fiber of $H_{\Z}\backslash H_{\C} \ra (\C^{\times})^{2}$ over $(\bar{x},\bar{y}^2=\overline{2y})$ is in bijection under $g$ with the fiber of $\tilde{H}_{\Z}\backslash H_{\C}$ over $(\bar{x},\bar{y})$. The former fiber is the union of the $H_{\Z}$-right cosets 
\[  \left(\begin{array}{ccc} 1 & \tpi\, l & (\tpi)^2n \\  & 1 &  \tpi\, m \\  & & 1 \\ \end{array}\right) \left(\begin{array}{ccc} 1 & x & z \\  & 1 &  2y\\  & & 1 \\ \end{array}\right)
= \left(\begin{array}{ccc} 1 & x+\tpi\, l & 2(\frac{z}{2}+\tpi\, ly+(\tpi)^2\frac{n}{2}) \\  & 1 &  2(y+\tpi\, m) \\  & & 1 \\ \end{array}\right) \] 
with $z$ running through $\C$ ($l,m,n\in\Z$).
But, for each $z\in\C$, the inverse image under $g$ of this subset:
\[ \left\{ \left(\begin{array}{ccc} 1 & x+\tpi\, l & \frac{z}{2}+\tpi\, ly+(\tpi)^2\frac{n}{2} \\  & 1 &  y+\tpi\, m \\  & & 1 \\ \end{array}\right) \ | \ l,m,n\in\Z \right\} \] 
is precisely the $\tilde{H}_{\Z}$-right coset of $\left(\begin{array}{ccc} 1 & x & \frac{z}{2} \\  & 1&  y \\  & & 1 \\ \end{array}\right)$.
\end{proof}

\begin{cor}
There exist canonical isomorphisms of principal $\C^{\times}$-bundles over $D(\partial M)$:
\[ q^{\ast} E(\partial M) \cong \langle \nu, \mu\rangle \backslash \C^{2h}\times\C^{\times}\cong \zeta^{\ast}(H_{\Z}^{(h)}\backslash H_{\C}^{(h)}).  \]
\end{cor}

Let us introduce another integral subgroup $\bar{H}_{\Z}^{(n)}$ of $H_{\Z}^{(n)}$ containing  $H_{\Z}^{(n)}$:
\[ \bar{H}_{\Z}^{(n)}=\left\{ \left(\begin{array}{ccc}  \Onen & a & c\\  & \In & b^t \\  & & \Onen^t \\ \end{array}\right) \,\middle\vert\, \begin{array}{c} a,b\in 2\Z(1)^{n}, \\ c\in 4\Z(2)\end{array} \right\},\]
and define a map $d:H_{\C}^{(n)} \ra H_{\C}^{(n)}$ by
\[ \left(\begin{array}{ccc} \Oneh & x & z \\  & \Ih &  y^t \\  & & \Oneh^t \\ \end{array}\right) \mapsto \left(\begin{array}{ccc} \Oneh & 2x & 4z \\  & \Ih &  2y^t \\  & & \Oneh^t \\ \end{array}\right). \]

%%%%%%%%%%%%%%%%%%%%
\begin{prop} \label{prop:KK=MT-CS_sections} 
(1) The two sections $q^{\ast}C_{\KK}$, $\zeta^{\ast}C_{\MoTe}$ of $q^{\ast}E(\partial M) \cong \zeta^{\ast}(H_{\Z}^{(h)}\backslash H_{\C}^{(h)})$ over $\widetilde{X}(M)$ are equal. 

(2) For any (sufficiently small) Dehn surgery space $V$ of $Y(M;\mathcal{T})$ which thus admits a map into $\widetilde{X}(M)$, the two maps $V\ra \bar{H}_{\Z}^{(h)} \backslash H_{\C}^{(h)}$, one induced by $C_{\MoTe}:V\ra H_{\Z}^{(h)} \backslash H_{\C}^{(h)}$ and the original Morishita-Terashima section $Y(M;\mathcal{T}) \ra \bar{H}_{\Z}^{(h)} \backslash H_{\C}^{(h)}$ \cite[Thm.4.1]{MorishitaTerashima07} coincide.
\end{prop}

\begin{proof}
These are obvious from Theorem 4.1 in \cite{MorishitaTerashima07}, taking into account that their $u(z)$ and $v(z)$ are $\mfu(z)$ and $\mfv(z)$ in our notations (\ref{eq:mfu,mfv}) which are equal to our $2u(z)$ and $2v(z)$ (\ref{eq:u,v_mfu,mfv}), and thus the Chern-Simons invariant $CS_X(z)$ used in their definition of $\tilde{\Phi}(z)$ is four times our $cs_X(z)$.
\end{proof}

\[ \xymatrix{ E(\bdM) \ar[d] & \ar@{->>}[l]  \langle \nu, \mu\rangle \backslash \C^{2h}\PLH\C^{\times} \ar[d] \ar@{->>}[r]^{g\circ f} & H_{\Z}^{(h)} \backslash H_{\C}^{(h)} \ar[d] \ar[r]^d & \bar{H}_{\Z}^{(h)} \backslash H_{\C}^{(h)} \ar[d] \\ 
X(\partial M) & \ar[l]_(.5){q} D(\partial M) \ar[r]^{\zeta}_(.5){(x,y)\mapsto (x,y^2)} & \C^{\times}\times\C^{\times} \ar[r]_{(x,y)\mapsto (x^2,y^2)} & \C^{\times}\times\C^{\times} \\
X(M)  \ar[u] & \ar[u] \ar@/^1.5pc/[uu]^(.72){q^{\ast}C_{\KK}} \ar[l] \widetilde{X}(M) \ar[u]  \ar@/_1.5pc/[uu]_(.72){\zeta^{\ast}C_{\MoTe}} & V \ar[l] \ar@{^(->}[r] & Y(M;\mathcal{T}) \ar[u] \ar@/_1.8pc/[uu]  } \]

Here, all square diagrams are pull-back diagrams. 

Before finishing this subsection, we give another description of the Heisenberg bundle $H_{\Z}^{(h)} \backslash H_{\C}^{(h)}$.
Let $\langle \nu', \mu'\rangle$ be the free abelian group of rank $2h$ with generators $\{\nu_1',\mu_1',\cdots, \nu_h',\mu_h'\}$ endowed with the action on $\C^{2h}\times\C^{\times}$ by
\begin{align} \label{eq:CS_linebundle-action2'}
\nu'_j :& (x,\, y\, ;\, z) \mapsto (x+\tpi\,  e_j,\, y\, ;\, ze^{y_j/2}) \\
\mu'_j :&  (x,\, y\, ;\, z) \mapsto (x,\, y +\tpi\,  e_j\, ;\, ze^{-x_j/2}), \nonumber
\end{align}
where $x=(x_1,\cdots,x_h)$, $y=(y_1,\cdots,y_h)$, and $\{e_1,\cdots,e_h\}$ is the standard basis of $\Z^h$.

\begin{prop}
The (multi-valued) map  $\C^{2h}\times\C^{\times} \ra H_{\C}^{(h)}$ defined by
\[ (x,y;z) \mapsto \left(\begin{array}{ccc} \Oneh & x & \tpi \log z+\frac{1}{2}x\cdot y^t \\ & \Ih &  y^t \\  & & \Oneh^t \end{array}\right) \]
induces an isomorphism
\[ P:=\langle \nu', \mu'\rangle \backslash \C^{2h}\times \C^{\times} \ra H_{\Z}^{(h)} \backslash H_{\C}^{(h)} \] 
of principal $\C^{\times}$-bundles over $D(\partial M)=(\C^{\times})^{2h}$. 

The principal $\C^{\times}$-bundle $\langle \nu, \mu\rangle \backslash \C^{2h}\times\C^{\times}$ is $P^{\otimes2}$, tensor square of the line bundle $P$.
\end{prop}

The proof is an obvious modification of the argument of Proposition \ref{prop:P=H_ZbsH_C}. The second statement follows from comparison of the two actions (\ref{eq:CS_linebundle-action}), (\ref{eq:CS_linebundle-action2'}).

%%%%%%%%%%%%%%%%%%%%%%%%%%%%%%%%%%%%%%%%
%%%%%%%%%%%%%%%%%%%%%%%%%%%%%%%%%%%%%%%%
\section{Chern-Simons variation of mixed Hodge structure over \texorpdfstring{$\tilde{X}(M)_0^{\mathrm{sm}}$}{X(M)0sm}}

\subsection{Tate variation of mixed Hodge structure} \label{subsec:TateVMHS}
In the rest of this article, we use the notation $H_{\Z}:=H_{\Z}^{(1)}$, $H_{\C}:=H_{\C}^{(1)}$ for the Heisenberg groups of order $1$ (Subsection \ref{subesc:Heisenberg-L.B.&MT-CS-section}).
For the split mixed $\Z$-Hodge structure $V_0=\Z(0)\oplus \Z(1)\oplus \Z(2)$, the simply connected $\Q$-algebraic group $G$ with Lie algebra $W_{-1}\End(V_0)_{\Q}$ is $H_{\Z}\otimes\Q$.
So, by \cite[Prop.9.1]{Hain94} (cf. Proposition \ref{prop:moduli_of_MTH}), the complex manifold $H_{\Z}\backslash H_{\C}$ can be regarded as a moduli space of mixed Hodge structures endowed with an isomorphism of the associated weight-graded quotient with $\Z(0)\oplus \Z(1)\oplus \Z(2)$. With this interpretation, the projection map $H_{\Z}\backslash H_{\C}\ra \C^{\times}\times\C^{\times}$ sends a mixed Hodge structure $V\in H_{\Z}\backslash H_{\C}$ to
\[ (V/W_{-4}V, W_{-2}V)\in \Ext^1(\Z, \Z(1)) \times \Ext^1(\Z(1), \Z(2))\simeq \C^{\times} \times \C^{\times}. \]

We refer to \cite{HainZucker87a}, \cite{HainZucker87b} for discussions of ``a good unipotent variation of mixed Hodge structure".
Its associated graded quotient variations of pure Hodge structures are all known to be constant. A \textit{Tate variation of mixed Hodge structure}%%
\footnote{or a \textit{mixed Tate variation}, following \cite{BeilinsonDeligne94}} 
is a good unipotent variation of mixed Hodge structure whose associated graded quotient Hodge structures have only even weights and $\Gr_{2p}^W\mbV$ is of Tate Hodge type $(-p,-p)$.

%%%%%%%%%%%%%%%%%%%%
\begin{thm} \cite{HainZucker87a}
If $\mbV\ra X$ is a Tate variation of mixed Hodge structure, then its canonical extension $\mcVb\ra \overline{X}$ (\`a la P. Deligne) is trivial as a holomorphic vector bundle and the extended weight and Hodge filtrations are also constant. Namely, there exists a complex vector space $V$ endowed with filtrations $F^{\bullet}V$, $W_{\bullet}V$ such that there exists an isomorphism of vector bundles $\mcVb\simeq V\otimes\mcO_{\overline{X}}$ over $\overline{X}$ under which the extended Hodge and weight filtrations are $F^{\bullet}V\times \overline{X}$, $W_{\bullet}V\times\overline{X}$.
\end{thm}

%%%%%%%%%%%%%%%%%%%%
\begin{thm} \cite[Prop. 9.4]{Hain94} \label{thm:Hain94_Prop.9.4}
For a smooth variety $X$, a map 
\[ f:X\ra H_{\Z}\backslash H_{\C} \] 
is the classifying map of a variation of mixed Hodge structure over $X$ with weight-graded quotients canonically isomorphic to $\Z(0)$, $\Z(1)$, $\Z(2)$ if and only if

(i) $f$ is holomorphic;

(ii) the composite $X\stackrel{f}{\ra} H_{\Z}\backslash H_{\C}\ra \C^{\times}\times\C^{\times}$ of $f$ with the canonical projection is algebraic;

(iii) the map $f:X\ra H_{\Z}\backslash H_{\C}$ is a flat section of the bundle $H_{\Z}\backslash H_{\C} \ra \C^{\times}\times\C^{\times}$
\end{thm}

%%%%%%%%%%%%%%%%%%%%
\subsection{Chern-Simons variation of mixed Hodge structure} \label{subsec:CSVMHS}

From now on, in the rest of this article we assume that $h=1$ so that the canonical component $X(M)_0$ of $X(M)_{\Qb}$, the irreducible component containing all lifts of $[\rho_0]$, is a curve. We define the \textit{augmented canonical (curve) component} of $\tilde{X}(M)$ as
\begin{equation} \label{eq:aug_can_comp}
\tilde{X}(M)_0:= X(M)_0\times_{X(\partial M)_{\Qb}}D(\partial M)_{\Qb}. 
\end{equation}
Let $X$ denote the smooth locus $\tilde{X}(M)_0^{\mathrm{sm}}$ of the irreducible curve $\tilde{X}(M)_0$. It is equipped with regular functions $m(x)$, $l(x)$ (Definition \ref{defn:m(z),l(z)}).

We now describe the variation $(\mcV,\nabla,\mbW_{\bullet},\mcF^{\bullet})$ of mixed Hodge structure over $X$ corresponding, via Theorem \ref{thm:Hain94_Prop.9.4}, to the Chern-Simons section $C_{\MoTe}:X\ra H_{\Z}\backslash H_{\C}$ (Theorem \ref{thm:MT09_4.1}), which turns out to be a Tate variation of mixed Hodge structure (cf. \cite[$\S$7]{Hain94}, \cite{HainZucker87b}):
the canonical extension $\mcVb$ of $\mcV$ is the trivial bundle $\mathcal{O}_{\overline{X}}^{\oplus 3}=\mathcal{O}_{\overline{X}}\otimes\C^3$. Let $\{e_0,e_1,e_3\}$ be a basis of $\C^3$. Define \textit{complex} weight filtration and Hodge filtration on $V=\C^3$ by:
\begin{align*}
W_{-4}V=W_{-3}V=\{ e_2\} &\quad \subset\quad W_{-2}V=W_{-1}V=\{ e_1,e_2 \} \quad \subset\quad W_0V=V \\
F^{-p}V=& \langle e_0, \cdots, e_p \rangle \quad (p=0,\cdots, 2),
\end{align*}
and then, after choosing a smooth projective variety $\overline{X}$ containing $X$ with $\overline{X}\backslash X$ being a normal crossing divisor, we put $\overline{\msF}^p:=F^pV\times \overline{X} \subset \mcVb$, $\overline{\mbW}_l:=W_lV\times\overline{X} \subset \mcVb$.

For $x\in X$, we put 
\begin{equation} \label{eq:w_i}
 \left(\begin{array}{c} w_0 \\  w_1 \\  w_2  \end{array}\right) = \left(\begin{array}{ccc} 1 & v(x) & (\tpi)^2cs(x)+u(x)v(x) \\  & \tpi &  (2\pi i)\, 2u(x) \\  & & (\tpi)^2 \\ \end{array}\right) \left(\begin{array}{c} e_0 \\  e_1 \\  e_2  \end{array}\right) 
\end{equation}
where $cs(x)=cs(A_x)$ for an $1$-form $A_x\in \mcA^1(M,\Lsl_{2,\C})$ as defined in Theorem \ref {thm:KK93_3.2}: in particular, it depends on $u(x)$ and $v(x)$. We define $\mbV_x$ to be the $\Q$-subspace of $V$ spanned by $\{w_0,w_1,w_2\}$ endowed with the filtration by $\Q$-subspaces:
\[ W_{-4}\mbV_x=W_{-3}\mbV_x=\{ w_2\} \, \subset\, W_{-2}\mbV_x=W_{-1}\mbV_x=\{ w_1,w_2 \} \, \subset\, W_0\mbV_x=\mbV_x \]
This $\Q$-filtration $W_{\bullet}\mbV_x$ is well-defined, i.e. independent of the choice of the branches $u(x)$, $v(x)$, and is a $\Q$-structure of the $\C$-filtration $W_{\bullet}V$. So, for each $x\in X$, we obtain a mixed $\Q$-Hodge structure $(\mbV_x,\msF^{\bullet}_x,\mbW_{\bullet,x}:=W_{\bullet}\mbV_x)$ whose weight-graded quotients are canonically identified with $\Z(0)$, $\Z(1)$, $\Z(2)$.

Now, we show that this is a Tate variation of mixed $\Q$-Hodge structure.
Let $\Lambda(x)$ denote the $3\times3$-matrix in (\ref{eq:w_i}). Then, the holomorphicity of $f$ implies that the connection $\nabla$ on the trivial bundle $\mcV=\C^3\otimes\mcO_X$ for which the \emph{row} vectors of $\Lambda(x)$ are horizontal sections is holomorphic (so that each $\mbW_l$ is flat): we have 
\[ \nabla s=ds-s\omega, \quad (s\in  (\mcO_X)^{\oplus 3}), \]
for 
\[ \omega:=\Lambda(x)^{-1}\cdot d\Lambda(x). \] 
Secondly, the fact that $f$ defines a flat section implies that $(\mcV,\nabla,\msF^{\bullet})$ satisfies the Griffiths transversality. To see this, we only need to verify that $\nabla e_0\in \msF^{-1}=\mcO_X e_0 +\mcO_X e_1$, as the other cases of the transversality are trivial. This amounts to vanishing of the $(1,3)$-entry of $\omega$. It is pleasing to see that this entry is the connection $1$-form $dc-adb$:
\[ \text{if }\Lambda(x)=\left(\begin{array}{ccc} 1 & a(x) & c(x) \\  & \tpi &  2\pi i\, b(x) \\  & & (\tpi)^2 \\ \end{array}\right), \text{ then }\omega:=\Lambda(x)^{-1}\cdot d\Lambda(x) = \left(\begin{array}{ccc} 0 & da(x) & dc(x)-a(x)db(x) \\  & 0 &  \, db(x) \\  & & 0 \\ \end{array}\right) .\]
For $a(x)=v(x)$, $b(x)=2u(x)$, $c(x)=(\tpi)^2 cs(x)+u(x)v(x)$, the vanishing of $dc(x)-a(x)db(x)$ is \cite[2.7]{KirkKlassen93}:
\[ (\fp)^2 d\, cs(x)=v(x)du(x)-u(x)dv(x). \] 
Note that $\omega$ is a nilpotent matrix in $1$-forms which have simple poles at each point at infinity as $m(x),l(x)\in \mcO_X^{\times}$ and have values $0$ or $\infty$ at any point at infinity in $\overline{X}\backslash X$.
Since the variation of pure Hodge structure $\mathrm{Gr}^{W}_k\mbV$ is constant, the variation of mixed Hodge structure $(\mcV,\nabla,\mbW_{\bullet},\msF^{\bullet})$ over $X$ is unipotent \cite[(1.3)]{HainZucker87b}, and clearly good \cite[(1.9)]{HainZucker87b}. 
Hence it is a Tate variation of mixed Hodge structure.

%%%%%%%%%%%%%%%%%%%%
\subsection{Computation of the limit mixed Hodge structure of CS VMHS at an ideal point} \label{subsec:CS-VMHS}

\begin{defn} \label{defn:ideal_point}
Let $\overline{X}$ be the smooth projective completion of the smooth irreducible affine curve $X$, the smooth locus $\tilde{X}(M)_0^{\mathrm{sm}}$ of the canonical curve component. 

An \emph{ideal point} of $X$ is any point in $\overline{X}\backslash X$ which is a zero or pole of either  of $m(x)$ and $l(x)$ regarded as rational functions on $\overline{X}$.
\end{defn}

For a path $\gamma:[0,1]\ra X(\C)$, the parallel transport of $V=\mcV_{\gamma(0)}$ along $\gamma$ is the multiplication from the right on $V=\C^3=\oplus_i\C e_i$ by
\[ T(\gamma)=1+\int_{\gamma}\omega + \int_{\gamma}\omega\omega + \cdots \quad \in \GL_3(\C) \]
where the integrals are the Chen's iterated integral of $\omega$ over $\gamma$ (cf. \cite[Lem.2.5]{Hain87}, \cite{HainZucker87b}).
As $\omega^2= \left(\begin{array}{ccc} 0 & 0 & 2dv(x) du(x) \\  & 0 &  0 \\  & & 0 \\ \end{array}\right)$ and $\omega^3=0$, we have
\begin{equation} \label{eq:monodromy=I.I.}
T(\gamma)= \left(\begin{array}{ccc}  1 & \int_{\gamma}dv(x) & 2\int_{\gamma}dv(x) du(x) \\  & 1 & 2\int_{\gamma}du(x) \\  & & 1 \\ \end{array}\right)
\end{equation}
(again the integrals in the entries are iterated integrals).

Let $P\in X(\C)$ and $\gamma$ a simple loop in $X(\C)$ which goes around $P$ once counterclockwise, starting and ending at a point close enough to $P$ such that it does not turn around any zero or pole of $l(x)$ and $m(x)$, possibly except for $P$. 
Then, as $u(x)=\log m(x)$, $v(x)=\log l(x)$,  we have
\begin{equation} \label{eq:I.I_{gamma}(uv)}
\int_{\gamma} \frac{d l(x)}{l(x)}\frac{d m(x)}{m(x)} = \frac{(\fp)^2}{2} b_1a_1, 
\end{equation}
for
\begin{equation}  \label{eq:a_1,b_1}
b_1=v_P(l(x)),\ a_1=v_P(m(x))\ \in \Z,
\end{equation} 
where $v_P=\mathrm{val}_P$ is the valuation on $\C(X)$ attached to $P$ (\cite[3.2]{Horozov14a}).

Now, we compute the limit mixed Hodge structure of the Chern-Simons variation of mixed Hodge structure $(\mcV,\nabla,\mbV_{\bullet},\msF^{\bullet})$ at a tangent vector at some ideal point $P$ of $X(\C)$. 
Let $\overline{X}$ be a smooth projective completion of $X$. Choose a local coordinate $z$ on $\Xb$ around $P$ and a branch of $\log z$. For a simple closed curve $\gamma$ around $P$ (as above), let $T:=T(\gamma)$ denote the parallel transport(or local monodromy) of $\omega$ along $\gamma$: by (\ref{eq:monodromy=I.I.}) and (\ref{eq:I.I_{gamma}(uv)}), we have
\[ T=I+\left(\begin{array}{ccc} 0 & \fp\, b_1 & (\tpi)^2 b_1 a_1 \\  & 0 & 4\pi i\, a_1  \\ & & 0 \end{array}\right) \]
and 
\begin{align} \label{eq:log_monodromy_N}
N:=& -\log T = -(T-I)+\frac{(T-1)^2}{2} \\
=& \left(\begin{array}{ccc} 0 & -\fp\, b_1 & 0 \\  & 0 & -4\pi i\, a_1  \\ & & 0 \end{array}\right) \nonumber
\end{align}

Then, the canonical extension $\mcVb$ is endowed with a new connection $\nabla^c=\nabla +\frac{N}{\fp}\frac{dz}{z}$ whose horizontal sections are $\lambda(z)\exp(tN)$, where $\lambda(z)$ are horizontal sections of $\mcV$. The underlying $\Q$-space of the limit mixed Hodge structure on the fiber at the tangent vector $\tv$ is spanned by the row vectors of the limit $\lim_{z\ra 0}  \Lambda(z)\exp(tN)$ ($t:=\frac{\log z}{\fp}$) (cf. \cite[$\S$7]{Hain94}). We have
\begin{align} 
\Lambda(z)\exp(tN)=& \Lambda(z)(1+tN) \quad (N^2=0) \label{eq:LimitMHS} \\
=&  \left(\begin{array}{ccc} 1 & v(z) & (\tpi)^2 cs(z)+u(z)v(z)\\  & \tpi &  4\pi i\, u(z) \\  & & (\tpi)^2 \\ \end{array}\right) \cdot  \left(\begin{array}{ccc} 1 & -b_1\log z & 0 \\  & 1 & -2a_1\log z  \\ & & 1 \end{array}\right) \nonumber \\
=& \left(\begin{array}{ccc} 1 & v(z) -b_1\log z & (\tpi)^2cs(z)+(u(z) -2a_1\log z)v(z)
\\  & \tpi &  4\pi i\, (u(z)-a_1\log z) \\  & & (\tpi)^2 \\ \end{array}\right) \nonumber
\end{align}
(Here, as variable we use $z$ instead of $x$).

The $(1,2)$ and $(2,3)$ entries of this matrix have well-defined limits (as $z\ra 0$) which depend only on $\tv$ (mod $\fp$ or $(\fp)^2$), rather than on the particular choice of a local uniformizer $z$. That is, in terms of the factorizations $l(z)=z^{b_1}\tilde{l}(z)$, $m(z)=z^{a_1}\tilde{m}(z)$, these entries are given by
\begin{equation} \label{eq:a_2,b_2}
b_2:=\lim_{z\ra 0} (v(z)-b_1\log z)=\log\tilde{l}(0),\ a_2:=\lim_{z\ra 0} (u(z)-a_1\log z)=\log\tilde{m}(0);
\end{equation}
there constants are taken modulo $\Z(1)$ 
($a_2$ and $b_2$ are also the logarithms of the leading coefficients of $m(z)$ and $l(z)$, respectively.)

The $(1,3)$-entry is equal to the sum $s_1(z)+s_2(z)$ of two terms, where
\begin{equation} \label{eq:s_2}
\ s_2(z)=(u(z) -a_1\log z)(v(z)-b_1\log z)
\end{equation}
and
\begin{align*} s_1(z) =& (\tpi)^2 cs(z) + b_1\log z (u(z) -a_1\log z) -a_1 v(z) \log z  \\
=& (\tpi)^2 cs(z) + b_1\log z \log \tilde{m}(z) -a_1 (b_1\log z+\log \tilde{l}(z)) \log z  \\
\equiv &  (\tpi)^2 cs(z) + (a_2 b_1-a_1 b_2) \log z -a_1b_1(\log z)^2\, + O(z\log z)
\end{align*}
Hence, if we put
\begin{equation} \label{eq:s(z)}
s(z) := (\tpi)^2 cs(z) + (a_2 b_1-a_1 b_2) \log z -a_1b_1(\log z)^2 
\end{equation}
since $\lim_{z\ra 0}s_2(z)=a_2b_2$, the limt $\lim_{z\ra 0}s(z)$ must exist by the existence of limit mixed Hodge structure. 
In other words, we have 
\begin{equation}  \label{eq:LMH_at_ideal_pt}
\lim_{z\ra 0}  \Lambda(z)\exp(tN) = \left(\begin{array}{ccc} 1 & b_2 & a_2b_2+\lim_{z\ra 0}s(z) \\ & \tpi & 4\pi i\cdot a_2 \\ & & (\tpi)^2 \end{array}\right).
\end{equation}

%%%%%%%%%%%%%%%%%%%%
\subsection{Results from three-dimensional topology: Culler-Shalen theory, A-polynomial}

Recall the constants $a_1, b_1\in \Z$ (\ref{eq:a_1,b_1}), $a_2, b_2\in \C/\Z(1)$ (\ref{eq:a_2,b_2}). 

%%%%%%%%%%%%%%%%%%%%
\begin{thm} \label{thm:a_2,b_3inQ(1)}
If we choose a holomorphic coordinate $z$ around $P$ such that $\tilde{m}(0)=1$, we have $a_2, b_2\in \Q(1)$.
\end{thm}

The proof of this theorem uses results in three-dimensional topology, mainly from \cite{CullerShalen83}, \cite{CullerShalen83} and \cite{CCGLS94}, which give us non-trivial properties of the constants $a_i$, $b_i$ $(i=1,2)$ through their topological interpretations.  For basic definitions in three-manifold theory appearing in the proof which are not covered in the above sources, we refer to \cite[Ch.1]{Kapovich01}.

An \textit{essential} surface in an oriented three-manifold $M$ with (possibly empty) boundary is an incompressible, $\partial$-incompressible, non-$\partial$-parallel, orientable surface which are neither a sphere nor a disk. Suppose that $\partial M$ is a torus. A \emph{slope} is an isotopy class of unoriented simple closed curves on $\bdM$; thus, each slope corresponds to a pair $\pm\gamma$ of primitive elements of $H_1(\bdM,\Z)$ (i.e. an element that cannot be represented as a non-trivial multiple of some other element).
For an incompressible surface $S$ in $M$, if $S\cap \bdM\neq\emptyset$, this intersection $\partial S$ is a family of parallel simple closed curves. If an ordered basis $\{\mu,\nu\}$ of $H_1(\partial M,\Z)$ has been chosen and the slope of $\partial S$ is $\pm(r\mu-s\nu)$, we call the ratio $-\frac{r}{s}\in \Q\cup\{\infty\}$, the \textit{boundary slope} of the surface.

The first topological meaning of the two constants $a_1$, $b_1$ is that $-\frac{b_1}{a_1}$ is the boundary slope of some essential surface inside $M$. More precisely, let $X_0$ be any irreducible curve inside the character variety $X(M)$ 
and $\tilde{X}_0$ the smooth projective birational model of $X_0$. Then for any point $P$ of $\tilde{X}_0$ where the birational map to $X_0$ is not regular,%%
\footnote{Such point is often called \textit{ideal point} in the Culler-Shalen theory. But in this work, we reserve this terminology for our definition \ref{defn:ideal_point} which has more strict meaning and is consistent with the definition used in \cite[2.2]{CCGLS94}.}
Culler and Shalen \cite[2.3.1]{CullerShalen83} associate an essential surface $S$ with non-empty boundary. Their method is by constructing a $\pi_1(M)$-action on the $\SL_2(F)$-tree, where $F$ is the completion at $w$ of a finite extension of the function field $\C(X_0)$ over which a universal(or tautological) representation $\tilde{\rho}$ (\ref{eq:univ_rep'n}) exists, and $w$ is a valuation lying above the valuation $v$ of $\C(X_0)=\C(\tilde{X}_0)$ corresponding to the ideal point $P$. 
Now, let us assume that $X$ is the smooth part of the canonical curve component (\ref{eq:aug_can_comp}). Then it is easy to see that the theory of Culler-Shalen works for $X$.
In this case, if $P$ is an ideal point of $X$ in the sense of Definition \ref{defn:ideal_point}, one has the additional information about the boundary slope of the essential surface $S$ obtained by the Culler-Shalen method, namely that its boundary slope is $-\frac{v(l)}{v(m)}$, where $m$, $l\in F^{\times}$ are the eigenvalues of $\mu$ and $\nu$ (\cite[3.1]{CCGLS94}). Indeed, it follows from the construction that if the image of $\partial S$ in $H_1(\partial M,\Z)$ is $\pm k \gamma$ in for a primitive $\gamma\in H_1(\partial M,\Z)\cong\pi_1(\partial M)$, the eigenvalues of $\tilde{\rho}(\gamma)$ are finite valued at $P$ (\cite[2.2.1]{CullerShalen83}). But, since one of $v(m)$, $v(l)$ is non-zero ($v$ is the valuation at an ideal point), there exists a \emph{unique}, up to sign, primitive $\gamma=a\mu+b\nu \in\pi_1(\partial M)$ such that the eigenvalues $(m^al^b)^{\pm1}$ of $\tilde{\rho}(\gamma)$ are finite valued at $P$, i.e. such that $v(m^al^b)=0$. As we always have $v(m^{-v(l)}l^{v(m)})=0$, by uniqueness and primitiveness, the pair $(-v(l),v(m))$ of integers must be proportional to $(a,b)$, i.e. we have
\begin{equation} \label{eq:boundary_slope}
-v(l)\mu+v(m)\nu=\pm d\gamma.
\end{equation}
where $d:=\mathrm{gcd}(v(l),v(m))$.

Another important property of $S$ is that the image $\mathrm{Im}(\pi_1(S)\ra \pi_1(M))$ lies in an ``edge group'', i.e. the stabilizer group of a directed edge in the $\SL_2(F)$-tree (\cite[2.3.1]{CullerShalen83}).
We recall the notion of ``tree-action eigenvalue''%
\footnote{This is our terminology coined to distinguish it from another use of the terminology of eigenvalue, i.e. that of $\tilde{\rho}$.} 
of an element of an edge group. Let $v:F^{\times}\ra\Z$, $\mathcal{O}$, $\pi$, $k:=\mathcal{O}/(\pi)$ be respectively the discrete valuation, the valuation ring, a uniformizer, and the residue field. A point of the $\SL_2(F)$-tree $\mathbb{T}$ is a similarity class of lattices in $F^2$ and a (directed) edge is defined by lattices $L,L'$ of $F^2$ with $\pi L\subsetneq L' \subsetneq L$, in which case $[L]$, $[L']$ are respectively the initial and terminal vertices of the directed edge $e$. When an element $A\in\SL_2(F)$ fixes a directed edge $e$ of $\mathbb{T}$ corresponding to lattices $L,L'$ of $F^2$ with $\pi L\subsetneq L' \subsetneq L$, $A$ leaves the line $L'/\pi L\subset L/\pi L (\simeq k^2)$ stable, so $A\ (\text{mod}\pi)$ acts on $L'/\pi L(\simeq k)$ by some $\lambda\in k^{\times}$. We call $\lambda$ the \textit{tree-action eigenvalue} of $A$ attached to the fixed directed edge $e$. 
It follows from the definition of the $\SL_2(F)$-tree action and the identity $\tilde{\rho}(\gamma^d)=\left(\begin{array}{cc} l^{v(m)}m^{-v(l)} & \ast \\ 0 & l^{-v(m)}m^{v(l)} \end{array}\right)^{\pm 1}$ (\ref{eq:boundary_slope}) that the tree-action eigenvalue $\lambda$ of $\gamma\in\pi_1(\partial M)$ satisfies $\lambda^d=(\frac{l^{v(m)}}{m^{v(l)}}|_{z=0})^{\pm1}$. Then, up to taking inverse, 
\[ \lambda^{d}=\frac{l^{v(m)}}{m^{v(l)}}|_{z=0}=\frac{l(z)^{v_z(m(z))}}{m(z)^{v_z(l(z))}}|_{z=0}=\exp(a_2 b_1-a_1 b_2). \]

Now, we use the non-trivial fact \cite[$\S$5.7, Corollary]{CCGLS94} that the tree-action eigenvalue $\lambda$ is a root of unity whose order divides $n(S)$, the greatest common divisor of the numbers of the boundary components of the various connected components of $S$.%% 
\footnote{We also remark that by definition, the exponential $\exp(a_2 b_1-a_1 b_2)$ equals the tame symbol at $P$ of the element of $\{l,m\}\in K_2(\C(M))$, which is also equal to the local monodromy around $P$ of the pull-back line bundle $(l,m)^{\ast}\mcL$ of the Heisenberg line bundle, cf. \cite[6.3]{Hain94}, \cite[$\S$4]{Ramakrishnan89}.} 

Therefore, we deduce that when we choose a holomorphic coordinate $z$ around $P$ such that $\tilde{m}(0)=1$ (i.e. $a_2\in (\fp)\Z$ (\ref{eq:a_2,b_2})), we have
\[ b_2\equiv (a_1 b_2-a_2 b_1)/a_1\equiv 0\text{ mod }\Q(1) \]
which completes the proof of Theorem \ref{thm:a_2,b_3inQ(1)}.

When $a_1=1$, such choice of the parameter $z$ amounts to a choice of a tangent vector $\frac{\partial }{\partial z}$; in general, there are several (i.e. as many as $|a_1|$) choices for a tangent vector $\frac{\partial }{\partial z}$ giving $a_2\in (\fp)\Z$. We also note in passing that when $a_1=1$, $a_2\in\Z(1)$ (the second can be always achieved by a suitable choice of $z$),
one has $b_2\equiv  a_1 b_2-a_2 b_1 \equiv 0\text{ mod }\Z(1)$ if $n(S)=1$, e.g. the essential surface $S$ is connected and has a single boundary component.

To proceed, we now propose a conjecture.

%%%%%%%%%%%%%%%%%%%%
\begin{conj} \label{conj:CS-VMHS_at_ideal_pt}
There exists an ideal point $P$ on $X=\tilde{X}(M)_0^{\mathrm{sm}}$ such that for some local uniformizer $z$ with $\tilde{m}(0)=1$, the limit of the $(1,3)$-entry in (\ref{eq:LimitMHS}) belongs to $\Q\cdot (\fp )^2$:
\[ \lim_{z\ra 0}\, (\tpi)^2 cs(z)+(u(z) -2a_1\log z)v(z)\, \in\, \Q\cdot (\fp)^2, \]
equivalently (\ref{eq:s(z)}) 
\[ \lim_{z\ra 0}s(z)=\lim_{z\ra 0}\, (\tpi)^2 cs(z)+ (a_2 b_1-a_1 b_2) \log z -a_1b_1(\log z)^2\, \in\, \Q\cdot (\fp)^2. \] 
In other words, there exists an ideal point such that for a suitable choice of parameter $z$, the limit $\Q$-mixed Hodge structure of the CS-VMHS at the tangent vector $\frac{\partial }{\partial z}$ is split: $\Q\oplus\Q(1)\oplus\Q(2)$.
\end{conj}

Compare this conjecture with the corresponding statement for the polylogarithm variation of mixed Hodge structure \cite[Thm.7.2]{Hain94}.

In Appendix A, we confirm this conjecture for the ``figure-eight knot'' complement (labelled by $4_1$), in which case it turns out that all four ideal points $P$ satisfy the conjecture.%%
\footnote{We do not know and have any expectation about whether this will be always the case. In the polylogarithm case, this is not true (the limit MHS at the tangent vecotr $-\partial/\partial z$ at $1$ is not split, \cite[Thm.7.2]{Hain94}). In that respect, we remark that the figure-eight knot is rather special, i.e. is an arithmetic hyperbolic manifold.}
We remark that we also verified the conjecture for the knot complement $5_1$ (for at least one ideal point), and believe it to be always true, in view of some evident patterns appearing in these computations.

For the discussion in the next subsection of the global monodromy of the Chern-Simons VMHS, we say that a slope (class in $H_1(\partial M,\Z)$ of an unoriented simple closed curve) is \textit{strongly detected} if it is the boundary slope of an essential surface which is obtained by the method of Culler-Shalen \cite{CullerShalen83} (using the action of $\pi_1(M)$ on the $\SL_2$-tree constructed from an ideal point of the character variety). 
When the ideal point lies on an irreducible component $X'$ of $X(M)_{\Qb}$ (or $\tilde{X}(M)_{\Qb}$), we will say that the essential surface is \textit{strongly detected by $X'$}.

Culler and Shalen \cite[Thm.1]{CullerShalen84} showed that for a large class of three-manifolds  including all knot-complements in $S^3$, $M$ contains a separating connected essential surface with non-empty boundary; being separated, this is not a Seifert surface and thus its slope is not the longitude. In fact, it follows from their method of proof of \textit{loc. cit.} that this essential surface is strongly detected by the \emph{canonical component} $X(M)_0$ of $X(M)_{\Qb}$. We prove a stronger fact:%%
\footnote{Often, in literatures people assert a weaker statement where one does not require the slopes to be detected by the geometric component $X_0(M)$ (not just $X(M)$), referring to \cite[Thm.3]{CullerShalen84}. We believe that its proof has a gap.}
%%

%%%%%%%%%%%%%%%%%%%%
\begin{prop} \label{prop:two_sd_bounary-slopes}
Suppose that $M$ is a (complete) hyperbolic three-manifold (of finite volume) with a single cusp. 
Then, there exist at least two distinct boundary slopes \emph{strongly detected by $X=\tilde{X}(M)_0^{\mathrm{sm}}$}.
\end{prop}

Our proof uses the A-polynomial theory developed in \cite{CCGLS94}, which we review here. 
The $A$-polynomial $A_M$ of a compact three-manifold $M$ with single torus boundary is a polynomial with $\Q$-coefficients in two variables $m$, $l$, the eigenvalue functions of a chosen ordered basis $\mu$, $\nu$ of $H_1(\bdM,\Z)$, constructed as follows: 
Let $r:X(M)\ra X(\bdM)$ be the natural map between the $\SL_2$-character varieties (over $\Q$) induced by the restriction $\pi_1(M)\ra \pi_1(\bdM)$, and recall the subvariety of $D(\bdM)$ consisting of diagonal representations; let $t_{\bdM}:D(\bdM)\ra X(\bdM)$ denote the embedding. Let $X'(M)$ be the union of the irreducible components $Z'$ of $X(M)_{\Qb}$ such that the Zariski closure $\overline{r(Z')}$ of the image $r(Z')$ is $1$-dimensional:
\[\xymatrix@R=1em{  \tilde{X}(M)_{\Qb} \ar[r]^{t_M} \ar[d] & X(M)_{\Qb} \ar[d]^{r} \ar@{}[r]|-*[@]{\supset}  & Z'  \ar@{}[r]|-*[@]{\subset} \ar[d] & X'(M) \\
D(\bdM)_{\Qb} \ar[r]_{t_{\bdM}} & X(\partial M)_{\Qb} \ar@{}[r]|-*[@]{\supset} & r(Z') \ar@{_(->}[ld]  & \\ Z \ar@{}[u]|-*[@]{\subset} \ar[r] & \overline{r(Z')}  \ar@{}[u]|-*[@]{\subset} & & } \]
Then, we define a $\Qb$-subvariety (reduced closed subscheme) of $D(\bdM)_{\Qb}$ by
\[ D_M:=\bigcup_{Z'\subset X'(M)} Z \quad \subset\ D(\bdM)_{\Qb} \]
where the union is over the irreducible components $Z'$ of $X'(M)$, and for such $Z'$, $Z:=t_{\bdM}^{-1}(\overline{r(Z')})$, the inverse image of the Zariski closure of $r(Z')$. The closure of $D_M$ in $\A^2$ is a plane algebraic curve, thus is defined by a single polynomial, which is by definition the $A$-polynomial $A_M=A_{M,\mathcal{B}}$ of $M$. Up to a constant multiple, $A_M$ is defined over $\Z$ (ibid. Prop.2.3). We consider $A_M\in \Q[l,m]$ as being defined up to multiplication by an element in $l^{\Z}m^{\Z}$. 

If $Y$ is an irreducible component of $D_M$, we write $\bar{Y}$
for the completion of $Y$ in $\mathbb{P}^2_{\Qb}$ and $\tilde{Y}$ for the smooth projective model of $\bar{Y}$.
In \textit{ibid.}, an ideal point is defined to be any point in $\tilde{Y}$ for some irreducible component of $Y$ of $D_M$ which corresponds to $\bar{Y}\backslash Y$. As the ratonal map $\tilde{Y}\ra Y\subset \A^2$ is $(l,m)$, an ideal point of $Y$ is a point on $\tilde{Y}$ where one of the rational functions $l(z)$, $m(z)$ on $\tilde{Y}$ has zero or pole. For any irreducible component $Y'$ of $X'(M)$, $Y:=t_{\bdM}^{-1}(\overline{r(Y')})$ is also the Zariski closure of the image in $D(\bdM)_{\Qb}$ of $t_M^{-1}(Y')\subset \tilde{X}(M)_{\Qb}$, hence for each ideal point $P$ on $Y$ in this sense of \textit{ibid.}, there exists an ideal point of $t_M^{-1}(Y')$ in our sense of Definition \ref{defn:ideal_point} whose associated valuation extends that attached to $P$. It follows that the essential surfaces associated with them by Culler-Shalen theory are equal. 

\begin{proof} (of Proposition \ref{prop:two_sd_bounary-slopes}) 
For $\gamma\in \pi_1(M)$, let $I_{\gamma}:X(M)\ra \A^1$ be the regular function defined by the character at $\gamma$: $\chi\mapsto \chi(\gamma)$. If $Y'$ is an irreducible component of $X(M)_{\Qb}$ which contains the character of a discrete faithful representation $\rho_0$, then $Y'$ is 1-dimensional and for any non-trivial $\gamma\in\pi_1(\bdM)$, the function $I_{\gamma}$ is non-constant on $Y'$. Indeed, according to \cite[Prop.2:2nd assertion]{CullerShalen84}, if $V\subset X(M)_0$ is the subvariety defined by $I_{\gamma}^2=4$, any irreducible component of $V$ passing through $\chi(\rho_0)$ is a single point. 

We first show that there exists at least one boundary slope strongly detected by $X:=\tilde{X}(M)_0^{\mathrm{sm}}$. If $Y_0$ is the smooth projective birational model of the canonical component $X(M)_0$, the rational function $I_{\mu}$ on $Y_0$ must have a pole, say $P_0$ (then, the birational map from $Y_0$ to $X(M)_0$ is not regular at $P_0$, either). Hence, we can find an ideal point (in the sense of Definition \ref{defn:ideal_point}) for the augmented canonical component $\tilde{X}(M)_0= X(M)_0\times_{X(\partial M)_{\Qb}}D(\partial M)_{\Qb}$, and by Culler-Shalen method \cite{CullerShalen83}, there exists an essential surface with non-empty boundary. 

Let $Y_0:=t_{\bdM}^{-1}(\overline{r(Y_0')})$ for the canonical component $Y_0'=X(M)_0$ of $X'(M)$ and $A_M^0\in \Qb[l,m]$ the defining polynomial of the closure of $Y_0$ in $\A^2$. According to \cite[Thm. 3.4]{CCGLS94}, the set of the slopes of the sides of the Newton polytope of $A_M$ equals the set of the boundary slopes of the essential surfaces strongly detected by some ideal point of $A_M$. In fact, the proof of this theorem works just fine for $A_M^0$, since Proposition 3.3 of \textit{ibid.} is true for any polynomial in $\Qb[l,m]$.
Hence, we have to show that the Newton polygon of $A_M^0$ has at least two distinct sides. If the Newton slope is a line segment, then there exist $a,b\in\Z$ and $p(t)\in \Q[t]$ such that $A_M^0(l,m)=p(l^am^b)$ up to $l^{\Z}m^{\Z}$. It follows that the canonical component $Y_0'$ is the curve defined by the equation $I_{\gamma}=\alpha$, where $\gamma=b\mu+a\nu$ and $\alpha\in\Qb$ is a zero of $p(t)$. But this contradicts the already mentioned fact that 
the trace function $I_{\gamma}$ is non-constant on $Y_0'$ for any non-trivial peripheral element $\gamma$.
\end{proof}

%%%%%%%%%%%%%%%%%%%%
\subsection{Local and global monodromy of CS VMHS}

Let $\sigma$ be a simple loop around a point $P$ on $X(\C)$ with a base point $Q$. Suppose $\sigma=\gamma\sigma_{\epsilon}\gamma^{-1}$, where $\sigma_{\epsilon}$ is a small loop around $P$ with a base point $R$ as used in the computation of local monodromy (\ref{eq:I.I_{gamma}(uv)}), and $\gamma$ is a path from $R$ to $Q$. 
Then, we have the following formula for the iterated integral $\int_{\gamma}dv(x) du(x)$, \cite[formula before Lem.1.5]{Horozov14b} (cf. \cite[2.5]{MorishitaTerashima08}):
\begin{align} 
\int_{\sigma} \frac{d l(x)}{l(x)}\frac{d m(x)}{m(x)} 
=& 2\pi i  \left[ t(P) -  u(P;Q) +\pi i b_1a_1 \right] \label{eq:II_path} 
\end{align} 
where 
\[ t(P):=a_1b_2-a_2b_1,\quad u(P;Q):=\log(\frac{l^{a_1}}{m^{b_1}}(Q)) \] 
($a_2=a_2(P)=\log \tilde{m}(P)$, $b_2=b_2(P)=\log \tilde{l}(P)$, (\ref{eq:a_2,b_2})). 
Thus, we see that

%%%%%%%%%%%%%%%%%%%%
\begin{lem} \label{lem:local_monodromy}
For the loop $\sigma$ as above, around $P$ and based at $Q$ and $n\in\Z$, the monodromy (\ref{eq:monodromy=I.I.}) around $\sigma^n$ equals
\[ T(\sigma^n)= \left(\begin{array}{ccc}  1 & n\cdot \fp\, b_1 & v(P,Q;n) \\  & 1 & n\cdot 4 \pi i\, a_1 \\  & & 1 \\ \end{array}\right),\]
where $v(P,Q;n):=n\cdot 4\pi i[t(P) -  u(P;Q)] +n^2 (\fp)^2 a_1b_1$.

(2) For two points $P$, $P'$, let $\sigma$ (resp. $\sigma'$) be the loop around $P$ (resp. $P'$) and based at $Q$ defined as above. Then, we have
\[ T(\sigma^{b_1'}\sigma'^{-b_1} \sigma^{-b_1'}\sigma'^{b_1})
=\left(\begin{array}{ccc}  1 & 0 & (\fp)^2 2b_1b_1'(a_1b_1'-a_1'b_1)\\  & 1 & 0 \\  & & 1 \\ \end{array}\right). \]
\end{lem}

\begin{proof}
The first statement is immediate from (\ref{eq:monodromy=I.I.}), (\ref{eq:II_path}). 
From (1), for $k\in\Z$, we have
\begin{align*} &\, T(\sigma^{k}\sigma'^{l}) \\
= & \left(\begin{array}{ccc}  1 & k \cdot \fp\, b_1 &  v(P,Q;k) \\  & 1 & k \cdot 4 \pi i\, a_1 \\  & & 1 \\ \end{array}\right) \cdot 
\left(\begin{array}{ccc}  1 & l\cdot \fp\, b_1' & v(P',Q;l) \\  & 1 & l \cdot 4 \pi i\, a_1' \\  & & 1 \\ \end{array}\right) \\
= & \left(\begin{array}{ccc}  1 & 0 & \ast \\  & 1 & 4\pi i \cdot (k a_1+l a_1') \\  & & 1 \\ \end{array}\right),
\end{align*}
where $\ast=4\pi i\cdot  \left[ k( t(P) -  u(P;Q)) +l (t(P') -  u(P';Q))\right]+ (\fp)^2 \cdot (k^2 a_1b_1 +l^2 a_1'b_1'+2kla_1'b_1)$. This implies (2).
\end{proof}

For the next proposition, we use the notion of ``base point at infinity' or ``tangential base point'' for fundamental groups of smooth complex algebraic curves and monodromy of their representations, cf. \cite[$\S$15]{Deligne89}.

%%%%%%%%%%%%%%%%%%%%
\begin{prop} \label{prop:global_monodromy}
For any single cusped hyperbolic three-manifold $M$ for which Conjecture \ref{conj:CS-VMHS_at_ideal_pt} holds, there exist an ideal point $P$ and a tangential base point $\frac{\partial }{\partial z}$ such that the monodromy at $\frac{\partial }{\partial z}$ of the Chern-Simons variation of mixed Hodge structure (Subsection \ref{subsec:CSVMHS}) contains (unipotent) elements $T_1$, $T_2$ of the following form: 
\[ T_1=\left(\begin{array}{ccc}  1 & \fp\cdot b & (\fp)^2 \cdot c \\  & 1 & \ast \\  & & 1 \\ \end{array}\right),\quad T_2=\left(\begin{array}{ccc}  1 & 0 & (\fp)^2 \cdot c' \\  & 1 & \ast \\  & & 1 \\ \end{array}\right), \]
where $b,c,c'\in\Z$ and $bc'\neq0$.
\end{prop}

\begin{proof}
Choose an ideal point $P$ on $X$ satisfying the assumption of Conjecture \ref{conj:CS-VMHS_at_ideal_pt}. We may further choose a rational function $z$ on $X$ which becomes a local uniformizer at $P$ whose tangent vector $\frac{d }{d z}$ satisfies the assumption of Conjecture \ref{conj:CS-VMHS_at_ideal_pt}. 
Let $\sigma_P$ be a small simple closed loop around $P$ which leaves and returns to $P$ in the direction $\frac{\partial }{\partial z}$: $\sigma_P=\gamma \sigma_{\epsilon}\gamma^{-1}$ as introduced above with $R$ converging to $Q=P$.
In the notation from Subsection \ref{subsec:CS-VMHS}, we have $a_1(P)b_1(P)\neq0$ (by the very definition of ideal point), hence $T(\sigma_P)$ is of the form $T_1$ by Lemma \ref{lem:local_monodromy} (take the limit as $Q\ra P$).

By Proposition \ref{prop:two_sd_bounary-slopes}, there exists another ideal point $P'$ with different slope $-b_1'\mu+a_1'\nu$ (with the obvious notation). Suppose first that $b_1'=0$; then, $a_1'\neq 0$ and
\[ T(\sigma_{P'})=\left(\begin{array}{ccc} 1 & 0 & 4\pi i [t(P')-u(P';P)] \\  & 1 & 4\pi i a_1' \\  & & 1 \\ \end{array}\right). \]
If $t(P')-u(P';P)\neq 0$, we are done. Otherwise, we have
\begin{align*} 
T(\sigma_P\sigma_{P'}^k\sigma_P^{-1})=& \left(\begin{array}{ccc} 1 & \fp\cdot b_1 & (\fp)^2 \cdot 2(a_1+2ka_1')b_1 \\  & 1 & \ast \\  & & 1 \\ \end{array}\right) \cdot
\left(\begin{array}{ccc}  1 & -\fp\cdot b_1 & (\fp)^2 \cdot a_1b_1 \\  & 1 & -4\pi i\cdot a_1 \\  & & 1 \\ \end{array}\right) \\
=& \left(\begin{array}{ccc}  1 & 0 & (\fp)^2 \cdot 2ka_1'b_1 \\  & 1 & \ast \\  & & 1 \\ \end{array}\right) 
\end{align*}
which proves the claim.
If $b_1'\neq 0$, the claim follows immediately from Lemma \ref{lem:local_monodromy}, (2).
\end{proof}

For a survey of the theory of mixed Hodge structures on the pro-unipotent fundamental groups of complex algebraic varieties, we refer to \cite{Hain87}, \cite{Hain94} and the references therein.

%%%%%%%%%%%%%%%%%%%%
\begin{thm} \label{thm:CS-VMHS_is_motivic}
Keep previous notations and assumptions, including Conjecture \ref{conj:CS-VMHS_at_ideal_pt}.

The Chern-Simons variation of mixed Hodge structure is a quotient of the variation of mixed Hodge structure on the completed (w.r.t. the augmentation ideal) path torsor at the tangential base point $\orav:=\frac{\partial }{\partial z}$ 
\[ \{ \Q[P_{\orav,x}X]^{\wedge} \}_{x\in X} \ra X. \]
\end{thm}

Compare this theorem with the corresponding statement for the polylogarithm variation of mixed Hodge structure \cite[Thm.11.3]{Hain94} whose strategy of proof we follow.

\begin{proof}
We have to show that the fiber over $\orav$ of the (completed) path torsor is, as a MHS, a quotient of $\Q\pi_1(X,\orav)^{\wedge}$.
By Conjecture \ref{conj:CS-VMHS_at_ideal_pt}, we have $\mbV_{\orav}=\oplus_{i=0}^2 \Q(i) e_i$, which thus contains a copy of $\Q(0)$ (spanned by the vector $e_0$). 
Since the parallel transport $\mbV_{\orav}\otimes \Q[P_{\orav,z}X]^{\wedge} \ra \mbV_x$ is a morphism of mixed Hodge structure \cite[11.2]{Hain94},
we need to show that the first rows of the $\Q$-vector space of upper-triangular matrices which is generated by the monodromy at $\orav$ span $\Q\oplus\Q(1)\oplus\Q(2)$.
This follows from Proposition \ref{prop:global_monodromy}.
\end{proof}

%%%%%%%%%%%%%%%%%%%%%%%%%%%%%%%%%%%%%%%%
%%%%%%%%%%%%%%%%%%%%%%%%%%%%%%%%%%%%%%%%
\begin{appendix}

\section{Confirmation of Conj. \ref{conj:CS-VMHS_at_ideal_pt} for knot complements $4_1$}
We confirm Conjecture \ref{conj:CS-VMHS_at_ideal_pt} for the (four) ideal points of $4_1$. 
Along the way, we also verify that the limt $\lim_{z\ra 0}s(z)$ exists.

%%%%%%%%%%%%%%%%%%%% \subsection{Example: $\mathbf{4_1}$}
Here, in the case of the figure-eight knot complement $4_1$, we confirm Conjecture \ref{conj:CS-VMHS_at_ideal_pt}, by direct computation based on an explicit description of the Neumann-Zagier potential on the deformation curve as provided in \cite{Hikami07}.

The equation for the deformation curve of the complete hyperbolic structure obtained by Hikami (using the so-called saddle point method) is
\begin{equation} \label{eq:4_1;gluing-eq}
  1=y^2x\cdot x^{-1} (1-x^{-1})(1-y^2x) \quad (ibid.\, (4.5)) ; 
\end{equation}
here, we used the letter $y$ instead of the original letter $m$, which is the eigenvalue of the holonomy ($\SL_2$-representation) of the meridian, since we reserved $m$ and $l$ for (the derivatives of) the holonomies of the meridian and the longitude (so $m=y^2$).
The eigenvalue of the holonomy of the longitude (which we denote by $\eta$ so that $l=\eta^2$) is given by (ibid. (4.6)):
\begin{equation} \label{eq:4_1;log(l)}
\eta=\frac{1}{y^2x(y^2x-1)}
\end{equation}
The potential function is given by (ibid. (4.4)):
\begin{align} 
V’(x,y)=& \Li_2(x)-\Li_2(1/xy^2)-4\log y \log(xy) \nonumber \\
=& \Li_2(x)+\Li(xy^2)+\frac{\pi^2}{6} + \frac{1}{2}\log^2(-xy^2) -4\log y \log(xy) \label{eq:4_1;NZ_potential1} \\
=& -\Li_2(x^{-1})-\Li_2(1/xy^2)-\frac{\pi^2}{6} -\frac{1}{2}\log^2(-x) -4\log y \log(xy), \label{eq:4_1;NZ_potential2}
\end{align}
where $\Li_2(z)=\sum_{k=1}^{\infty}\frac{z^2}{k^2}=-\int_0^z\frac{\log(1-t)}{t}dt$ is the dilogarithm function (we take the principal branch of it with the branch cut being the real axis $[1,\infty)$; this is defined by the principal branch of the logarithm whose branch cut is the negative real axis). Here, for the (second) equality we used the transformation properties of dilogarithm function
\begin{align} \label{eq:Li_2-trans.law}
\Li_2(z^{-1}) =& - \Li_2(z) -\frac{\pi^2}{6} -\frac{1}{2}\log^2(-z)\quad (z\notin [0,\infty)).
\end{align}
Because of the choice of variables \cite[(3.14)]{Hikami07}, the potential function $V(x,y)$ used by Hikami differs from the Neumann-Zagier potential $\Phi$ by \[ \Phi=4V. \]

The usual gluing equation of $4_1$ is expressed as: $zw(1-z)(1-w)=1$, in terms of which we have \cite[(63),(64)]{NeumannZagier85}
\begin{equation}  \label{eq:4_1;(m,l)}
m=w(1-z),\quad l=[z(1-z)]^{-2}
\end{equation}
(we took the inverse of the Neumann-Zagier’s $l$ to match it with (\ref{eq:4_1;log(l)})). The relation to the equation (\ref{eq:4_1;gluing-eq}) is given by:
\begin{equation}  \label{eq:4_1;(z,w),(y,x)} 
z=1-y^2x,\quad w=x^{-1}.
\end{equation}
Namely, the variety defined by  (\ref{eq:4_1;gluing-eq}) (which is the $\SL_2$-character variety) is a double covering of the deformation variety: $zw(1-z)(1-w)=1$ (which can be regarded as a $\mathrm{PSL}_2$-character variety). From now, we work with the latter deformation variety.
There are four ideal points: $(z,w)=(0,\infty)$, $(1,\infty)$, $(\infty,0)$, $(\infty,1)$.

%%%%%%%%%%%%%%%%%%%%
\textsc{Case 1}. The ideal point $P$ is $(z,w)=(1,\infty)$ (which corresponds to $(x,y)=(0,0)$)
We make change of variables $(z,w)=(1-s,-t^{-1})$: so $P$ becomes $(s,t)=(0,0)$, and the new equation is
\begin{equation} \label{eq:4_1:NZ-gluing_eq1}
st^{-1}(1-s)(1+t^{-1})=-1, 
\end{equation}
(i.e. $s[(1-s)(1+t)]=-t^2$), in terms of which we have
\[ l=[s(1-s)]^{-2},\quad m=-st^{-1}. \] 
The relation to the equation (\ref{eq:4_1;gluing-eq}) is $s=y^2x=mx$, $t=-x$ (\ref{eq:4_1;(z,w),(y,x)} ). At $P$, $t$ is a uniformizer and we have $s=s(t)=-t^2f(t)$ with $f(t)\in 1+\C\{\!\!\{t\}\!\!\}$. So, we see that
\begin{align*} 
m=tf(t),\qquad &\quad l=t^{-4}f(t)^{-2}(1-t^2f(t))^{-2}, \\
\boxed{a_1=1,\, a_2=\log(1)=0}, &\quad \boxed{b_1=-4,\, b_2=\log(1)=0}.
\end{align*}
and $(a_2b_1-a_1b_2) \log t-a_1b_1 (\log t)^2= 4\log^2 t$.

Now, since $x=x(t)=-t$, $y^2=st^{-1}=tf(t)$ (so that $x,xy^2\in O(t)$),
\begin{align*} 
& \frac{1}{2}\log^2(-xy^2) -4\log y \log(xy) = \frac{1}{2}\log^2(-xy^2) - \log(y^2) \log(x^2y^2)  \\ 
=& \frac{1}{2}(2\log t+\log f(t))^2 -(\log t+\log(f(t))) (3\log t +\log(f(t))) \\ 
=& -\log^2t -2\log t\log f(t) - \frac{1}{2}\log^2 f(t) \\
=& -\log^2t +r(t)
 \end{align*}
where $r(t)\in t\C\{\!\!\{t\}\!\!\}$ (i.e. $\lim_{t\ra 0}r(t)=0$) since $f(0)=1$, and
we have \begin{align*}  & (\fp)^2 cs(t)+(a_2b_1-a_1b_2) \log t-a_1b_1 (\log t)^2 \\ 
=\, & \Phi(s(t),t)  +4\log^2 t  \\ 
=\, & 4V(x(t),y(t))  +4\log^2t  + O(t) \\
=\, & \frac{2}{3}\pi^2 + O(t) 
\end{align*} which confirms the claim.

%%%%%%%%%%%%%%%%%%%%
\textsc{Case 2}. $P=(z,w)=(0,\infty)$. We change the variables $(z,w)$ in \cite[(63)]{NeumannZagier85}  into $(s,t^{-1})$: $(s,t)=(0,0)$ and the new equation becomes
\begin{equation} \label{eq:4_1:NZ-gluing_eq2}
st^{-1}(1-s)(1-t^{-1})=1, 
\end{equation}
and 
\[ l=[s(1-s)]^{-2},\quad m=(1-s)t^{-1}. \] 
and $s=1-y^2x=1-mx$, $t=w^{-1}=x$, by (\ref{eq:4_1;(z,w),(y,x)}).
Again, the uniformizer is $t$ and we have $s=s(t)=-t^2f(t)$ with $f(t)\in 1+\C\{\!\!\{t\}\!\!\}$. So, 
\begin{align*} 
m=t^{-1}(1+t^2f(t)),\qquad &\quad l=t^{-4}f(t)^{-2}(1+t^2f(t))^{-2}, \\
a_1=-1,\, a_2=\log(1)=0, &\quad b_1=-4,\, b_2=\log(1)=0.
\end{align*}
so that $(a_2b_1-a_1b_2) \log t-a_1b_1 (\log t)^2= -4\log^2 t$.

Since $x=x(t)=t$, $y^2=(1-s)t^{-1}=t^{-1}g(t)$ with $g(t)=1+t^2f(t)$, we have $x(t)\in O(t)$ but $xy^2=g(t)$. From
\begin{align*} 
& -4\log y \log(xy) =  - \log(y^2) \log(x^2y^2)  \\ 
=& -(-\log t+\log g(t)) (\log t +\log g(t)) = \log^2t  - \log^2 g(t),
 \end{align*}
we see that 
\begin{align*}  & (\fp)^2 cs(t)+(a_2b_1-a_1b_2) \log t-a_1b_1 (\log t)^2 \\ 
=\, & \Phi(s(t),t)  -4\log^2 t  \\ 
=\, & 4V(x(t),y(t))  -4\log^2t  \\
=\, & 4\left( \Li_2(x)-\Li_2(1/xy^2)-4\log y \log(xy) \right)  -4\log^2t  \\
=\, &-4\Li_2(g(t)^{-1}) +O(t),
\end{align*} 
confirming the claim.
In particular, we see that the $(1,3)$-entry of (\ref{eq:LimitMHS}) is 
\[ -4\lim_{z\ra 1}\Li_2(z)=-4\zeta(2)=-\frac{2\pi^2}{3}. \]

%%%%%%%%%%%%%%%%%%%%
\textsc{Case 3}. The ideal point $P$ is $(z,w)=(\infty,0)$.
We make change of variables $(z,w)=(-s^{-1},t)$: so $P$ is $(s,t)=(0,0)$, and the new equation is
\begin{equation} \label{eq:4_1:NZ-gluing_eq3}
ts^{-1}(1-t)(1+s^{-1})=1, 
\end{equation}
and 
\[ l=[z(1-z)]^{-2}=s^4(1+s)^{-2},\quad m=(1-z)w=s^{-1}(1+s)t. \] 
The relation to the equation (\ref{eq:4_1;gluing-eq}) is $1+s^{-1}=y^2x$, $t=x^{-1}$, by (\ref{eq:4_1;(z,w),(y,x)}). At $P=(0,0)$, $s$ is a uniformizer and we have $t=t(s)=s^2f(s)$ with $f(s)\in 1+\C\{\!\!\{s\}\!\!\}$. So, we see that
\begin{align*} 
m=s(1+s)f(s),\qquad &\quad l=s^4(1+s)^{-2}, \\
a_1=1,\, a_2=\log(1)=0, &\quad b_1=4,\, b_2=\log(1)=0.
\end{align*}
and $(a_2b_1-a_1b_2) \log s-a_1b_1 (\log s)^2= -4\log^2 s$.

Since $x=x(s)=s^{-2}f(s)^{-1}$, $y^2=(1+s)s^{-1}t=sg(s)$ with $g(t)=(1+s)f(s)$, we have
$x^{-1}, (xy^2)^{-1}\in O(s)$. If $h(s)=g(s)f(s)^{-2}$ ($x^2y^2=s^{-3}h(s)$),
\begin{align*} 
& -\frac{1}{2}\log^2(-x) -4\log y \log(xy) = -\frac{1}{2}\log^2(-x) - \log(y^2) \log(x^2y^2)  \\ 
=& -\frac{1}{2} (2\log s +\log (-f(s)))^2 -(\log s+\log g(s)) (-3\log s +\log h(s)) \\ 
=& \log^2s -\log(f(s)^2h(s)/g(s)^3) \log s -\frac{1}{2}\log^2 (-f(s)) -\log g(s)\log h(s) \\
=& \log^2s -\frac{1}{2}\log^2 (-1) +r(s)
 \end{align*}
where $\lim_{s\ra 0}r(s)=0$. Thus, by (\ref{eq:4_1;NZ_potential2}), we have 
\begin{align*}  
& (\fp)^2 cs(s)+(a_2b_1-a_1b_2) \log t-a_1b_1 (\log s)^2 \\ 
=\, & \Phi(s,t(s))  -4\log^2 s  \\ 
=\, & 4V(x(t),y(t))  -4\log^2s  + O(s) \\
=\, & -\frac{2}{3}\pi^2 -2\log^2(-1) +O(s)
\end{align*} which confirms the claim.
Note that $-\frac{2}{3}\pi^2 -2\log^2(-1)\equiv  -\frac{8}{3}\pi^2 \mod 4 (2\pi i)^2$.

%%%%%%%%%%%%%%%%%%%%
\textsc{Case 4}.  The ideal point $P$ is $(z,w)=(\infty,1)$.
We make change of variables $(z,w)=(-s^{-1},1-t)$: so $P$ is $(s,t)=(0,0)$, and the new equation is
\begin{equation} \label{eq:4_1:NZ-gluing_eq4}
ts^{-1}(1-t)(1+s^{-1})=1, 
\end{equation}
and 
\[ l=[z(1-z)]^{-2}=s^4(1+s)^{-2},\quad m=(1-z)w=s^{-1}(1+s)(1-t). \] 
The relation to the equation (\ref{eq:4_1;gluing-eq}) is $1+s^{-1}=y^2x$, $1-t=x^{-1}$, by (\ref{eq:4_1;(z,w),(y,x)}). At $P=(0,0)$, $s$ is a uniformizer and we have $t=t(s)=s^2f(s)$ with $f(s)\in 1+\C\{\!\!\{s\}\!\!\}$. So, we see that
\begin{align*} 
m=s^{-1}(1+s)(1-s^2f(s)), &\qquad l=s^4(1+s)^{-2}, \\
a_1=-1,\, a_2=\log(1)=0, &\quad b_1=4,\, b_2=\log(1)=0.
\end{align*}
and $(a_2b_1-a_1b_2) \log s-a_1b_1 (\log s)^2= 4\log^2 s$.

Since $x=x(s)=(1-s^2f(s))^{-1}$, $y^2=(1+s)s^{-1}(1-t)=s^{-1}g(s)$ with $g(s)=(1+s)(1-s^2f(s))$, we have
$(xy^2)^{-1}\in O(s)$, but $x(0)=1$. If $h(s)=(1-s^2f(s))^{-2}g(s)=(1+s)(1-s^2f(s))^{-1}$, from
\begin{align*} 
& -4\log y \log(xy) =  - \log(y^2) \log(x^2y^2)  \\ 
=& -(-\log s+\log g(s)) (-\log s +\log h(s)) \\
=& -\log^2s  + \log g(s)h(s) \log s-\log g(s) \log h(s) \\
=& -\log^2s+r(s)
 \end{align*}
where $\lim_{s\ra 0}r(s)=0$, 
we see that 
\begin{align*}  & (\fp)^2 cs(s)+(a_2b_1-a_1b_2) \log s-a_1b_1 (\log s)^2 \\ 
=\, & \Phi(s,t(s))  +4\log^2s  \\ 
=\, & 4V(x(t),y(t))  +4\log^2s  \\
=\, & 4\left( \Li_2(x)-\Li_2(1/xy^2)-4\log y \log(xy) \right)  +4\log^2s  \\
=\, & 4\Li_2(x(s)) +O(s),
\end{align*} 
confirming the claim.
In particular, we see that the $(1,3)$-entry of (\ref{eq:LimitMHS}) is 
\[ 4\lim_{z\ra 1}\Li_2(z)=4\zeta(2)=\frac{2\pi^2}{3}. \]
\end{appendix}

%%%%%%%%%%%%%%%%%%%%%%%%%%%%%%%%%%%%%%%%

\textit{Email:} machhama@gmail.com


\begin{thebibliography}{99}

\bibitem[BBD82]{BBD82}
Beilinson A., Bernstein J., Deligne P., 
\newblock Faisceaux pervers. 
\newblock Analyse et topologie sur les espaces singuliers, Ast\'erisque, vol. 100, SMF, 1982.

\bibitem[BMS87]{BMS87}
Beĭlinson, A.; MacPherson, R.; Schechtman, V.
\newblock Notes on motivic cohomology.
\newblock Duke Math. J.54 (1987), no.2, 679-710.

\bibitem[BGSV90]{BGSV90}
Beĭlinson, A. A.; Goncharov, A. B.; Schechtman, V. V.; Varchenko, A. N.
\newblock Aomoto dilogarithms, mixed Hodge structures and motivic cohomology of pairs of triangles on the plane.
\newblock The Grothendieck Festschrift, Vol. I, 135-172. Progr. Math., 86. Birkhäuser Boston, Inc., Boston, MA, 1990.

\bibitem[BD92]{BeilinsonDeligne92}
Beilinson, A.; Deligne, P. 
\newblock Motivic Polylogarithm and Zagier Conjecture. 
\newblock preprint 1992

\bibitem[BD94]{BeilinsonDeligne94}
Beilinson, A.; Deligne, P. 
\newblock Interpr\'etation motivique de la conjecture de Zagier reliant polylogarithmes et r\'egulateurs. 
\newblock Motives (Seattle, WA, 1991), 97-121, Proc. Sympos. Pure Math., 55, Part 2, Amer. Math. Soc., Providence, RI, 1994.

\bibitem[Bei12]{Beilinson12}
Beilinson, A.
\newblock Remarks on Grothendieck's standard conjectures.
\newblock Regulators, 25–32. Contemp. Math., 571 American Mathematical Society, Providence, RI, 2012.

\bibitem[Ben20]{Benard20} B\'enard, L\'eo. 
\newblock Reidemeister torsion form on character varieties. 
\newblock Algebr. Geom. Topol. 20 (2020), no. 6, 2821-2884.

\bibitem[BP92]{BenedettiPetronio92}
Benedetti, R.; Petronio, C.
\newblock Lectures on hyperbolic geometry.
\newblock Universitext, Springer-Verlag, Berlin, 1992. xiv+330 pp.

\bibitem[Blo00]{Bloch00}
Bloch, S. J.
\newblock Higher regulators, algebraic K-theory, and zeta functions of elliptic curves.
\newblock CRM Monogr. Ser., 11, American Mathematical Society, Providence, RI, 2000. x+97 pp.

\bibitem[Bor53]{Borel53}
Borel, A. 
\newblock Sur la cohomologie des espaces fibr\'es principaux et des espaces homog\`enes de groupes de Lie compacts. 
\newblock Ann. of Math. (2)57(1953), 115-207.

\bibitem[Bor55]{Borel55}
Borel, A. 
\newblock Topology of Lie groups and characteristic classes. 
\newblock Bull. Amer. Math. Soc.61(1955).

\bibitem[Bor77]{Borel77}
Borel, A. 
\newblock Cohomologie de $\SL_n$ et valeurs de fonctions zeta aux points entiers.
\newblock Ann. Scuola Norm. Sup. Pisa Cl. Sci. (4)4(1977), no.4, 613-636.

\bibitem[BZ98]{BoyerZhang98}
Boyer, S.; Zhang, X.
\newblock On Culler-Shalen seminorms and Dehn filling.
\newblock Ann. of Math. (2)148(1998), no.3, 737-801.

\bibitem[Bro82]{Brown82}
Brown, E. H., Jr.
\newblock The cohomology of BSOn and BOn with integer coefficients.
\newblock Proc. Amer. Math. Soc.85(1982), no.2, 283-288.

\bibitem[Bro13]{Brown13}
Brown, Francis C. S.
\newblock Dedekind zeta motives for totally real number fields.
\newblock Invent. Math.194 (2013), no.2, 257–311.

\bibitem[Bur02]{BurgosGil02}
Burgos Gil, José I.
\newblock The regulators of Beilinson and Borel. 
\newblock CRM Monogr. Ser., 15 American Mathematical Society, Providence, RI, 2002. xii+104 pp.

\bibitem[BGF]{BurgosGilFresan}
Burgos Gil, J. I.; Fresan, J.
\newblock Multiple zeta values: from numbers to motives.
\newblock Clay Mathematics Proceedings, to appear, available at {\url{http://javier.fresan.perso.math.cnrs.fr/publications.html}}

\bibitem[CS74]{ChernSimons74}
Chern, S. S.; Simons, J.
\newblock Characteristic forms and geometric invariants.
\newblock Ann. of Math. (2) 99 (1974), 48-69.

\bibitem[CD19]{CisinskiDeglise19}
Cisinski, D.-C.; D\'eglise, F. 
\newblock Triangulated categories of mixed motives. 
\newblock Springer Monographs in Mathematics. Springer, Cham, 2019.

\bibitem[CMJ03]{Cisneros-MolinaJones03}
Cisneros-Molina, J. L.; Jones, J. D. S.
\newblock The Bloch invariant as a characteristic class in $B(SL2(\C),\mathfrak{T})$.
\newblock Homology Homotopy Appl.5(2003), no.1, 325-344.

\bibitem[CCGLS94]{CCGLS94}
Cooper, D.; Culler, M.; Gillet, H.; Long, D. D.; Shalen, P. B. 
\newblock Plane curves associated to character varieties of 3-manifolds.
\newblock Invent. Math. 118 (1994), no. 1, 47-84.

\bibitem[CL96]{CooperLong96}
Cooper, D.; Long, D. D. 
\newblock Remarks on the A-polynomial of a knot. 
\newblock J. Knot Theory Ramifications 5 (1996), no. 5, 609–628.

\bibitem[CS83]{CullerShalen83}
Culler, M.; Shalen, P. B. 
\newblock Varieties of group representations and splittings of 3-manifolds. 
\newblock Ann. of Math. (2) 117 (1983), no. 1, 109-146.

\bibitem[CS84]{CullerShalen84}
Culler, M.; Shalen, P. B. 
\newblock Bounded, separating, incompressible surfaces in knot manifolds. 
\newblock Invent. Math. 75 (1984), no. 3, 537-545.

\bibitem[Del89]{Deligne89}
Deligne, P. 
\newblock Le groupe fondamental de la droite projective moins trois points.
\newblock Galois groups over $\Q$ (Berkeley, CA, 1987), 79-297, Math. Sci. Res. Inst. Publ., 16, Springer, New York, 1989.

\bibitem[DG05]{DeligneGoncharov05}
Deligne, P.; Goncharov, A. B. 
\newblock Groupes fondamentaux motiviques de Tate mixte. 
\newblock Ann. Sci. \'Ecole Norm. Sup. (4) 38 (2005), no. 1, 1-56.

\bibitem[DM82]{DeligneMilne82}
Deligne, P.; Milne, J. S.; Ogus, A.; Shih, K.-Y.
\newblock Hodge cycles, motives, and Shimura varieties.
\newblock Lecture Notes in Math., 900, Springer-Verlag, Berlin-New York, 1982. ii+414 pp.

\bibitem[DW90]{DijkgraafWitten90}
Dijkgraaf, R.; Witten, E.
\newblock Topological gauge theories and group cohomology.
\newblock Comm. Math. Phys.129 (1990), no.2, 393-429.

\bibitem[DN89]{DrezetNarasimhan89}
Drezet, J.-M.; Narasimhan, M. S. 
\newblock Groupe de Picard des vari\'et\'es de modules de fibr\'es semi-stables sur les courbes alg\'ebriques. 
\newblock Invent. Math.97 (1989), no.1, 53–94.

\bibitem[DF23]{DupontFresan23}
Dupont, C.; Fresán, J.
\newblock A construction of the polylogarithm motive
\newblock {\url{https://arxiv.org/abs/2305.00789}}

\bibitem[Dup78]{Dupont78}
Dupont, J. L.
\newblock Curvature and characteristic classes.
\newblock Lecture Notes in Math., Vol. 640 Springer-Verlag, Berlin-New York, 1978. viii+175 pp.

\bibitem[Dup82]{Dupont82}
Dupont, J. L.
\newblock Algebra of polytopes and homology of flag complexes.
\newblock Osaka J. Math.19(1982), no.3, 599-641.

\bibitem[DS82]{DupontSah82}
Dupont, J. L.; Sah, C. H.
\newblock Scissors congruences. II.
\newblock J. Pure Appl. Algebra25(1982), no.2, 159-195.

\bibitem[Dup87]{Dupont87}
Dupont, J. L.
\newblock The dilogarithm as a characteristic class for flat bundles.
\newblock Proceedings of the Northwestern conference on cohomology of groups (Evanston, Ill., 1985)
J. Pure Appl. Algebra44(1987), no.1-3, 137-164.

\bibitem[Dup01]{Dupont01}
Dupont, J. L.
\newblock Scissors congruences, group homology and characteristic classes.
\newblock Nankai Tracts Math., 1 World Scientific Publishing Co., Inc., River Edge, NJ, 2001. viii+168 pp.

\bibitem[EP88]{EpsteinPenner88}
Epstein, D. B. A.; Penner, R. C.
\newblock Euclidean decompositions of noncompact hyperbolic manifolds.
\newblock J. Differential Geom. 27 (1988), no.1, 67-80.

\bibitem[Fra04]{Francaviglia04}
Francaviglia, S.
\newblock Algebraic and geometric solutions of hyperbolicity equations.
\newblock Topology Appl.145(2004), no.1-3, 91-118.

\bibitem[Fre95]{Freed95}
Freed, D.
\newblock Classical Chern-Simons theory. I.
\newblock Adv. Math.113(1995), no.2, 237-303.

\bibitem[Fuj05]{Fujii05} 
Fujii, M. 
\newblock Degeneration of hyperbolic structures on the figure-eight knot complement and points of finite order on an elliptic curve. 
\newblock J. Math. Kyoto Univ. 45 (2005), no. 2, 343-354.

\bibitem[Gon98]{Goncharov98}
Goncharov, A.
\newblock Mixed elliptic motives. 
\newblock Galois representations in arithmetic algebraic geometry (Durham, 1996), 147-221. London Math. Soc. Lecture Note Ser., 254. Cambridge University Press, Cambridge, 1998.

\bibitem[Gon99]{Goncharov99}
Goncharov, A.
\newblock Volumes of hyperbolic manifolds and mixed Tate motives. 
\newblock J. Amer. Math. Soc. 12 (1999), no. 2, 569-618.

\bibitem[Gon05]{Goncharov05}
Goncharov, A.
\newblock Galois symmetries of fundamental groupoids and noncommutative geometry.
\newblock Duke Math. J.128 (2005), no.2, 209-284.

\bibitem[GZ18]{GoncharovZhu18}
Goncharov, A.; Zhu, G.
\newblock The Galois group of the category of mixed Hodge-Tate structures.
\newblock Selecta Math. (N.S.)24(2018), no.1, 303-358.



\bibitem[Hai86]{Hain86}
Hain, R. 
\newblock On a generalization of Hilbert’s 21st problem. 
\newblock Ann. Sci. \'Ecole Norm. Sup. (4) 19 (1986), no. 4, 609-627.

\bibitem[HZ87a]{HainZucker87a}
Hain, R.; Zucker, S. 
\newblock Unipotent variations of mixed Hodge structure. 
\newblock Invent. Math. 88 (1987), no. 1, 83-124.

\bibitem[HZ87b]{HainZucker87b}
Hain, R.; Zucker, S. 
\newblock A guide to unipotent variations of mixed Hodge structure. 
\newblock Hodge theory (Sant Cugat, 1985), Lecture Notes in Math., 1246, 1987.

\bibitem[Hai87]{Hain87}
Hain, R. 
\newblock The geometry of the mixed Hodge structure on the fundamental group. 
\newblock Algebraic geometry, Bowdoin, 1985 (Brunswick, Maine, 1985), 247-282, Proc. Sympos. Pure Math., 46, Part 2, Amer. Math. Soc., Providence, RI, 1987.

\bibitem[Hai94]{Hain94}
Hain, R.
\newblock Classical polylogarithms. 
\newblock Motives (Seattle, WA, 1991), 3-42, Proc. Sympos. Pure Math., 55, Part 2, Amer. Math. Soc., Providence, RI, 1994.

\bibitem[Hai03]{Hain03}
Hain, R.
\newblock Lectures on the Hodge–De Rham Theory of $\pi_1(\mathbb{P}^1\backslash\{0,1,\infty\}$
\newblock Arizona Winter School, 2003.

\bibitem[Han95]{Hanamura95}
Hanamura M. 
\newblock Mixed motives and algebraic cycles I, 
\newblock Math. Res. Lett. 2 (6) (1995) 811–821. 
%See also II, Inv. Math. 158 (1) (2004) 105–179. MR2090362

\bibitem[Har11]{Harada11}
Harada, S. 
\newblock Hasse-Weil zeta function of absolutely irreducible $SL_2$-representations of the figure $8$ knot group. 
\newblock Proc. Amer. Math. Soc. 139 (2011), no. 9, 3115-3125.

\bibitem[Hik07]{Hikami07}
Hikami, K. 
\newblock Generalized volume conjecture and the A-polynomials: the Neumann-Zagier potential function as a classical limit of the partition function. 
\newblock J. Geom. Phys. 57 (2007), no. 9, 1895-1940.

\bibitem[Hor14a]{Horozov14a}
Horozov, I. 
\newblock Non-abelian reciprocity laws on a Riemann surface. 
\newblock Int. Math. Res. Not. IMRN 2011, no. 11, 2469-2495.

\bibitem[Hor14b]{Horozov14b}
Horozov, I. 
\newblock Reciprocity laws on algebraic surfaces via iterated integrals. With an appendix by Horozov and Matt Kerr. 
\newblock J. K-Theory 14 (2014), no. 2, 273–312.

\bibitem[HMS17]{HuberMuller-Stach17}
Huber, A.; Müller-Stach, S.
\newblock Periods and Nori motives. With contributions by Benjamin Friedrich and Jonas von Wangenheim
\newblock Ergeb. Math. Grenzgeb. (3), 65. Springer, Cham, 2017. xxiii+372 pp.

\bibitem[HW98]{HuberWildeshaus98}
Huber, A.; Wildeshaus, J.
\newblock Classical motivic polylogarithm according to Beilinson and Deligne. 
\newblock Doc. Math. 3 (1998), 27-133.

\bibitem[Hub95]{Huber95}
Huber, A. 
\newblock Mixed Motives and their Realization in Derived Categories, 
\newblock Lecture Notes in Math., vol. 1604, Springer-Verlag, 1995.

\bibitem[Hub00]{Huber00}
Huber, A. 
\newblock Mixed Realization of Voevodsky’s motives, 
\newblock J. Algebraic Geom. 9 (2000) 755-799. Corrigendum: Ibid. 13 (1) (2004) 115-207.

\bibitem[Kap01]{Kapovich01}
Kapovich, M.
\newblock Hyperbolic manifolds and discrete groups.
\newblock Progr. Math., 183 Birkhäuser Boston, Inc., Boston, MA, 2001. xxvi+467 pp.

\bibitem[Kir89]{Kirby89}
Kirby, R. C.
\newblock The topology of 4-manifolds.
\newblock Lecture Notes in Math., 1374, Springer-Verlag, Berlin, 1989. vi+108 pp.

\bibitem[KK90]{KirkKlassen90}
Kirk, Paul A.; Klassen, Eric P. 
\newblock Chern-Simons invariants of 3-manifolds and representation spaces of knot groups. 
\newblock Math. Ann. 287 (1990), no. 2, 343-367. 

\bibitem[KK93]{KirkKlassen93}
Kirk, Paul; Klassen, Eric.
\newblock Chern-Simons invariants of 3-manifolds decomposed along tori and the circle bundle over the representation space of T2. 
\newblock Comm. Math. Phys. 153 (1993), no. 3, 521-557.

\bibitem[KZ01]{KontsevichZagier01}
Kontsevich, M.; Zagier, D.
\newblock Periods.
\newblock Mathematics unlimited-2001 and beyond, 771-808. Springer-Verlag, Berlin, 2001.

\bibitem[Lev93]{Levine93}
Levine, M. 
\newblock Tate motives and the vanishing conjectures for algebraic K-theory. 
\newblock Algebraic K-Theory and Algebraic Topology, Lake Louise, 1991, NATO Adv. Sci.
Inst. Ser. C Math. Phys., vol. 407, Kluwer, 1993, pp. 167-188.

\bibitem[Lev98]{Levine98}
Levine, M. 
\newblock Mixed Motives, 
\newblock Math. Surveys and Monographs, vol. 57, AMS, 1998.

\bibitem[Lev10]{Levine10}
Levine, M. 
\newblock Tate motives and the fundamental group.
\newblock Cycles, motives and Shimura varieties. Srinivas, V. (ed.), Proceedings of the international colloquium,
Mumbai, India, January 3–12, 2008. Studies in Mathematics. Tata Institute of Fundamental Research 21, 265-392 (2010).

\bibitem[Lev13]{Levine13}
Levine, M.
\newblock Six lectures on motives.Autour des motifs.
\newblock École d'\'et\'e Franco-Asiatique de G\'eom\'etrie Alg\'ebrique et de Th\'eorie des Nombres. Vol. II, 1–141. Panor. Synth\`eses, 41[Panoramas and Syntheses]. Soci\'et\'e Mathématique de France, Paris, 2013.

\bibitem[Li02]{Li02}
Li, T.
\newblock Immersed essential surfaces in hyperbolic 3-manifolds.
\newblock Comm. Anal. Geom. 10 (2002), no. 2, 275–290.

\bibitem[Mar16]{Marche16}
March\'e, J.
\newblock Character varieties in $\SL_2$ and skein algebras,
\newblock Topology, Geometry and Algebra of low dimensional manifolds, 25-29 May 2016, RIMS, Kyoto.

\bibitem[MPvL11]{MPvL11}
Macasieb, M. L.; Petersen, K. L.; van Luijk, R. M.
\newblock On character varieties of two-bridge knot groups.
\newblock Proc. Lond. Math. Soc. (3)103(2011), no.3, 473-507.

\bibitem[Mey86]{Meyerhoff86}
Meyerhoff, R.
\newblock Density of the Chern-Simons invariant for hyperbolic 3-manifolds.
\newblock Low-Dimensional Topology and Kleinian Groups (Coventry/Durham, England, 1984), London Math. Soc. Lecture Note Ser. 112, Cambridge Univ. Press, Cambridge, 1986, 217-239.

\bibitem[MT91]{MimuraToda91}
Mimura, M.; Toda, H.
\newblock Topology of Lie groups. I, II.
Translated from the 1978 Japanese edition by the authors 
\newblock Transl. Math. Monogr., 91 American Mathematical Society, Providence, RI, 1991. iv+451 pp.

\bibitem[MT07]{MorishitaTerashima07}
Morishita, M.; Terashima, Y. 
\newblock Arithmetic topology after Hida theory. 
\newblock Intelligence of low dimensional topology 2006, 213–222, Ser. Knots Everything, 40, World Sci. Publ., Hackensack, NJ, 2007.

\bibitem[MT08]{MorishitaTerashima08}
Morishita, M.; Terashima, Y. 
\newblock Geometry of polysymbols. 
\newblock Math. Res. Lett. 15 (2008), no. 1, 95–115.

\bibitem[MT09]{MorishitaTerashima09}
Morishita, M.; Terashima, Y. 
\newblock Chern-Simons variation and Deligne cohomology. 
\newblock Spectral analysis in geometry and number theory, 127-134, Contemp. Math., 484, Amer. Math. Soc., Providence, RI, 2009.

\bibitem[Mor01]{Morita01}
Morita, S.
\newblock Geometry of differential forms. Translated from the two-volume Japanese original (1997, 1998) by Teruko Nagase and Katsumi Nomizu. 
\newblock Transl. Math. Monogr., 201 Iwanami Ser. Mod. Math. Amer. Math. Soc., Providence, RI, 2001. xxiv+321 pp.

\bibitem[Nak00]{Nakamoto00}
Nakamoto, K. 
\newblock Representation varieties and character varieties. 
\newblock Publ. Res. Inst. Math. Sci. 36 (2000), no. 2, 159–189.

\bibitem[Nee88]{Neeman88}
Neeman, A. 
\newblock Analytic questions in geometric invariant theory.
\newblock Invariant theory (Denton, TX, 1986), 11–23. Contemp. Math., 88.

\bibitem[NZ85]{NeumannZagier85}
Neumann, W.; Zagier, D.
\newblock Volumes of hyperbolic three-manifolds, 
\newblock Topology 24, (1985), no. 3, 307-332.

\bibitem[Neu92]{Neumann92}
Neumann, W. D.
\newblock Combinatorics of triangulations and the Chern-Simons invariant for hyperbolic 3-manifolds.
\newblock Topology '90 (Columbus, OH, 1990), 243-271. Ohio State Univ. Math. Res. Inst. Publ., 1. Walter de Gruyter \& Co., Berlin, 1992.

\bibitem[Neu98]{Neumann98}
Neumann, W. D.
\newblock Hilbert's 3rd problem and invariants of 3-manifolds. 
\newblock 
The Epstein birthday schrift, 383-411. Geom. Topol. Monogr., 1 Geometry \& Topology Publications, Coventry, 1998.

\bibitem[NY99]{NeumannYang99}
Neumann, W. D.; Yang, J.
\newblock Bloch invariants of hyperbolic 3-manifolds.
\newblock Duke Math. J.96(1999), no.1, 29-59.

\bibitem[Neu04]{Neumann04}
Neumann, W. D.
\newblock Extended Bloch group and the Cheeger-Chern-Simons class.
\newblock Geom. Topol.8(2004), 413-474.

\bibitem[Pro98]{Procesi98}
Procesi, C. 
\newblock Deformations of representations. 
\newblock Methods in ring theory (Levico Terme, 1997), 247-276, Lecture Notes in Pure and Appl. Math., 198, Dekker, New York, 1998.

\bibitem[PS00]{PS00}
J. H. Przytycki, A. Sikora. 
\newblock On skein algebras and $Sl_2(\C)$-character varieties. 
\newblock Topology 39 (2000), no. 1, 115-148.

\bibitem[RSW89]{RSW89}
Ramadas, T. R.; Singer, I. M.; Weitsman, J. 
\newblock Some comments on Chern-Simons gauge theory. 
\newblock Comm. Math. Phys. 126 (1989), no. 2, 409-420.

\bibitem[Ram82]{Ramakrishnan82}
Ramakrishnan, D.
\newblock On the monodromy of higher logarithms.
\newblock Proc. Amer. Math. Soc. 85 (1982), no. 4, 596-599.

\bibitem[Ram89]{Ramakrishnan89}
Ramakrishnan, D.
\newblock Regulators, algebraic cycles, and values of L-functions. 
\newblock Algebraic K-theory and algebraic number theory (Honolulu, HI, 1987), 183-310, Contemp. Math., 83, Amer. Math. Soc., Providence, RI, 1989.

%\bibitem[Rap88]{Rapoport88} Rapoport, M. \newblock Comparison of the regulators of Beilinson and of Borel.\newblock Beilinson’s Conjectures on Special Values of L-Functions, Perspectives in Math., vol.4, Academic Press, 1988, pp. 169-192.

\bibitem[Sah81]{Sah81}
Sah, C. 
\newblock Han Scissors congruences. I. The Gauss-Bonnet map. 
\newblock Math. Scand. 49 (1981).pdf

\bibitem[Sai96]{Saito96}
Saito, K.
\newblock Character variety of representations of a finitely generated group in $\SL_2$. 
\newblock Topology and Teichmüller spaces (Katinkulta, 1995), 253-264, World Sci. Publ., River Edge, NJ, 1996.

\bibitem[Sch73]{Schmid73}
Schmid, W. 
\newblock Variation of Hodge structure: the singularities of the period mapping. 
\newblock Invent. Math. 22 (1973), 211-319.

\bibitem[Sik12]{Sikora12}
Sikora, A. S. 
\newblock Character varieties. 
\newblock Trans. Amer. Math. Soc. 364 (2012), no. 10, 5173-5208.

\bibitem[Sik14]{Sikora14}
Sikora, A. S. 
\newblock Character varieties of abelian groups.
\newblock Math. Z. 277 (2014), no.1-2, 241-256.

\bibitem[Sus90]{Suslin90}
Suslin, A. A.
\newblock K3 of a field, and the Bloch group.
\newblock Translated in Proc. Steklov Inst. Math. 1991, no. 4, 217-239.
Galois theory, rings, algebraic groups and their applications (Russian)
Trudy Mat. Inst. Steklov.183(1990), 180-199, 229.

\bibitem[Ste99]{Steenrod99}
Steenrod, N.
\newblock The topology of fibre bundles. Reprint of the 1957 edition.
\newblock Princeton Landmarks Math. Princeton Paperbacks
Princeton University Press, Princeton, NJ, 1999. viii+229 pp.

\bibitem[Thu97]{Thurston97}
Thurston, William P.
\newblock Three-dimensional geometry and topology. Vol. 1. Edited by Silvio Levy
\newblock Princeton Math. Ser., 35 Princeton University Press, Princeton, NJ, 1997. x+311 pp.

\bibitem[Voe00]{Voevodsky00}
Voevodsky, V. 
\newblock Triangulated categories of motives over a field.
\newblock Cycles, Transfer and Motivic Homology Theories, Ann. of Math. Studies, vol. 143, Princeton University Press, 2000, pp. 188-238.

\bibitem[Wan06]{Wang06}
Wang, Q.
\newblock Moduli spaces and multiple polylogarithm motives.
\newblock Adv. Math.206 (2006), no.2, 329-357.

\bibitem[Wei94]{Weibel94}
Weibel, Charles A.
\newblock An introduction to homological algebra.
\newblock Cambridge Stud. Adv. Math., 38 Cambridge University Press, Cambridge, 1994. xiv450 pp.

\bibitem[Yok02]{Yokota02}
Yokota, Y. 
\newblock On the potential functions for the hyperbolic structures of a knot complement. 
\newblock Invariants of knots and 3-manifolds (Kyoto, 2001), 303–311, Geom. Topol. Monogr., 4, Geom. Topol. Publ., Coventry, 2002.

\bibitem[Yok03]{Yokota03}
Yokota, Y. 
\newblock From the Jones polynomial to the A-polynomial of hyperbolic knots. 
\newblock Proceedings of the Winter Workshop of Topology/Workshop of Topology and Computer (Sendai, 2002/Nara, 2001). Interdiscip. Inform. Sci. 9 (2003), no. 1, 11–21. 
 
\bibitem[Yos85]{Yoshida85}
Yoshida, T.
\newblock The $\eta$-invariant of hyperbolic 3-manifolds.
\newblock Invent. Math. 81 (1985), no.3, 473-514.

\bibitem[Zag07]{Zagier07} 
Zagier, D.
\newblock The dilogarithm function.
\newblock Frontiers in number theory, physics, and geometry. II, 3-65. Springer-Verlag, Berlin, 2007.
 
\end{thebibliography}
\end{document}